\RequirePackage{amsmath, amssymb}

\documentclass[11pt, a4paper]{article}
\usepackage[margin=1in]{geometry}

\newif\ifanonymous
\anonymousfalse

\usepackage{mathrsfs} %
\usepackage{authblk}

\usepackage{longtable}

\usepackage[english]{babel}
\usepackage[utf8]{inputenc}
\usepackage[T1]{fontenc}

\usepackage{amsmath,amssymb,amsfonts,amsthm}

\usepackage{lmodern} %
\usepackage{thmtools}
\usepackage{hyperref}
\usepackage{algorithm}
\usepackage{algorithmic}
\usepackage{todonotes}
\usepackage{booktabs}
\usepackage{cleveref}
\usepackage{stmaryrd}
\usepackage{enumerate}
\usepackage{comment}

\hypersetup{colorlinks=true, citecolor=teal, linkcolor=purple, linktoc=page, urlcolor=purple}

\usepackage{array}
\usepackage{subcaption}
\usepackage{attachfile}
\usepackage{tikz} 
\usepackage{pgfplots} 
\usetikzlibrary{matrix}

\newtheorem{theorem}{Theorem}
\newtheorem{proposition}{Proposition}[section]
\newtheorem{lemma}[proposition]{Lemma}
\newtheorem{corollary}[proposition]{Corollary}
\newtheorem{computationalproblem}[proposition]{Problem}

\theoremstyle{definition}
\newtheorem{definition}[proposition]{Definition}
\newtheorem{remark}[proposition]{Remark}
\newtheorem{example}[proposition]{Example}

\renewcommand\vec{}

\newcommand{\lineref}[1]{\ref{#1}}

\newcommand{\sd}{\sigma}
\newcommand{\Gau}{\mathcal{G}}
\newcommand{\dGau}{\ddot{\Gau}}
\newcommand{\gaussian}{\rho}

\newcommand{\distrdiag}{\distr_{\textrm{diag}}}
\newcommand{\ddistrdiag}{\ddot{\distr}_{\textrm{diag}}}

\newcommand{\dH}{\ddot{H}}
\newcommand{\dHf}{\dH_{\textmd{fin}}}

\newcommand{\ddx}{\ddot{x}}
\newcommand{\ddX}{\ddot{X}}
\newcommand{\dTheta}{\ddot{\Theta}}
\newcommand{\drho}{\ddot{\rho}}
\newcommand{\dtheta}{\ddot{\theta}}
\newcommand{\da}{\ddot{a}}
\newcommand{\ddistr}{\ddot{\distr}}
\newcommand{\dw}{\ddot{w}}

\newcommand{\ddh}{\ddot{h}}

\newcommand{\eps}{\varepsilon}

\newcommand{\bx}{\mathbf{x}}
\newcommand{\by}{\mathbf{y}}
\newcommand{\bs}{\mathbf{s}}
\newcommand{\bt}{\mathbf{t}}
\newcommand{\mB}{\mathbf{B}}
\newcommand{\mI}{\mathbf{I}}
\newcommand{\mA}{\mathbf{A}}
\newcommand{\mH}{\mathbf{H}}

\newcommand{\Q}{\mathbb{Q}}
\newcommand{\Z}{\mathbb{Z}}
\newcommand{\Zhat}{\widehat{\mathcal{O}}}
\newcommand{\R}{\mathbb{R}}
\newcommand{\C}{\mathbb{C}}
\newcommand{\K}{\mathbb{K}}
\newcommand{\X}{\mathbb{X}}

\newcommand{\adel}{\mathbb{A}}

\newcommand{\p}{\mathfrak{p}}
\newcommand{\ida}{\mathfrak{a}}

\newcommand{\Oc}{\mathcal{O}}

\DeclareMathOperator{\size}{size}

\DeclareMathOperator{\GL}{GL}

\DeclareMathOperator{\SL}{SL}
\DeclareMathOperator{\U}{U}

\DeclareMathOperator{\SU}{SU}
\DeclareMathOperator{\Aut}{Aut}
\DeclareMathOperator{\Cl}{Cl}
\DeclareMathOperator{\id}{id}
\DeclareMathOperator{\Speh}{Speh}
\DeclareMathOperator{\diag}{diag}
\DeclareMathOperator{\vol}{vol}
\DeclareMathOperator{\tr}{tr}
\DeclareMathOperator{\Riem}{Riem}
\DeclareMathOperator{\SO}{SO}
\DeclareMathOperator{\op}{op}
\DeclareMathOperator{\Span}{span}

\DeclareMathOperator{\rank}{rank}
\DeclareMathOperator{\real}{Re} %
\DeclareMathOperator{\slope}{slope}
\DeclareMathOperator{\Gr}{Gr}

\DeclareMathOperator{\poly}{poly}

\DeclareMathOperator{\cd}{cd}  %

\newcommand{\triv}{\mathbf{1}}
\newcommand{\lquo}{\backslash}
\newcommand{\unitrk}{r_u}
\newcommand{\initial}{\varphi}
\newcommand{\fp}{\mathfrak{p}}
\newcommand{\ZK}{\mathcal O_K} %

\newcommand{\Lip}{\mbox{Lip}}
\newcommand{\minksum}{\boxplus}
\newcommand{\voronoi}{\mathcal{V}_0}

\newcommand{\Gaussian}{\mathcal{G}}
\newcommand{\Round}{\mathrm{Round}_{\mathrm{Lat}}}
\newcommand{\RoundPerf}{\mathrm{Round}_{\mathrm{Lat}}^{\mathrm{Perf}}}

\newcommand{\unif}{\mathcal{U}}
\newcommand{\distr}{\mathcal{D}}
\newcommand{\from}{\leftarrow}

\newcommand{\norm}[1]{\left\lVert#1\right\rVert}
\newcommand{\abs}[1]{\lvert #1 \rvert}
\newcommand{\inner}[1]{\left \langle #1 \right \rangle}

\newcommand{\ma}{\mathfrak a}

\usepackage[style = alphabetic, backend = bibtex, backref = true, abbreviate = true, arxiv = abs, doi = true, isbn = false, url = false, giveninits = true, maxcitenames = 2, maxbibnames = 5, minalphanames=3]{biblatex} %
\renewbibmacro{in:}{}

\ifanonymous
\author{Anonymous submission}
\else
\author[ \hspace{-1ex}]{Koen de Boer}
\author[1]{Aurel Page}
\author[2]{Radu Toma}
\author[3]{Benjamin Wesolowski}

\affil[1]{Inria, Univ. Bordeaux, CNRS, Bordeaux INP, IMB, UMR 5251,  F-33400, Talence, France}
\affil[2]{Sorbonne Univ. and Univ. Paris Cité, CNRS, IMJ-PRG, F-75005 Paris, France}
\affil[3]{ENS de Lyon, CNRS, UMPA, UMR 5669, Lyon, France}
\fi
\title{Average hardness of SIVP for module lattices of fixed rank}
\date{\today}

\begin{document}
\pagenumbering{roman}
\maketitle 
\begin{abstract}
The problem of finding short vectors in Euclidean lattices is a central hard problem in complexity theory. 
The case of module lattices (i.e., lattices which are also modules over a number ring) is of particular interest for cryptography and computational number theory.
The hardness of finding short vectors in the asymptotic regime where the rank (as a module) is fixed is supporting the security of quantum-resistant cryptographic standards such as ML-DSA %
and ML-KEM. %

In this article we prove the average-case hardness of this problem for uniformly random module lattices (with respect to the natural invariant measure on the space of module lattices of any fixed rank). More specifically, we prove a polynomial-time worst-case to average-case self-reduction for the \emph{approximate Shortest Independent Vector Problem} ($\gamma$-SIVP) where the average case is the (discretized) uniform distribution over module lattices, with a polynomially-bounded loss in the approximation factor, assuming the Extended Riemann Hypothesis.

This result was previously known only in the rank-$1$ case (so-called \emph{ideal lattices}). That proof critically relied on the fact that the space of ideal lattices is a compact group. In higher rank, the space is neither compact nor a group. Our main tool to overcome the resulting challenges is the theory of automorphic forms, which we use to prove a new quantitative rapid equidistribution result for random walks in the space of module lattices.
\end{abstract}

\tableofcontents

\newpage
\pagenumbering{arabic}
\section{Introduction}

\subsection{Motivation}
A lattice is a discrete subgroup in a Euclidean vector space. It is typically described by a \emph{basis}, a collection $\mathbf B = (b_1,\dots,b_n)$ of linearly independent vectors, and the lattice is the group $\Lambda = b_1\Z + \dots + b_n\Z$ obtained from all linear combinations with integer coefficients. Since it is discrete, a lattice contains a non-zero vector of smallest possible Euclidean norm, a \emph{shortest vector}.

The task of finding such a shortest vector (the \emph{Shortest Vector Problem}, SVP) is a central hard problem in complexity theory.
More generally, one can look for a lattice vector whose norm is within a small
factor $\gamma \geq 1$ of the shortest (the \emph{approximate Shortest Vector
Problem}, $\gamma$-SVP), or for a collection of $n$ short independent vectors (the
approximate \emph{Shortest Independent Vector Problem}, $\gamma$-SIVP).
For small enough approximation factors, problems of this type are believed to be hard, and the best known algorithms have exponential complexity in the dimension of the lattice, in both the classical and quantum paradigms. In their hardest regimes, they are even known to be NP-hard \cite{FOCS:Micciancio98,STOC:HavReg07}.
However, applications typically fall outside this NP-hard regime, often depend on the \emph{average hardness} of the problems, and mobilize lattices with additional algebraic structure, like \emph{module lattices}. The main question addressed in this article is:
\begin{quote}
How hard are lattice problems \emph{on average} in \emph{module lattices}?
\end{quote}

\paragraph*{Average hardness.}
No NP-hard problem is known to be hard \emph{on average} (for random instances), and generating random instances that appear consistently hard is a delicate task.
This property is critical for applications to cryptography: one needs randomly sampled instances of the problem to be hard with overwhelming probability.
 Lattice problems are a remarkable family of problems enjoying some proofs of average-case hardness. 
This property is typically ensured by proving a \emph{worst-case to average-case} reduction: a proof that if random instances of a problem $A$ can be solved efficiently with good probability, then all instances (even the ``worst'') of problem $B$ can be solved efficiently. Thus, if there exist hard instances of $B$, then random instances of $A$ are hard.
A self-reduction, when $A = B$, is particularly interesting, as it implies that random instances of the problem are, in a precise sense, as hard as they could possibly be. In approximation problems, like $\gamma$-SVP, a reduction might degrade the approximation factor. One strives to keep this loss as small as possible, to stay in a regime where the problem is hard. 

Ajtai~\cite{STOC:Ajtai96} launched the field of lattice-based cryptography by proving a worst-case to average-case reduction from the approximate shortest vector problem (the worst case, although in a regime unlikely to be NP-hard) to SIS (the average case, for some carefully designed distribution on the set of instances).
This line of research has since evolved into a front-runner of quantum-resistant cryptography, now making it into the real world~\cite{nist_fips204, nist_fips203}. SIS, and later LWE~\cite{STOC:Regev05}, have provided highly fertile ground for cryptography, leading to breakthroughs like fully homomorphic encryption~\cite{C:Gentry10}.

Beyond applications to cryptography, understanding worst-case to average-case reductions for lattice problems helps with the analysis of lattice algorithms. Algorithms such as LLL~\cite{lenstra82:_factor} have experimentally appeared to perform better than their worst-case analysis suggests, both in terms of the running time and the output quality. This mystery has found some explanation through the study of random lattices, see for instance~\cite{nguyen-stehle,kim-venkatesh}. Indeed, the analysis of algorithms is often eased by heuristics on their geometry, such as the \emph{Gaussian heuristic}. Such heuristics are only true in an \emph{average} sense, for random lattices. The average case being easier to analyze, a worst-case to average-case reduction provides a bridge to deduce information about the worst case. This approach has already been fruitful in the case of \emph{ideal lattices}~\cite{C:BDPW20,BPW25}, a particular case of \emph{module lattices}.

\paragraph*{Module lattices.}
Many variants of lattice problems have been studied, and applications to
cryptography and computational number theory motivated \emph{algebraically
structured} variants: finding short vectors in lattices which are ideals or
modules over a number ring (they are called \emph{module lattices}, of which
\emph{ideal lattices} are a special case). These can be thought of as lattices
with ``many symmetries'', offering opportunities for more complex algebraic
manipulations, faster arithmetic, and shorter representation of elements --- all
great features for the design of cryptosystems.

A module lattice over a number field $K$ is essentially a lattice $\Lambda
\subset K^r$ such that $x\Lambda \subseteq \Lambda$ for any integral element $x$
of the field $K$ (this last condition means that $\Lambda$ is a module over the
ring of integers $\ZK \subset K$).
This is a simplification of the definition provided in \Cref{sec:module-lats}.
For the present discussion, we further assume that $\Lambda$ has full-rank (it contains a basis of the vector space $K^r$), and we call $r$ the \emph{rank} of the lattice. Forgetting about its module structure, the lattice $\Lambda$ has $\Z$-rank $r\cdot \deg(K)$, where $\deg(K) = [K:\Q]$ is the degree of the number field, its dimension as a $\Q$-vector space. There is thus a spectrum of ways to construct large lattices: one can balance between choosing a field of large degree $\deg(K)$, or choosing a large rank $r$. In one extreme case, one can let $K = \Q$ so $\deg(K) = 1$, and we obtain generic lattices (with no additional module structure). At the other end of the spectrum, one can consider a large degree field $K$ and set $r=1$, and obtain ``rank one'' module lattices $\Lambda \subset K$, also known as \emph{ideal lattices}; they are in a sense the ``most structured'' case.

The computational study of module lattices started in the context of computational number theory, as the efficient manipulation of ideals in number fields requires seeing them as lattices (see~\cite{cohen2013course} for a variety of examples). The domain accelerated after its introduction to cryptography, first with Ring-LWE~\cite{FOCS:Micciancio02} (proven to be at least as hard as an ideal version of SIVP), then with Module-LWE~\cite{DCC:LanSte15} (proven to be at least as hard as a module version of SIVP).
The digital signature scheme ML-DSA~\cite{nist_fips204, TCHES:DKLLSS18}, and the key-encapsulation mechanism ML-KEM~\cite{nist_fips203, EUROSP:BDKLLSSSS18}, both based on module lattices, recently became the first public-key cryptosystems standardized by the American National Institute of Standards and Technologies for resistance against quantum adversaries. These cryptosystems are proven secure under the assumption that some module-variants of lattice problems are hard, and it has become critical for cryptographers to understand this presumed hardness.
The modules at play in these schemes have small rank (at most five). This regime of ``small rank'' module lattices is precisely the focus of the present paper.

\paragraph*{The invariant probability measure.}
To study the average hardness of lattice problems, one first needs to specify a probability measure on the space of instances: what is a \emph{random} lattice?
In this paper, we work with arguably the most natural choice, a measure on the space of lattices that is both mathematically canonical, and practically relevant.

Every lattice can be described by a basis, an element of $\GL_n(\R)$. Rescaling has no impact on the difficulty of finding short vectors, so we only consider lattices of volume $1$, with basis in $\SL_n(\R)$. Now, two bases describe the same lattice if and only if they differ by a \emph{change of basis}: a matrix in $\SL_n(\Z)$.
Therefore, the space of lattices (of volume $1$) can be identified with the quotient $X_n = \SL_n(\Z) \backslash \SL_n(\R)$ (see Section \ref{sec:module-lats} for the case of module lattices).
This is a homogeneous space for the group $\SL_n(\R)$ and it inherits the Haar measure.
A fundamental result in reduction theory is that the space of lattices has finite volume and we can thus normalize the measure to a probability measure. It is also referred to as the $\SL_n$-invariant measure, or simply \emph{the invariant measure}, written $\mu$ in this introduction.
Several facts motivate the study of this particular measure.

\begin{itemize}
\item This measure was first introduced by Siegel \cite{siegel} to prove that the expected value of the number of lattice points inside a ball centered at zero is approximated by the volume of the ball.
An important line of research then continued to study more refined such statistics \cite{Rogers1955}, as well as interactions with algorithms (e.g. \cite{FOCS:Ajtai02,nguyen-stehle}). This distribution is often the most natural and convenient choice when speaking of ``random lattices''.

\item
The invariance of the measure 
also allows for the theory of automorphic forms to be used as a tool. This rich theory unlocks the spectral analysis of the space $L^2(X_n)$ of square-integrable functions $f : X_n \to \C$.
As Section~\ref{sec:quantitative-equidistribution-bigsec} shows, we take full advantage of this.

\item
    Lattice problems are believed to be hard, so to hope for
    a worst-case to average-case reduction, \emph{easy instances must be rare}: more
    precisely, the probability to sample a lattice for which the problem is easy
    must be negligible. There are easy instances for SVP: for instance, if a
    lattice contains one particularly small vector (exponentially smaller than
    all other independent vectors), the LLL algorithm~\cite{lenstra82:_factor}
    will find it in polynomial time. Such ``very imbalanced'' lattices should
    have small measure. Conveniently, this is the case for the invariant
    measure. Sections of the space $X$ containing very imbalanced lattices are
    referred to as \emph{cusps}, and they do have very small $\mu$-volume, a
    fact quantified in \Cref{sec:self-red-bulk-algorithmic} for module lattices.
In fact, most of the $\mu$-random lattices, forming the \emph{bulk} of the space, are rather \emph{balanced}.

\item
For a worst-case to average-case reduction, we need the average-case distribution to be efficiently sampleable.
Conveniently, the invariant measure naturally comes up as the limit distribution of 
simple random processes. In particular, one can start from an 
arbitrary lattice (say $L_0 = \Z^n \in X_n$), select a ``large'' 
prime number $p$, and sample a uniformly random sublattice 
$L \subseteq L_0$ of index $p$. The probability distribution of 
$L$ is, in a precise sense, \emph{close} to the invariant 
measure~\cite{clozel-ullmo,goldstein-mayer} --- we call this phenomenon \emph{Hecke equidistribution}.\footnote{This is short for the \emph{equidistribution of Hecke points}, as it is commonly referred to in the literature.}
This convenient construction is deceptively simple, as it compares a discrete distribution to a continuous distribution, and hides some computational difficulties. It is nevertheless a powerful idea at the heart of our results, and at least suggests that sampling from the invariant measure should be easy. 
\end{itemize}

Finally, let us point out the main \emph{downside} of the invariant measure: it is continuous. In a computational context, we do not actually manipulate continuous values. Continuity is extremely convenient for algorithmic design and analysis, but in the end, all needs to be discretized, and one must prove that the analysis carries through this discretization. In particular, the average-case distribution for lattice problems is actually a discretized version of the invariant measure. 
These issues are the object of \Cref{sec:discretization}.

\paragraph*{Prior work, and the inspiring case of ideal lattices.}
As fruitful as the worst-case to average-case reduction of Ajtai~\cite{STOC:Ajtai96} has been, 
it has drawbacks. SIS can be posed as a shortest vector problem, so Ajtai's
reduction can essentially be seen as a self-reduction (not quite, but an SIVP
variant achieves that~\cite{DBLP:conf/icalp/Ajtai99}) to an average-case
distribution that does not resemble the invariant distribution (the SIS
distribution is supported on a ``small'' subset of carefully designed lattices).
Yet, the reduction does not preserve the dimension of the lattice. This
dimension change incurs a loss in the approximation factor --- an obstacle
towards approaching an NP-hard threshold. In our regime, this causes an
additional issue: we work in fixed rank, and the analogous reductions for
modules do not preserve the rank~\cite{DCC:LanSte15}. Changing the rank makes
for weaker asymptotic statements --- especially since there seems to be a
hardness gap between rank $1$ and other small ranks (see \cite{AC:LPSW19} for an
analysis on the relative hardness across ranks).

Self-reductions of SVP and variants have successfully been developed for ideal lattices (i.e., rank $1$). 
They first arise in the work of Gentry~\cite{C:Gentry10} on fully homomorphic encryption. There, he develops a worst-case to average-case reduction for the \emph{closest vector problem} (CVP, a problem closely tied to SVP) in ideal lattices. The distribution he is considering is the uniform distribution on prime ideals of bounded norm. Translating this result to SIVP through the quantum equivalence of Regev~\cite{STOC:Regev05} results in a worst-case to average-case (quantum) reduction for SIVP where the average-case distribution is uniform on the \emph{inverse} of prime ideals of bounded norm.

The ideal shortest vector problem was then approached by de Boer, Ducas, Pellet-Mary, and Wesolowski \cite{C:BDPW20}.
They prove a random self-reduc\-tion for the average-case distribution defined by the invariant measure, assuming the Extended Riemann Hypothesis (ERH). 
Their reduction is based on a continuous random walk on the space of ideal lattices, viewed as the so-called \emph{Arakelov class group}.
The use of this rich structure was a fruitful addition to the literature on lattice-based cryptography.
It was used in the article \cite{TCC:FPSW23} to extend Gentry's work to the uniform distribution on prime ideals (instead of their inverse), with applications to the NTRU cryptosystem. Surprisingly, this work critically relies on the results of \cite{C:BDPW20} on the invariant measure to analyze a different distribution on ideal lattices.
The work \cite{C:BDPW20} provides a rigorous understanding of random ideal lattices (assuming the Extended Riemann Hypothesis) which has unlocked algorithmic advances. It was used by de Boer~\cite{boer} to develop the first polynomial time algorithm to compute power residue symbols, and more recently, it has unlocked the first rigorous subexponential algorithms for some of the most fundamental problems in algebraic number theory like the computation of class groups and unit groups~\cite{BPW25}. We are hoping that our generalization from the ideal case to the module case will find such varied applications.

The article \cite{C:BDPW20} is a direct precursor of our paper, both through its
choice of the natural invariant measure, and through its methods. They transfer
computational problems in an ideal lattice to random sublattices, effectively
performing a random walk in the space of ideal lattices. This space is a compact
and abelian topological group, and the study of this random walk boils down to a
study of generalized class groups and Fourier analysis.

Extending this strategy to modules of higher rank presents significant 
challenges, related to the fact that the space of module lattices in 
rank $>1$ is no longer compact, nor is it a group. A key insight is that 
the Fourier analysis underlying \cite{C:BDPW20} is the theory of 
automorphic forms for $\GL(1)$. The much deeper automorphic machinery 
for the non-commutative group $\GL(r)$, $r > 1$, provides a higher rank analog, as already observed in \cite{DK}. 
However, exploiting it has proved considerably more delicate due to the necessity of studying important, yet historically overlooked aspects with high precision.

A concrete and fundamental issue arising with $r>1$ is imbalancedness. On one hand, ideal lattices cannot have extremely short vectors: their shortest vectors are not much shorter than the vectors of their shortest bases --- we say that these lattices are balanced. On the other hand, module lattices of higher rank can be arbitrarily imbalanced. Topologically, this manifests into the fact that the space of module lattices for rank $r>1$ is not compact.
This is an entirely new dimension of the problem, and it leads to serious limitations to a naive generalization of the random walk.\\

We note that the idea of random walks giving rise to reductions and their study using automorphic theory also emerged in another branch of cryptography based on abelian varieties, in particular elliptic curves.
See for instance \cite{jao-miller-venkatesan} and \cite{EC:PagWes24}.
In contrast to the above, this setting is discrete by nature, given by graphs of abelian varieties connected through isogenies.
For example, in the case of supersingular elliptic curves, one may study random
walks using automorphic forms on definite quaternion algebras~\cite{EC:PagWes24}.

\subsection{Results}

We obtain in this paper the first random self-reduction of a shortest vector problem for module lattices beyond ideal lattices, \Cref{thm:main}.
This also marks the first application of automorphic forms on $\GL(n)$ to the complexity theory of lattice problems.

Let us start by formalizing the main computational problem we consider in this article, $\gamma$-SIVP. Recall that the successive minima of a lattice $L$ of dimension $n$ are defined as
\[\lambda_j(L) = \min \left\{ \lambda \in \R_{>0} \;\middle|
\begin{array}{l}
    \text{there exist $\R$-linearly independent vectors } (x_i)_{i=1}^j \\
    \text{such that } x_i \in L \text{ and } \|x_i\| \le \lambda \text{ for all } i
  \end{array}
  \right\},\]
for $j \in \{1,\dots,n\}$.
Given as input a basis $\mathbf{B}$ of an $n$-dimensional lattice $L$ and an approximation factor $\gamma \in \R_{\geq 1}$, the $\gamma$-\emph{shortest independent vector problem} (or $\gamma$-SIVP) is the computational task of finding $\R$-linearly independent lattice vectors $\vec{x}_1, \ldots,\vec{x}_n \in L$ that satisfy $\|\vec{x}_i\| \leq \gamma \cdot \lambda_n(L)$ for all $i \in \{1,\ldots,n\}$.
The problem remains the same when we look at module lattices: the input is a module lattice $M$, and we require the same condition $\|\vec{x}_i\| \leq \gamma \cdot \lambda_n(M)$.

Let us now briefly introduce the average-case distribution: the discretized version of the invariant probability measure $\mu$ on the space $X_r(K)$ of module lattices of rank $r$ over a number field $K$.
It can be described through a rounding algorithm, which we call $\Round$ and defined in Section~\ref{sec:rounding}.
Given an arbitrary lattice $L$, the output $\Round(L)$ is a randomly generated
\emph{rational} module lattice (one which can be represented and manipulated on
a computer or, more formally, on a Turing machine) that is geometrically close to $L$.
We write $\Round(\mu_{\mathrm{cut}})$ for the distribution on rational module lattices coming
from applying $\Round$ to $\mu$-random lattices (with a tail-cut, removing a negligibly
small section of the space, to ensure that the distribution is supported on a
compact set).
This defines the average case; see \Cref{sec:conclusion} for the precise definition.

We insist that $\Round$ replaces any lattice with a ``very close'' one: the distinction between $\mu_{\mathrm{cut}}$ and $\Round(\mu_{\mathrm{cut}})$ is similar to the distinction between the continuous uniform distribution on $[0,1]$, and its discretization by rounding real numbers in $[0,1]$ to a certain number of bits of precision.

\begin{theorem} \label{thm:main}
Let $K$ be a number field of degree $d$ and discriminant $\Delta_K$. Fix a rank $r \in \Z_{>1}$, and let $n = rd$. Assume ERH for the $L$-function of every Hecke character of~$K$ of trivial modulus.
Let $\mathscr O$ be an oracle for $\gamma'$-SIVP which succeeds with probability\footnote{The oracle is \emph{Monte Carlo} in the sense that when it does not succeed, it might still return an incorrect response. The assumption that the error probability satisfies $p = 2^{-o(n)}$ is not fundamental: the problem can be solved in time $2^{O(n)}$ anyway, and we did not attempt to fine-tune our approach for the narrow regime where $p$ is in $2^{-O(n)}$ but not in $2^{-o(n)}$.} $p = 2^{-o(n)}$ when its input follows distribution $\Round(\mu_{\mathrm{cut}})$.
There is a probabilistic polynomial time algorithm for $\gamma$-SIVP over any module lattice of rank $r$ over $K$  with $\gamma = \poly_r(|\Delta_K|^{1/d},d) \cdot \gamma'$, making an expected number of\footnote{The notation $f = \poly_r(g)$ means that $|f| = |g|^{O(1)}$ where the implicit constants in $O(1)$ may depend on $r$ (but on no other parameter).} $ \poly_r(\log |\Delta_K|) \cdot p^{-1}$ queries to $\mathscr O$.
\end{theorem}

This result relies heavily on a quantitative Hecke equidistribution theorem for specific natural test functions that is uniform in all parameters.
The full statement is given in Theorem~\ref{thm:main-equidistro}, and we believe it is of independent interest.
See a special case of this theorem in a simplified version, Theorem \ref{thm:informal-hecke}, in the next section.

Indeed, the problem of equidistribution of Hecke points has a rich history: it was already considered by Linnik and Skubenko \cite{linnik} and became particularly influential at the turn of the century through its generalizations (see e.g. Sarnak's ICM address \cite{sarnak}).
Solutions of greater and greater generality were proven using a wealth of methods, from representation theoretic in \cite{COU} to ergodic theoretic, based on measure rigidity in \cite{eskin-oh} (see the cited papers for more references).

Most of the literature has focused on proving statements as general as possible, with explicit and sharp rates of equidistribution for \emph{general} test functions.
However, we return to the classical interpretation in terms of lattices and ask the following natural question.
Let $L$ be any given lattice and consider a smoothened $\delta$-distribution centered at $L$ (i.e. take a bump function as test function).
How does the equidistribution rate of Hecke operators applied to this distribution depend on $L$ and, in particular, on the rank and balancedness of $L$?
Adapting the work of Clozel--Ullmo \cite{clozel-ullmo}, introducing geometry of numbers and carefully making constants explicit, we give an answer to this question.
We expect further interesting refinements to be possible.

We rely on another result that we believe to be of independent interest: we prove
that random module lattices are somewhat balanced with overwhelming probability;
it is the content of our Theorem~\ref{thm:randbalanced}.
This is related to recent work of Gargava, Serban, Viazovska and
Viglino~\cite{nihar1,nihar2} but our methods are different.
Precisely, relying on
computations by Thunder~\cite{thunder} and generalizing work of Shapira and
Weiss~\cite{shapira-weiss}, we bound the proportion of semistable lattices in
the sense of Grayson--Stuhler~\cite{grayson}.

\subsection{Technical overview}\label{sec:tech-overview}

In this section, we give an overview of our worst-case to average-case reduction for $\gamma$-SIVP.

\paragraph*{Randomizing lattices.}
For the moment, let us forget about modules, and consider generic lattices.
The starting point of our strategy is rather simple: we leverage the fact that
given a lattice $L_0$ and a large prime $p$, a uniformly random sublattice $L
\subseteq L_0$ of index~$p$ is equidistributed in the space of lattices, with
respect to the measure $\mu$. Before properly quantifying this property and
translating them to module lattices, let us sketch how it can be used to build
worst-case to average-case reductions.

Suppose we have an algorithm for $\gamma$-SIVP that works well on average: given a $\mu$-random lattice $L$, the algorithm finds linearly independent vectors $(x_i)_{i=1}^n$ such that $\|x_i\| \leq \gamma \cdot \lambda_n(L)$ with good probability. Now, we are given a lattice $L_0$, a worst-case instance. A straightforward idea would be to pick a large enough prime $p$ and a sublattice $L \subset L_0$ of index $p$, and use our algorithm on $L$. As $L$ is equidistributed, we expect the algorithm to find linearly independent vectors $(x_i)_{i=1}^n$ such that $\|x_i\| \leq \gamma \cdot \lambda_n(L)$ with good probability. Since $L \subset L_0$, these vectors are also in $L_0$. However, proposing $(x_i)_{i=1}^n$ as a solution of SIVP for $L_0$, the lengths must be compared to $\lambda_n(L_0)$ instead of $\lambda_n(L)$.

In general, we only have $\lambda_n(L) \leq p\lambda_n(L_0)$ (an inequality
reached with $L_0 = \Z^n \supset \Z^{n-1} \oplus p\Z = L$), which suggests that
$(x_i)_{i=1}^n$ only solves $p \gamma$-SIVP, a considerable loss in the quality
of the solution. However, the extreme case $\lambda_n(L) \approx
p\lambda_n(L_0)$ is actually rare, and in a precise sense, for random
sublattices $L$, one expects $\lambda_n(L) \approx p^{1/n}\lambda_n(L_0)$.
Indeed, as $L$ is equidistributed, the Gaussian heuristic applies, thus we
expect $\lambda_n(L)$ to be of the order of $\det(L)^{1/n} = p^{1/n}\det(L_0)^{1/n} = O(p^{1/n} \lambda_n(L_0))$.
This ``balancedness of random lattices'' is studied in more detail in Section
\ref{sec:self-red-bulk-algorithmic}, where we prove
Theorem~\ref{thm:randbalanced}.

\paragraph{The problem of imbalancedness.}
In conclusion, this simple strategy appears to provide a worst-case to average-case reduction for SIVP, with a loss of $p^{1/n}$ in the approximation factor.
Now, what does \emph{$p$ sufficiently large} mean? 
On one hand, we want it to be small, to stay in a regime where SIVP is as hard
as possible: the smaller the better, but let us aim for an approximation factor
that is polynomial in the dimension $n$ (a regime in which all known algorithms
have exponential complexity). In other words, we require that $p^{1/n} =
n^{O(1)}$, i.e., $p = n^{O(n)}$.
On the other hand, we require $p$ to be large enough for the random sublattice $L$ to be equidistributed.
This is where difficulties arise, as this constraint actually depends on the initial lattice $L_0$.

For instance, consider the lattice $L_\varepsilon = \varepsilon\Z \oplus
\Z^{n-1}$, where $\varepsilon \in \R_{>0}$ is very small. It contains the small
vector $x_\varepsilon = (\varepsilon,0,\dots,0)$. For any index-$p$ sublattice
$L \subset L_\varepsilon$, we have $px_\varepsilon \in L$, so $\lambda_1(L) \leq
\|px_\varepsilon\| = p\varepsilon$. If $L$ were equidistributed, we would expect
$\lambda_1(L)$ to be of the order of $\det(L)^{1/n}$, yet $\lambda_1(L) \leq
p\varepsilon$ and $\det(L)^{1/n} = (p\varepsilon)^{1/n}$. Therefore, for index-$p$ sublattices of $L_\varepsilon$ to be equidistributed, we need $p$ to be at least as large as $\varepsilon^{-1}$.

These lattices $L_\varepsilon$, with vanishingly small $\varepsilon$, are \emph{imbalanced}, they contain unusually short vectors. We can think of these imbalanced lattices as living in a remote corner of the space of lattices, so far away that to reach the rest, we need to take a gigantic step of index $p > \varepsilon^{-1}$. For such initial lattices $L_0 = L_\varepsilon$, the simple reduction sketched above cannot work, as the loss $p^{1/n}$ in the approximation factor could be arbitrarily large.

\paragraph{A trichotomy.}
This notion of \emph{balancedness} is key, and to quantify it properly, let us return to our actual objects of interest: module lattices. For the rest of the article, we fix a number field $K$ of degree $d$, we fix a rank $r = O(1)$, and we will consider module lattices  of rank $r$ over $K$. Such a module lattice $M$ is still a lattice in the usual sense, of dimension $dr$, and we can speak of its successive minima $\lambda_i(M)$. Its module structure gives rise to a convenient variant of this notion, the $K$-successive minima $\lambda_i^K(M)$ defined as
\[\lambda_j^K(M) = \min \left\{ \lambda \in \R_{>0} \;\middle|
\begin{array}{l}
    \text{there exist $K$-linearly independent vectors } (x_i)_{i=1}^j \\
    \text{such that } x_i \in M \text{ and } \|x_i\| \le \lambda \text{ for all } i
  \end{array}
  \right\},\]
 for $j\in\{1,\dots,r\}$. Each $\lambda_i(M)$ is approximately as large as $\lambda_{\lceil i/d \rceil}^K(M)$ (see \Cref{lem:Kminima}).
Now, we say that a module lattice $M$ is \emph{$\alpha$-balanced} if $\lambda^K_{j+1}/\lambda^K_j \leq \alpha$ for all $j \in\{1,\dots,r-1\}$.

As discussed above, the straightforward reduction consisting in taking random sublattices fails for very imbalanced lattices like $L_\varepsilon$. The notion of $\alpha$-balancedness measures this precisely, and allows us to divide our reduction into three regimes, illustrated in \Cref{fig:schematicimage}.

\begin{itemize}
\item \textbf{The bulk.}
``Most'' lattices are fairly balanced: they form what we call the \emph{bulk} of the space. 
Informally, we say that $M$ belongs to the bulk if it is $\alpha$-balanced with $\alpha = d^{O(1)}$.
We prove that for such $M$, the simple strategy sketched above (reducing SIVP to random sublattices) can be made to work. In \Cref{overview:bulk}, we give more details on this regime and an overview of the proof of equidistribution. The full proof is the object of \Cref{sec:quantitative-equidistribution-bigsec} and \Cref{sec:self-red-bulk-algorithmic}.
\item \textbf{The cusp.} As illustrated with $L_\epsilon$, the simple strategy fails for imbalanced lattices. When the imbalancedness is extreme enough, it can be detected and exploited by polynomial time algorithms like LLL. This region of the space is referred to as the \emph{cusp}, and roughly consists of lattices which are \emph{not} $\alpha$-balanced for some threshold $\alpha = 2^{O(d)}$. 
This region has very small $\mu$-measure, and can be thought of as very ``thin'' and ``elongated'', with lattices like $L_\epsilon$ going ``to infinity'' as $\varepsilon\to 0$ (see \Cref{fig:schematicimage}).
In \Cref{overview:cusp}, we give an overview of the strategy to reduce SIVP from the cusp to the balanced case. The full proof is the object of \Cref{sec:cutting-cusps}.
\item \textbf{The flare.} Between the bulk (where randomization works well) and the cusp (where algorithms like LLL come in handy) remains a region of moderately-balanced lattices: the \emph{flare}\footnote{This notion is not canonical. It comes from the gradual widening in Figure \ref{fig:schematicimage}.}. It consists of lattices that are $2^{O(d)}$-balanced, but not $d^{O(1)}$-balanced. From such a lattice, we prove that we can reduce SIVP to another lattice which is in the bulk. We give an overview of this step in \Cref{overview:flare}. The full proof is the object of \Cref{sec:middle-to-bulk}.
\end{itemize}

\begin{figure}
\centering
\includegraphics[scale=0.8]{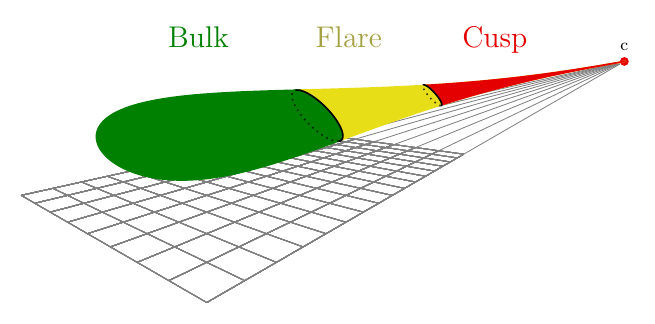}
\caption{Schematic illustration of a single connected component of the space of module lattices.} \label{fig:schematicimage}
\end{figure}

\subsubsection{The bulk}\label{overview:bulk}

In this section, we explain our technique in the regime where lattices are balanced: we start from a lattice $L_0$ in the bulk. As already explained, we reduce SIVP in $L_0$ to SIVP in a random sublattice $L \subseteq L_0$ --- but we are now in the context of module lattices. We work in the space $X_r$ of rank-$r$ module lattices (over $K$). It is defined analogously to the space $\SL_n(\Z) \backslash \SL_n(\R)$ of generic lattices, but the module structure introduces technicalities, the details of which are deferred to \Cref{sec:module-lats}.

Instead of a prime number $p$, and a sublattice $L \subseteq L_0$ of index $p$
(i.e., $L_0/L \cong \Z/p\Z$), we consider a prime ideal $\mathfrak p$ (in the
ring of integers $\ZK$ of $K$) and consider sub-module lattices $L
\subseteq L_0$ of index $N(\p)$ with $L_0/L \cong \ZK/\mathfrak p$
as~$\ZK$-modules (we might say that $L$ has ``index $\mathfrak p$'' in $L_0$).

\paragraph*{Random processes and Hecke operators.}
This process of taking random sublattices can be thought of as a random walk in the space $X_r$. It can be formalized as an operator on probability distributions of $X_r$, or, more conveniently, on the Hilbert space $L^2(X_r)$ (of square-integrable functions $X_r \to \C$ for the measure $\mu$).
Given a prime ideal $\mathfrak p$ in $\mathcal O_K$, the Hecke operator $T_\mathfrak p:L^2(X_r) \to L^2(X_r)$ is defined for each $f \in L^2(X_r)$ and each $L \in X_r$ as
\[
	T_\p f(L) = \frac{1}{D_\p}\sum_{\substack{M \subset L \\ L/M\cong \ZK/\p}}f(M),
\]
where $D_\p = 1 + N(\mathfrak p) + \ldots + N(\mathfrak p)^{r-1}$ is the number of terms in the sum.
This operator is averaging over all ``index $\mathfrak p$'' sublattices.
Its probabilistic interpretation is as follows. Suppose that $f \in L^2(X_r)$ is a probability density function. Then, $T_\p f$ is the probability density function for the experiment which consists in sampling a lattice $L$ with density $f$, then selecting a uniformly random sublattice of $L$ of index $\mathfrak p$. 
Assuming that the measure $\mu$ is a probability measure implies that the constant function $\triv$ is its probability density function. Note that $T_\mathfrak p \triv = \triv$.

The idea that, for $\mathfrak p$ large enough, sublattices of index $\mathfrak p$ are equidistributed can be formalized as follows.
For an initial probability distribution $f\in L^2(X_r)$, the $L^1$-distance (which is the notion of statistical distance we use in this paper) $\|T_\mathfrak pf - \triv\|_1$
converges to $0$ as $N(\mathfrak p)$ grows. 
To apply this, we must now ask for an explicit rate of convergence, which will depend on $f$.

Note that, starting from a lattice $L_0$, it is tempting to consider the Dirac distribution $\delta_{L_0}$ centered at $L_0$, and to study the distribution $T_\p\delta_{L_0}$ of uniformly random sublattices of $L_0$. However, these are discrete distributions in a continuous space: no matter how large $\mathfrak p$ is, the distribution $T_\mathfrak p\delta_{L_0}$ remains discrete and $\|T_\mathfrak p\delta_{L_0} - \triv\|_1 = 1$.
Instead, we ``smoothen'' the distribution $\delta_{L_0}$, and consider a continuous distribution $\initial_{L_0}$ that is highly concentrated around $L_0$. One can think of it as the uniform distribution in a small ball around $L_0$: it samples random lattices which are, geometrically, very close to $L_0$. The precise definition of $\initial_{L_0}$ is the object of \Cref{sec:starting-distribution}.  

\paragraph*{Equidistribution via the theory of automorphic forms.}
The question becomes: given a lattice $L_0$ and a probability density function $\initial_{L_0}$ concentrated around $L_0$ as above, how fast does $T_\mathfrak{p} \initial_{L_0}$ tend to $\triv$ in $L^1$-norm as $N(\mathfrak p)$ grows? 
In other words, how large does $N(\mathfrak p)$ need to be for the distance $\|T_{\mathfrak p} \initial_{L_0} - \triv\|_1$ to be negligibly small?

To answer this, we follow the ideas of Clozel--Ullmo in their work on Hecke equidistribution~\cite{clozel-ullmo}.
They apply the principle behind the Weyl criterion, which suggests spectrally decomposing $\initial_{L_0}$ and analyzing the action of $T_\mathfrak{p}$ on the spectral components, given in terms of automorphic forms or automorphic representations.
For $\GL(r)$, doing this relies on deep theorems by Langlands and Moeglin--Waldspurger, who made the decomposition explicit enough for computations.
We review this theory in Section~\ref{sec:automorphic-theory}.

The main input is the spectral gap for $T_\mathfrak{p}$, an important object of study in number theory (see the Ramanujan Conjecture \cite{BB-Bull}), consisting in strong bounds for its eigenvalues.
However, generalizing Clozel--Ullmo \cite{clozel-ullmo} to number fields requires some care due to the fact that $L^2(X_r)$ contains a large subspace behaving like $L^2(X_1)$, where $T_\p$ acts by unitary characters.
Its eigenvalues thus have absolute value~$1$ there.

To make this formal, we introduce a ``splitting'' of $\GL(r)$ into $\SL(r)$ and $\GL(1)$ using the determinant function (see Section \ref{sec:riem-geom}).
It corresponds to a decomposition
\begin{displaymath}
    L^2(X_r) = L^2_{\det}(X_r) \oplus L^2_{\det}(X_r)^\perp.
\end{displaymath}
On $L^2_{\det}(X_r)^\perp$, the operator $T_\p$ has small eigenvalues, whilst the space $L^2_{\det}(X_r)$ captures the spectral theory of $\GL(1)$.
We also use this splitting through corresponding ``distance functions'' that
allow us to define $\initial_{L_0}$: 
first, take a basis $z \in \GL_r(K_\R)$ for $L_0$ and construct the normalized bump function $f_z$ that is the characteristic function of a ball in the $\SL(r)$-direction and a Gaussian in the $\GL(1)$-direction, both centered at elements corresponding to $z$.
Then we average $f_z$ over all bases of $L_0$ to obtain $\initial_{L_0}$ (this is sometimes called an automorphic kernel).
See Section \ref{sec:starting-distribution} for more details. 

Schematically, we now do the following.
We choose a special basis of automorphic forms $(\varpi)$ for the space $L^2_{\det}(X_r)^\perp$ and $(\chi)$ for the space $L^2_{\det}(X_r)$.
In particular, there is the constant function $\triv = \chi_0$.
These spaces have discrete, as well as continuous spectrum, and we informally write the decomposition of $\initial_{L_0}$ as
\begin{displaymath}
    \initial_{L_0} = \int_\chi \inner{\initial_{L_0}, \chi} \chi + \int_\varpi \inner{\initial_{L_0}, \varpi} \varpi.
\end{displaymath}

Crucially, the operator $T_\p$ acts on $\varpi$ and $\chi$ by scalars.
It is normalized so that $T_\p \triv = \triv$.
By representation theoretic methods, one may compute that $T_\p$ acts on $\varpi$ with eigenvalue bounded in absolute value by $r N(\p)^{-3/8}$, fact which relies on bounds towards the Ramanujan conjecture.

However, on $\chi$ it acts by $\chi(\p)$, a number of absolute value one.
Fortunately, we have a phenomenon generalizing the orthogonality of characters: a sum of the shape $\sum_\p \chi(\p)$ exhibits cancellation for all $\chi \neq \chi_0$.
A strong quantitative version of this fact was proved and used in \cite{C:BDPW20} to treat the case of ideal lattices (the $\GL(1)$ case) under the Extended Riemann Hypothesis.
We therefore study an average of Hecke operators 
\begin{displaymath}
    T_\mathcal{P} = \frac{1}{|\mathcal{P}|} \sum_{\p \in \mathcal{P}} T_\p,
\end{displaymath}
where $\mathcal{P}$ consists of all primes of norm at most $B$ for some $B > 1$.
At the level of the algorithm, this means we randomize the prime $\p$.

The spectral gap and the results of \cite{C:BDPW20}, together with Parseval's
identity, show a bound of the rough shape (see below for a more precise statement)
\begin{multline*}
    \norm{T_\mathcal{P} \initial_{L_0} - \triv} = \norm{ \int_{\chi \neq \chi_0} \inner{\initial_{L_0}, \chi} T_\mathcal{P} \chi + \int_{\varpi} \inner{\initial_{L_0}, \varpi} T_\mathcal{P} \varpi} \\
     \leq rd^{3/2} B^{-3/8} \norm{\initial_{L_0} - \triv} \leq rd^{3/2} B^{-3/8} \norm{\initial_{L_0}}.
\end{multline*}
This generalizes the work of Clozel--Ullmo to number fields.

However, we turn to the question of how this rate of equidistribution depends on $L_0$: we must bound $\norm{\initial_{L_0}}$, which is one of our new contributions.
We reduce this to a counting problem that has been encountered in other contexts of analytic number theory.
When $K = \Q$, taking a basis matrix $z \in \SL_n(\R)$ to represent $L_0$, it asks for a bound on the number of $\gamma \in \SL_n(\Z)$ such that $\gamma z$ lies in a ball of small radius around $z$ in the symmetric space $\SL_n(\R) / \SO(n)$.
While a generalization of the classical circle problem asks for bounds uniform in the radius, in this case we require uniformity in the center of the ball.
Indeed, if $z$ goes deeper into the cusp, defining a very imbalanced lattice $L_0$, then this number of $\gamma$ grows.

We solve this problem over any number field in Section \ref{sec:counting-problem}, producing bounds in terms of the $K$-successive minima of the lattice $L_0$.
Considering lattices defined by diagonal matrices and unipotent $\gamma$, our bounds seem to be essentially sharp.

Piecing everything together, we obtain the quantitative equidistribution theorem, Theorem \ref{thm:main-equidistro}, with very explicit dependence on all parameters.
We give a simplified variant here to point out the quantities that show up.

\begin{theorem}[Hecke equidistribution theorem: special case] \label{thm:informal-hecke}
    For $\ell$ prime, let $K$ be the $\ell$-th cyclotomic field, and let $d = \ell-1$ be its degree. 
    Let $X_r(K)$ be the space of rank $r$ module lattices over $K$ equipped with the invariant probability measure $\mu$. 
    Let $\initial_{L_0}$ be the bump function on $X_r(K)$ centered around a lattice $L_0$ defined in Section \ref{sec:starting-distribution}, and let $\triv$ denote the constant $1$ function on $X_r(K)$.
    If $\p$ is a prime ideal of norm $p$, define $T_\p$ as the Hecke operator averaging over submodules with quotient space given by $\ZK / \p$.
    For a large parameter $B \gg d \log d$, let
    \begin{displaymath}
        T_\mathcal{P} = \frac{1}{|\mathcal{P}|} \sum_{\p \in \mathcal{P}} T_\p,
    \end{displaymath}
    where $\mathcal{P}$ is the set of primes of norm at most $B$.
    Finally, assume ERH for the $L$-function of every Hecke character of~$K$ of trivial
    modulus.
    Then, for any $\varepsilon > 0$, if $L$ is $\alpha$-balanced, we have
    \begin{displaymath}
        \norm{T_\mathcal{P} \initial_{L_0} - \triv} = O\left( d^{3/2+\varepsilon} B^{-1/2 + \varepsilon} + (rd)^{O(r^2 d)} \alpha^{O(r^3d)} B^{-3/8 + \varepsilon} \right) 
    \end{displaymath}
    where the implied constants depend only on $\varepsilon$.
\end{theorem}

Note that the $\Q$-dimension of a module lattice $L$ over a degree $d$ field $K$ is given by $n = dr$. 
For the quantity $\norm{T_\mathcal{P} \initial_{L_0} - \triv}$ to be negligible, e.g. smaller than $2^{-n}$, we must have that
\begin{equation} \label{eq:B-lower-bound}
    B \gg \max((rd)^{Cr^2d}, \alpha^{Cr^3d})
\end{equation}
for some large enough constant $C > 0$.
With such $B$, the process of choosing a random sublattice $L\subset L_0$ produces a $\mu$-random lattice $L$ (up to a negligible error).

Following the steps and observations described above, we can now solve SIVP for $L_0$ by solving it in $L$: we find linearly independent vectors $(x_i)_{i=1}^n$ in $L$ such that $\|x_i\| \leq \gamma \lambda_n(L)$, and since $L$ is $\mu$-random, we have with overwhelming probability that     $\lambda_n(L)$ is roughly bounded by $\det(L)^{\frac{1}{n}} = O(B^{1/n} \lambda_n(L_0))$ (\Cref{thm:randbalanced}).
This results in a loss of $B^{1/n}$ in the approximation factor.
Minimizing $B$ above, this remains polynomial in $n$ if $\alpha \ll n^{O(1/r)}$.
This means that the randomization works well when $L_0$ is $\alpha$-balanced for\footnote{Note that the dependence in the rank is exponential, hinting at difficulties for the regime asymptotic in $r$.} $\alpha = \poly_r(d)$.
This region of the space with polynomially-bounded $\alpha$ constitutes the \emph{bulk} as defined above.

As this randomization requires $L_0$ to be balanced, we next show how to reduce SIVP in imbalanced lattices to the balanced case.

\subsubsection{The cusp}\label{overview:cusp}

Very imbalanced lattices are generally speaking easier instances of short vector problems due to the existence of the LLL algorithm, as previously noted.
However, applying such an algorithm in our situation still requires some careful lattice ``surgery'', cutting and glueing together instances following a divide and conquer strategy. 
We obtain in \Cref{thm:cusp-to-flare} a reduction from SIVP in any module lattice of rank $r$ to SIVP in at most $r$ module lattices, still of rank $r$, but now with the guarantee that they are somewhat balanced (they are in the flare).

\paragraph*{Finding dense sublattices.} The idea is the following. Consider a lattice $L$ that is \emph{not} $\alpha$-balanced, for some $\alpha$ (large enough, and part of our task here is determining what \emph{large enough} means). There is an index $k$ such that $\lambda_{k+1}^K(L) > \alpha \lambda_{k}^K(L)$. This translates to a gap of the form $\lambda_{j+d}(L) > \alpha \lambda_{j}(L)$ between the standard successive minima. One can compute an LLL-reduced basis $(b_i)_i$ of $L$, with the guarantee that
$$\|b_i\| \leq 2^{(rd-1)/2}\lambda_i(L),$$
for all $i$. If $\alpha > 2^{(rd-1)/2}$, we obtain that $\|b_i\| < \lambda_{j+d}(L)$ for all $i \leq j$. This means that the first $j$ vectors founds by LLL are all in the subspace generated by the first $j+d-1$ smallest vectors of the lattice.
This is not sufficient yet, but under a slightly 
stronger bound for $\alpha$, and looking at \emph{module} 
sublattices, we realize that LLL reveals $K$-independent vectors 
$(b_i')_{i = 0}^k$ such that $\|b_i'\| < \lambda_{k+1}^K(L)$. 
In particular, these vectors span the same $K$-subspace $V$ as the $k$ 
first $K$-minima. We can therefore deduce a basis of $L' = L \cap V$: a sublattice of 
$L'$ whose $K$-minima are exactly $\lambda_1^K(L),\dots,\lambda_k^K(L)$.
The detailed proof is the object of \Cref{lem:basis-of-dense-submodule-K-gap}.
Finding short 
vectors in $L'$ immediately reveals short vectors in $L$. Furthermore, while $L'$ might 
still not be $\alpha$-balanced, it has at least one fewer gaps than $L$ 
(the one separating $\lambda_{k}^K(L)$ from $\lambda_{k+1}^K(L)$).
A recursive application of this strategy ultimately leads to a lattice with no gaps left: an $\alpha$-balanced lattice.

\paragraph*{Lattice surgery.}
Note that $\lambda_1(L') = \lambda_1(L)$, so a solution for SVP can easily be transferred. But for SIVP, we need to find $n = rd$ independent vectors of $L$, that is more than exist in $L'$. A solution of SIVP in $L'$ does not give a complete solution for $L$: we also need to solve SIVP in a ``complementary lattice'': the orthogonal projection $\pi(L)$ of $L$ along $V = \mathrm{span}_K(L')$. The successive minima of $\pi(L)$ are very close to $\lambda_{k+1}^K(L),\dots,\lambda_{r}^K(L)$; the small discrepancy causes a small loss in the approximation factor. This is proved in \Cref{lemma:skewed-split-in-two}.

In summary, LLL can detect large gaps between $\lambda_{k}^K(L)$ and $\lambda_{k+1}^K(L)$, and can effectively split the lattice $L$ ``around that gap'', resulting in two lattices $L'$ (or rank $k$) and $L''$ (of rank $r-k$) such that the minima of $L'$ are 
$\lambda_{1}^K(L),\dots,\lambda_{k}^K(L)$, and the minima of $L''$ are \emph{almost} $\lambda_{k+1}^K(L),\dots,\lambda_{r}^K(L)$.
To solve SIVP in $L$, it is sufficient to solve SIVP in $L'$ and $L''$.
Applying this recursively results in a collection of lattices $L_1,\dots,L_t$ whose successive minima have no remaining large gap, and whose ranks sum to $r$ (\Cref{lem:to-balanced-smalled-dim}). This reduction of the dimension sounds good in practice, but to ultimately achieve a worst-case to average-case reduction, we wish to preserve the dimension. Therefore, in a final step, we show how each $L_i$ can be embedded in a module lattice of rank $r$ in a way that preserves its balancedness (\Cref{lem:small-dim-to-original-dim}).

\subsubsection{The flare}\label{overview:flare}

If $\alpha$ is larger than some polynomial in $d$ but not exponentially large, we must proceed differently.
The idea is still to take sublattices with the goal of reducing the size of $\alpha$, which measures ``gaps'' in between the successive minima.
The principle is as follows.

Take a unimodular lattice $L$ of rank $2$ over $\Q$ with shortest vector $v$ of very small size $\lambda_1$.
There exists a reduced basis $(v,w)$ of $L$ with $w$ a vector of much larger size $\lambda_2 \asymp 1/\lambda_1$.
Choosing a sublattice of index $p$ amounts to multiplying either $v$ or $w$ by $p$ and taking some linear combinations to form a new basis.
Put another way, one chooses a subspace of dimension one inside the $2$-dimensional $\Z/p\Z$-vector space $L/pL$.

There are $p+1$ possibilities to do so, yet only one that contains the projection of $v$: indeed, $v$ is a primitive vector and spans a unique one-dimensional subspace in $L/pL$.
Thus, with very high probability, i.e. $p/(p+1)$, the sublattice does not contain $v$, but it contains $pv$ --- the basis of the new sublattice is of the form $(pv, w + kv)$ for some $0 \leq k \leq p-1$.
If $p \lambda_1 < \lambda_2$, then $p \lambda_1$ must be the shortest length in the sublattice, and $\lambda_2$ remains a good approximation for the second successive minimum.
The gap between $\lambda_1$ and $\lambda_2$ can thus be reduced by $1/p$.

To generalize this idea to higher rank $n$, we must use different types of Hecke operators.
This corresponds to taking sublattices with different, fixed structures of the quotient space: consider those sublattices corresponding to subspaces of $L/pL$ of dimension $k$ for some $1 \leq k \leq n-1$.
Now, let $1 \leq i \leq n-1$ and assume we have a large $i$-th gap $\lambda_{i+1}/\lambda_{i}$ and $p < \lambda_{i+1}/\lambda_{i}$. 
Then we can prove that, with high probability, sublattices $L' \subset L$ such that $L/L' \cong (\Z/p\Z)^i$ have $i$-th gap reduced by $1/p$, i.e., equal to $\lambda_{i+1}/p \lambda_{i}$, and all other gaps remain approximately the same (see Section \ref{sec:closing-one-gap}).

With this technique, we could close ``exponentially large'' gaps, but it requires knowing at which index the gaps are, and how large they are.
Gaps in the flare are \emph{moderately} large, so they cannot be detected in the same way as gaps in the cusp (\Cref{overview:cusp}).
We therefore ``guess'' the dyadic sizes and apply the process.
There are $O(n \log n)$ dyadic intervals for each gap, but there are $r$-many gaps.
The number of possible guesses is thus exponential in the rank --- this is fine in our fixed-rank regime, but is another big obstacle to treating generic lattices.
Once the correct guess has been found, SIVP is reduced from the flare to the bulk (see Section \ref{sec:middle-to-bulk}).

\ifanonymous
\else
\subsection{Acknowledgments}
The authors would like to thank Thibault Monneret for his comments, which helped improve this article.
R.T. was supported by European Research Council Advanced Grant 101054336 and the European Union’s Horizon 2020 research and innovation programme under the Marie Skłodowska-Curie grant agreement No 101034255. At the time of this research, K.dB. was affiliated with the Mathematical Institute, Universiteit Leiden. 
B.W. was supported by the PEPR quantique France 2030 programme
(ANR-22-PETQ-0008), and by the HQI initiative (ANR-22-PNCQ-0002). 
B.W. was supported by the European Research Council under grant No. 101116169 (AGATHA CRYPTY).
B.W. and A.P. were supported by the ANR CHARM project (ANR-21-CE94-0003).
A.P. was supported by the ANR AGDE (ANR-20-CE40-0010).
\fi

\section{Preliminaries}

\subsection{Notation}
For every abelian group~$A$, let~$A_\R$ denote~$A \otimes \R$.
For a representation~$V$ of a group~$G$, let~$V^G$ denote the space of fixed
points~$\{v\in V \mid gv = v \text{ for all }g\in G\}$.

For two complex-valued functions $f$ and $g$, we occasionally write $f(x) \ll g(x)$ to mean $f = O(g)$.
We write $f = O_r(g)$ of the implicit constants depend on a parameter $r$.
We write $f = \poly(g)$ to signify $|f| = |g|^{O(1)}$ and $f = \poly_r(g)$ to signify $|f| = |g|^{O_r(1)}$.
For $n \in \Z_{>0}$ we denote $[n] = \{ 1,\ldots,n \}$.
The expression $\log x$ denotes the natural logarithm of $x$ and $\log_2 x$ denotes the base $2$ logarithm.
For a finite set $X$, we denote by $|X|$ its cardinality. 

\subsection{Number fields} \label{sec:number-fields}
Let $K$ be a number field of degree $d$ with signature~$(r_1,r_2)$ (i.e., there are $r_1$ real embeddings $K \to \R$, and $2r_2$ complex embeddings $K \to \C$).
Let $\ZK$ be its ring of integers with discriminant $\Delta_K$ and denote by $\Cl(K)$ the ideal class group.
Let $h_K$ be the class number, $R_K$ the regulator, and $w_K$ the number of roots of unity in $K$.
Let $\unitrk = r_1+r_2-1$ be the rank of the group of units $\ZK^\times$.

We fix a set of $r_2$~complex embeddings $\{\sigma_1, \ldots \sigma_{r_2}\}$ such that the union of $\{\sigma_i, \overline{\sigma_i}\}$ exhausts all complex embeddings.
We have that $K_\R \cong \R^{r_1}\times\C^{r_2}$ and there is a natural embedding $K \to K_\R$, with the real components given by the $r_1$ real embeddings $x \mapsto \rho(x)$ and the complex components given by the $r_2$ embeddings $x \mapsto \sigma(x)$ fixed above.
We call this map the Minkowski embedding.

There is a unique positive involution~$a\mapsto a^*$ on $K_\R$ given by complex conjugation in each factor under the isomorphism with $\R^{r_1}\times\C^{r_2}$.
The canonical metric on $\R^{r_1}\times\C^{r_2}$ is given by
\begin{displaymath}
	\langle x, y \rangle_0 = \sum_{\rho} x_\rho y_\rho + \sum_{\sigma} 2 \real (x_\sigma \overline{y_\sigma}) = \tr_{K_\R/\R}(x \cdot y^\ast).
\end{displaymath}

At the non-archimedean places, given by prime ideals $\p \in \ZK$, we have the completions $K_{\p}$ of $K$ and $\Oc_\p$ of $\ZK$.
Let~$\Zhat_K = \prod_\p \Oc_\p$ denote the profinite completion of~$\ZK$ and~$\adel_K = K_\R \times \prod'_{\p} K_\p$ the ring of ad\`eles of~$K$.
We refer to \cite{neukirch} for more details on these constructions.

\subsection{Lattices}

\subsubsection{Euclidean $K_\R$-modules}
A \emph{Euclidean $K_\R$-module} of rank~$r$ is a
pair~$(V,\langle\cdot,\cdot\rangle)$ where~$V$ is a free~$K_\R$-module of
rank~$r$
and~$\langle\cdot,\cdot\rangle\colon V\otimes_{\R} V \to \R$ is a positive
definite inner product on the real vector space~$V$
such that
\[
\langle ax, y\rangle = \langle x, a^*y\rangle \text{ for all }x,y\in V
\text{ and }a\in K_\R.
\]

\begin{example}
	The module $V_0 = K_\R^r$ equipped with the standard inner product
	\begin{displaymath}
		\langle x,y\rangle_0 = \sum_{i=1}^r \tr_{K_\R/\R}(x_i y_i^*)
        = \sum_{i=1}^r \langle x_i,  y_i \rangle_0
	\end{displaymath} 
	is a Euclidean $K_\R$-module of rank~$r$.
	More generally, for~$g\in \GL_r(K_\R)$, the $K_\R$-module $V = K_\R^r$ equipped
	with~$\langle x,y\rangle = \langle gx,gy\rangle_0$ is a Euclidean~$K_\R$-module of
	rank~$r$, and $g \colon V \to V_0$ is an isomorphism.
\end{example}

Any abstract Euclidean~$K_\R$-module is isomorphic to the more concrete~$V_0$. 
To see this, first let~$(e_v)_v$ be the primitive idempotents of the $\R$-algebra $K_\R$ indexed by
the infinite places~$v$ of~$K$,  meaning that $e_v
K_\R \cong \R$ if~$v$ is a real place and $e_v K_\R \cong \C$ if~$v$ is a
complex place. Note that~$e_v^* = e_v$ for all~$v$.

Now let~$V$ be a Euclidean $K_\R$-module of rank~$r$. Then the~$e_v$ act as
self-adjoint idempotents on~$V$, i.e. they induce an orthogonal decomposition
\[
V = \bigoplus e_v V
\]
that commutes with the action of~$K_\R$.

Let~$v$ be a real place of~$K$. Then~$e_v V$ is a real Euclidean space of
dimension~$r$, and therefore there exists an isomorphism~$g_v \colon e_v V \to
K_v^r \cong \R^r$ to the standard Euclidean space of dimension~$r$.

Let~$v$ be a complex place of~$K$, fix an isomorphism~$K_v\cong \C$, and let~$W$
be the $\C$-vector space~$e_v V$ of dimension~$r$.
For~$x,y\in W$, define
\[
H(x,y) = \langle x,y\rangle - i \langle ix, y\rangle \in \C,
\]
so that~$\langle x,y\rangle = \real H(x,y)$.
Then the identity~$\langle ax,y\rangle = \langle x,\bar{a}y\rangle$ for~$x,y\in
W$ and~$a\in \C$ implies that~$H$ is a positive definite Hermitian form on~$W$.
Therefore, there exists an isomorphism~$g_v\colon W \to \C^r$ to the standard
Hermitian space of dimension~$r$. In particular, we have
\[
\langle x,y\rangle = \real H(g_v x, g_v y)
\]
for all~$x,y\in e_v V$.

Putting all places together, there exists an isomorphism~$g\colon V \to V_0$ of
Euclidean $K_\R$-modules.

\subsubsection{Module lattices} \label{sec:module-lats}
A \emph{module lattice} of rank~$r$ is a pair~$(M,\langle\cdot,\cdot\rangle)$
where~$M$ is a projective $\ZK$-module of rank~$r$
and~$(M_\R,\langle\cdot,\cdot\rangle)$ is a Euclidean~$K_\R$-module.
We will often omit~$\langle\cdot,\cdot\rangle$ from the notation.

\begin{example}
	Let~$M_0\subset V_0$ be an $\ZK$-sub-module such that~$M_0 \cdot \R = V_0$, i.e.
	that is also a lattice in~$V_0$. Then~$(M_0,\langle\cdot,\cdot\rangle_0)$ is a
	module lattice. We refer to those as \emph{embedded} module lattices.
\end{example}

Let~$M$ be an arbitrary module lattice.
By the previous section, there exists an isomorphism~$g\colon M_\R \to V_0$
of Euclidean~$K_\R$-modules. Since~$M$ is projective, the restriction of~$g$
to~$M$ is injective. In other words, $(M,\langle\cdot,\cdot\rangle)$ is
isomorphic to an embedded module lattice.

Let~$\lambda\in\R_{>0}$. A \emph{similitude}~$f\colon M_1\to M_2$ of factor~$\lambda$ is an isomorphism
of~$\ZK$-modules that multiplies the inner product by~$\lambda$.
An \emph{isomorphism of module lattices} is a similitude of factor~$1$, i.e. an isomorphism
of~$\ZK$-modules that preserves the inner product.

Let~$X_r(K)$, also denoted~$X_r$ when~$K$ is clear from the context, be the
space of similarity classes of modules lattices of rank $r$. 
We recall that any such module lattice is isomorphic as an~$\ZK$-module to
$\ZK^{r-1} \oplus \ida$ for an ideal~$\ida$ in some fixed set of representatives
of the ideal class group of~$K$.
Using the Minkowski embedding $\ZK \to K_\R$, this implies that we have an isomorphism
\begin{equation} \label{eq:adelic-quot-conn-comps}
	X_r(K) \cong \bigsqcup_{\ida \in \Cl(K)} \GL_r(\ZK,\ida) \lquo
	\GL_r(K_\R) / (\U_r(K_\R)\cdot\R_{>0}),
\end{equation}
where~$\U_r(K_\R) = \{g\in \GL_r(K_\R) \mid g(g^t)^* = \id\}$, the group~$\R_{>0}$ is
embedded via~$\lambda \mapsto (1\otimes \lambda)\id \in \GL_r(K_\R)$,
and~$\GL_r(\ZK, \ida) = \Aut(\ZK^{r-1}\oplus \ida)$.
We write
\begin{equation} \label{eq:xra}
X_{r,\ida} = \GL_r(\ZK,\ida) \lquo \GL_r(K_\R) / (\U_r(K_\R)\cdot\R_{>0})
\end{equation}
and we often use the notation $\Gamma_\ida = \GL_r(\ZK, \ida)$, when $r$ and $K$ are understood from context.
Note that we also have a map~$X_r(K) \to X_n(\Q)$ where~$n = rd$, obtained by forgetting the structure of~$\ZK$-module.

Choosing a representative $\ida \in \Cl(K)$ and a matrix $z \in \GL_r(K_\R)$ we uniquely determine a class of module lattices, which we call $L_{z, \ida}$.
By abuse of notation, we let $L_{z, \ida}$ denote the representative in this class given by 
\begin{displaymath}
	L_{z, \ida} = (\ZK^{r-1}\oplus \ida) \cdot z \subset K_\R^r,
\end{displaymath}
viewing $\ZK \subset K_\R$ through the Minkowski embedding.

The Haar measure on~$\GL_r(K_\R)$ induces a measure~$\mu$ on~$X_r(K)$, whose total
volume is finite. 
We normalize~$\mu$ to be a probability measure and we often refer to it as the \emph{uniform} measure.
The measure $\mu$ gives a meaning to \emph{random module lattices}, distributed
according to~$\mu$, or with distribution given by a density function~$f\in
L^1(X_r)$ with respect to~$\mu$.
For computations, however, we often work with a more explicit normalization of $\mu$, namely $\mu_{\Riem}$, descending from~$\GL_r(K_\R) / (\U_r(K_\R)\cdot\R_{>0})$ and defined in~\Cref{sec:riem-geom}. \label{def:riemriem}

In order to use representation theoretic arguments, it will often be easier to
use the ad\'elic version of the space of lattices.
We can rewrite our union of double quotients as a single ad\'elic double
quotient as follows:
\[
X_r(K) \cong \GL_r(K) \lquo \GL_r(\adel_K) /
(\GL_r(\Zhat_K)\cdot\U_r(K_\R)\cdot\R_{>0}).
\]
Let~$\X_r = \X_r(K) = \GL_r(K) \lquo \GL_r(\adel_K) / \R_{>0}$.
We can therefore write
\[
L^2(X_r) = L^2(\X_r)^{\GL_r(\Zhat_K)\cdot\U_r(K_\R)}
\]
where~$\GL_r(\adel_K)$ acts on~$L^2(\X_r)$ by~$g\cdot f(x) = f(xg)$.

\subsubsection{Representation and sizes of elements, ideals and modules} \label{sec:sizes}
\paragraph{Assumptions.}
In this paper, we assume that $K = \Q[x]/f(x)$ is represented by a polynomial $f \in \Z[x]$ satisfying $\log \max |f_i| \leq \poly(\log |\Delta_K|)$. Additionally, we assume 
that we have a $\Z$-basis of $\ZK$, written $(\beta_1,\ldots,\beta_d)$. Without loss of generality (by applying LLL and \Cref{lemma:rank1prop}), we may assume that it satisfies $\max_i \|\beta_i \| \leq  2^{d} \cdot |\Delta_K|^{1/d}$.

\paragraph{Representations.} We represent an element $\alpha \in K$ by its coordinates $(a_1,\ldots,a_d) \in \Q^d$ with respect to the $\Z$-basis $(\beta_1,\ldots,\beta_d)$. This means that $\alpha = \sum_{i=1}^d a_i \beta_i$. A (fractional) ideal $\mathfrak{a}$ of $\ZK$ can be represented by a generating matrix $B_{\mathfrak{a}} = (\alpha_1,\ldots,\alpha_d) \in K^d$, for which we have that $\mathfrak{a}$ is generated %
by these $\alpha_i$ as a $\Z$-module. Each of these generators $\alpha_i$ is then represented by $(a^{(i)}_1,\ldots,a^{(i)}_d) \in \Q^d$ and hence $B_{\mathfrak{a}}$ can be written as a matrix in $\Q^{d \times d}$.
In this paper, we choose to have a unique representation for an ideal by always demanding that $B_{\mathfrak{a}}$ (as a matrix in $\Q^{d \times d}$) is in Hermite normal form. That is, we write the generating matrix of $\mathfrak{a}$ as $\frac{1}{m} \cdot \left( m \cdot B_{\mathfrak{a}}\right) $, where $m \in \Z_{>0}$ and $m \cdot B_{\mathfrak{a}} \in \Z^{d \times d}$ is in Hermite normal form.

A rank $r$ module lattice $M$ is represented by its \emph{pseudo-basis} (see, for example, \cite{cohen1999advanced}), which consists of a matrix $\mA \in K^{r \times r}$ (with columns $\mA_i \in K^r$) and a sequence of $r$ ideals $\mI := (\mathfrak{a}_1,\ldots,\mathfrak{a}_r)$. Again, each of the $\mA_{ij} \in K$ can be represented by a sequence in $\Q^d$, and each of the ideals of $\mI$ can be represented by its generating matrix. The module lattice is then defined by the rule 
\[ M = \left\{ \sum_{i = 1}^r  \mA_i \cdot \alpha_i \in K^r ~|~ \alpha_i \in \mathfrak{a}_i \right\}. \] 
\paragraph{Sizes of elements, ideals and modules.}
For $n \in \Z$, we define $\size(n) = 1 + \lceil \log_2(|n|) \rceil$ (where the extra $1$ is for encoding the sign). 
For $q = \frac{a}{b} \in \Q$ with $a,b \in \Z$ coprime, we set $\size(q) = \size(a) + \size(b)$. For $\alpha \in K$ represented by $(a_1,\ldots,a_d) \in \Q^d$ we put $\size(\alpha) = \sum_{i=1}^d \size(a_i)$. For an ideal $\mathfrak{a}$ of $K$,
we define $\size(\mathfrak{a}) := \size(m B_{\mathfrak{a}}) + \size(m)$, where the generating matrix equals $\frac{1}{m} \cdot (m B_{\mathfrak{a}}) $ and $(m B_{\mathfrak{a}})  \in \Z^{d \times d}$ is in Hermite normal form. 

For a rank $r$ module lattice $M$ with pseudo basis $(\mA,\mI)$
with $\mA \in K^{r \times r}$ and $\mI =  (\mathfrak{a}_1,\ldots,\mathfrak{a}_r)$ we put 
\[ \size(M) := \sum_{i,j = 1}^r \size(\mA_{ij}) + \sum_{i = 1}^r \size(\mathfrak{a}_i). \]
\paragraph{Rules for sizes.}
For the $\Z$-basis $(\beta_1,\ldots,\beta_d)$ of $\ZK$, we surely have, by Cauchy-Schwarz, $\|\beta_i \beta_j \| \leq \left( \sum_{\sigma} |\sigma(\beta_i)|^2 |\sigma(\beta_j)|^2 \right)^{1/2} \leq  \sum_{\sigma} |\sigma(\beta_i)| |\sigma(\beta_j)| \leq \|\beta_i\| \|\beta_j\| \leq  2^{2d} \cdot |\Delta_K|^{2/d}$ per assumption. Additionally, we can deduce that, writing $B = (\sigma(\beta_j))_{\sigma,j}$ as a basis in the Minkowski space $K_\R^{d \times d}$, and using \Cref{lemma:wellconditioned} and the fact that $\lambda_1(\ZK) = \sqrt{d}$,
\[  \|B^{-1}\| \leq \sqrt{d} \prod_i \|\beta_i\| \leq \sqrt{d} \cdot 2^{d^2} \cdot |\Delta_K|. \]
Hence, $\|B^{-1} (\beta_i \cdot \beta_j)\| \leq  \|B^{-1}\| \|\beta_i \cdot \beta_j\| \leq \sqrt{d} \cdot2^{d^2+2d} \cdot |\Delta_K|^{1 + 2/d}$. So, the co-ordinates of the product $\beta_i \beta_j$ in terms of the basis $(\beta_1,\ldots,\beta_d)$ are bounded by $\sqrt{d} \cdot2^{d^2+2d} \cdot |\Delta_K|^{1 + 2/d}$.

\begin{lemma}[Rules on sizes of elements] \label{lemma:rules-on-sizes}
	For fractional $\ZK$ ideals $\ma, \ma_i$ of $K$, we have $\size(\ma) \leq d^2 \size(N(\ma)) \ll d^2 \log N(\ma)$ and $\size(\prod_{i=1}^k \ma_i) \leq d^2 \sum_{i =1}^k \size(\ma_i)$.
\end{lemma}
\begin{proof} 
	See Section \ref{sec:appendix-sizes} in the Appendix.
\end{proof}

\paragraph{Sizes and Module-HNF}
In this paper, we will make use of a Hermite-normal form algorithm that works over module-lattices, and thus applies basis operations that are compatible with the module structure \cite[Section 1.4]{cohen1999advanced}. Computing this Hermite-normal form of a given pseudo-basis of a module lattice can be done within polynomial time of the input size \cite{BF12}. This Hermite-normal form can be made unique (i.e., not depending on the specific pseudo-basis given) with no significant overhead \cite[Theorem 1.4.11]{cohen1999advanced}.

Due to the polynomial time algorithm of \cite{BF12} it must surely be true that the output $(\mH, (\mathfrak{h}_i)_{i \in [r]})$ of the module Hermite Normal Form algorithm must have size polynomially bounded in the size of the input module lattice $(\mB, (\mathfrak{a}_i)_{i \in [r]} )$, i.e., 
\[ \size(H, (\mathfrak{h}_i)_{i \in [r]}) \leq \poly(\size(\mB,\mathfrak{a}_i)_{i \in [r]}) ). \]

\subsubsection{Sublattices}

We record the following standard definition and refer to \cite[App. B.2]{AC:FelPelSte22} for a proof of the equivalences.
\begin{definition} \label{def:primitive} 
	Let $M$ be a $\ZK$-module. A sub-module $N \subseteq M$ is said to be
	\emph{primitive} if it satisfies any of the following equivalent conditions:
	\begin{itemize}
		\item The module $N$ is maximal for the inclusion relation in the set of submodules of $M$ of rank at most $\rank(N)$.
		\item There is a module $N'$ with $M = N + N'$ and $\rank(M) = \rank(N) + \rank(N')$.
		\item There is a module $N'$ with $M = N \oplus N'$.
		\item We have $N = M \cap \Span_K(N)$.
	\end{itemize}
\end{definition}

\paragraph{Algorithm for taking a random index $N(\fp)$ sub-module lattice}

\begin{algorithm}[ht]
    \caption{Computing a random sub-module $M'$ of $M$ such that $M/M' \simeq \ZK/\fp$}
    \label{alg:randomsublattice}
    \begin{algorithmic}[1]
    \REQUIRE ~\\ \vspace{-.3cm} \begin{itemize} 
              \item A pseudos $(\mathbf{B},\mI)$ of a rank $r$ module lattice $M$, with $\mathbf{B} = (\mathbf{b}_1, \ldots,\mathbf{b}_r) \in K^{r \times r}$ and $\mI = (\ma_1,\ldots,\ma_r)$. \vspace{-0.2cm}
              \item A prime ideal $\fp$ of $\ZK$, 
             \end{itemize}
    \ENSURE A pseudo-basis $(\mathbf{B}',\mI')$ of a module $M'$ that satisfies $M/M'  \simeq \ZK/\fp$.
    \STATE Draw a random integer $u$ from $\{1, \ldots, \sum_{i=0}^{r-1} q^{i}\}$, with $q = N(\fp)$ and pick the smallest $j \geq 1$ such that $\sum_{i=0}^{j-1} q^{i} \geq u$.
    \STATE Put $\mI' = (\ma_1,\ma_2,\ldots,\fp \ma_j,\ldots,\ma_r)$. I.e., multiply the $j$-th ideal in $\mI$ by $\fp$ to obtain $\mI'$.
    \STATE Put, for all $i < j$,  $\mathbf{b}_i' = \mathbf{b}_i + \gamma_i \mathbf{b}_j$ where $\gamma_i$ is uniformly drawn from a set of representatives of $\ma_i / \fp\ma_i$, and put $\mathbf{b}_i' = \mathbf{b}_i$ for $i \geq j$. Assemble $\mathbf{b}_i'$ into a matrix $\mathbf{B}'$. 
    Equivalently, we put $\mathbf{B}' = \mathbf{B} \cdot T$ where $T = I + \sum_{i < j} \gamma_i \mathbf{e}_{ji}$ where $\mathbf{e}_{ji}$ is the matrix with a one on place $ji$ and zeroes elsewhere.
    \label{line:randomsublattice:rep}
    \RETURN $(\mathbf{B}',\mI')$.
    \end{algorithmic}
\end{algorithm}
\begin{lemma}\label{lem:alg:randomsublattice} \Cref{alg:randomsublattice} is correct, outputs a uniformly random sub-module $M' \subseteq M$ satisfying $M/M' \simeq \ZK/\fp$ and runs in time polynomial in the input size.
\end{lemma}
\begin{proof} (Correctness) We have that $M' \subseteq M$ since any element of $M'$ can be written as (with $\alpha_i \in \ma_i$ for $i \neq j$ and $\alpha_j \in \fp \ma_j \subseteq \ma_j$)
\[ \sum_{i = 1}^r \alpha_i \mathbf{b}_i' = \sum_{i=1}^r \alpha_i (\mathbf{b}_i + \gamma_j \mathbf{b}_j) =\sum_{i=1, i\neq j}^r \alpha_i \mathbf{b}_i + (1 +  \sum_{i = 1}^{j-1} \gamma_i) \mathbf{b}_j \in M,  \]
 since $\gamma_i \in \ma_i$ for all $i < j$. Additionally, a set of representatives of $M/M'$ can be given by $\{ \gamma_j \mathbf{b}_j \} \subseteq M'$ with $\gamma_j$ from a set of representatives of $\ma_j/\fp\ma_j$. Hence $M/M' \simeq \ma_j/\fp\ma_j \simeq \ZK/\fp$.
 
 (Uniformly random sub-module) The number of submodules $M' \subseteq M$ satisfying $M/M' \simeq \ZK/\fp$ corresponds with the number of hyper planes in $M/\fp M$, which equals (by the $q$-binomial theorem) $\binom{r}{1}_q = \sum_{i=0}^{r-1} q^i$. One can readily verify that the number of $M'$ that \Cref{alg:randomsublattice} outputs is indeed $\sum_{i=0}^{r-1} q^i$, and that the way that they are all picked with equal probability.
 
 (Polynomial time) Each of the operations can be reasonably seen to be able to be computed in time polynomial in the input size. We spend some extra words on line \lineref{line:randomsublattice:rep}, where a random representative of $\ma_i/\fp \ma_i$ needs to be chosen. This can be done by computing the Hermite normal form of both $\ma_i$ and $\fp\ma_i$ (after scaling up), take random elements in the finite quotient group (of these two lattices) of order $N(\fp) = q$ (seen as a subgroup of $\Z^d$ with $d = [K:\Q]$) and lift the elements to $\ma_i$.  
\end{proof}

\subsubsection{Successive minima} \label{subsec:prelimminima}

It is useful for us to work with two different notions of successive minima, corresponding to linear independence with respect to $\Q$ and, respectively, $K$.

\begin{definition}
	For a $\ZK$-module lattice $M$ of rank $r$ we put
	\begin{displaymath}
		\lambda_j(M) = \min \{ \lambda \in \R_{>0} \mid \dim_\Q \Span_\Q (B_\lambda \cap M) \geq j \},
	\end{displaymath}
	for $j = 1, \ldots, rd$, where $B_\lambda$ is the ball of radius $\lambda$ with respect to the norm on $M$.
	In other words, $\lambda_j(M)$ is the minimal $\lambda$ such that there exist $j$ vectors in $M$ of length at most $\lambda$ that are $\Q$-linearly independent.
\end{definition}

\begin{definition} \label{def:Kminima}
	For a $\ZK$-module lattice $M$ of rank $r$ we put
	\begin{displaymath}
		\lambda^K_j(M) = \min \{ \lambda \in \R_{>0} \mid \dim_K \Span_K (B_\lambda \cap M) \geq j \},
	\end{displaymath}
	for $j = 1, \ldots, r$, where $B_\lambda$ is the ball of radius $\lambda$ with respect to the norm on $M$.
	We often call these quantities $K$-minima.
\end{definition}

\begin{definition} \label{def:alpha-bal}
	A $\ZK$-module lattice $M$ with $K$-minima $\lambda^K_1,\dots,\lambda^K_r$ is \emph{$\alpha$-balanced} if
	$\lambda^K_{i+1}/\lambda^K_i \leq \alpha$ for all $1 \leq i< r$.
\end{definition}

When comparing the two types of successive minima, we also require the sup-norm successive minima of the underlying ring of integers.

\begin{definition}
	Consider $\ZK$ as a $\Z$-lattice through the Minkowski embedding. Define
	\begin{displaymath}
		\lambda^\infty_j(\ZK) = \min \{ \lambda \in \R_{>0} \mid \dim_\Q \Span_\Q (B^\infty_\lambda \cap \ZK) = j \},
	\end{displaymath}
	for $j = 1, \ldots, r$, where $B^\infty_\lambda$ is the ball of radius
    $\lambda$ with respect to the sup-norm on $K_\R$.
\end{definition}

The reason for using the sup-norm is the following bound.
\begin{lemma}\label{lem:ZKactionbound}
	Let~$L$ be a module lattice, $x\in L$ and~$a\in K_\R$. Then
	\[
	\|a x\| \le \|a\|_\infty \cdot \|x\|.%
	\]
\end{lemma}
\begin{proof}
	Embed~$L$ in~$K_\R^r$. Then the action of $a \in K_\R$ is component-wise.
\end{proof}

The successive minima of the ring of integers $\ZK$ can be estimated.

\begin{definition} \label{def:imbalance-gamma}
	Let $\Gamma_K = \sup_L \lambda_d(L)/\lambda_1(L)$, where $L$ ranges over all ideal lattices in $K$.
\end{definition}

\begin{lemma}[{\cite[Lemma 2.13]{BPW25}}] \label{lemma:rank1prop}
	For any $\ZK$-module lattice $L$ of rank 1, we have:
	\begin{enumerate}[(i)]
		\item \label{item:gap-bound} $\lambda_d(\ZK)/\sqrt{d} \leq \Gamma_K \leq \lambda^\infty_d(\ZK) \leq |\Delta_K|^{1/d}$. 
		\item \label{item:gap-bound-cyc} If $K$ is a cyclotomic field, then $\Gamma_K = 1$.
		\item \label{item:covering-bound} $\lambda_d(L) \leq \sqrt{d} \cdot \Gamma_K \cdot \det(L)^{1/d}$.
		\item \label{item:lower-bound-lambda1} $\lambda_1(L) \geq
		\sqrt{\frac{d}{|\Delta_K|^{1/d}}}   \cdot \det(L)^{1/d}$.
	\end{enumerate}
\end{lemma}

The relation between the two types of successive minima is given in the following result.

\begin{lemma}\label{lem:Kminima}
	For any~$1\le j\le rd$, we have
	\[
	\lambda_{\lceil j/d \rceil}^K(L) \le \lambda_{j}(L) \le \Gamma_K\lambda_{\lceil j/d \rceil}^K(L).
	\]
\end{lemma}
\begin{proof}
	Fix~$(k,i)$.
	Let~$\lambda = \lambda_{d(k-1)+i}(L)$ and~$S$ the set of vectors in~$L$ of
	length at most~$\lambda$. By definition~$\dim_\Q \Span_\Q(S)\ge d(k-1)+i$.
	But then~$\dim_K\Span_K(S) \ge \frac{d(k-1)+i}{d} > k-1$ and
	therefore~$\dim_K\Span_K(S)\ge k$, so that~$\lambda_k^K(L) \le \lambda$: this
	proves the first inequality.

	Let~$(u_i)_{i=1}^k$ be $K$-linearly independent vectors in~$L$ of length at most~$\lambda_k^K(L)$.
	For each $i$, let $(v_{ij})_{j=1}^d$ be $\Q$-linearly independent vectors in $\ZK u_i$ of length at most 
	$$\lambda_d(\ZK u_i) \leq \Gamma_K \lambda_1(\ZK u_i) \leq \Gamma_K \lambda_k^K(L).$$
	The family $(v_{ij})_{i,j}$ contains $dk$ many $\Q$-linearly independent vectors of length at most $\Gamma_K \lambda_k^K(L)$, hence $\lambda_{d(k-1)+i}(L) \le \lambda_{dk}(L) \leq \Gamma_K \lambda_k^K(L)$.
\end{proof}

The $n$-th successive minima of $\alpha$-balanced lattices can be bounded by a power of $\alpha$ multiplied by the root determinant of that lattice.
\begin{lemma} \label{lemma:balancedlambdanbound} Let $M$ be an $\alpha$-balanced $\ZK$-module lattice of rank $r$. Then 
 \[ \lambda_{rd}(M) \leq \Gamma_K \cdot \sqrt{rd} \cdot \alpha^{r-1} \cdot \det(M)^{1/(rd)}. \]
\end{lemma}
\begin{proof} We use Minkowski's second theorem \cite[Chap. VIII, Thm. 5]{cassels-intro-gon} and \Cref{lem:Kminima}, we write $\lambda_j = \lambda_j(M)$ and $n = rd$, to obtain 
\[ \frac{\det(M)^{1/n}}{\lambda_{n}} \geq \frac{\sqrt{\pi}}{2 \cdot \Gamma(\frac{n}{2} + 1)^{1/n}} \left( \prod_{j =1}^{n}\frac{\lambda_j}{\lambda_{n}} \right)^{1/n} \geq \sqrt{e\pi/(4n)}\left( \prod_{j =1}^{rd}\frac{\lambda^K_{\lceil j/d \rceil}}{\Gamma_K \lambda^K_{r}} \right)^{1/(rd)}   \]
\[ \geq \sqrt{e\pi/(4n)}\left( \prod_{j =1}^{r}\frac{(\lambda^K_j)^d}{\Gamma_K^d (\lambda^K_{r})^d} \right)^{1/(rd)}= \sqrt{e\pi/(4n)} \cdot \frac{1}{\Gamma_K} \cdot \left( \prod_{j =1}^{r}\frac{\lambda^K_j}{ \lambda^K_{r}} \right)^{1/r} \]
\[ \geq \sqrt{e\pi/(4n)} \cdot \frac{1}{\Gamma_K} \cdot \left( \prod_{j =1}^{r} \alpha^{-(j-1)} \right)^{1/r} = \sqrt{e\pi/(4n)} \cdot \frac{\alpha^{-(r-1)}}{\Gamma_K}    \]
Rewriting and using that $4/(e\pi) < 1$ yields the result.
\end{proof}

We will also use the following simple bound on the balancedness of submodules.

\begin{lemma} \label{lemma:qsparsificationqbalanced} Let $M$ be an $\alpha$-balanced rank $r$ module lattice and let $M' \subseteq M$ be an index $q$ sub-module lattice. Then $M'$ is $\alpha  \cdot q$-balanced. 
\end{lemma}
\begin{proof} Suppose, to derive a contradiction, that $M'$ is \emph{not} $\alpha \cdot q$-balanced, i.e., $\frac{\lambda^K_{i+1}(M')}{\lambda^K_{i}(M')} > \alpha \cdot q$ for some $i \in \{1, \ldots,r-1\}$. Write $j$ for the smallest $i$ satisfying this imbalancedness property.

Write $v'_1, \ldots,v'_j \in M'$ for the vectors in $M'$ attaining $\lambda^K_{1}(M'), \ldots, \lambda^K_{j}(M')$ and $v_1, \ldots,v_j,v_{j+1}$ for the vectors in $M$ attaining $\lambda^K_{1}(M), \ldots, \lambda^K_{j}(M),\lambda^K_{j+1}(M)$.

By definition there exists a $k \in \{1,\ldots,j+1\}$ so that $v_k \notin  v'_1 \ZK +  \ldots + v'_j  \ZK$ (which is the module lattice generated by $v'_1,\ldots,v'_j$). We claim that $a \cdot v_k + M'$ for $a \in \{0,\ldots,q\}$ are all different cosets in $M$. Indeed, if two cosets were the same, we would have that $a \cdot v_k \in M'$ for some $a \in \{0,\ldots,q\}$, and hence 
\[ \lambda_{j+1}^K(M') \leq \| a \cdot v_k \| \leq  q \cdot \lambda_k^K(M) \leq q \cdot \lambda_{j+1}^K(M) . \]
But then 
\[  \frac{\lambda^K_{j+1}(M')}{\lambda^K_{j}(M')} \leq \frac{q \cdot \lambda^K_{j+1}(M)}{\lambda^K_{j}(M')} \leq \frac{q \cdot \lambda^K_{j+1}(M)}{\lambda^K_{j}(M)} \leq \alpha \cdot q, \]
which leads to a contradiction.

Hence, indeed, $a \cdot v_k + M'$ for $a \in \{0,\ldots,q\}$ are all different cosets in $M$, and count to $q+1$. But $|M/M'| = q$, which in turn is a contradiction. Hence, $M'$ must be $\alpha \cdot q$-balanced.
\end{proof}

\subsection{Probability}
\subsubsection{Probability distributions}

\begin{definition} For an $n$-dimensional Euclidean vector space $V$ and $\mathbf{x} \in V$, we write $\gaussian_\sd(\mathbf{x}) := \exp(- \pi \|\mathbf{x}\|^2/\sd^2)$
for the Gaussian function.
\end{definition}

\begin{lemma}[Gaussian weight lemma]
\label{lemma:total-gaussian-weight} \label{lemma:smoothing}
Let $\Lambda$ be an $n$-dimensional full-rank lattice in an $n$-dimensional Euclidean vector space, let $\mathbf{c} \in \mbox{span}(\Lambda)$ and 
$\sd \geq \sqrt{\frac{\log(2n(1+1/\eps))}{\pi}} \cdot \lambda_n(\Lambda)$
for some $\eps > 0$. Then we have
\[ \sum_{\ell \in \Lambda} \gaussian_\sd(\ell + \mathbf{c}) =
  \gaussian_\sd(\Lambda + \mathbf{c}) \in \left[1-\eps, 1+\eps\right] \cdot
  \frac{\sd^n}{\det(\Lambda)}.\]
\end{lemma}
\begin{proof} This is a combination of the bound on the smoothing parameter \cite[Lemma 3.3]{MicciancioRegev2007} and the proof of \cite[Lemma 4.4]{MicciancioRegev2007}.
\end{proof}

\begin{definition}[Gaussian distribution] \label{def:gaussian}
Let $V$ be a Euclidean vector space. For $\sd \in \R_{>0}$, we denote $\Gau_{V,\sd}(x) := \sd^{-n} \cdot e^{-\pi \|x\|^2/\sd^2} = \sd^{-n} \cdot \gaussian_{\sd}(x)$ for the Gaussian distribution over $V$, where $n = \dim(V)$ and where $\| \cdot \|$ is the length notion over $V$. 
\end{definition}

\begin{definition}[Discrete Gaussian distribution] \label{def:discretegaussian}
Let $\Lambda \subseteq V$ be a full-rank lattice in a Euclidean vector space. For $\sd \in \R_{>0}$, 
we define the discrete Gaussian over $\Lambda$ with center $c \in V$ by the rule
\[ \Gau_{\Lambda,\sd,c}(\ell) := \frac{\Gau_{V,\sd}(\ell + c)}{\Gau_{V,\sd}(\Lambda+c)} =  \frac{\gaussian_{\sd}(\ell + c)}{\gaussian_{\sd}(\Lambda+c)},  ~\mbox{ for } \ell \in \Lambda ,\]
 where $\Gau_{V,\sd}(\Lambda+c) := \sum_{\ell \in \Lambda} \Gau_{V,\sd}(\ell+c)$. In the case that $ c =0$, the center $c$ is omitted in the notation.
\end{definition}

\begin{lemma} \label{lemma:taildiscretegaussian} Let $\Lambda \subseteq V$ be a full-rank lattice in a Euclidean vector space. For $\sd \geq \sqrt{\frac{\log(8n)}{\pi}} \cdot \lambda_n(\Lambda)$ and $\kappa \geq 1/(2\pi)$, we have 
\[ \underset{v \from \Gau_{\Lambda,\sd,c}}{\mathbb{P}}[ \|v - c\| > \sqrt{n} \cdot \kappa \cdot \sd] \leq 4 (\kappa \sqrt{2 \pi e})^n \cdot e^{-\pi \kappa^2n} \]
 
\end{lemma}
\begin{proof} We have
\begin{align*} \underset{v \from \Gau_{\Lambda,\sd,c}}{\mathbb{P}}[ \|v - c\| > \sqrt{n} \cdot \kappa \cdot \sd] & = \frac{\gaussian_\sd( (\Lambda + c)\backslash \sqrt{n} \cdot \kappa \cdot \sd \cdot B_2  )}{\gaussian_\sd(\Lambda+c)} \leq 2 C^n \frac{\gaussian_\sd(\Lambda+c)}{\gaussian_\sd(\Lambda)} \\ &\leq 2 \cdot \frac{1 + 1/3}{1 - 1/3} \cdot C^n  \leq 4C^n \leq 4 (\kappa \sqrt{2 \pi e})^n e^{-\pi \kappa^2n}, \end{align*}
 where the first inequality follows from \cite[Lemma 2.10]{MicciancioRegev2007} and the second inequality follows from \Cref{lemma:smoothing} (with $\eps = 1/3$). Here $C = \kappa \sqrt{2 \pi e} \cdot e^{-\pi \kappa^2}$. This yields the claim.
\end{proof}

\subsubsection{Statistical distance and the data processing inequality} \label{sec:statistical-distance}
\begin{definition}[Statistical distance] \label{def:statisticaldistance}
Let $(\Omega,\mathcal{S})$ be a measurable space with probability measures $P,Q$. The statistical distance between $P$ and $Q$ is defined by the rule 
\[ SD(P,Q) = \sup_{X \in \mathcal{S}} |P(X) - Q(X)|. \]
In the present work we only consider discrete or continuous domains $\Omega$. For a discrete space $\Omega$, we have 
\[ SD(P,Q) =  \frac{1}{2} \sum_{x \in \Omega} |P(x) - Q(x)| =: \frac{1}{2} \|P - Q\|_1. \] 
For a continuous space $\Omega$ with probability densities $P,Q$, we have
\[ SD(P,Q) =  \frac{1}{2} \int_{x \in \Omega} |P(x) - Q(x)| =: \frac{1}{2} \|P - Q\|_1. \] 
Often, in this work, due to the equivalence of these notions (up to a constant $\frac{1}{2}$) we will describe closeness of probability distributions in terms of the distance notion $\| \cdot \|_1$, instead of $SD(  \cdot, \cdot)$.

\end{definition}

The data processing inequality captures the idea that an algorithm (by just processing a single query) cannot
increase the statistical distance between two probability distributions. A proof can be found in, for example, \cite[\textsection 2.8]{Cover2006}.
\begin{proposition}[Data processing inequality] \label{theorem:dataprocessinginequality}
Let $(\Omega,\mathcal{S})$ be a measurable space with probability measures $P,Q$. Let $f$ be a (potentially probabilistic) function on $\Omega$. Then
\[  \|f(P) - f(Q) \|_1 \leq \|P-Q\|_1. \]
\end{proposition}

Statistical distance is well compatible with conditional events. If two distributions are close, the conditional counterparts are also close, where the statistical distance is multiplied by the probability of the conditioned event happening.

\begin{lemma} \label{lemma:conditionalvariation} 
Let $(\Omega,\mathcal{S})$ be a measurable space with probability measures $P,Q$. Let $U \in \mathcal{S}$ be an event with non-zero weight for both $P$ and $Q$. Then 
\[  SD(P|_{U},Q|_{U}) \leq 2 \cdot P\left(U %
\right) ^{-1} SD(P ,Q),  \]
where $P|_{U},Q|_{U}$ denotes $P$ respectively $Q$ conditioned on the event $U$.
\end{lemma}

\begin{proof}
We have, by the law of conditional probability, writing $p = P(U)$ and $q = Q(U)$,
\begin{align*} SD(P|_{U},Q|_{U}) & = \sup_{X \in \mathcal{S}} \left |\frac{P(X \cap U)}{p}  - \frac{Q(X \cap U)}{q} \right|,
\\ 
& = \frac{1}{p} \sup_{X \in \mathcal{S}} \left |P(X \cap U)  - \frac{p}{q} \cdot Q(X \cap U) \right| \\
& \leq \frac{1}{p} \sup_{X \in \mathcal{S}} \left(  |P(X \cap U)  - Q(X \cap U) | +  |1 - \frac{p}{q} |\cdot Q(X \cap U) \right) 
\\ 
& \leq \frac{1}{p} \sup_{X \in \mathcal{S}}  |P(X \cap U)  - Q(X \cap U) |  +  \frac{q-p}{p}  \leq \frac{2}{p} \cdot  SD(P,Q).
\end{align*}

\end{proof}

\subsection{Computational problems}
We consider the following three types of ``shortest vector problems'' in lattices.

\begin{computationalproblem}[Shortest Vector Problem (SVP$_\gamma$)] \label{problem:SVP}
Given as input a basis $\mathbf{B}$ of a lattice $L$ and a $\gamma \in \R_{\geq 1}$, the $\gamma$-\emph{shortest vector problem} is the computational task of finding a non-zero lattice vector $\vec{x} \in L$ that satisfies $\|\vec{x}\| \leq \gamma \cdot \lambda_1(L)$.
\end{computationalproblem}

\begin{computationalproblem}[Shortest Independent Vector Problem (SIVP$_\gamma$)] \label{problem:SIVP}
Given as input a basis $\mathbf{B}$ of an $n$-dimensional lattice $L$ and a $\gamma \in \R_{\geq 1}$, the $\gamma$-\emph{shortest independent vector problem} is the computational task of finding $\R$-linearly independent lattice vectors $\vec{x}_1, \ldots,\vec{x}_n \in L$ that satisfy $\|\vec{x}_i\| \leq \gamma \cdot \lambda_n(L)$ for all $i \in \{1,\ldots,n\}$.
\end{computationalproblem}

The parameter $\gamma$ in the definitions above is called \emph{approximation factor} and is generally written as a function in the dimension $n$ of the lattice.
No known polynomial time algorithm can solve these problems for $\gamma = \mbox{poly}(n)$. However, they are easy for $\gamma = 2^{O(n)}$: by~\cite[Proposition~1.12]{lenstra82:_factor}, the LLL algorithm finds a basis $(x_i)$ of $L$ such that $\|x_i\| \leq 2^{(n-1)/2}\lambda_i(L)$ for any $i$.

\subsection{Riemannian geometry, the determinant map, volumes} \label{sec:riem-geom}

The space of lattices $X_r(K)$ comes equipped with an invariant probability measure.
For computing with this measure, it is useful to work with an explicit realization coming from a Riemannian metric.
The latter generalizes the canonical metric for Minkowski space and allows us to compute the volumes of different spaces that show up while proving the Hecke equidistribution theorem.
We also consider in this section a way of ``splitting'' $\GL(r)$ into $\SL(r)$ and $\GL(1)$, as announced in Section~\ref{overview:bulk}.

\subsubsection{Riemannian structure} \label{sec:riemannian-structure}
We introduce a Riemannian metric on $\GL_r(K_\R)$ and its quotients. For this, we equip the Lie algebra~$M_r(K_{\R})$
of~$\GL_r(K_\R)$ with the positive definite inner product
\[
(x,y) \mapsto \tr_{K_\R/\R}\tr({}^tx^*y).
\]
This gives $\GL_r(K_\R)$ the structure of a Riemannian manifold with a
metric that is left-invariant by arbitrary elements and right-invariant
by~$\U_r(K_\R)\cdot \R_{>0}$.
In particular, it defines a volume form~$\mu_{\Riem}$ on $\GL_r(K_\R)$ that is a Haar measure and we note that $\GL_r(K_\R)$ is unimodular.
This also induces a Riemannian metric and measure on the quotient
\begin{equation} \label{eq:Yr}
	Y_r = \GL_r(K_\R) / \U_r(K_\R)\cdot \R_{>0}.
\end{equation}
The measure $\mu_{\Riem}$ further descends to $X_{r, \ida} = \Gamma_{\ida} \backslash Y_r$ and $X_r(K)$.
The probability measure $\mu$ on $X_r(K)$ is then equal to $\mu_{\Riem}(X_r(K))^{-1} \mu_{\Riem}$.
Throughout this section and much of Section \ref{sec:quantitative-equidistribution-bigsec}, we endow all spaces with the corresponding Riemannian measures and this defines all norms and inner products where the dependence on the space is given as a subscript.
Unless specified otherwise, this is the measure implicit in the notation $L^2(X_r)$ and the other $L^2$-spaces.

The map
\[
K_\R^\times/\U_1(K_\R) \to \R^{r_1+r_2}
\]
given by~$g = (g_v)_v \mapsto \log|\det(g)| = (\log|\det g_v|)_v$ is an
isometry of Riemannian manifolds, where for~$x = (x_i)_i\in \R^{r_1+r_2}$ we define
\begin{displaymath}
	\|x\|^2 =
\sum_{i=1}^{r_1}x_i^2 + 2\sum_{i=r_1+1}^{r_2}x_i^2.
\end{displaymath}

Let $H\subset \R^{r_1+r_2}$ be orthogonal to~$(1,\dots,1)$, so that the logarithmic embedding of
units  lies in~$H$. Let~$\pi_H  \colon \R^{r_1+r_2} \to H$ denote the orthogonal
projection onto~$H$.
We obtain an isometry $Y_1 \to H$ given by~$g\mapsto \pi_H (\log |\det g|)$. \label{def:hyperplane}

\subsubsection{The determinant map} \label{sec:det-map}

Let $\Delta \colon Y_r \to Y_1$
be the map induced by the determinant~$\GL_r(K_\R) \to K_\R^\times$.
For an ideal $\ida\subset \ZK$, this restricts to a map
\[
\Delta_\ida \colon X_{r,\ida} \to X_{1,\ida}.
\]
Pulling back functions, we obtain an injective map
\[
\Delta_\ida^* \colon L^2(X_{1,\ida}) \to L^2(X_{r,\ida})
\]
defined by $(\Delta_\ida^* f)(x) = f(\Delta_\ida x)$.
We denote the image of the pull-back by
\[
L^2_{\det}(X_{r,\ida}) = \Delta_{\ida}^*(L^2(X_{1,\ida})) \subset L^2(X_{r,\ida}),
\]
and, putting all connected components together,
\[
L^2_{\det}(X_{r}) = \bigoplus_{\ida\in\Cl(K)}L^2_{\det}(X_{r,\ida})  =
\Delta^*(L^2(X_{1})) \subset L^2(X_{r}).
\]

\begin{lemma}
	For non-negative measurable functions $f \colon Y_r \to \R_{\geq 0}$, we have the integration formula
	\begin{equation} \label{eq:coarea-formula}
		\int_{y\in Y_r}f(y)dy
		= \frac{1}{r^{\unitrk/2}} \int_{\delta\in Y_1} \left(\int_{y\in
		\Delta^{-1}(\delta)} f(y)dy\right)d\delta,
	\end{equation}
	and an analogous formula for~$X_{r,\ida}$.
	Integration on the fibers of $\Delta$ is done with respect to the restriction of the Riemannian metric.
\end{lemma}
\begin{proof}
	At the level of the Lie algebra, notice that the complement of the kernel of the derivative $D_\Delta$ of $\Delta$ consists of the scalar matrices.
	Locally, at one place $v$, the vector $X = \diag(1, \ldots, 1)$ has length $\sqrt{r}$ in the Riemannian metric.
	Since $\det(\exp(\lambda X)) = \exp(r \lambda)$, we see that $D_\Delta(X) = r \cdot 1$, where $1$ is the unit vector in the Lie algebra of $\R^\times$ or $\C^\times$.
	These computations now show that the Jacobian of $\Delta$ is $\sqrt{r}^{\unitrk}$.
	Put another way, $\Delta$ is a Riemannian submersion when the metric on~$Y_1$ is scaled by~$\frac{1}{\sqrt{r}}$.
	The statement now follows from the coarea formula (see \cite[Thm. 2.1]{nicolaescu}) applied to $\Delta$.
\end{proof}

It is convenient to have an explicit form of the fibers of $\Delta$.
Let~$g\in \GL_r(K_{\R})$ and~$\delta = \Delta(g)$. 
We have
\begin{equation} \label{eq:fibre}
	\Delta^{-1}(\delta) = \left(\SL_r(K_\R) \cdot\R_{>0} / (g\SU_r(K_\R)g^{-1}) \cdot \R_{>0}\right) g
\end{equation}
and the analogous formula
\begin{displaymath}
	\Delta_\ida^{-1}(\delta) = \left(\Gamma_\ida \backslash \Gamma_\ida \cdot \SL_r(K_\R) \cdot \R_{>0}/ (g\SU_r(K_\R)g^{-1})\cdot\R_{>0}\right) g.
\end{displaymath}
We shall often encounter the volume of these fibers in computations.

\begin{lemma}
	The volume $\mu_{\Riem}(\Delta_\ida^{-1}(\delta))$ is equal to $\mu_{\Riem}(\Delta_\ida^{-1}(1))$ and is therefore independent of $\delta$.
\end{lemma}
\begin{proof}
	Define $\Gamma_\ida^1 = \SL_r(\ZK,\ida)$ and notice that a fundamental domain for the left action of $\Gamma_\ida^1$ on $\SL_r(K_\R)$ serves as a fundamental domain for the left action of $\Gamma_\ida$ on $\Gamma_\ida \cdot \SL_r(K_\R)$, as well.
	Using that $\mu_{\Riem}$ is bi-invariant several times, we have that
	\begin{align*}
		\mu_{\Riem}(\Delta_\ida^{-1}(\delta))
		&= \mu_{\Riem}\left(\Gamma_\ida \lquo \Gamma_\ida \SL_r(K_\R) / (g\SU_r(K_\R)g^{-1}) \right) \\
		&= \frac{\mu_{\Riem} (\Gamma_\ida^1 \lquo \SL_r(K_\R))} {\mu_{\Riem}(g\SU_r(K_\R)g^{-1})} =  \mu_{\Riem}(\Delta_\ida^{-1}(1))
	\end{align*}
	for all $\delta$.
\end{proof}

Using the integration formula above, we now construct the orthogonal projection onto $L^2_{\det}(X_{r})$.
For  this, define~$\Delta_\ida' \colon L^2(X_{r,\ida}) \to L^2(X_{1,\ida})$ by
\[
\Delta_\ida'(f)(\delta) = \int_{x\in \Delta_\ida^{-1}(\delta)}f(x) \, dx.
\]
Next, define the operator
\begin{equation} \label{def:pi-det}
	\pi_{\det}^{\ida} = \mu_{\Riem}(\Delta_{\ida}^{-1}(1))^{-1} \Delta_{\ida}^* \Delta_{\ida}' \colon L^2(X_{r,\ida}) \to L^2(X_{r,\ida}),
\end{equation}
and let $\pi_{\det}$ be the direct sum of the operators $\pi_{\det}^{\ida}$ over the class group. 

\begin{lemma}
	The operator $\pi_{\det}^{\ida}$ is the orthogonal projection onto~$L^2_{\det}(X_{r,\ida})$.
\end{lemma}
\begin{proof}
	Let~$f\in L^2(X_{r,\ida})$ and~$g\in L^2(X_{1,\ida})$.
	Using the integration formula \eqref{eq:coarea-formula}, we have
	\begin{align*}
		\langle \Delta_{\ida}'f, g\rangle_{X_{1,\ida}}
		&= \int_{\delta\in X_{1,\ida}} \Delta_{\ida}'f(\delta)\overline{g(\delta)} \, d\delta 
		= \int_{\delta\in X_{1,\ida}} \left(\int_{x\in \Delta_{\ida}^{-1}(\delta)}f(x)dx\right)\overline{g(\delta)} \, d\delta \\
		&= \int_{\delta\in X_{1,\ida}} \left(\int_{x \in \Delta_{\ida}^{-1}(\delta)} f(x) \overline{g(\Delta_{\ida}(x))} \, dx \right) \, d\delta \\
		&= \int_{\delta\in X_{1,\ida}} \left(\int_{x\in \Delta_{\ida}^{-1}(\delta)}f(x)\overline{(\Delta_{\ida}^* g)(x)} \, dx\right) \, d\delta \\
		&= r^{\frac{\unitrk}{2}} \int_{x\in X_{r,\ida}}f(x)\overline{(\Delta_{\ida}^* g)(x)} \, dx 
		= r^{\frac{\unitrk}{2}} \langle f,\Delta_{\ida}^* g \rangle_{X_{r,\ida}}.
	\end{align*}
	Moreover, we compute that
	\begin{align*}
		\Delta_{\ida}'\Delta_{\ida}^*g(\delta)
		&= \int_{x\in \Delta_{\ida}^{-1}(\delta)}(\Delta_{\ida}^* g)(x) \, dx 
		= \int_{x\in \Delta_{\ida}^{-1}(\delta)}g(\Delta_{\ida}(x)) \, dx \\
		&= \int_{x\in \Delta_{\ida}^{-1}(\delta)}g(\delta) \, dx
		= \mu_{\Riem}(\Delta_{\ida}^{-1}(\delta)) g(\delta)
		= \mu_{\Riem}(\Delta_{\ida}^{-1}(1)) g(\delta).
	\end{align*}
	In other words, we have shown that
	\[
	\Delta_{\ida}'\Delta_{\ida}^* = \mu_{\Riem}(\Delta_{\ida}^{-1}(1)) \cdot \id.
	\]
	We finally deduce the formula
	\begin{equation} \label{eq:norms-under-Delta-ida}
		\norm{\Delta_{\ida}^\ast f}_{X_r}^2 = \mu_{\Riem}(\Delta_{\ida}^{-1}(1)) r^{\frac{-\unitrk}{2}} \norm{f}_{X_1}^2,
	\end{equation}
	for any function $f \in L^2(X_1)$, by piecing together the computations above.

	We have~$\pi_{\det}^{\ida}(L^2(X_{r,\ida})) \subset L^2_{\det}(X_{r,\ida})$ and, by the properties above, $\pi_{\det}^{\ida}$ is self-adjoint and restricts to the identity on~$L^2_{\det}(X_{r,\ida})$.					
\end{proof}

\subsubsection{Distance functions and volumes} \label{sec:volume-comps}
Although the Riemannian structure provides a notion of distance, we require two finer ways of measuring it.

\begin{definition} \label{def:tau-rho}
	For~$x\in\R^{r_1+r_2}$, let~$\|x\|_H = \|\pi_H(x)\|$, and
	for~$g\in\GL_r(K_\R)$ define
	\[
	\tau(g) = \| \log \lvert \det g \rvert \|_H^2.
	\]
	For~$\K =\R$ or~$\C$, let $\|\cdot\|_{\op}$ denote the operator norm with respect to the Euclidean norm
	on~$\K^r$. For~$g = (g_v)_v \in \GL_r(K_\R) = \prod_v \GL_r(K_v)$, define
	\[
	\rho(g) = \max_v \log \max \left(\frac{\|g_v\|_{\op}}{\abs{\det
			g_v}^{\frac{1}{r}}}, \frac{\|g_v^{-1}\|_{\op}}{\abs{\det g_v^{-1}}^{\frac{1}{r}}}\right).
	\]
\end{definition}

The functions defined above satisfy the following properties.
For all~$g, h\in \GL_r(K_\R)$ we have the inequality $\rho(gh) \le \rho(g)+\rho(h)$.
Moreover, for all~$g\in\GL_r(K_\R)$, $u\in\U_r(K_\R)$ and~$a\in K_\R^{\times}$ we have
\[
\rho(g) = \rho(gu) = \rho(ug) = \rho(ag) = \rho(g^{-1})
\]
and if~$a\in \R_{>0}$ then
\[
\tau(g) = \tau(gu) = \tau(ug) = \tau(ag) = \tau(g^{-1}).
\]
In particular, $\rho$ and~$\tau$ both descend to~$Y_r$.
One should think of~$\rho$ and~$\tau$ as being a ``distance to identity'' on
the~$\SL(r)$ part, respectively on the~$\GL(1)$ part.
We also define balls for the former as
\begin{equation} \label{def:Bt}
	B(t) = \{g\in\SL_r(K_\R) \mid \rho(g)\le t\}.
\end{equation}

We now compute the volumes of certain spaces, including the balls defined above and the full space $X_r$.
For this, we use two estimates from~\cite{MaireP}.

\begin{lemma}
	We have
	\[
	\log \mu_{\Riem}(\SU_r(K_\R)) \geq -\frac{d}{4}r^2\log r.
	\]
\end{lemma}
\begin{proof} We apply \cite[Proposition 11]{MaireP} to lower bound the volume of the local parts, after which we sum over complex and real places. We use the notation $a$ for $r$ in \cite[Proposition 11]{MaireP}, and $r$ for $d$ in \cite[Proposition 11]{MaireP}.

For the real case holds that $a = r/2$ if $r$ is even, and $(r-1)/2$ otherwise. Using the bound $j! \leq j^j$, and using the fact that $m_k \leq r$, we have 
\begin{align*} - \log \mu_{\Riem}(\SU_r(\R)) & \leq \sum_{k=1}^a \log(m_k!) \leq \sum_{k=1}^a m_k \log(r) \leq \log(r) r^2/4.
\end{align*}
Indeed, in the case that $r$ is odd, $m_k = 2k-1$, yielding $\sum_{k=1}^a m_k = a^2 \leq r^2/4$. In the case that $r$ is even, $m_k = 2k-1$ except for $m_a = a-1$, for which we can then deduce that $\sum_{k=1}^a m_k = (a-1)^2 + (a-1) = a(a-1) = \frac{(r-1)(r-3)}{4} \leq r^2/4$.

For the complex case, we have $a = r-1$ and $m_k = k$, yielding
\begin{align*} - \log \mu_{\Riem}(\SU_r(\C)) & \leq \sum_{k=1}^a \log(m_k!) \leq \sum_{k=1}^a m_k \log(r) \leq \log(r) a(a+1)/2 \leq \log(r) r^2/2
\end{align*}
Hence, summing over all places, $-\log \mu_{\Riem}(\SU_r(K_\R)) \leq d\log(r)r^2/4$, which finishes the proof.

\end{proof}

\begin{lemma} \label{lemma:lowerboundintegral} Let $\K$ be either $\C$ or $\R$. Then, for $t \leq 1$, we have 
 \begin{align*} I := \int_{(a_i)_i \in \Delta_t^*} \prod_{1 \leq i < j \leq r}
   \sinh(a_i - a_j)^{[\K:\R]} \geq
 \left(\frac{t}{4r^2}\right)^{\frac{(r-1)(r[\K:\R]+2)}{2}}
 \end{align*}
 where $\Delta^*_t = \{ (a_1,\ldots,a_{r-1}) \in \R ~|~  t > a_1 > \ldots > a_{r-1} > a_{r} := -\sum_{i = 1}^{r-1} a_i > -t\}$.
\end{lemma}
\begin{proof} We follow the same steps as in the proof of \cite[Proposition
  14]{MaireP}, where we use the assumption $t \leq 1$ instead. We write $g = [\K:\R] \in \{1,2\}$ for conciseness.

We apply \cite[Lemma 13]{MaireP} with $k = r-1$ to find intervals $[\alpha_i,\beta_i]$ (for $i \in \{1,\ldots,r-1\}$) satisfying the properties (1) - (6) of  \cite[Lemma 13]{MaireP}. For a certain reordering $\sigma \in S_{r-1}$, we put $Q := \prod_{i = 1}^{r-1} [t \alpha_{\sigma(i)}, t \beta_{\sigma(i)}]$. By properties (1) and (2) of \cite[Lemma 13]{MaireP} we can deduce that 
 \[ I := \int_{(a_i)_i \in \Delta_t^*} \prod_{1 \leq i < j \leq r} \sinh(a_i - a_j)^{g} da_i \geq  \int_{(a_i)_i \in Q}  \prod_{1 \leq i < j \leq r} \sinh(a_i - a_j)^{g} da_i. \]
 
The function $x \mapsto \sinh(x) \exp(-x) = \frac{1 - \exp(-2x)}{2}$ is increasing, and we have $\frac{1-\exp(-2x)}{2} \geq x/2$ for $x < 1/2$. Hence, for $x \geq \frac{t}{4r^2}$, 
\[ \sinh(x) = \frac{1-\exp(-2x)}{2} \cdot \exp(x) \geq \frac{1-\exp(-\frac{t}{2r^2})}{2} \cdot \exp(x) \geq \frac{t}{4r^2} \cdot \exp(x),   \]
since $\frac{t}{2r^2} \leq 1/2$. This yields
\begin{align*} I & \geq \left(\frac{t}{4r^2}\right)^{r(r-1)g/2} \int_{(a_i)_i \in Q}  \prod_{1 \leq i < j \leq r} \exp(a_i - a_j)^{g} da_i
\\ & \geq \left(\frac{t}{4r^2}\right)^{r(r-1)g/2} \int_{(a_i)_i \in Q}  \exp(g \underbrace{\sum_{i<j} (a_i - a_j)}_{\beta}) da_i.\end{align*}
In the proof of \cite[Proposition 14]{MaireP}, we see that $\beta = 2\sum_{i=1}^{r-1} (r-i)a_i$. By properties (3) and (6) of  \cite[Lemma 13]{MaireP} we deduce that $a_i \geq t/4$ for at least $\lfloor r/5 \rfloor$ intervals $[t\alpha_i,t\beta_i]$ (meaning, $\alpha_i \geq 1/4$ for these intervals), and that all intervals have width at least $t/(4r^2)$. Hence, 
\begin{align*} I & \geq \left(\frac{t}{4r^2}\right)^{r(r-1)g/2} \int_{(a_i)_i \in Q}  \exp(2g \sum_{i<j}(r-i)a_i) da_i \\ 
&\geq \left(\frac{t}{4r^2}\right)^{r(r-1)g/2} \cdot \left( \frac{t}{4r^2} \right)^{r-1} \cdot \exp( 2gt \sum_{i =1}^{\lfloor r/5 \rfloor} i/4)  \\
& \geq \left(\frac{t}{4r^2}\right)^{r(r-1)g/2} \cdot \left( \frac{t}{4r^2} \right)^{r-1} \cdot \exp( \frac{gr^2t}{200}) \geq \left(\frac{t}{4r^2}\right)^{\frac{(r-1)(rg+2)}{2}}, \end{align*}
which is what we wanted to proof.
\end{proof}

\begin{lemma}\label{lem:ballvol}
	Let~$t\leq 1$, and~$r\ge 2$.
	We have
	\[
	\log \mu_{\Riem}(B(t)) \ge - \log(4r^2/t) \cdot dr^2.
	\]
\end{lemma}
\begin{proof} Noting that $\mu_{\Riem}(B(t)) = \prod_{\nu} I^{(\nu)}$, where
  $I^{(\nu)} = I$ as in \Cref{lemma:lowerboundintegral} with $\K = K_{\nu}$ (which is $\R$ or $\C$), we compute (using $\sum_{\nu} [K_\nu:K] = d$ and $\sum_{\nu} 1 \leq d$),

\begin{align} \log \mu_{\Riem}(B(t)) &  \geq - \log(4r^2/t) \cdot (r-1) \cdot (d + r d/2) \nonumber \\ & \geq -\log(4r^2/t) \cdot d \cdot (r-1)(r+2)/2  \geq - \log(4r^2/t) \cdot dr^2. \nonumber   \end{align}
\end{proof}

\begin{proposition}
	We have
	\[
	\mu_{\Riem}(X_r) = \sqrt{d}r^{\frac{r_2+1}{2}}2^{-\frac{r_2}{2}}\mu_{\Riem}(\SU_r(K_\R))^{-1}
	h_KR_K |\Delta_K|^{\frac{r^2-1}{2}}\prod_{j=2}^r \zeta_K(j).
	\]
\end{proposition}
\begin{proof}
	Let~$\ida\subset\ZK$ be an ideal. 
	By the computations and integration formula in Section \ref{sec:det-map}, we have
	\begin{align*}
		\mu_{\Riem}(X_{r,\ida}) &= \int_{x\in X_{r,\ida}}dx = r^{-\frac{\unitrk}{2}}\int_{\delta\in X_{1,\ida}}\int_{x\in \Delta_\ida^{-1}(\delta)}d\delta \\
		&= r^{-\frac{\unitrk}{2}}\int_{\delta\in X_{1,\ida}}\mu_{\Riem}(\Delta_\ida^{-1}(\delta))d\delta \\
		&= r^{-\frac{\unitrk}{2}}\int_{\delta\in X_{1,\ida}}\mu_{\Riem}(\Delta_\ida^{-1}(1))d\delta \\
		&= r^{-\frac{\unitrk}{2}} \mu_{\Riem}(X_{1,\ida})
		\frac{\mu_{\Riem}\left(\SL_r(\ZK,\ida) \lquo \SL_r(K_\R) / \R_{>0}\right)} {\mu_{\Riem}(\SU_r(K_\R))}.
	\end{align*}
	
	By Prasad's formula, %
	we have (see~\cite[Proposition 18]{MaireP}, which is also valid in the
	non-compact case):
	\[
	\mu_{\Riem}\left(\SL_r(\ZK,\ida) \lquo \SL_r(K_\R) / \R_{>0}\right)
	= r^{\frac d 2}|\Delta_K|^{\frac{r^2-1}{2}}\prod_{j=2}^r \zeta_K(j).
	\]
	Finally, we have %
	\[
	\mu_{\Riem}(X_{1,\ida}) = \sqrt{d}2^{-\frac{r_2}{2}}R_K.
	\]
\end{proof}

\begin{lemma}\label{lem:louboutin-residue}
  The residue~$\zeta_K^*(1)$ of~$\zeta_K$ at~$1$ satisfies
  \[
    \zeta_K^*(1) \le \left(\frac{e \log |\Delta_K|}{2(d-1)}\right)^{d-1}
    \le |\Delta_K|^{\frac 12}.
  \]
\end{lemma}
\begin{proof}
  The first inequality is~\cite[Equation (2)]{louboutin00}.
  The second one follows from applying the inequality  $\frac{e \log |x|}{|x|}
  \leq 1$, which holds for all $x$, to $x = |\Delta_K|^{\frac{1}{2(d-1)}}$.
\end{proof}

\begin{lemma}\label{lem:louboutin}
	We have
	\[
	\log(h_KR_K) \le \log|\Delta_K| + O(1)
	\]
\end{lemma}
\begin{proof}
    Apply Lemma~\ref{lem:louboutin-residue} and the analytic class number
    formula.
\end{proof}

\begin{remark} 
	In the original statement of Louboutin \cite[Equation (2)]{louboutin00}, one can see that, next to $\sqrt{|\Delta_K|}$, the dominant factor is $\left(\frac{e \log |\Delta_K|}{2(d-1)}\right)^{d-1}$ which might be much smaller than $\sqrt{|\Delta_K|}$. 
	Hence, the bound above, though simple, is not tight and might be improved to get a better approximation factor in the main result of this paper.
\end{remark}

The results above finally imply a key inequality. 
\begin{lemma}\label{lem:spacevol}
	We have
	\[
	\log \mu_{\Riem}(X_r) \le \frac{dr^2}{4}\log r + \frac{r^2}{2}\log|\Delta_K| + O(\log\log|\Delta_K| + dr^2).
	\]
\end{lemma}

\subsection{Automorphic theory} \label{sec:automorphic-theory}
The purpose of this section is to explain the spectral decomposition of $L^2(\X_r)$ and analyze the action of Hecke operators on the different components.
Our main references here are \cite[Sec. 3.2, Sec. 4.1]{clozel-ullmo} and \cite[Sec. 10]{getz-hahn}. 

\subsubsection{The spectral decomposition}
We recall first that standard parabolic subgroups $P \subset \GL_r$ are in correspondence with partitions $\sum_i r_i = r$.
Given such a partition, called $\wp = (r_i)$ for short, the Levi subgroup
$M_{\wp}$ of the corresponding parabolic $P_{\wp}$ is isomorphic to $\prod_i
\GL_{r_i}$ (the group~$P_\wp$ is the group of blockwise upper triangular
matrices with blocks of sizes given by the partition).
We attach to $M_{(r_i)}$ a certain space of characters, denoted by $\ida_{M_{\wp}}^\ast$, which we can interpret as a tuple of complex numbers or parameters.
We denote the subspace of purely imaginary parameters by $\Im(\ida_{M_{\wp}}^\ast)$.
If, for each $i$, we have an irreducible automorphic representation~$\pi_i$ appearing in the discrete spectrum of $L^2(\X_{r_i})$, and $\lambda \in \Im(\ida_{M_{\wp}}^\ast)$, then we can construct the induced representation $I(\bigotimes_i \pi_i, \lambda)$, as in \cite[Sec. 10.1]{getz-hahn}.

A celebrated theorem of Moeglin and Waldspurger describes the discrete spectrum in terms of Speh representations.
To introduce the latter, let $N \in \Z_{>0}$ and let $s$ be a divisor of $N$.
There is a unique standard Levi $M \subset \GL_N$ isomorphic to $\prod_{i = 1}^{N/s} \GL_s$.
Given a cuspidal automorphic representation $\sigma$ of $\GL_s$, we can define the Speh representation $\Speh(\sigma, N/s)$ for $M$ (see \cite[Sec. 10.7]{getz-hahn}) that occurs in the discrete spectrum of $L^2(\X_N)$.

The spectral theorem of Langlands now states that $L^2(\X_r)$ is a sub-module of
\begin{equation} \label{eq:spectral-thm}
	\bigoplus_{\wp \colon \sum_i r_i = r} \, \bigoplus_{s_i\mid r_i} \widehat{\bigoplus_{\sigma_i}}\int_{\lambda \in \Im(\ida_{M_{\wp}}^\ast)}^{\oplus}
	I(\bigotimes_i\Speh(\sigma_i,r_i/s_i),\lambda)d\lambda,
\end{equation}
where~$\sigma_i$ ranges over cuspidal automorphic representations of $\GL_{s_i}(K)$, and $\int^\oplus$ denotes a direct integral decomposition.
We can identify the components that make up $L^2_{\det}(\X_r)$ by noting that
\begin{displaymath}
	L^2(\X_1) = \bigoplus_{\chi} \C \cdot \chi,
\end{displaymath}
where $\chi$ runs over all (unitary) Hecke characters of $K$, and that $L^2(\X_1)$ is isomorphic to $L^2_{\det}(\X_r)$ by the map $\Delta^\ast$.
It is known that $\Speh(\chi, r)$, for a Hecke character $\chi$, is the one-dimensional representation of $\GL_r(\adel_K)$ given by $\chi \circ \det$.
Thus, $L^2_{\det}(\X_r)$ is the contribution of the terms corresponding to the trivial partition $r = r$ and $s = 1$ in \eqref{eq:spectral-thm}.

\subsubsection{Hecke operators} \label{sec:Hecke-ops}

In terms of lattices, the Hecke operator $T_\p$ corresponds to uniform averaging over submodules $N\subset M$ such that  at $M/N\cong \ZK/\p$ at every module lattice $M$.
More precisely, for a function $f \in L^2(X_r)$, we define
\begin{displaymath}
	T_\p f(M) = \frac{1}{D_\p}\sum_{\substack{N \subset M \\ M/N\cong \ZK/\p}}f(N),
\end{displaymath}
where, if $q = N(\p)$, the number of terms in the average is $D_\p = 1 + q + \ldots + q^{r-1}$.

Interpreting $X_r$ adelically, the operator $T_\p$ acts only at the place $\p$.
More precisely, let~$\pi_\p \in \Oc_\p$ be a uniformizer at~$\p$, and write
\begin{equation} \label{eq:def-Hecke-coset}
	\GL_r(\Oc_\p) \begin{pmatrix}
		\pi_\p & 0 & \cdots & 0 \\
		0 & 1 & & \\
		\vdots & & \ddots & 0 \\
		0 & \cdots & 0 & 1
	\end{pmatrix}\GL_r(\Oc_\p)
	= \bigsqcup_{g \in R_\p} \GL_r(\Oc_\p) g
\end{equation}
for some finite set~$R_\p$ (of size~$D_\p$).
For a function $f\in L^2(X_r)$ and $x \in \GL_r(\adel_K)$ we have
\[
T_\p f(x) = \frac{1}{D_\p}\sum_{g\in R_\p}f(xg^{-1}) = \frac{1}{D_\p}\sum_{g\in R_\p} g^{-1}\cdot f.
\]
This is well-defined because the action (from the right) of $\GL_r(\Oc_\p)$ is trivial on $X_r$ by definition.

Note that $T_\p \triv = \triv$, where $\triv$ is the constant $1$ function on $X_r$. 
We remark also that, by definition, a Hecke operator also acts on the irreducible representations occurring in $L^2(X_r)$.
On such representations, it acts by a scalar, the \emph{Hecke eigenvalue}.
In particular, it is an endomorphism of $L^2_{\det}(X_r)$.

Let~$\pi$ be an irreducible automorphic representation of~$\GL_r(\adel_K)$.
The representations relevant in our case are unramified at all primes $\p$, meaning that, they contain nonzero vectors fixed by $\GL_r(\Oc_\p)$.
Clearly, all automorphic representations appearing in the decomposition of $L^2(X_r) \subset L^2(\X_r)$ are unramified.
In this case, for every $\p$ we can attach to $\pi$ (in fact, to the component $\pi_\p$ at $\p$) its \emph{Satake parameters} $\alpha_1,\dots,\alpha_r\in \C$.
These describe the action of the Hecke operators at $\p$.
For instance (see \cite[(7.2)]{getz-hahn}), the eigenvalue of $T_\p$ on $\pi$ is
\[
\frac{1}{D_\p}q^{\frac{r-1}{2}}\sum_{k=1}^r \alpha_k.
\]
In our application, we are satisfied with bounding the eigenvalue of $T_\p$ by
\begin{displaymath}
	q^{-(r-1)/2} \sum_{k=1}^r \alpha_k,
\end{displaymath}
using the formula for $D_\p$.

For example, the Satake parameters of the trivial representation of $\GL_r(K_\p)$ are
\[
q^{-\frac{r-1}{2}}, q^{-\frac{r-1}{2}+1}, \dots, q^{\frac{r-1}{2}},
\]
and this corresponds to the fact that $T_\p$ acts by the scalar $1$ on constant functions.
However, $T_\p$ acts on most unramified automorphic representations with smaller eigenvalues and this is the source of equidistribution. 

\subsubsection{Eigenvalue bounds}
We now analyze the action of Hecke operators using the explicit spectral decomposition.
This is formally contained in the work of Clozel and Ullmo \cite{clozel-ullmo}, who treated the case $K = \Q$.
We follow their method and adjust it to cover the general number field case.

\begin{proposition} \label{prop:eigenvalue-bd}
	The operator norm of $T_\p$, defined with respect to the $L^2$-norm on $X_r$ endowed with $\mu_{\Riem}$, acting on the orthogonal complement of $L^2_{\det}(X_r) \subset L^2(X_r)$ is bounded by $r q^{-3/8}$.
\end{proposition}
\begin{proof}
	First, it is important to understand the Satake parameters of unramified cuspidal representations $\pi$ of $\GL_N$, since these are the building blocks of the spectral decomposition.
	The Generalized Ramanujan Conjecture (GRC) states that, if $\alpha_1, \ldots, \alpha_N$ are the Satake parameters of $\pi$ at $\p$, then $\abs{\alpha_i} = 1$ for all $i$.
	This conjecture seems far out of reach (see \cite{BB-Bull} for a survey), but there are useful bounds towards it.

	Let $\theta_N \geq 0$ be the exponent in the best known bound towards GRC, that is,
	\begin{displaymath}
		\abs{\alpha_i} \leq q^{\theta_N}
	\end{displaymath}
	for all $i$.
	We have (see \cite{BB-Annals}) $\theta_1 = 0$, $\theta_2 \leq 7/64$, $\theta_3 \leq 5/14$, $\theta_4 \leq 9/22$, and more generally $\theta_N \leq 1/2$ for all $N$.

	We can now compute eigenvalues of Hecke operators on the representations occurring in the spectral decomposition \eqref{eq:spectral-thm}, as in \cite[Sec. 4.1]{clozel-ullmo}.
	Let~$\p$ be a prime of~$K$ and let~$q = N(\p)$. 
	A representation~$\Speh(\sigma,m)$ is unramified at~$\p$ if an only if the cuspidal representation $\sigma$ is.
	In this case, its Satake parameters at~$\p$ are
	\[
	\alpha_{i}q^{\frac{m+1}{2}-j},\ i=1,\dots,s,\ j=1,\dots,m,
	\]
	where~$\alpha_1,\dots,\alpha_s$ are the Satake parameters of~$\sigma$ at~$\p$.

	If~$I(\otimes_i \pi_i,\lambda)$ contains nonzero~$\GL_r(\Oc_\p)$-invariant
	vectors, then all~$\pi_i$ are unramified at~$\p$ and the Satake parameters
	of the irreducible sub-quotients of~$I(\otimes_i \pi_i,\lambda)$ unramified
	at~$\p$ are the
	\[
	\alpha_{i,k}\zeta_i
	\]
	where the~$\alpha_{i,k}$ are the Satake parameters of the~$\pi_i$
	and~$|\zeta_i| = 1$ (because~$\lambda \in \Im(\ida_{M})$ is imaginary).
	Therefore, the eigenvalue of~$T_\p$ acting on the~$\GL_r(\Oc_\p)$-fixed
	points of the representation~$I(\otimes_i\Speh(\sigma_i,r_i/s_i),\lambda)$ is bounded in absolute value by
	\[
	\sum_i \sum_{j=1}^{r_i/s_i} \sum_{k=1}^{s_i} q^{\theta_{s_i} +
		\frac{r_i/s_i+1}{2}-j - \frac{r-1}{2}}.
	\]
	We bound this crudely by
	\begin{displaymath}
		r q^{\max(\theta_{s_i} + r_i / 2 s_i - r/2)}.
	\end{displaymath}

	We now estimate the exponent in various cases.
	First, note the exceptional case when $i = 1$, $r_1 = r$, and $s_1 = 1$, which occurs when considering representations in $L^2_{\det}(X_r)$.
	In that case our bound is simply $r$, which does not give any saving.
	Excluding this case, we can have for some $i$ that
	\begin{itemize}
		\item $s_i = 1$ and $r_i < r$: the exponent $\theta_1 + r_i/2 - r/2$ is at most $-1/2$;
		\item $s_i = 2$ (so that $r \geq r_i \geq 2$): the exponent $\theta_2 +
		r_i/4 - r/2 \leq 7/64 - r/4$ is bounded by $-25/64 \leq -3/8$;\footnote{We take $-3/8$ here only to make expressions in the rest of the paper cleaner. Any non-trivial bound gives the same qualitative result in this paper.}
		\item $s_i \ge 3$ (so that $r \geq r_i \geq 3$): using the general bound $\theta_{s_i} \leq 1/2$, the exponent is at most $\frac{1}{2} + r/6 - r/2 \leq 1/2 - r/3 \leq -1/2$.
	\end{itemize}
	This finishes the proof.
	Note that the first of these cases shows that, even assuming GRC, the best exponent we could hope for in general is $-1/2$.
\end{proof}

In the following, we consider averages of Hecke operators.
If $\mathcal{P}$ is a finite subset of the prime ideals of $K$, we write
\begin{equation} \label{eq:def-T-mathcal-P}
	T_{\mathcal{P}} = \frac{1}{\abs{\mathcal{P}}} \sum_{\p \in \mathcal{P}} T_\p.
\end{equation}
For $B \geq 0$, we can define $\mathcal{P}(B)$ as the set of prime ideals with norm at most $B$. \label{def:boundB}
The Extended Riemann Hypothesis implies that
\begin{displaymath}
	|\mathcal{P}(B)| \geq \frac{B}{2 \log B}
\end{displaymath}
for $B \geq \max((12 \log \abs{\Delta_K} + 8d + 28)^4, 3 \cdot 10^{11})$, where we recall that $d$ is the degree of $K$ and $\Delta_K$ is its discriminant.
This was shown in Lemma A.3 of \cite{C:BDPW20}. 
We also have the trivial upper bound $|\mathcal{P}(B)| \leq d B$ that follows from unique factorization.

\begin{corollary} \label{cor:spec-gap}
	Let $B \geq \max((12 \log \abs{\Delta_K} + 8d + 28)^4, 3 \cdot 10^{11})$. The operator norm of $T_{\mathcal{P}(B)}$ acting on the orthogonal complement of $L^2_{\det}(X_r) \subset L^2(X_r)$ is bounded by $20 \cdot rd \cdot B^{-3/8} \log(B)$.
\end{corollary}
\begin{proof}
	This follows from the standard technique of splitting the average into dyadic intervals.
	With $\alpha = -3/8$, we write
	\begin{displaymath}
		\sum_{N(\p) \leq B} N(\p)^{\alpha} \leq \sum_{k=1}^{\log_2(B)-1} \sum_{2^k \leq N(\p) \leq 2^{k+1}} N(\p)^{\alpha}.
	\end{displaymath}
	The inner sum is at most $d 2^{k+1} 2^{k \alpha}$ by the trivial upper bound on $\mathcal{P}(B)$.
	Removing constant factors, the outer sum now becomes a geometric series, namely
	\begin{displaymath}
		\sum_{k=1}^{\log_2(B)-1} 2^{k(1+ \alpha)} \leq 10 B^{1+\alpha},
	\end{displaymath}
	where $10$ is large enough considering the size of $\alpha$.
	
	Plugging in the lower bound for $\mathcal{P}(B)$, Proposition \ref{prop:eigenvalue-bd} implies the upper bound
	\begin{displaymath}
		20 rd \cdot \frac{\log(B)}{B} \cdot B^{1+\alpha} = 20 rd \cdot B^\alpha \log(B)
	\end{displaymath} 
	on the operator norm of $T_{\mathcal{P}(B)}$.
\end{proof}

\section{Rounding module lattices} \label{sec:rounding}
 \label{subsec:canonicalrep}

\newcommand{\mat}[1]{\mathbf{#1}}
\subsection{Introduction}
In the worst-case to average-case reduction of the present paper,
the average-case stems from a Haar-uniform distribution over 
the space of module lattices (see \Cref{sec:module-lats}). 
Due to the continuity of this latter space, such a uniform distribution 
cannot be adequately represented by a computer; indeed, computers (or, more
formally, Turing machines) can only process module lattices that are represented by rational numbers (bounded in size). We tackle this issue by 
using a probabilistic algorithm that, for an input module lattice $M$, outputs a random sample from a specific distribution $\distr(M) := \RoundPerf(M)$ over \emph{rational module lattices}. This algorithm (\Cref{alg:canonical}) can be seen as a probabilistic way of rounding the input module lattice $M$ to a geometrically close rational module lattice. The average-case distribution considered in this paper can be described by $\RoundPerf(M)$, where $M$ is sampled Haar-uniform over the space of module lattices. 

This rounding algorithm is a generalization of \cite[Algorithm 1]{C:BDPW20} to module lattices. It also closely resembles \cite[Algorithm 3.1]{AC:FelPelSte22}, with the difference that our version of the rounding algorithm forces the output module to be full-rank and is proven to be H\"older continuous; properties indispensable for the purposes of the current paper. 

This specific distribution $\RoundPerf(M)$ over rational module lattices has special properties in order to indeed 
resolve the issues coming from the continuity of the module-lattice space. Specifically the distribution $\RoundPerf$ satisfies 
\begin{enumerate}[(i)]
 \item Discreteness, efficiency and rationality: For each $M$, we have that $\RoundPerf(M)$ is a random module lattice supported on discrete set $S$ of \emph{rational module lattices}, each of which can be represented by a tuple of rational entries. Additionally, for any module lattice $M$, almost all of the weight of $\RoundPerf(M)$ is on a finite set $S' \subseteq S$. Moreover, the algorithm computing a sample from $\RoundPerf(M)$ is efficient.
 \item Independence of module representation: The distribution $\RoundPerf(M)$ does not depend on the choice of pseudo-basis of the module $M$. This makes $\RoundPerf$ a map from $X_r(K)$ to the distribution space $L^1(S)$ over the set $S$ of rational modules.
 \item Preservation of geometry: With high probability, a rational module lattice sample $R \from \RoundPerf(M)$ has almost the same geometry as $M$, meaning that solving respectively SVP, SIVP, etc., on $R$ allows for solving SVP, SIVP, etc., on $M$ (and vice versa). 
 \item Continuity: If $M$ and $M'$ are almost isomorphic, their associated distributions $\RoundPerf(M)$ and $\RoundPerf(M')$ are close in total variation distance.
\end{enumerate}
The discreteness, efficiency and rationality makes that any module lattice $M$ can be efficiently ``represented'' by a computer via the distribution $\RoundPerf(M)$, even if $M$ itself cannot be. The independence of module representation makes that $\RoundPerf(-)$ a map truly on modules (and not a pseudo-basis representation thereof). The geometry preservation makes this distribution representation \emph{useful} for the particular context of this paper: SVP-like problems are not (too much) distorted by the distribution representation. Lastly, continuity of $\RoundPerf$ allows quantifying the effects of \emph{discretization} of the input of $\RoundPerf$, which will be treated in detail in \Cref{sec:discretization}.

An intuitive way of thinking about a sample of $\RoundPerf(M)$ (which is the output distribution of \Cref{alg:canonical}) is by seeing it as a randomized rounding of the module $M$ to a close, rational module $M'$. This probabilistic rounding is then done in such a way that the continuity in the module lattice space is transferred to a continuous change of the probability weights on the rational output modules. 

The pseudo-algorithm computing the ``perfect'' rounding distribution $\RoundPerf(M)$ (\Cref{alg:canonical}) involves real arithmetic and continuous distributions, and can therefore not be computed by a Turing machine. Instead, we resort to a discrete variant of $\RoundPerf(M)$, an actual algorithm called $\Round(M)$, which approximates $\RoundPerf(M)$ within arbitrarily small statistical distance. The precise description of this discretization can be found in the proof of \Cref{lemma:runtime}.

The precise main result of this section is the following proposition, in which the properties (i)-(iv) precisely match those just explained.
\begin{proposition} \label{prop:rounding-algo}
	There exists an algorithm $\Round$ with balancedness parameter $\alpha \in \R_{>1}$ and error parameter $\eps_0 \in (0,1/2)$ that takes $\alpha$-balanced rank $r$ module lattices $(\mB_M,\mI = (\ma_i)_{i \in [r]})$ over $K$ as input, whose output distribution satisfies, for any $M$,
	\[ \|\Round(M) - \RoundPerf(M)\| < \eps_0, \]
	for a certain perfect distribution $\RoundPerf(M)$,  and where $\Round$ and $\RoundPerf$ satisfy the following properties.
	\begin{enumerate}[(i)]
	\item The output $(\mat{H}_R,(\mathfrak{h}_i)_{i \in [r]})$ of $\Round(M)$ is a rational module lattice that is bounded in size by
	$\poly(\size(\mB_M,\mI), \log(1/\eps_0))$. Moreover, the algorithm runs in time $\poly(\size(\mB_M), \allowbreak\max_i \size(\ma_i), \allowbreak\log(1/\eps_0))$.
	\item If $(\mB_M,\mI)$ and $(\mB'_M,\mI')$ represent the same module, we have that 
	\begin{displaymath}
	\RoundPerf(\mB_M,\mI) = \allowbreak\RoundPerf(\mB'_M,\mI'), 
	\end{displaymath}
	meaning that their output distributions are identical.
	\item For any $N \from \Round(M)$ there exists a full-rank matrix $Y \in K_\R^{r \times r}$ so that $M = Y \cdot N$, which satisfies $\cd(Y) := \|Y\| \cdot \|Y^{-1} \| \leq 1 + \frac{1}{2n} \leq 2$ and hence preserves SVP-like problems.

    For example, solving $\gamma$-SVP (resp. 
    SIVP) in $N$ allows for solving $2 \gamma$-SVP (resp. 
    SIVP) in $M$, with probability at least $1-\eps_0^{n^4}$.
	\item We have, for any module lattices  $M,M'$ we have
	\[ \|\RoundPerf(M) - \RoundPerf(M') \|_1 \leq 92n^3  \sqrt[4]{\log(12r/\eps_0)} \sqrt{d(M,M')}, \]
	where $d(M,M') := \min_\phi(\|\phi - I\|_2,\|\phi^{-1} - I \|_2)$ if there exists a module isomorphism $\phi: M \rightarrow M'$ between $M$ and $M'$ and $d(M,M') = \infty$ otherwise.
\end{enumerate}
\end{proposition}
\begin{proof} Item (i) is proven in \Cref{lemma:runtime}, item (ii) in \Cref{lemma:pseudobasisinvariant}, item (iii) in \Cref{lemma:matrixcomputations} and \Cref{lemma:preservesnonexact} using $(1 + \frac{1}{8n})(1 + \frac{1}{4n}) \leq (1 + \frac{1}{2n})$ (see also the proof of \Cref{lemma:randomizationbalanced}), and item (iv) in \Cref{lemma:holdercontinuous}. 
\end{proof}

\subsection{The rounding algorithm and its properties}
The rounding algorithm of this section is described in \Cref{alg:canonical}. We now prove the properties listed in the introduction.

\begin{algorithm}[ht]
    \caption{Rounding a module lattice to a near rational module lattice}
    \label{alg:canonical}
    \begin{algorithmic}[1]
    	\REQUIRE ~\\ \vspace{-2mm}
    	\begin{itemize}
    	 \item A balancedness parameter $\alpha \in \R_{\geq 1}$, \vspace{-2mm}
    	 \item A pseudo-basis $(\mB_M, (\mathfrak{a}_i)_{i \in [r]})$ of an $\alpha$-balanced rank $r$ module lattice $M$ with $\mathfrak{a}_i \subseteq \ZK$. \vspace{-2mm}
    	 \item An error parameter $\eps_0 \in (0,1/2)$.
    	\end{itemize}
    	\ENSURE A pseudo-basis $(H_R,(\mathfrak{h}_i)_{i \in [r]}))$ of a module lattice $R$ of rank $r$ where $H_R$ has coefficients in $K$ and $\mathfrak{h}_i \subseteq \ZK$ for all $i \in [r]$.
    	\STATE Put $\varsigma = 3 \cdot 2^{n} 
		\cdot \alpha^{r-1} \cdot \Gamma_K \cdot n \cdot \det(M)^{1/{n}}$ and $T = 8n^4 \cdot \sqrt{\log(12r/\eps_0)} \cdot \varsigma$.
    		\FOR{$i = 1$ to $r$} \label{line:loopfor}
                \STATE Pick $\hat{c} \in \{ x \in K_\R ~|~  \|x_\sigma\| = T \mbox{ for all } \sigma \}$ uniformly.   ~~\textcolor{gray}{Sample a center} \label{line:uniformcenter}
                \STATE Put $c = (\underbrace{0,\ldots,0}_{i-1}, \hat{c}, \underbrace{0 , \ldots, 0}_{r-i}) \in K_\R^r$.  
                \STATE Sample $\vec{v}_i \leftarrow \Gaussian_{M,\varsigma,c}$ from the discrete Gaussian (see \Cref{def:discretegaussian}) over $M$ with center $c$. Repeat until $\vec{v}_i$ is $K_\R$-linearly independent of $(\vec{v}_1,\dots,\vec{v}_{i-1})$. \label{line:gaussian}
    		\ENDFOR
    		\STATE Define the free $r$-module $N = \bigoplus_{i = 1}^r \ZK \cdot \vec{v}_i$; construct its basis $\mB_N$ by stacking $\vec{v}_i$ as columns. \label{line:freemodule}
        \RETURN the Hermite normal form $(H_R,(\mathfrak{h}_i)_{i \in [r]}))$ of the module $R$ generated by the pseudo-basis $(\mB_R := \mB_N^{-1} \mB_M,  (\mathfrak{a}_i)_{i \in [r]})$.
    \end{algorithmic}
\end{algorithm}
\newcommand{\dC}{\ddot{\mathcal{C}}}

The following lemma on the matrix $2$-norm and the determinant of the basis $\mB_N$ will turn out useful.
\begin{lemma} \label{lemma:matrixcomputations} We have, except with probability $(\eps_0)^{n^4}$,
\[  \|\mB_N\| \leq (1 + \frac{1}{8n}) T \mbox{, }  \|\mB_N^{-1} \| \leq (1 + \frac{1}{4n}) \frac{1}{T}  \]
 and 
 \[  |\det(\mB_N^{-1})|^{1/n} \leq (1 + \frac{1}{4}) T^{-1} \]
\end{lemma}
\begin{proof}
For conciseness, we write $\mu = \log(2r/\eps_0)$.
We start with computing the matrix $2$-norm of the matrices $\mB_N$ and $\mB_N^{-1}$. By the very definition of $\mB_N$, 
we can write $\mB_N = T \cdot J - E$, where $J$ is diagonal with on the diagonal entries elements of $\{ x \in K_\R ~|~ |x_\sigma| = 1 \mbox{ for all }  \sigma \}$ and where $T \in \R_{>0}$. Hence $\|\mB_N\|_2 = \| T \cdot J \| + \|E\| = T + \|E\|$. By the fact that $N$ is constructed by stacking Gaussian samples (see lines \ref{line:gaussian} and \ref{line:freemodule}), a single component $e_{ij}$ of $E \in K_\R^{r \times r}$ must satisfy 
$e_{ij} \in \{ x \in K_\R ~|~ |x_\sigma| \leq n^2 \cdot \mu \cdot \varsigma \mbox{ for all }  \sigma \}$ except with probability (by putting $\kappa = n^{3/2} \mu$), 
$4 (n^{3/2} \mu \sqrt{2\pi e})^n e^{-\pi \mu n^4}$ 
by \Cref{lemma:taildiscretegaussian}

Hence \emph{all} of these components satisfy this property except for probability %
$4n^2 \cdot (n^{3/2} \mu \sqrt{2\pi e})^n e^{-\pi \mu n^4} = (4^{1/n} n^{2/n} \cdot n^{3/2} \mu \sqrt{2 \pi e} \cdot e^{-\pi \mu n^3} )^n \leq  (24 \kappa e^{-\pi \kappa^2})^n \leq (e^{-\kappa^2})^n = e^{-\mu n^4} \leq (\eps_0)^{n^4}$ with $\kappa = n^{3/2} \mu \geq 1.3$ (since $\mu = \sqrt{\log(12r/\eps_0)} \geq \sqrt{\log(24)} \approx 1.78$ as $\eps_0 \in (0,1/2)$). The inequality $24 \kappa e^{-\pi \kappa^2} \leq e^{-\kappa^2}$ for $\kappa \geq 1.3$ follows from graphical inspection.

For the remainder of this proof (where we account for the failure probability $(\eps_0)^{n^4}$) we assume that indeed, for all $(e_{ij})_{ij}$, $(e_{ij}) \in \{ x \in K_\R ~|~ |x_\sigma| \leq n^2 \cdot \mu \cdot \varsigma \mbox{ for all }  \sigma \}$.

Hence, writing $\delta = n^3 \cdot \mu \cdot \varsigma/T \leq 1/(8n) < 1/2$, we obtain 
\[ \|\mB_N\|_2 \leq T + \|E\|_2 \leq T + \sqrt{\|E\|_1 \|E\|_\infty} \leq T + n^3 \cdot \mu \cdot \varsigma = T(1 + \delta) \]
Now for $\mB_N^{-1}$, we use that $T^{-1} \cdot J^{-1} \mB_N = I - T^{-1} E$. Hence 
\begin{align*} T \| \mB_N^{-1}\| & =  \| T \mB_N^{-1} \cdot J \| =  \|(I - T^{-1} E)^{-1}\| = \| \sum_{j = 0}^\infty (T^{-1} E)^j \| \leq 1 +  \frac{n^3 \cdot \mu \cdot \varsigma/T}{1 - n^3 \cdot \mu \cdot \varsigma/T} \\ & \leq   1 + \frac{\delta}{1-\delta} \leq 1 + 2 \delta  \end{align*}
Therefore, $\| \mB_N^{-1}\| \leq \frac{1}{T} \cdot \left ( 1 +  2\delta \right)$.

The last computation is on the determinant of $\mB_N^{-1}$. We have, by the fact that $\det(J) = 1$, \[ \det(\mB_N) = \det(T \cdot J - E) = T^n \det(I - T^{-1}J^{-1}E), \] for which we have the bound 
\[ |\det(I -  T^{-1}J^{-1}E) - 1| \leq 2n \| T^{-1} J^{-1} E\| = \frac{2n\|E\|}{T} \leq 2 n \delta \] for $\| T^{-1} J^{-1} E\| = \frac{1}{T} \|E\| \leq 1/n$ (see \cite{ipsenrehman08}, together with $(x+1)^n -1 \leq 2nx$ for $nx \leq 1$). This inequality $\frac{1}{T} \|E \| \leq 1/n$ is clearly satisfied since we assumed that all components of $E$ satisfy $(e_{ij})_{ij}$, $(e_{ij}) \in \{ x \in K_\R ~|~ |x_\sigma| \leq n^2 \cdot \mu \cdot \varsigma \mbox{ for all }  \sigma \}$.

Hence, 
\[ |\det(\mB_N^{-1})| = T^{-n} |\det(I - T^{-1} J^{-1} E)| \leq T^{-n} (1 - 2n\delta)^{-1}. \]
Since $\delta = n^3 \mu \varsigma/T = 1/(8n)$, we can easily deduce the claims, since $2n\delta \leq 1/4$.
\end{proof}

\begin{lemma} \label{lemma:runtime}
The pseudo-algorithm described in \Cref{alg:canonical}, of which we will call the output distribution $\RoundPerf(M)$ on input $M$, is correct. 
Furthermore, 
there exists an algorithm, called $\Round$, that, given $\eps_0 \in (0,1/2)$, and given any rational input $(\mB_M,\mI)$, approximates the output distribution $\RoundPerf$ of \Cref{alg:canonical} (with the same input) within statistical distance $\eps_0$ within bit complexity
\[ \poly(\size(\mB_M), \max_i \size(\ma_i),
\log(1/\eps_0)).\]
Moreover, any output module $R$ with pseudo-basis $(H_R, (\mathfrak{h}_i)_{i \in [r]})$ of this latter algorithm ($\Round$)
satisfies $\size(H_R),\max_i \size(\mathfrak{h}_i) \leq \poly( \size( \mB_M, (\mathfrak{a}_i)_{i \in [r]}), 
\log(1/\eps_0)) $.
\end{lemma}
\begin{proof} \textbf{Correctness.} To prove correctness, we need to show that the output $R$ is a rank $r$ module lattice with coefficients in $K$. The output $R$ is a rank $r$ module lattice by definition (as it has a pseudo-basis $(\mB_R,  (\mathfrak{a}_i))$). This is forced by the repeated sampling until linearly independence in line \lineref{line:gaussian}. Note that the choice of $\varsigma$ in combination with \Cref{lemma:balancedlambdanbound} and the $\alpha$-balancedness of $M$ implies 
\begin{equation} \varsigma > 3 \cdot 2^{n} \cdot \sqrt{n} \cdot \lambda_{n}(M). \label{eq:varsigmabound} \end{equation} %

For the coefficients, we observe instead the matrix $\mB_R^{-1} = \mB_{M}^{-1} \mB_N \in \ZK^{r \times r}$.
Since $\vec{v}_i \in M$, we can write $\vec{v}_i =  \mB_{M} \cdot \vec{w}_i$ where $\vec{w}_i \in \mathfrak{a}_1 \times \cdots \times \mathfrak{a}_r \subseteq \ZK^r$. Hence putting $W = (\vec{w}_1,\ldots,\vec{w}_r)$ (where the $\vec{w}_i$ are columns), we have $\mB_N = \mB_M W$ and thus $\mB_R^{-1} = \mB_{M}^{-1} \mB_N = W$. Hence, by the formula for the inverse via the adjugate, we see that $\mB_R = \frac{1}{\det_{K_\R}(W)} \mbox{adj}(W)$ which must have coefficients in $\frac{1}{\det_{K_\R}(W)} \ZK^{r \times r}$. Hence, $\mB_R$ and thus $H_R$ can be represented by rational numbers (in the field $K$, and hence, by picking any basis of $K$, by rational numbers in $\Q$). By scaling, one can demand the ideals to be integral, see also the text on Module-HNF in \Cref{sec:sizes}.

\textbf{Approximation of \Cref{alg:canonical} with small statistical distance.}
Next, we prove that the output distribution of \Cref{alg:canonical}, $\RoundPerf$, can be 
approximated by an efficient algorithm $\Round$ using bit-operations and within statistical distance $\eps_0$.
There are two lines in \Cref{alg:canonical} that cannot be computed with bit-operations 
due to their real or infinite nature: Line \lineref{line:uniformcenter} and line \lineref{line:gaussian}. The former, because a computer cannot sample from a uniform ball, and the latter because a computer cannot process arbitrarily large elements of the lattice $M$.

We resolve the first issue by discretizing the set $\mathcal{C} =  \{ x \in K_\R ~|~  \|x_\sigma\| = T \mbox{ for all } \sigma \}$, into the finite set $\dC$, in such a way that $\mathcal{C} = \dC + F$ with $F$ some fundamental domain satisfying $\|f\| \leq \frac{\varsigma \eps_0^2}{32r^3d}$ for all $f \in F$ (with $r = \rank(M)$ and $d = \deg(K)$). I.e., every element $c \in \mathcal{C}$ can uniquely be written as $c = \ddot{c} + f$ with $\ddot{c} \in \dC$ and $f \in F$; with $\vol(F) = \frac{\vol(\mathcal{C})}{|\dC|}$. One can efficiently sample in $\dC$ by sampling $x \in K_\R$ per embedding separately.

Hence, the statistical distance of the two methods of sampling $\vec{v}_i \from \Gau_{M,\varsigma,c}$, with $c = (\underbrace{0,\ldots,0}_{i-1}, \hat{c}, \underbrace{0 , \ldots, 0}_{r-i})$, where $\hat{c} \from \mathcal{C}$ or $ \hat{c} \from \dC$ can then be computed by (where the statistical distance, or, equivalently, the norm $\| \cdot \|_1$, is over $m \in M$)
\begin{align*} \left \| \frac{1}{\vol(\mathcal{C})}\int_{\hat{c} \in \mathcal{C}}  \Gau_{M,\varsigma,c}  d\hat{c} - \frac{1}{|\dC|} \sum_{\hat{c} \in \dC} \Gau_{M,\varsigma,c} \right\|_1  =  \left \|  \frac{1}{\vol(\mathcal{C})} \int_{f \in F}\sum_{\hat{c} \in \dC}  \Gau_{M,\varsigma,c+f} -  \Gau_{M,\varsigma,c}  df   \right \|_1
\\  \leq   \frac{1}{\vol(\mathcal{C})}\int_{f \in F} \sum_{\hat{c} \in \dC} \| \Gau_{M,\varsigma,c} - \Gau_{M,\varsigma,c + f} \|_1 df \leq  \frac{1}{\vol(\mathcal{C})} \int_{c \in \mathcal{C}} 4\sqrt{\frac{n \| f \|}{\varsigma}} \leq \eps_0/(2r)  \end{align*}
where the last inequality follows from the result \cite[Lemma 2.3]{AC:PelSte21} by Pellet-Mary and Stehl\'e. The premise of this result, $\eta_{1/2}(M) \leq \varsigma/2$, follows from \cite[Lemma 3.3]{MicciancioRegev2007}, as $\eta_{1/2}(M) \leq \sqrt{\frac{\log(2n(1+2))}{\pi}} \cdot \lambda_{n}(M) \leq 2rd \lambda_{n}(M) \leq  \varsigma/2$ (see \Cref{eq:varsigmabound}, and where we use that $\sqrt{\log(6x)/\pi} \leq 2x$ for all $x > 0$).

We resolve the second issue by using an algorithm computing an approximation of the discrete Gaussian as in \cite[Lemma A.7]{TCC:FPSW23} (see also \cite[Theorem 4.1]{STOC:GenPeiVai08}) with error $\eps_0/(12r)$. This means that, instead of sampling $\vec{v}_i \leftarrow \Gaussian_{M,\varsigma,c}$ in line \lineref{line:gaussian}, we sample
$\vec{v}_i \leftarrow \widehat{\Gaussian}_{M,\varsigma,c}$ for which 
\[  \|\widehat{\Gaussian}_{M,\varsigma,c} - \Gaussian_{M,\varsigma,c}  \|_1 < \eps_0/(12r), \]
for which it additionally holds that $\| \vec{v}_i  - c \| \leq \varsigma \sqrt{\log(12r/\eps_0) + 4n}$.

At the end each loop occurrence, at line \lineref{line:gaussian}, a $v_i$ is sampled that is a discrete Gaussian \emph{conditioned} on being independent to the earlier samples $(v_1,\ldots,v_{i-1})$. By \Cref{lemma:Gaussian-evenly-distributed}, the success probability of a single try of $v_i$ must be bounded from below by $1/3$ (by the fact that $\varsigma > 3 \cdot \sqrt{n} \cdot \lambda_{n}(M)$, see \Cref{eq:varsigmabound}). Hence, by \Cref{lemma:conditionalvariation}, the statistical distance between the two \emph{conditioned samples} (meaning, repetition until success), must be upper bounded by 
\[ 2 \cdot (1/3)^{-1} \cdot  \|\widehat{\Gaussian}_{M,\varsigma,c} - \Gaussian_{M,\varsigma,c}  \|_1 \leq  \eps_0/(2r) .\]

For fixed input $M$, write $\RoundPerf(M)$ for the output distribution of \Cref{alg:canonical} over rank $r$ modules represented by $(H_R,(\mathfrak{h}_i)_{i \in [r]}))$. And, for the same fixed input, write $\Round(M)$ for the same output distribution of \Cref{alg:canonical} except that $\hat{c}$ is sampled according to a discrete circle and $\vec{v}_i \leftarrow \widehat{\Gaussian}_{M,\varsigma,c}$ is sampled from an approximate discrete Gaussian. Then we have, by the fact that the loop in line \lineref{line:loopfor} consists of $r$ repetitions,  
\[ \| \RoundPerf(M) - \Round(M)\| \leq r \cdot (\eps_0/(2r) + \eps_0/(2r)) = \eps_0  \]
Hence, indeed, the output distribution $\RoundPerf(M)$ of \Cref{alg:canonical} can be approximated within statistical distance $\eps_0$. The bound on the run time is shown at the very end of this proof.

\textbf{Bound on size of $H_R$ and $\mathfrak{h}_i$.}
Due to the polynomial time algorithm for the Module-HNF by 
Biasse and Fieker \cite{BF12}, it is sufficient to find a polynomial 
size bound on $\mB_R$ in order to bound the sizes of $H_R$ and $\mathfrak{h}_i$, since 
\[ \size(H_R, (\mathfrak{h}_i)_{i \in [r]}) \leq \poly(\mB_R, (\mathfrak{a}_i)_{i \in [r]}).  \]
We bound the size of $\mB_R \in \frac{1}{\det_{K_\R}(W)} \ZK^{r \times r} \subseteq K_\R^{r \times r}$ (with $W = \mB_R^{-1}$, see the beginning of this proof) by proving an upper bound on the length of the vectors it consists of, as well as an upper bound on the (norm of the) denominators of its coefficients.

We have, by a similar computation as in 
\Cref{lemma:matrixcomputations} 
(with $\delta = n^3 
\cdot \varsigma \cdot \sqrt{\mu}/T \leq 1/(8n)$, writing $\mu = \sqrt{\log(12r/\eps_0)}$ ), using that the approximate discrete Gaussian samples indeed always satisfy $e_{ij} \leq n^2 \cdot \sqrt{\log(12r/\eps_0)} 
\cdot \varsigma = n^2 \mu \varsigma$, (contrarily to the perfect discrete Gaussian samples, for which this happens with high probability) 
\[  \|\mB_R\| = \|\mB_N^{-1} \mB_M\| \leq \|\mB_N^{-1}\| \|\mB_M\| \leq \frac{1}{T} (1 + 2 \delta) \|\mB_M\|. \]
Additionally, again, by similar determinant computations as in \Cref{lemma:matrixcomputations},
\begin{align*} |\det(W)| & = |\det(\mB_M^{-1} \mB_N)| = |\det(M)^{-1}| \cdot |\det(\mB_N)| \leq \det(M)^{-1} \cdot T (1 + 2n\delta) \\ & \leq 2 \cdot \det(M)^{-1} \cdot T. \end{align*}
Since $W \in \ZK^{r \times r}$, we have $\det(W) \in \Z$ and hence we see that $|\det(W)| = |N(\det_{K_\R}(W))| \leq 2 \det(M)^{-1} \cdot T$. 

Hence, the size of $\mB_R$ is bounded by $\poly(\size(\mB_M), \log(T)) = \poly(\size(\mB_M),
\allowbreak \log(1/\eps_0)
)$, which proves the claim. 

\textbf{Run-time.}
We finish the proof that the approximated algorithm is efficient. For lines 1-6, the efficiency follows from the efficiency of sampling the discrete circle and the efficiency of the approximate discrete Gaussian algorithm as in \cite[Lemma A.7]{FPMSW_eprint}. The fact that the sample from the discrete Gaussian is required to be conditioned on being linearly independent of earlier samples, does not give a significant overhead, by \Cref{lemma:Gaussian-evenly-distributed}.
We can conclude that these lines run in time $\poly(\size(\mB_M), \max_i \size(\ma_i), 
\log(1/\eps_0))$.

An additional note on computing this (approximate) discrete Gaussian, is that before sampling $\vec{v}_i \from \hat{\Gaussian}_{M,\varsigma,c}$, the basis of $M$ is first LLL-reduced (for this purpose only), in order to have smaller basis elements. This LLL reduction does not need to be module-compatible, and an efficient algorithm to find such an LLL reduced basis for approximate bases is described in \cite{buchmann87,buchmann96}. This allows for computing a $\Z$-basis $(m_1, \ldots,m_{n})$ of the lattice $M$ satisfying $\|m_i\| \leq 2^{n} \lambda_{i}(M)$ for all $i \in [n]$ \cite[Corollary 4.1]{buchmann96}.

Line 7 is just stacking columns and causes no real overhead. The last line, line 8, involves the computation of a Hermite normal form, which can be computed in polynomial time \cite{BF12}.
Hence the overall bit-wise approximation algorithm (of \Cref{alg:canonical}) runs within polynomial time in $\size(\mB_M), \size(\mathfrak{a}_i)$. 
\end{proof}

\begin{lemma}[The output distribution $\RoundPerf$ of \Cref{alg:canonical} does not depend on the pseudo-basis representation of $M$]
\label{lemma:pseudobasisinvariant}
Let $\alpha \in \R_{\geq 1}$ and
let $\RoundPerf(\mB_M, (\mathfrak{a}_i)_{i \in [r]})$ be the output distribution of \Cref{alg:canonical} on input $(\mB_M, (\mathfrak{a}_i)_{i \in [r]})$. Let $(\mB_M, (\mathfrak{a}_i)_{i \in [r]})$ and $(\mB'_M, (\mathfrak{a}'_i)_{i \in [r]})$ be two pseudo-basis representations of an $\alpha$-balanced module lattice $M$. Then
\[  \RoundPerf(\mB_M, (\mathfrak{a}_i)_{i \in [r]})  =   \RoundPerf(\mB'_M, (\mathfrak{a}'_i)_{i \in [r]}) \]
\end{lemma}
\begin{proof} Since the sample of  $\vec{v}_i \leftarrow \Gaussian_{M,\varsigma,c}$ in line \ref{line:gaussian} is independent on 
the choice of pseudo-basis of $M$, the distribution of the free module $N$ in line \ref{line:freemodule} is also independent on this pseudo-basis choice. Therefore, the module-lattice $R$ is independent of this pseudo-basis choice (but its representation $(\mB_R := \mB_N^{-1} \mB_M,  (\mathfrak{a}_i)_{i \in [r]})$ generally not). As the output is the Hermite normal form basis of $R$ (which is unique for each module lattice), the output is indeed independent of the pseudo-basis choice of the module $M$.
\end{proof}

\begin{lemma}[$\RoundPerf$ preserves short-vector problems] \label{lemma:preserves} Let $R$ be a module lattice produced as the output of \Cref{alg:canonical} with input $M$. Then, given a vector $\vec{v} \in R$ satisfying $\| \vec{v} \| \leq \gamma \lambda_1(R)$ (respectively $\| \vec{v} \| \leq \gamma' \det(R)^{1/(dr)}$), the vector $\vec{m} = \mB_N v \in M$ satisfies 
\[ \| \vec{m} \| \leq 2\gamma \lambda_1(M) \mbox{ (respectively } \| \vec{m} \| \leq 2\gamma' \det(M)^{1/(dr)}\mbox{)}, \]
with probability at least $1 - (\eps_0)^{n^4}$.
\end{lemma}
\begin{proof} 
By \Cref{lemma:matrixcomputations}, we have 
$\|\mB_N\|_2 \leq (1 + \frac{1}{8n})T$, $\|\mB_N^{-1}\|_2 \leq (1 + \frac{1}{4n})\frac{1}{T}$ and $|\det(\mB_N^{-1})|^{1/n} \leq (1 + \frac{1}{4}) T^{-1}$,  with probability at least $1 - (\eps_0)^{n^4}$.

Since $\mB_N$ is a module isomorphism from $R$ to $M$  (and thus $\mB_N^{-1}$ from $M$ to $R$), we obtain that for any $\vec{v} \in R \backslash \{ 0 \}$ attaining $\lambda_1(R)$, we have
$\mB_N \vec{v} \in M \backslash \{ 0 \}$ and hence 
\[ \lambda_1(M) \leq \| \mB_N \vec{v} \| \leq T(1 + 1/(8n)) \|\vec{v}\| = T\left(1 + \frac{1}{8n} \right) \lambda_1(R),  \]
and similarly, for $\vec{m} \in M \backslash \{ 0\}$ attaining $\lambda_1(M)$, 
\[ \lambda_1(R) \leq \| \mB_N^{-1} \vec{m}\| \leq \|\mB_N^{-1}\| \|\vec{m}\| \leq  \frac{1}{T} \cdot  \left ( 1 +  \frac{1}{4n} \right) \cdot \lambda_1(M). \]

After these computations, we turn back to the original task at hand: showing that a short vector in $R$ gives means of computing a short vector of $M$.
Suppose $\vec{v}$ satisfies $\|\vec{v}\| \leq \gamma \lambda_1(R)$, i.e., $\vec{v} = \mB_R \vec{w}$ with  $\vec{w} \in \mathfrak{a}_1 \times \ldots \times \mathfrak{a}_r$. Let now $m = \mB_N \vec{v} = \mB_N \mB_N^{-1} \mB_M \vec{w} = \mB_M \vec{w} \in M$.
Then by the computations on the norms on the matrix, we obtain
\[ \|\vec{m}\| = \|  \mB_N \vec{v}\| \leq T(1 + \frac{1}{8n}) \|\vec{v}\| \leq \gamma \cdot T(1 + \frac{1}{8n}) \cdot \lambda_1(R) \leq \gamma \cdot (1 + \frac{1}{8n})(1 +\frac{1}{4n}) \lambda_1(M).  \]
For the determinant variant, the same type of sequence of inequalities occurs:
\begin{align*} \|\vec{m}\| &= \|  \mB_N \vec{v}\| \leq T(1 + \frac{1}{8n}) \|\vec{v}\| \leq \gamma \cdot T(1 + \frac{1}{8n}) \cdot \det(R)^{1/n} \\ & \leq \gamma \cdot (1 + \frac{1}{8n})(1 + \frac{1}{4}) \cdot \det(M)^{1/n}.  \end{align*}
Here we use that $\det(R) = \det(\mB_N^{-1} \mB_M) = |\det(\mB_N)^{-1}| |\det(\mB_M)| \leq T^{-1} (1+\frac{1}{4}) \det(M)$. 
Now we use that $(1 + 1/(8n))(1+1/(4n)) \leq (1 + 1/(8n)) (1+1/4) \leq 2$ to obtain the final claim.
\end{proof}

\begin{lemma}[$\Round$ preserves short-vector problems] \label{lemma:preservesnonexact} Let $R$ be the module lattice represented by the output of the approximation $\Round$ of \Cref{alg:canonical} with input $M$. Then, if $\vec{v} \in R$ satisfying $\| \vec{v} \| \leq \gamma \lambda_1(R)$ respectively $\| \vec{v} \| \leq \gamma' \det(R)^{1/(dr)}$ allows for finding $\vec{m} \in M$ satisfying 
\[ \| \vec{m} \| \leq 2\gamma \lambda_1(M) \mbox{ respectively } \| \vec{m} \| \leq 2\gamma' \det(M)^{1/(dr)}, \]
with probability $1$.
\end{lemma}
\begin{proof} This follows from the proof of \Cref{lemma:preserves} and the fact that (as can be seen in \Cref{lemma:runtime}) the tails of the discrete Gaussians are cut in $\Round$, which takes away the probability that arbitrarily large samples from these Gaussians can cause the short-vector problems not to be preserved.
\end{proof}

\begin{lemma}[$\RoundPerf$ is $1/2$-H\"older continuous] 
\label{lemma:holdercontinuous}
Let $\alpha \in \R_{\geq 1}$, $\eps_0 \in (0,1/2)$ and
let $M,M'$ be $\alpha$-balanced module lattices of rank $r$. Denote $\mathcal{D}(M)$ for the output distribution of \Cref{alg:canonical} on input $(\mB_M, (\mathfrak{a}_i)_{i \in [r]})$, a pseudo-basis of $M$.

Then we have
\[ \| \mathcal{D}(M) - \mathcal{D}(M')\|_1 \leq 92 n^3 \cdot \sqrt[4]{\log(12r/\eps_0)} \sqrt{d(M,M')},  \]
where $d(M,M') :=  \min( \| \phi - I \|_2, \|\phi^{-1} - I \|_2)$ if there exists a module isomorphism $\phi: M \rightarrow M'$  between $M$ and $M'$ and $d(M,M') = \infty$ otherwise.
\end{lemma}
\begin{proof} 
Assume that $M' = \phi M$, where $\phi \in K_\R^{r \times r}$ serves as a module isomorphism (and is thus invertible); otherwise the lemma is trivially true. We may without loss of generality assume that $\det(\phi) = 1$ (and hence $\det(M') = \det(M)$), by replacing $M'$ by $\det(\phi)^{-1/(rd)}M'$ and $\phi$ by $\phi \cdot \det(\phi^{-1/(rd)})$. This holds because \Cref{alg:canonical} is scaling-independent, i.e., it does not matter whether $M$ or $qM$ is the input for $q \in \R_{>0}$.

We use that $d(M,M') = \min(\|\phi - I\|_2,\|\phi^{-1} - I\|_2)$ (where $M'$ is scaled so that $\det(M) = \det(M')$).
Then, for the same $c$,
the samples $(\vec{v}_i)_{i \in [r]}$ from $G_{M,\varsigma,c}$ (for \Cref{alg:canonical} on input $M$) and $(\phi \vec{v}_i)_{i \in [r]}$ from $G_{\phi M,\varsigma,c}$ (for \Cref{alg:canonical} on input $M' = \phi M$) 
lead to the same output module. Indeed, a pseudo-basis of the output module $R$ 
in the first case can be described by $(\mB_R := \mB_N^{-1} \mB_M,  (\mathfrak{a}_i)_{i \in [r]})$ with $\mB_N$ is constructed by stacking $\vec{v}_i$;
whereas in the second case it can be described by $((\phi \mB_N)^{-1} (\phi \mB_M),  (\mathfrak{a}_i)_{i \in [r]})$, which is equal to the pseudo-basis in the first case since $(\phi \mB_N)^{-1} (\phi \mB_M) = \mB_N^{-1} \mB_M$.

Hence, by the data processing inequality, the total variation distance in the output distribution of \Cref{alg:canonical} on input $M$ and $M'$ can be bounded above by the total variation distance between $G_{M,\varsigma,c}$ and $\phi^{-1} G_{\phi M, \varsigma,c}$, 
which are both distributions over $M$ (where we mean with $\phi^{-1} G_{\phi M, \varsigma,c}$ the distribution obtained by multiplying the output of $G_{\phi M ,\varsigma, c}$ by $\phi^{-1}$).

By rewriting, one obtains that $\phi^{-1} G_{\phi M, \varsigma,c}$ is equal to the distribution $G_{M, \phi^{-1} \varsigma, \phi^{-1} c}$, where $\phi^{-1}\varsigma$ serves as a sort of variance matrix. The probability of sampling $\vec{v}$ from $G_{M, \phi^{-1} \varsigma, \phi^{-1} c}$ is proportional to $\exp(-\|\phi/\varsigma \cdot (\vec{v} - \phi^{-1} c)\|^2)$. We use a result from Stehl\'e and Pellet-Mary \cite[Lemma 2.4]{AC:PelSte21}, where we instantiate $\mathbf{S}_1 = \varsigma, \mathbf{S}_2 = \phi^{-1} \varsigma, \mathbf{c}_1 = c$ and $\mathbf{c}_2 = \phi^{-1} c$ in \cite[Lemma 2.4]{AC:PelSte21}; we use here that, by the definition of $\varsigma$ we have $\eta_{1/2}(M) \leq \sqrt{\frac{\log(6n)}{\pi}}\lambda_n(M) \leq \varsigma$ (see \cite[Lemma 3.3]{MicciancioRegev2007}) and similarly for $M'$  
(see also \cite[Equation (2.1)]{AC:PelSte21}). This yields the following bound:
\begin{align} & \| G_{M, \phi^{-1} \varsigma, \phi^{-1} c} - G_{M, \varsigma, c}\|  \leq 4 \sqrt{n} \left( \sqrt{\mathbf{S}_2^{-1} \mathbf{S}_1 - I_n \| } + \sqrt{\mathbf{S}_2^{-1} (\mathbf{c}_1 - \mathbf{c}_2)} \right) \\ 
\leq & 4 \sqrt{n} \left( \sqrt{\|\phi  - I\|} + \sqrt{\|\varsigma^{-1} (\phi c - c) \| } \right) \\ 
& \leq 4 \sqrt{n}\sqrt{\|\phi - I\|} (1 + \sqrt{nT/\varsigma}) \end{align}
Note, though, that in line \lineref{line:gaussian}, instead the samples 
are \emph{conditional} on being linearly independent of the former samples. By \Cref{lemma:Gaussian-evenly-distributed}, the success probability of sampling such $\vec{v}_i$ being linearly independent to the former samples is at least $1/3$. Hence, by \Cref{lemma:conditionalvariation} the statistical distance between the \emph{conditioned} Gaussian samples must be upper bounded by 
\begin{equation}
 \label{eq:conditionedgaussian2} 
 2 \cdot (1/3)^{-1} \cdot  \| G_{M, \phi^{-1} \varsigma, \phi^{-1} c} - G_{M, \varsigma, c}\|  \leq 24 \sqrt{n}\sqrt{\|\phi - I\|} (1 + \sqrt{nT/\varsigma}) .
\end{equation}

Hence, for $r \leq rd = n$ of such samples, the total variation distance can be bounded by, using that $nT/\varsigma = 8n^5  \cdot \mu$ (with $\mu = \sqrt{\log(12r/\eps_0)}$) and $24 \sqrt{n}(1 + \sqrt{8n^5 \cdot \mu}) \leq 16 n^3 \cdot \sqrt{\mu}$ for $n \geq 1$, we obtain a total variation distance of
\[ \| \mathcal{D}(M) - \mathcal{D}(M')\|_1  = \| \mathcal{D}(M) - \mathcal{D}(\phi M)\|_1 \leq  92 n^3 \sqrt{\mu} \sqrt{ \|\phi - I\| } . \]
By analogously comparing the Gaussians over $M'$ and $\phi^{-1}M' = M$, one arrives at the exact same bound, except that $\|\phi - I\|$ is replaced by $\|\phi^{-1} - I \|$. Hence, replacing $\mu = \sqrt{\log(12r/\eps_0)}$ the claim of the lemma follows.  
\end{proof}

\section{Self-reducibility in the bulk: analytic tools} \label{sec:quantitative-equidistribution-bigsec}
The goal of this section is to prove an explicit, quantitative Hecke equidistribution theorem for test functions concentrated at arbitrary lattices.
It is one of the main drivers of our reduction, but the result is of independent interest.

First, take a module lattice $L_{z, \ida}$ for $z \in \GL_r(K_\R)$ and $\ida$ in the class group, as in Section \ref{sec:module-lats}.
We define a probability measure on $X_{r, \ida}$, extended trivially to $X_r$, that is concentrated around $L_{z, \ida}$.
It is given by an ``initial distribution'' function $\initial_z$.

Applying Hecke operators $T_\p$ to $\initial_z$ corresponds to randomizing $L_{z, \ida}$ or a geometrically close lattice by taking certain sublattices with index $N(\p)$.
As $\p$ grows large, the measures we obtain spread out to the whole of $X_r$ and start to converge to the uniform probability measure $\mu$.
In other words, the sublattices of $L_{z, \ida}$ of large index equidistribute.

For our purposes, it is essential to understand the rate of convergence to the uniform measure.
We do so by applying the bounds on Hecke eigenvalues given in Section \ref{sec:automorphic-theory}, importing the quantitative equidistribution results of \cite{C:BDPW20}, and bounding the $L^2$-norm of the initial distribution concentrated around $L_{z, \ida}$.
The latter depends heavily on the balancedness of the lattice (recall Definition~\ref{def:alpha-bal}).

\begin{theorem} \label{thm:main-equidistro}
    Assume ERH for the $L$-function of every Hecke character of~$K$ of trivial
    modulus.
  Let $\initial_z$ be the function defined in Definition \ref{def:initial-distro}, with defining parameters $\sigma \leq 1/\sqrt{d}$ and $t = 1$.
  Let $B, \kappa$ be positive parameters such that $\kappa \geq \sigma^{-1} \sqrt{\unitrk/4 \pi}$ and $B \gg \log \abs{\Delta_K} + d$, with large enough implied constant.
  Recall that $\mathcal{P}(B)$ is the set of prime ideals of $K$ with norm up to $B$.
	Assume that the associated lattice $L_{z, \ida}$ is $\alpha$-balanced.
	Then
	\begin{align*}
		\norm{T_{\mathcal{P}(B)} \initial_z - \mu_{\Riem}(X_r)^{-1} \cdot \triv_{X_r}}^2_{X_r} &\ll
    \max(\unitrk (\log \unitrk)^3, 1/\sigma)^{\unitrk} \cdot (C_1^2 + e^{-2(\kappa / \sqrt{d})^2}) \\
    &+ (rd)^2 \cdot B^{-3/4} \log(B)^2 \cdot C_2^2,
	\end{align*}
	where
	\begin{displaymath}
		C_1 = O\left( \frac{\log(B) \log[B^d \cdot \abs{\Delta_K} \cdot (4 + 2 \pi \kappa/ \sqrt d)^d]}{\sqrt{B}} \right)
	\end{displaymath}
	and
	\begin{displaymath}
		C_2 \leq \exp \left( \frac{r^3 d}{6} \log \alpha + r^2 \log|\Delta_K| + \frac{d}{2} \log d + O(r^2 d \log r + \log\log|\Delta_K|) \right).
	\end{displaymath}
\end{theorem}

\begin{remark}
  For our purposes, Theorem~\ref{thm:main-equidistro} is strong enough when the
  starting lattice is~$\alpha$-balanced with~$\alpha$ at most polynomial in~$d$
  (see Section~\ref{sec:conclusion} for details).
  In fact, as explained in Section~\ref{sec:tech-overview}, we should not
  expect it to apply to very imbalanced lattices.
\end{remark}

\subsection{Initial distribution} \label{sec:starting-distribution}

Define the natural projection
\begin{equation} \label{eq:pushforwardpia}
	\pi_\ida \colon Y_r \to X_{r,\ida}.
\end{equation}
The initial distribution will be the push-forward under~$\pi_\ida$ of a distribution on~$Y_r$, which we define from its density with respect to $\mu_{\Riem}$.
The latter is informally constructed by splitting $\GL_r$ into $\SL_r$ and $\GL_1$ and taking characteristic functions or bump functions on neighborhoods of the identity on each part.
For this, we recall the functions $\rho$ and $\tau$ from Section \ref{sec:volume-comps}, measuring distance on the $\SL_r$ and $\GL_1$-parts, respectively.

Let~$t>0$ and~$\sigma>0$.
We define~$\tilde f \in L^2(Y_r)$ by
\begin{equation} \label{eq:defftilde}
\tilde f (x) = \triv_{[0,t]}(\rho(x))\exp(-\frac{\pi}{\sigma^2}\tau(x)),
\end{equation}
and let $I_f = \int_{Y_r}\tilde f d\mu_{\Riem}$.
Notice that the first factor is the characteristic function of the ball $B(t)$, as defined in \ref{sec:volume-comps}.
Moreover, $\tilde{f}$ has rapid decay, which implies that all integral manipulations appearing below are valid and there are no convergence issues. 

\begin{lemma}
	We have that
	\begin{equation} \label{eq:If} 
		I_f = \int_{Y_r}\tilde f d\mu_{\Riem} =
		\frac{\mu_{\Riem}(B(t))}{\mu_{\Riem}(\SU_r(K_\R))} 
		\Bigl(\frac{\sigma}{\sqrt{r}}\Bigr)^{\unitrk}.
	\end{equation}
\end{lemma}
\begin{proof}
	Let
	\begin{displaymath}
		S \colon K_\R^\times \to K_\R^\times \U_r(K_\R)
	\end{displaymath}
	be a section of the determinant~$\GL_r(K_\R) \to K_\R^\times$ that takes
	values in~$K_\R^\times \U_r(K_\R)$.
	This latter condition is useful for employing the invariance properties of $\rho$.
	Using the integration formula \eqref{eq:coarea-formula}, we compute that
	\begin{align*}
		&\int_{Y_r} \tilde f(x) \, d\mu_{\Riem}(x)   \\
		&= \int_{x\in Y_r}\triv_{[0,t]}(\rho(x))\exp(-\frac{\pi}{\sigma^2}\tau(x)) \, dx \\
		&= r^{\frac{-\unitrk}{2}}\int_{\delta\in Y_1} \left(\int_{x\in \Delta^{-1}(\delta)}\triv_{[0,t]}(\rho(x))\exp(-\frac{\pi}{\sigma^2}\tau(x))
		dx \right) \, d\delta.
	\end{align*}
	Plugging in definitions and using the isometry between $Y_1$ and $H$ given in section \ref{sec:riemannian-structure}, the expression above equals
	\begin{align*}
		&r^{\frac{-\unitrk}{2}}\int_{\delta\in Y_1} \left(\int_{x\in \Delta^{-1}(\delta)}\triv_{[0,t]}(\rho(x)) \, dx \right)
		\exp(-\frac{\pi}{\sigma^2}\|\log\abs{\delta}\|_H^2) \, d\delta \\
		&= r^{\frac{-\unitrk}{2}} \int_{\delta \in Y_1} \mu_{\Riem}(B(t)S(\delta)/\SU_r(K_\R))
		\exp(-\frac{\pi}{\sigma^2}\|\log\abs{\delta}\|^2) \, d\delta  \\
		&= r^{\frac{-\unitrk}{2}} \int_{\delta \in Y_1}
		\frac{\mu_{\Riem}(B(t))}{\mu_{\Riem}(S(\delta)\SU_r(K_\R)S(\delta)^{-1})}
		\exp(-\frac{\pi}{\sigma^2}\|\log\abs{\delta}\|^2) \, d\delta  \\
		&= r^{\frac{-\unitrk}{2}}\int_{x\in H}
		\frac{\mu_{\Riem}(B(t))}{\mu_{\Riem}(\SU_r(K_\R))}
		\exp(-\frac{\pi}{\sigma^2}\|x\|^2) dx
	\end{align*}
	The claim follows from a standard formula for the integral of a Gaussian.
\end{proof}

For~$z\in Y_r$, notice that $\int_{Y_r} \tilde f(z^{-1} x) = I_f$ by invariance of the measure.
Thus, writing
\begin{equation} \label{eq:deffz}
f_z(x) = I_f^{-1}\tilde f(z^{-1}x),
\end{equation}
we have that $f_z \cdot \mu_{\Riem}$ defines a probability measure on $Y_r$, concentrated around the point $z$.

\begin{definition} \label{def:initial-distro}
	Let the \emph{initial distribution} around a point~$z\in X_{r,\ida}$
	be $\initial_z = (\pi_\ida)_* f_z$.
	Explicitly,
	\begin{equation} \label{eq:def-initial-z}
		\initial_z (w) = \sum_{\gamma \in \GL_r(\ZK,\ida)} f_z(\gamma w),
	\end{equation}
	and we note the dependence on $\ida$ and on the two parameters, $\sigma$ and $t$, which we leave out of notation for simplicity.
	It extends trivially to $X_r$ by setting its value to be $0$ on all other components.
\end{definition} 

\begin{lemma}
	The measure $\initial_z \cdot \mu_{\Riem}$ is a probability measure on~$X_r$.
\end{lemma}
\begin{proof}
	Indeed, we can compute that
	\begin{displaymath}
		\int_{X_r} \initial_z(w) \,d\mu_{\Riem} = \int_{X_{r, \ida}} \initial_z(w) \, d\mu_{\Riem} = \int_{Y_r} f_z(w) \, d\mu_{\Riem} = 1
	\end{displaymath}
	using the formal integration rule
	\begin{displaymath}
		\int_{\Gamma \backslash X} \sum_{\gamma \in \Gamma} f(\gamma x) \, dx = \int_X f(x) \, dx.
	\end{displaymath}
	The latter is often called the unfolding method and is valid in all cases we consider.
\end{proof}

\subsubsection{Determinant projection of the initial distribution}
We now compute the projection of $\initial_z$ onto the space $L^2_{\det}(X_{r, \ida})$ by applying the results in Section \ref{sec:riem-geom}.
For this, let~$\initial_{z,1} = \mu_{\Riem}(\Delta_{\ida}^{-1}(1))^{-1} \Delta_{\ida}'\initial_z$,
so that, by \eqref{def:pi-det},
\begin{equation} \label{eq:initial-z-1}
	\pi_{\det}\initial_z = \Delta_{\ida}^*\initial_{z,1}.
\end{equation}

\begin{lemma}
	For the initial distribution defined in \eqref{eq:def-initial-z}, we have
	\begin{equation} \label{eq:initial-det}
		\Delta_\ida'\initial_z(\delta) = \sum_{\xi \in \ZK^\times} \sigma^{-\unitrk} \exp(-\frac{\pi}{\sigma^2}\|\log \abs{\delta} + \log \abs{\xi} - \log \abs{\det(z)} \|_H^2).
	\end{equation}
	The measure $\Delta_\ida'\initial_z \cdot \mu_{\Riem}$ is a probability measure on $X_{1, \ida}$.
\end{lemma}
\begin{proof}
	Recall the description \eqref{eq:fibre} of $\Delta^{-1}(\delta)$ and that
	\begin{displaymath}
		\Delta_{\ida}^{-1}(\delta) = \Gamma_\ida \backslash \Gamma_\ida \Delta^{-1}(\delta).
	\end{displaymath}
	For each $\xi \in \ZK^\times$, choose an element $\gamma(\xi) \in \Gamma_\ida$ such that $\det(\gamma(\xi)) = \xi$.
	Using this, we parametrize
	\begin{displaymath}
		\Gamma_\ida \Delta_\ida^{-1}(\delta) = \bigcup_{\xi \in \ZK^\times} \gamma(\xi) \Delta_\ida^{-1}(\delta).
	\end{displaymath}
	The unfolding method with respect to the measure $\mu_{\Riem}$ now implies that
	\begin{align*}
		\Delta_\ida'\initial_z(\delta) &= \int_{x\in \Delta_\ida^{-1}(\delta)}\initial_z(x) \, dx
		= \int_{\Gamma_\ida \Delta_\ida^{-1}(\delta)} f_z(x) \, dx \\
		&= r^{-\frac{\unitrk}{2}} I_f^{-1} \sum_{\xi \in \ZK^\times} \int_{\Delta_\ida^{-1}(\delta)} \tilde f (z^{-1} \gamma(\xi) x) \, dx.
	\end{align*}
	The same computation as for $I_f$ and the fact that 
	\begin{displaymath}
		\mu_{\Riem}(z \gamma(\xi) B(t)) = \mu_{\Riem}(B(t))
	\end{displaymath} 
	now imply that
	\begin{multline*}
		\int_{\Delta^{-1}(\delta)} \tilde f (z^{-1} \gamma(\xi) x) \, dx
		= \frac{\mu_{\Riem}B(t)}{\mu_{\Riem}(\SU_r(K_\R))} \exp(-\frac{\pi}{\sigma^2}\|\log \abs{\delta} + \log \abs{\xi} - \log \abs{\det(z)} \|_H^2)
	\end{multline*}
	Plugging in our formula for $I_f$, we obtain the formula in the claim.
	The fact that $\Delta_\ida'\initial_z \cdot \mu_{\Riem}$ defines a probability measure can be checked by standard properties of the Gaussian function.
\end{proof}

\subsection{Bound on the norm of the initial distribution}
It is essential in our method to have uniform bounds on the norm $\norm{\initial_z}$.
We obtain them by first reducing to a problem of counting matrices in $\Gamma_\ida$ with certain size constraints that depend on $z$.
We then use techniques based on counting lattice points in balls, where the dependence on the lattice manifests through the appearance of successive minima.

\subsubsection{Reduction to a counting problem} \label{subsec:partialsum}
We first use generic notation in this section and we specialize later.
Let $Y$ be a space equipped with a measure~$\nu$. 
Let~$\Gamma$ be a discrete group acting on~$Y$ properly discontinuously.
This induces an action of $\Gamma$ on the space of functions on $Y$, defined by
\begin{displaymath}
	(\gamma f)(y) = f(\gamma^{-1}y)
\end{displaymath}
for $f\colon Y \to \R$, $y\in Y$ and~$\gamma\in\Gamma$.

Let~$\pi \colon Y \to \Gamma\lquo Y$ be the canonical projection.
For a function $f$ on $Y$ we let $\pi_*f$ be the push-forward function on $\Gamma \lquo Y$, i.e.
$\pi_*f(x) = \sum_{y\in\pi^{-1}(x)}f(y)$ for~$x\in\Gamma\lquo Y$.
Assume here that $f$ is measurable and has rapid decay so that all the sums and integrals we consider converge.

\begin{lemma}
	We have $\|\pi_*f\|^2 = \sum_{\gamma\in\Gamma}\langle f, \gamma^{-1} f\rangle_{Y}$.
\end{lemma}
\begin{proof}
	We compute that
\begin{align*}
	\|\pi_*f\|^2
	&= \int_{\Gamma\lquo Y}(\pi_*f(x))^2d\nu(x)
	= \int_{\Gamma\lquo Y}\Bigl(\sum_{y\in\pi^{-1}(x)}f(y)\Bigr)^2d\nu(x) \\
	&= \int_{\Gamma\lquo Y}\sum_{y\in\pi^{-1}(x)}f(y)\sum_{y'\in\pi^{-1}(x)}f(y')d\nu(x) \\
	&= \int_{\Gamma\lquo Y}\sum_{y\in\pi^{-1}(x)}f(y)\sum_{\gamma\in\Gamma}f(\gamma y)d\nu(x) \\
	&= \int_{Y}f(y)\sum_{\gamma\in\Gamma}f(\gamma y)d\nu(y) = \sum_{\gamma\in\Gamma}\langle f, \gamma^{-1} f\rangle_{Y}.
\end{align*}
\end{proof}

Suppose we have a function~$\tau \colon \Gamma \to \R_{\ge 0}$ (measure of
size) such that the sets $B_f(t) = \{\gamma \in \Gamma \mid
\tau(\gamma) \leq t \text{ and }\langle f,\gamma^{-1} f\rangle_Y\neq 0\}$ are finite.
Define~$C_f(t) = \abs{B_f(t)}$.
In addition, assume that we have a bound of the form
\[
\abs{\langle f,\gamma^{-1} f\rangle_Y} \le F(\tau(\gamma))
\]
for some smooth function $F \colon \R_{\ge 0} \to \R_{\ge 0}$.

\begin{corollary}
	In the notation above, we have
	\begin{equation} \label{eq:formal-L2-bound}
		\|\pi_*f\|^2 \le (F\cdot C_f)(\infty) + \int_0^\infty C_f(t)(-F'(t))\, dt.
	\end{equation}
\end{corollary}
\begin{proof}
	Since $C_f(t)$ is clearly monotone, we may use the Riemann--Stieltjes integral to state the inequality
	\[
	\Bigl|\sum_{\substack{\gamma\in\Gamma \\ \tau(\gamma) \leq T}} \langle f, \gamma^{-1} f \rangle_Y \Bigr|
	\le \sum_{\gamma \in B_f(T)} F(\tau(\gamma))
	= \int_0^T F(t) \, dC_f(t),
	\]
	for any $T > 0$.
	Integration by parts gives
	\[
	\int_0^T F(t) \, dC_f(t)
	= F(T+)C_f(T+) - F(0-)C_f(0-) - \int_0^T C_f(t)F'(t) \, dt.
	\]
	Finally, assuming the quantities below converge, letting $T \to \infty$, we obtain the claim.
\end{proof}

We now specialize the discussion to~$Y = Y_r$ with the measure~$\nu = \mu_{\Riem}$, the function $f = f_z$, the discrete subgroup $\Gamma = \GL_r(\ZK,\ida)$, and $\tau$ as in Definition \ref{def:tau-rho}.
We first compute the function $F$ that gives a bound on the inner products.

\begin{lemma}
	We have that $\langle f_z, \gamma^{-1}f_z \rangle_Y \leq F(\tau(\gamma))$, where
	\begin{equation} \label{eq:F-function}
		F(\tau) =  \frac{\mu_{\Riem}(X_r)^2 \mu_{\Riem}(\SU_r(K_\R))}{\mu_{\Riem}(B(t))}
		\left(\frac{\sqrt{r}}{\sigma\sqrt{2}}\right)^{\unitrk}
		\exp\left(-\frac{\pi}{2\sigma^2}\tau\right).
	\end{equation}
\end{lemma}

\begin{proof}
	We begin with writing explicitly
	\begin{align*}
		& \langle f_z, \gamma^{-1}f_z \rangle_Y \\
		&= \frac{1}{I_f^2} \int_{x\in Y_r}\tilde f(z^{-1}x) \tilde f(z^{-1}\gamma x) \, d \mu_{\Riem}(x) \\
		&= \frac{1}{I_f^2} \int_{x\in Y_r}
		\triv_{[0,t]}(\rho(z^{-1}x))\triv_{[0,t]}(\rho(z^{-1}\gamma x))
		\exp\left(-\frac{\pi}{\sigma^2}(\tau(z^{-1}x)+\tau(z^{-1}\gamma x))\right)
	\end{align*}
	Let~$z_H = \pi_H \log|\det z|$ and~$\gamma_H = \log |\det\gamma|\in H$.
	We now estimate the intersection of the balls $z B(t)$ and $\gamma^{-1} z B(t)$ trivially to obtain, using the same techniques as when calculating $I_f$ and $\Delta_\ida' \initial_z$, that
	\begin{align*}
		& \mu_{\Riem}(X_r) \cdot \langle f_z, \gamma^{-1}f_z \rangle_Y \\
		&\leq \frac{1}{I_f^2} \int_{x\in Y_r}
		\triv_{[0,t]}(\rho(z^{-1}x))
		\exp\left(-\frac{\pi}{\sigma^2}(\tau(z^{-1}x)+\tau(z^{-1}\gamma x))\right) \, dx \\
		&= \frac{1}{I_f^2r^{\frac{\unitrk}{2}}} \int_{\delta\in Y_1}\int_{x\in \Delta^{-1}(\delta)} \triv_{[0,t]}(\rho(z^{-1}x))
		\exp\left(-\frac{\pi}{\sigma^2}(\tau(z^{-1}x)+\tau(z^{-1}\gamma x))\right)
		\, dx \, d\delta \\
		&= \frac{1}{I_f^2r^{\frac{\unitrk}{2}}} \int_{h\in H}
		\frac{\mu_{\Riem}(B(t))}{\mu_{\Riem}(\SU_r(K_\R))}
		\exp\left(-\frac{\pi}{\sigma^2}(\|h-z_H\|^2+\|h+\gamma_H -z_H\|^2)\right)
		\, d\delta.
	\end{align*}
	Focusing in on the integral, we have
	\begin{align*}
		&\int_{h\in H} \exp\left(-\frac{\pi}{\sigma^2}(\|h-z_H\|^2+\|h+\gamma_H -z_H\|^2)\right) \, d\delta \\
		&= 
		\int_{h\in H} \exp\left(-\frac{\pi}{\sigma^2}(2\|h-z_H+\tfrac{\gamma_H}{2}\|^2+\tfrac{1}{2}\|\gamma_H\|^2)\right)
		\, d\delta \\
		&= \exp\left(-\frac{\pi}{2\sigma^2}\tau(\gamma)\right)
		\int_{h\in H} \exp\left(-\frac{\pi}{\sigma^2}(2\|h\|^2)\right) \, d\delta \\
		&= 
		\left(\frac{\sigma}{\sqrt 2}\right)^{\unitrk} \exp\left(-\frac{\pi}{2\sigma^2}\tau(\gamma)\right).
	\end{align*}
	We finish by plugging in our formula \eqref{eq:If} for $I_f$.
\end{proof}

To apply our formalism above and obtain a bound for $\norm{\initial_z}$, we are thus left with estimating $C_{f_z}(\tau)$.
For this, observe that if~$\langle f_z, \gamma^{-1}f_z \rangle_Y\neq 0$, then there
exists a point~$x\in Y_r$ such that~$\rho(z^{-1}x)\le t$ and~$\rho(z^{-1}\gamma x)\le t$.
By the properties of $\rho$, we deduce that 
\begin{equation} \label{eq:gamma-supp-condition}
	\rho(z^{-1}\gamma z) = \rho(z^{-1}\gamma x x^{-1} z) \le 2t.
\end{equation}
We use this in the next section to count the elements $\gamma$ that contribute to the $L^2$-norm.

\subsubsection{The counting problem} \label{sec:counting-problem}

Counting elements of $\Gamma_\ida$ lying in $B_{f_z}(\tau)$ can be reduced to counting lattice points in balls.
The following lemma is well-known, and we cite a version that features explicit constants.

\begin{lemma} \label{thm:henk} %
  If~$L$ is a lattice of rank~$n$ and~$R\in\R_{>0}$, we have
	\[
	|\{v\in L\mid \|v\| \le R\}| \le 2^{n-1}\prod_{i=1}^n \left(\frac{2R}{\lambda_i(L)}+1\right).
	\]
\end{lemma}
\begin{proof}
	This is Theorem 1.5 in \cite{henk}.
\end{proof}

\begin{lemma}
	Let~$L = L_{z, \ida}$ and let~$\tau>0$. Then
	\[
	C_{f_z}(\tau) \le 
	2^{r^2d-r}\prod_{k=1}^r\prod_{i=1}^{rd}
	\left(2\exp(\tfrac{1}{r}\sqrt{\tau}+2t)\frac{\lambda_k^K(L)}{\lambda_i(L)}+1\right).
	\]
\end{lemma}
\begin{proof}
	Recall that for all~$\gamma \in \Gamma = \GL_r(\ZK,\ida)$,
	\begin{displaymath}
		\tau(\gamma) = \|\log\abs{\det\gamma}\|_H^2 = \|\log\abs{\det\gamma}\|^2, 
	\end{displaymath}
	and $C_f(\tau) = |B_f(\tau)|$, where
	\[
	B_f(\tau) = \{\gamma \in \Gamma \mid \tau(\gamma)\le\tau \text{ and }\langle f,\gamma^{-1} f\rangle_Y\neq 0\}.
	\]
	Let~$\gamma\in B_f(\tau)$. 
	Then the non-vanishing of the inner product condition implies that $\rho(z^{-1}\gamma z) \le 2t$, as in \eqref{eq:gamma-supp-condition}.
  Writing $z=(z_v)_v,\gamma=(\gamma_v)_v$ by viewing $\GL_r(K_\R)$ as $\prod_v \GL_r(K_v)$, we have that $\|z^{-1} \gamma z\|_{\op} = \max_v \|z_v^{-1} \gamma_v z_v\|_{\op}$ since for $w \in K_\R^r$ we have $\norm{w}^2 = \sum_v [K_v:\R] \norm{w_v}^2$. 
	We obtain,
	\begin{align*}
		\|z^{-1} \gamma z\|_{\op}
		&= \max_v \|z_v^{-1} \gamma_v z_v\|_{\op} 
		= \max_v \frac{\|z_v^{-1} \gamma_v z_v\|_{\op}}{|\det \gamma_v|^{\frac{1}{r}}}
		|\det \gamma_v|^{\frac{1}{r}} \\
		&\le \exp\left(\tfrac{1}{r}\log \max_v |\det \gamma_v|\right)
		\max_v \frac{\|z_v^{-1} \gamma_v z_v\|_{\op}}{|\det \gamma_v|^{\frac{1}{r}}} \\
		&= \exp\left(\tfrac{1}{r}\|\log |\det \gamma|\|_\infty\right)
		\exp(\rho(\gamma)) \\
		&\leq \exp\left(\tfrac{1}{r}\sqrt{\tau(\gamma)} + \rho(\gamma)\right) \le \exp(\tfrac{1}{r}\sqrt{\tau} + 2t),
	\end{align*}
  using that the $L^\infty$-norm is at most the $L^2$-norm in finite dimensional spaces.
	In other words, the operator norm of~$\gamma$ acting on~$L$ is at
	most~$\exp(\tfrac{1}{r}\sqrt{\tau}+2t)$.
	
	Now let~$v_1,\dots,v_r\in L$ be $K$-independent with~$\|v_k\| \le
	\lambda_k^K(L)$.
	Each~$\gamma\in\Gamma$ is uniquely determined by the images of the~$v_k$.
	In addition, for every~$\gamma\in B_f(\tau)$, we have
	\[
	\|\gamma v_k\| \le \exp(\tfrac{1}{r}\sqrt{\tau}+2t)\lambda_k^K(L).
	\]
	By Lemma~\ref{thm:henk} the number of vectors in~$L$ satisfying this bound is
	at most
	\[
	2^{rd-1}\prod_{i=1}^{rd} \left(\frac{2\exp(\tfrac{1}{r}\sqrt{\tau}+2t)\lambda_k^K(L)}{\lambda_i(L)}+1\right).
	\]
	Using this bound for every $k$ leads to the claim.
\end{proof}

\begin{corollary}\label{cor:counting}
	Under the same hypothesis, assuming that~$t\ge \frac 1 4$ and that $L$ is $\alpha$-balanced (recall Definition \ref{def:alpha-bal}), we have
	\begin{displaymath}
		C_{f_z}(\tau) \leq \left(8 \alpha^{r/6}\right)^{r^2 d} \exp(rd\sqrt{\tau}+2 r^2d t).
	\end{displaymath}
\end{corollary}
\begin{proof}
	Rewrite the bound of the lemma as
	\begin{displaymath}
		C_{f_z}(\tau) \leq
		2^{r^2d-r}\prod_{k=1}^r\prod_{k'=1}^r\prod_{i=1}^{d}
		\left(2\exp(\tfrac{1}{r}\sqrt{\tau}+2t)\frac{\lambda_k^K(L)}{\lambda_{d(k'-1)+i}(L)}+1\right).
	\end{displaymath}
	Applying Lemma~\ref{lem:Kminima}, we get
	\begin{align*}
		C_{f_z}(\tau) &\leq
		2^{r^2d-r}\prod_{k=1}^r\prod_{k'=1}^r\prod_{i=1}^{d}
		\left(2\exp(\tfrac{1}{r}\sqrt{\tau}+2t)\frac{\lambda_k^K(L)}{\lambda_{k'}^K(L)}+1\right)
		\\
		&\leq
		2^{r^2d} (4 \exp(\tfrac{1}{r}\sqrt{\tau}+2t))^{dr(r+1)/2} \prod_{k=1}^r\prod_{k'=1}^{k-1}
		\left(4\exp(\tfrac{1}{r}\sqrt{\tau}+2t)\frac{\lambda_k^K(L)}{\lambda_{k'}^K(L)}\right)^d
		\\
		&\leq
		8^{r^2d} \exp(rd\sqrt{\tau}+2r^2dt) \prod_{k=1}^r\prod_{k'=1}^r \alpha^{d(k-k')} 
		\\
		&\leq
		8^{r^2d} \exp(rd\sqrt{\tau}+2r^2dt) \alpha^{dr(r^2-1)/6},
	\end{align*}
	which implies the claim.
\end{proof}

\subsubsection{The norm bound} \label{sec:norm-bound}
Before proving our bound for $\norm{\initial_z}$, we state a technical lemma that aids computation.

\begin{lemma}\label{lem:integralbound}
	Let~$a,b>0$.
	Then
	\[
	\int_0^\infty \exp(-ax+b\sqrt{x})dx \leq 
	\frac{2}{a}\left(2\exp(2b^2/a)+1\right).
	\]
\end{lemma}
\begin{proof}
	Let~$x\ge (2b/a)^2$. Then~$-ax+b\sqrt{x} \leq -\frac a 2 x$, so
	\[
	\int_{(2b/a)^2}^\infty \exp(-ax+b\sqrt{x})dx
	\leq \int_{(2b/a)^2}^\infty \exp(-\frac a 2 x)dx
	= \frac{2}{a}\exp(-2b^2/a) \leq \frac{2}{a}.
	\]
	
	On the other hand we have
	\begin{align*}
		\int_0^{(2b/a)^2} \exp(-ax+b\sqrt{x})dx
		&\leq \int_0^{(2b/a)^2} \exp(b\sqrt{x})dx
		= 2\int_0^{2b/a} y\exp(by)dy
		\\
		&\leq (4b/a)\int_{-\infty}^{2b/a} \exp(by)dy
		= (4/a)\exp(2b^2/a).
	\end{align*}
\end{proof}

We now sum up all the previous sections, recalling our construction for convenience, and conclude with one of the main estimates in our argument.
Namely, we define the function $\initial_z$ (see Definition \ref{def:initial-distro}) starting with data consisting of a matrix $z \in Y_r$, a class group representative $\ida$, the parameter $t$, which controls how much $\initial_z$ localizes in the $\SL(r)$-part, and the parameter $\sigma$, which controls how much it localizes in the $\GL(1)$-part.
The first two data also define a lattice $L = L_{z, \ida}$.
We have the following bound on the $L^2$-norm of the starting distribution, defined in terms of $\mu_{\Riem}$.

\begin{proposition}\label{prop:l2-norm-bd}
	Suppose $L_{z, \ida}$ is $\alpha$-balanced, and let~$1 \le t\le O(1)$ and~$\sigma = O(1/\sqrt{d})$.
	We have
	\begin{displaymath}
		\log \|\initial_z\|_{X_R} \leq \frac{r^3 d}{6} \log \alpha + r^2 \log|\Delta_K| + \frac{d}{2}\log d + O(r^2 d \log r + \log\log|\Delta_K|).
	\end{displaymath}
\end{proposition}
\begin{proof}
	We first prove the more precise bound
	\begin{displaymath}
	\|\initial_z\|^2 \le
	\frac{\mu_{\Riem}(X_r)^2}{\mu_{\Riem}(B(t))\sigma^{\unitrk}}
	\left(\frac{r}{2}\right)^{\frac{\unitrk}{2}}
	\left(8e^{2t}\alpha^{r/6}\right)^{r^2d}
	\left(4\exp\Bigl(\frac{(2\sigma rd)^2}{\pi}\Bigr)+2\right).
	\end{displaymath}
	For this, recall from \eqref{eq:F-function} that we have~$\langle f_z, \gamma^{-1}f_z \rangle_Y \le F(\tau)$, where
	\begin{displaymath}
	F(\tau) =  \frac{\mu_{\Riem}(X_r)^2\mu_{\Riem}(\SU_r(K_\R))}{\mu_{\Riem}(B(t))}
	\left(\frac{\sqrt{r}}{\sigma\sqrt{2}}\right)^{\unitrk}
	\exp\left(-\frac{\pi}{2\sigma^2}\tau\right).
	\end{displaymath}
	For applying the formal bound \eqref{eq:formal-L2-bound}, we first note that, by Corollary \ref{cor:counting}, the function $C_f(\tau)$ grows like $\exp(\sqrt \tau)$, whilst $F(\tau)$ decays like $\exp(-\tau)$.
	This implies that $F(\tau) C_f(\tau)$ vanishes as $\tau$ goes to infinity.
	The same observation shows that $F'(\tau) C_f(\tau)$ exhibits rapid decay and is integrable.
	Therefore, we obtain that
	\[
	\|\initial_z\|^2 \le  \int_0^\infty C_f(\tau)(-F'(\tau))d\tau.
	\]
	
	Ignoring the $\tau$-independent factors in the formula for $F'(\tau)$, we have
	\begin{equation*}
		\int_0^\infty C_f(\tau)\exp\left(-\frac{\pi}{2\sigma^2}\tau\right) \, d\tau
		\leq
		\left(8e^{2t}\alpha^{r/6}\right)^{r^2d}
		\int_0^\infty
		\exp(-\frac{\pi}{2\sigma^2}\tau + rd\sqrt{\tau}) \,
		d\tau.
	\end{equation*}
	Lemma~\ref{lem:integralbound} with~$a = \frac{\pi}{2\sigma^2}$ and~$b = rd$ now gives
	\[
	\int_0^\infty \exp(-\frac{\pi}{2\sigma^2}\tau + rd\sqrt{\tau}) \, d\tau
	\le 
	\frac{4\sigma^2}{\pi}\left(2\exp\Bigl(\frac{(2\sigma rd)^2}{\pi}\Bigr)+1\right).
	\]
	To finally arrive at the claimed bound, note simply that the factor $\frac{4\sigma^2}{\pi}$ in the previous display and the $\frac{\pi}{2\sigma^2}$ from differentiating $F$ cancel to give a factor of $2$. 
	
	By the assumption on~$\sigma$ we have
	\[
	4\exp\Bigl(\frac{(2\sigma rd)^2}{\pi}\Bigr)+2
	= 2^{O(r^2d)}.
	\]
	Introducing the volume computations of Section \ref{sec:volume-comps} to the bound we proved above, we obtain
	\begin{align*}
		2\log \|\initial_z\|
		&\leq 
		2 \log\mu_{\Riem}(X_r)
		-\log \mu_{\Riem}(B(t))
		+ \frac{\unitrk}{2}\log d
		+ \frac{r^3 d}{6} \log \alpha
		+ O(r^2d) \\
		& \leq 
		\frac{dr^2}{2}\log r + r^2\log|\Delta_K| + O(\log\log|\Delta_K|)
		+ \frac{9}{4} d r^2\log r \\
		& + \frac{\unitrk}{2}\log d + \frac{r^3 d}{6} \log \alpha + O(r^2d)
		\text{ by Lemmas~\ref{lem:ballvol} and~\ref{lem:spacevol}}
		\\
		&=
		r^2 \log|\Delta_K|
		+ \frac{d}{2}\log d + \frac{r^3 d}{6} \log \alpha + O(r^2 d \log r + \log\log|\Delta_K|)
	\end{align*}
	as claimed, by simplifying the expression using that $\unitrk \leq d$.
\end{proof}

\subsection{Quantitative equidistribution}
Recall that we are interested in showing that the measures on $X_r$, obtained by applying Hecke operators $T_\p$ and averages thereof to the initial probability distribution given by $\initial_z \cdot \mu_{\Riem}$, converge to the uniform measure $\mu$.
To understand the rate of convergence, we need an upper bound on 
\begin{displaymath}
	\norm{ T_{\mathcal{P}(B)} (\initial_z - \mu_{\Riem}(X_r)^{-1} \triv_{X_r})}_{X_r},
\end{displaymath}
where $B > 0$ is some parameter to be chosen later.
Here we are using the $L^2$-norm with respect to $\mu_{\Riem}$.

For this, we decompose the function into its projection onto $L^2_{\det}$ and its orthogonal complement.
Recall that $T_\p$ preserves such decompositions and notice also that the constant function is equal to its projection onto $L^2_{\det}$.
We therefore focus first on bounding
\begin{displaymath}
	\norm{ T_{\mathcal{P}(B)} (\pi_{\det} \initial_z - \mu_{\Riem}(X_r)^{-1} \cdot \triv_{X_r})}_{X_r}.
\end{displaymath}

We now import the results of \cite{C:BDPW20}, which essentially treat Hecke operators on the space $L^2_{\det}(X_r)$.
For that, we denote by $T_\p^1$ the Hecke operator on $L^2(X_1)$, which is adelically given by
\begin{displaymath}
	T_\p^1 f(x) = f(x \pi_\p^{-1}),
\end{displaymath}
where $\pi_\p$ is a uniformizer at $\p$.
This corresponds to the definition of a Hecke operator in \cite[Sec. 3]{C:BDPW20}, upon identifying $X_1$ with the \emph{additive} Arakelov class group $\operatorname{Pic}_K^0$.

Next, we recall that $L^2(X_1) = \prod_{\ida} L^2(X_{1,\ida})$ and that $X_{1, \ida} = X_{1, 1}$ for all representatives $\ida$ of the class group.
The definition of Hecke operators (see \eqref{eq:def-Hecke-coset}) directly implies that
\begin{equation} \label{eq:Tp-Delta}
	T_\p \Delta_{\ida}^\ast = \Delta_{\ida}^\ast T_\p^1.
\end{equation}
To be precise, we view $L^2(X_{N,\ida})$ embedded in $L^2(X_N)$, for $N = 1, r$, by extending functions by the constant zero function on all other components.
Note also that $T_\p$ sends $L^2(X_{N, \ida})$ to $L^2(X_{N, \p \cdot \ida})$ in this interpretation.

Next, we recall that $\pi_{\det} \initial_z = \Delta_{\ida}^\ast \initial_{z,1}$ (see \eqref{eq:initial-z-1}) and $\triv_{X_r} = \Delta_{\ida}^\ast \triv_{X_1}$.
Recall from the computation \eqref{eq:initial-det} that
\begin{displaymath}
	\langle \initial_{z,1}, \triv_{X_1} \rangle_{X_1} = \mu_{\Riem}(\Delta_{\ida}^{-1}(1))^{-1}.
\end{displaymath}
From Section \eqref{sec:volume-comps}, we gather that
\begin{displaymath}
	\mu_{\Riem}(X_r) = r^{-\frac{\unitrk}{2}} \cdot \mu_{\Riem}(\Delta_{\ida}^{-1}(1)) \cdot \mu_{\Riem}(X_1).
\end{displaymath}
The formula for how norms behave under $\Delta_{\ida}^\ast$, given in Section \ref{sec:det-map}, and the Hecke operator compatibility relation \ref{eq:Tp-Delta} now imply that
\begin{align} \label{eq:equid-det-part}
	& \norm{T_\p (\pi_{\det} \initial_z - \mu_{\Riem}(X_r)^{-1} \cdot \triv_{X_r})}_{X_r} \nonumber \\ 
	&= \sqrt{\frac{r^{-\frac{\unitrk}{2}} \cdot \mu_{\Riem}(\Delta_{\ida}^{-1}(1)) \mu_{\Riem}(X_1)}{\mu_{\Riem}(X_r)}} \norm{T_\p^1 (\rho_{\sigma} - \mu_{\Riem}(X_1)^{-1} \triv_{X_1})}_{X_1} \nonumber \\
	&= \norm{T_\p^1 (\rho_{\sigma} - \mu_{\Riem}(X_1)^{-1} \triv_{X_1})}_{X_1}
\end{align}
where we write
\begin{displaymath}
	\rho_{\sigma}(\delta) = \sum_{\xi \in \Z_K^\times} \sigma^{-\unitrk} \exp(-\frac{\pi}{\sigma^2}\|\log \abs{\delta} + \log \abs{\xi} - \log \abs{\det(z)} \|_H^2).
\end{displaymath}
The function $\rho_{\sigma}$ is defined on $X_{1, \ida}$ and extended, as usual, to all of $X_1$ by zero.

We now observe that $\rho_{\sigma}$ is the same test function as $\sigma^{\unitrk} \rho_{\sigma} \mid^T$ in Section 3.5 of \cite{C:BDPW20}, up to the shift $\delta \mapsto \delta \cdot \det(z)$.
Since the right regular representation is unitary, leaves constant functions invariant, and commutes with Hecke operators, we may ignore this shift.

For the next result, we introduce the natural notation
\begin{displaymath}
	T_{\mathcal{P}(B)}^1 = \frac{1}{|\mathcal{P}(B)|} \sum_{N(\p) \leq B} T_{\p}^1.
\end{displaymath}

\begin{proposition} \label{prop:GL1-equid-import}
    Assume ERH for the $L$-function of every Hecke character of~$K$ of trivial
    modulus.
	For positive parameters $B, \kappa, \sigma$ such that $\kappa \sigma > \sqrt{\unitrk/4 \pi}$, we have
	\begin{displaymath}
		\norm{ T_{\mathcal{P}(B)}^1 (\rho_{\sigma}) - \mu_{\Riem}(X_1)^{-1} \triv_{X_1}}^2_{X_1} 
		\ll \max(\unitrk (\log \unitrk)^3, 1/\sigma)^{\unitrk} \cdot (c^2 + e^{-2(\kappa \sigma)^2}),
	\end{displaymath}
	where
	\begin{displaymath}
		c = O\left( \frac{\log(B) \log[B^d \cdot \abs{\Delta_K} \cdot (4 + 2 \pi \kappa/ \sqrt d)^d]}{\sqrt{B}} \right).
	\end{displaymath}
\end{proposition} 
\begin{proof}
	This is Theorem 3.16 of \cite{C:BDPW20} with $N = 1$ and a few mild, additional constraints.
	For convenience, we note here that in loc. cit., $n$ is our $d$, $l$ is our $\unitrk$, $s$ is our $\sigma$, $r$ is our $\kappa$.
	Observe also the typo in (6) of loc. cit., where $n$ should be replaced by $l$.
	
	We use (9) of loc. cit. together with the bounds in the beginning of the proof of Corollary 3.4 in Appendix B of loc. cit. to obtain the bound $\unitrk (\log \unitrk)^3$ for $\eta_1(\Lambda_K^\ast)$, the smoothing number in the notation of that paper.
	We finish by applying the bound $\beta^{(\unitrk)}_{\sqrt 2 \kappa \sigma} \leq e^{-2 (\kappa \sigma)^2}$ from just before Lemma 2.10 in loc. cit., which is valid under our assumption.
\end{proof}

We continue with the orthogonal complement of $L^2_{\det}(X_r)$.
On this space, we use the spectral gap afforded by Corollary \ref{cor:spec-gap}.
Putting everything together we prove Theorem \ref{thm:main-equidistro}.

\begin{proof}[Proof of Theorem \ref{thm:main-equidistro}]
	As indicated at the beginning of this section, we prove this bound by decomposing the expression in the norm into its projection to $L^2_{\det}$ and its orthogonal complement.
	For the $L^2_{\det}$-part, we recall \eqref{eq:equid-det-part} and the previous result, Proposition \ref{prop:GL1-equid-import}.
	
	Let us temporarily denote $\initial_z^\perp = \initial_z - \pi_{\det} \initial_z$.
	We are left with bounding $\norm{T_{\mathcal{P}(B)}^N \initial_z^\perp}$.
	For this we apply the spectral gap as in Corollary \ref{cor:spec-gap} to get
	\begin{displaymath}
		\norm{T_{\mathcal{P}(B)} \initial_z^\perp}^2 \ll (rd)^2 \cdot B^{-3/4} \log(B)^2 \cdot \norm{\initial_z^\perp}^2.
	\end{displaymath}
	We then trivially bound $\norm{\initial_z^\perp}$ by $\norm{\initial_z}$ and apply Proposition \ref{prop:l2-norm-bd}, the conditions of which are satisfied.
\end{proof}

\section{Balancedness of random module lattices}\label{sec:self-red-bulk-algorithmic}

In this section, we prove that $\mu$-random module lattices are balanced (in a
weak sense) with high probability: the main result is
Theorem~\ref{thm:randbalanced}.
We will use the Grayson--Stuhler theory of stability of lattices, from
which we recall some definitions (cf.~\cite{grayson,bost-bourbaki}).
The role of this notion is that it is relatively easy to compute the probability
of a random lattice being unstable, and that stable lattices are balanced.
We compute this probability using work of Thunder~\cite{thunder}, with
inspiration from an article of Shapira and Weiss~\cite{shapira-weiss}, whose
result we generalize and sharpen.
Note that in recent work~\cite{nihar1,nihar2}, Gargava, Serban, Viazovska and
Viglino prove strong bounds on the shortest vectors of random module lattices;
our bounds are weaker but more widely applicable.

\begin{definition}
  Let~$L$ be a module lattice. The \emph{slope} of $L$ is
  \[
    \slope(L) = \frac{\log\det(L)}{\rank(L)}.
  \]
  Let~$t\ge 1$. A sub-module lattice~$L'\subset L$ (of arbitrary rank) is \emph{$t$-destabilising} if
  \[
    \slope(L') \le \slope(L) - \frac{\log(t)}{\rank(L')},
  \]
  i.e. if
  \[
    \bigl(t \cdot \det(L')\bigr)^{\frac{1}{\rank(L')}}
    \le \det(L)^{\frac{1}{\rank(L)}}.
  \]
  A lattice is \emph{semistable} if it does not contain any $t$-destabilising
  sub-module lattices for any~$t>1$, i.e. if
  \[
    \slope(L') \ge \slope(L)
  \]
  for every sub-module lattice~$L'\subset L$.
\end{definition}

\begin{remark}
  The notion of stability we use is with respect to the class of module lattices over a fixed field $K$.
  Throughout this section, keeping this remark in mind, we abbreviate the term \emph{sub-module lattice} to simply \emph{sublattice}.
\end{remark}

\begin{remark}
  Note that if there exists a $t$-destabilising sublattice~$L'$ in~$L$, then there
  also exists a primitive one of the same rank as $L'$, namely~$L'' = W\cap L' \supset L'$
  where~$W = K\cdot L'$.
\end{remark}

\begin{theorem}\label{thm:randbalanced}
  In the set of~$(K,r)$ such that
  \begin{itemize}
    \item $r\ge 4$, or
    \item $r\ge 3$ and~$|\Delta_K| \ge 57.5^d$, or
    \item $|\Delta_K| \ge 845^d$,
  \end{itemize}
  a $\mu$-random module lattice~$L$ is semistable
  with probability at least~$1-2^{-\Omega(n\log r)}$.
  
  Now assume~$r\le 3$, and let~$\delta = |\Delta_K|^{\frac 1d}$.
  If~$r=2$, let
  \[
    t = \max\left(
    1.01 \frac{\pi e}{2}\frac{\log \delta}{\delta^{\frac 12}}
    , 1\right)^{\frac d2};
  \]
  if~$r=3$, let
  \[
    t = \max\left(
    1.01 \frac{\pi^3 e}{6}\frac{\log \delta}{\delta}
    , 1\right)^{\frac d3};
  \]
  Then a $\mu$-random module lattice~$L$ has no $t$-destabilising sublattice
  with probability at least~$1-2^{-\Omega(d)}$.

  In all cases, a $\mu$-random module lattice~$L$ satisfies
  \[
    \lambda_1(L) \ge \Omega(|\Delta_K|^{-\frac{1}{2d}})\cdot \det(L)^{\frac{1}{n}}
    \text{ and }
    \lambda_n(L) \le O(n|\Delta_K|^{\frac{1}{2d}})\cdot \det(L)^{\frac{1}{n}}
  \]
  with probability at least~$1-2^{-\Omega(n\log r)}$.
\end{theorem}

\begin{remark}\hfill
  \begin{enumerate}
    \item It would be interesting to know whether there exists families of
      number fields in which the proportion of semistable lattices of rank~$2$
      (or~$3$) is not~$1-2^{-\Omega(d)}$. Such a family should have bounded root
      discriminant, and, as is visible from our proof, this proportion is
      directly related to the size of the residue of the Dedekind zeta function
      of these fields.
    \item Our methods meets its limits when the rank~$r$ is small and the fields
      have small root discriminant. Interestingly, the methods
      of~\cite{nihar1,nihar2} also have limitations in small rank and for
      families of fields that admits elements of small height. It would be
      interesting to investigate relations between these limitations.
  \end{enumerate}
\end{remark}

We break up the proof into several intermediate results.
We will use computations by Thunder~\cite{thunder} and we first explain how to
relate his ad\'elic computations to our case of interest. For the reader's
convenience, we provide the correspondence between notations: Thunder's $K_\adel$
is our $\adel_K$, his~$n$ is our~$r$, his~$d$ is our~$k$, he writes~$[K:\Q]$
for our~$d = \deg(K)$, and his~$\chi_t$ is our~$\triv_{[0,t]}$.
Recall from Section~\ref{sec:module-lats} that to each~$A\in\GL_r(\adel_K)$ we
can attach a module lattice embedded in~$K_\R^r$, which we will write~$L_A$. Let~$k\ge 1$ be an integer.
Thunder defines a function
\[
  f_{r,k} \colon \GL_r(\adel_K) \to \R_{>0}.
\]
Translated in module lattice language, $f_{r,k}(A)$ is the determinant of the
sub-module-lattice~$L'\subset L_A$ generated by the first~$k$ columns of the
basis of~$L_A$ determined by~$A$.

Define~$G_r = \{A\in \GL_r(\adel_K) : \prod_v|\det(A_v)|_v=1\}$ (corresponding to lattices
of determinant~$1$) and
\[
  G_{r,k} = \left\{\begin{pmatrix}A & B \\ 0 & D\end{pmatrix} : 
    A\in G_k, D\in G_{r-k}, B\in M_{k,r-k}(\adel_K)
    \right\}.
\]
Then the quotient~$\GL_r(K)/\GL_r(K)\cap G_{r,k}$ is in bijection with the
set~$\Gr_{r,k}(K)$ of $k$-dimensional subspaces of~$K^r$ via~$\gamma \mapsto
\gamma(K^k \times \{0\}^{r-k})$.
Thunder defines
\[
  c(r,k) = \int_{G_r/G_{r,k}}\triv_{[0,1]}(f_{r,k}(A)) d\mu(A).
\]

\begin{lemma}
  Let $t\in\R_{\ge 1}$ and~$k\in\Z_{\ge 1}$.
  The measure of the set of module lattices that admit a $t$-destabilising
  sublattice of rank~$k$ is at most~$c(r,k)t^{-r}$.
\end{lemma}
\begin{proof}
  We have
  \begin{eqnarray*}
    && \mu(\{L\in X_r(K) : L\text{ admits a }t\text{-destabilising sublattice of rank }k\}) \\
    &\le& \int_{G_r/\GL_r(K)} \left(\sum_{W\in\Gr_{r,k}(K)}\triv_{W\cap L_A
      \text{ is } t \text{-destabilising in }L_A}(A) \right)d\mu(A) \\
    &=& \int_{G_r/\GL_r(K)} \left(\sum_{\gamma \in \GL_r(K)/\GL_r(K)\cap G_{r,k}}
      \triv_{[0,\frac 1t]}(f_{r,k}(A\gamma)) \right)d\mu(A) \\
    &=& \int_{G_r/G_{r,k}} \triv_{[0,\frac 1t]}(f_{r,k}(A)) d\mu(A)
      \text{ by~\cite[Lemma~2.4.2]{weil-AAG}} \\
    &=& c(r,k)t^{-r} \text{ by~\cite[Lemma~5]{thunder}},
  \end{eqnarray*}
  proving the claim.
\end{proof}

For every dimension~$n$, let~$V_n$ be the volume of the Euclidean $n$-ball of
radius~$1$, i.e.~$V_n = \frac{\pi^{\frac{n}{2}}}{\Gamma(\frac{n}{2}+1)}$.
For every integer~$m\ge 1$, let~$\zeta_K^*(m)$ denote the leading coefficient
of~$\zeta_K(s)$ at~$s=m$, i.e. $\zeta_K^*(m) = \zeta_K(m)$ for~$m\ge 2$
and~$\zeta_K^*(1) = \frac{2^{r_1}(2\pi)^{r_2}R_Kh_K}{|\Delta_K|^{1/2}w_K}$ by the analytic class number formula (recall Section~\ref{sec:number-fields} for notation)
and define
\[
  R(m) =
  \frac{m^{\unitrk+1}2^{mr_2}V_m^{r_1}V_{2m}^{r_2}}{\zeta_K^*(m)|\Delta_K|^{m/2}}
  \text{, so that }
  R(1) = \frac{w_K}{h_KR_K}.
\]

\begin{lemma}\label{lem:crk}
  For every~$0<k<r$ we have
  \[
    c(r,k) = \frac{1}{r} \cdot \frac{\prod_{j=1}^r R(j)}{\prod_{j=1}^k
    R(j)\prod_{j=1}^{r-k} R(j)}.
  \]
\end{lemma}
\begin{proof}
  First note that for~$r>1$,
  \[
    \frac{1}{r}\cdot \frac{R(r)}{R(1)}
    =
    \frac{r^{\unitrk}2^{rr_2}V_r^{r_1}V_{2r}^{r_2}h_KR_K}{\zeta_K(r)|\Delta_K|^{r/2}w_K},
  \]
  which is indeed the value of~$c(r,1)$ by~\cite[Lemma~7]{thunder}.
  In addition, the RHS of the claimed equality clearly satisfies Thunder's
  recurrence relation~\cite[Theorem~3]{thunder}, so the equality holds for
  every~$r$ and~$k$.
\end{proof}

\begin{lemma}\label{lem:boundZ}
  For every~$m\ge 1$ we have
  \[
    \prod_{j=2}^m \zeta_K(j) \le (2.3)^d.
  \]
\end{lemma}
\begin{proof}
  Since there are at most $d$ prime ideals in $\ZK$ over a rational prime, we have~$\zeta_K(j) \le \zeta(j)^d$ for all~$j>1$.
  It is thus sufficient to
  prove the inequality for~$\zeta$. In addition, since the product increases
  with~$m$, it is enough to prove the inequality for~$m$ large enough.
  We have
  \[
    \zeta(j) \le 1 + 2^{-j} + \int_{2}^\infty t^{-j}dt \le 1+3\cdot 2^{-j}.
  \]
  Therefore, for all~$m\ge m_0$, we have
  \[
    \sum_{j=m_0}^m \log\zeta(j) \le 3\sum_{j=m_0}^m 2^{-j} \le 6\cdot 2^{-m_0},
  \]
  and thus
  \[
    \prod_{j=2}^m \zeta(j) \le \prod_{j=2}^{m_0-1} \zeta(j) \cdot \exp(6\cdot
    2^{-m_0}).
  \]
  For~$m_0=11$, this gives the claimed inequality.
\end{proof}

\begin{lemma}\label{lem:boundG}
  For any function $f\colon [1,+\infty) \to \R$, write~$S_m(f) = \sum_{j=1}^m
  f(j)$ for~$m\ge 1$
  and $S_{r,k}(f) = S_r(f) - S_k(f) - S_{r-k}(f)$ for~$r\ge 2$ and~$1\le k<r$.
  \begin{enumerate}
    \item For~$f(x) = (\frac{x}{2}+1)\log(\frac{x}{2}+1)$, we have
      \[
        S_{r,k}(f) \ge
          \tfrac 12k(r-k)\log(\tfrac{r}{2}+1)
        - \tfrac 14 k(r-k)
        - \tfrac{13}{8}\log(\tfrac r2+1)
        + \tfrac{13}{8}\log(\tfrac 32)
        - \tfrac{5}{8}.
      \]

    \item For~$f(x) = \log(\frac x2+1)$, we have
      $S_{r,k}(f) \le (r+2)\log 2 - \tfrac{11}{12} - \tfrac{5}{2}\log(\tfrac 32)$.

    \item For~$f(x) = \frac 1{\frac x2+1}$, we have
      $S_k(f)+S_{r-k}(f) \le 4 \log(\tfrac r4+1)$.

    \item For~$f(x) = \log\Gamma(\frac{x}{2}+1)$, we have
      \[
        S_{r,k}(f) \ge
          \tfrac 12k(r-k)\log(\tfrac{r}{2}+1)
        - \tfrac 34 k(r-k)
        - r\tfrac{\log 2}{2}
        - \tfrac{47}{24} \log(r+4)
        + 1.89.
      \]

      \item For~$f(x) = (x+1)\log(x+1)$, we have
        \[
          S_{r,k}(f) \ge
          k(r-k)\log(r+1) - \tfrac 12 k(r-k) - \tfrac 54 \log(r+1)
          + \tfrac 54 \log 2 - \tfrac 34.
        \]

      \item For~$f(x) = \log(x+1)$, we have~$S_{r,k}(f) \le
          (r+1)\log 2 + \tfrac 32 \log 2 - \tfrac 78$.

      \item For~$f(x) = \frac{1}{x+1}$, we have~$S_k(f)+S_{r-k}(f) \le 2\log(\tfrac r2 +1)$.

      \item For~$f(x) = \log\Gamma(x+1)$, we have
          \[
            S_{r,k}(f) \ge
          k(r-k)\log(r+1) - \tfrac 32 k(r-k)
          -r\log 2
          -\tfrac{17}{12}\log(r+2)
          -0.626.
        \]
  \end{enumerate}
\end{lemma}
\begin{proof}
  We will repeatedly use Euler--Maclaurin summation in the following
  form~\cite[Corollary~9.2.3 and Proposition~9.2.5]{cohen-NT-2}:
  if~$f$ is $C^4$ and both $f^{(2)}$ and~$f^{(4)}$ do not change sign
  on~$[1,+\infty)$, then
  \[
    S_m(f) = \int_1^m f(x)dx + \frac{f(1)+f(m)}{2} + R
    \text{ where } |R| \le \frac{|f'(m)-f'(1)|}{12}.
  \]
  \begin{enumerate}
    \item We have
      \begin{eqnarray*}
        && S_m(f) \\
        &=& (\tfrac{m}{2}+1)^2\bigl(\log(\tfrac m2+1)-\tfrac 12\bigr)
        - \tfrac 94 \bigl(\log(\tfrac 32)-\tfrac 12\bigr)
        + \tfrac 12 \bigl((\tfrac m2+1)\log(\tfrac m2+1) + \tfrac 32 \log(\tfrac 32)\bigr)
        + R \\
        &=& \tfrac 14(m+2)^2\bigl( \log(\tfrac m2+1)-\tfrac 12\bigr)
        - \tfrac 94 \bigl(\log(\tfrac 32)-\tfrac 12\bigr)
        + \tfrac 14 \bigl((m+2)\log(\tfrac m2+1) + 3\log(\tfrac 32)\bigr)
        + R
      \end{eqnarray*}
      where
      \[
        |R| \le \tfrac 1{24} \bigl(\log(\tfrac m2 +1) - \log(\tfrac 32)\bigr).
      \]
      We bound
      \begin{eqnarray*}
        && (r+2)^2\log(\tfrac{r}{2}+1)
        -(k+2)^2\log(\tfrac{k}{2}+1)
        -(r-k+2)^2\log(\tfrac{r-k}{2}+1) \\
        &\ge&
        (r+2)^2\log(\tfrac{r}{2}+1)
        -(k+2)^2\log(\tfrac{r}{2}+1)
        -(r-k+2)^2\log(\tfrac{r}{2}+1) \\
        &=&
        2(k(r-k)-2)\log(\tfrac{r}{2}+1)
      \end{eqnarray*}

      and

      \begin{eqnarray*}
        && (r+2)\log(\tfrac{r}{2}+1)
          - (k+2)\log(\tfrac{k}{2}+1)
          - (r-k+2)\log(\tfrac{r-k}{2}+1) \\
        &\ge& (r+2)\log(\tfrac{r}{2}+1)
          - (k+2)\log(\tfrac{r}{2}+1)
          - (r-k+2)\log(\tfrac{r}{2}+1) \\
        &=& -2\log(\tfrac{r}{2}+1)
      \end{eqnarray*}

      to obtain

      \begin{eqnarray*}
        S_{r,k}(f) &\ge&
          \tfrac 12(k(r-k)-2)\log(\tfrac{r}{2}+1)
        - \tfrac 14 (k(r-k)-2)
        - \tfrac 12\log(\tfrac{r}{2}+1) \\
        && + \tfrac 94 (\log(\tfrac 32)-\tfrac 12)
        - \tfrac 34 \log(\tfrac 32)
        - \tfrac 18 (\log(\tfrac r2 +1) - \log(\tfrac 32)) \\
        &=& 
          \tfrac 12k(r-k)\log(\tfrac{r}{2}+1)
        - \tfrac 14 k(r-k)
        - \tfrac{13}{8}\log(\tfrac r2+1)
        + \tfrac{13}{8}\log(\tfrac 32)
        - \tfrac{5}{8}.
      \end{eqnarray*}

    \item We have
      \[
        S_m(f) = (m+2)\bigl(\log(\tfrac m2+1)-1\bigr) - 3(\log(\tfrac 32)-1)
        + \tfrac 12 \log(\tfrac m2+1) + \tfrac 12\log(\tfrac 32)
        + R
      \]
      where
      \[
        |R| \le \tfrac{1}{12}(\tfrac 13 - \tfrac{1}{m+2}) \le \tfrac{1}{36}.
      \]
      We bound
      \begin{eqnarray*}
        && (r+2)\log(\tfrac r2+1)
        -(k+2)\log(\tfrac k2+1)
        -(r-k+2)\log(\tfrac {r-k}2+1) \\
        &\le& 
         (r+2)\log(\tfrac r2+1)
        -2(\tfrac r2+2)\log(\tfrac r4+1) \\
        &\le& (r+2)\log(\tfrac{2r+4}{r+4}) \\
        &\le& (r+2)\log 2
      \end{eqnarray*}
      and
      \[
        \tfrac 12 \log(\tfrac r2+1)
        -\tfrac 12 \log(\tfrac k2+1)
        -\tfrac 12 \log(\tfrac {r-k}2+1)
        \le 0.
      \]

      We get
      \[
        S_{r,k}(f) \le
        (r+2)\log 2
        + 2
        + 3(\log(\tfrac 32)-1)
        - \tfrac 12\log(\tfrac 32)
        + \tfrac{1}{12}
        = (r+2)\log 2 - \tfrac{11}{12} - \tfrac{5}{2}\log(\tfrac 32).
      \]

    \item We have
      \[
        S_m(f) \le \int_0^m \frac{dt}{\frac{t}{2}+1} = 2 \log(\tfrac m2+1),
      \]
      and therefore
      \[
        S_k(f) + S_{r-k}(f) \le 2 \log(\tfrac k2+1) + 2 \log(\tfrac {r-k}2+1)
        \le 4 \log(\tfrac r4+1).
      \]

    \item We use the following bound~\cite[Theorem~8]{alzer}: for all~$y>0$ we
      have
      \[
        (y-\tfrac 12)\log(y) -y + \tfrac{\log(2\pi)}{2} 
        < \log\Gamma(y) <
        (y-\tfrac 12)\log(y) -y + \tfrac{\log(2\pi)}{2} + \tfrac{1}{12y}.
      \]

      Summing the various contributions, we get
      \begin{eqnarray*}
        && S_{r,k}(f) \\
        &\ge&
          \tfrac 12k(r-k)\log(\tfrac{r}{2}+1) %
        - \tfrac 14 k(r-k) %
        - \tfrac{13}{8}\log(\tfrac r2+1)
        + \tfrac{13}{8}\log(\tfrac 32)
        - \tfrac{5}{8}
        - (r+2)\tfrac{\log 2}{2} \\ %
        && + \tfrac{11}{24}
        + \tfrac{5}{4}\log(\tfrac 32)
        - \tfrac{1}{2}k(r-k) %
        - \tfrac{1}{3} \log(\tfrac r4+1) \\ %
        &\ge&
          \tfrac 12k(r-k)\log(\tfrac{r}{2}+1)
        - \tfrac 34 k(r-k)
        - r\tfrac{\log 2}{2}
        - \tfrac{47}{24} \log(r+4)
        + 1.89.
      \end{eqnarray*}

      \item We have
        \[
          S_m(f) = \tfrac 12(m+1)^2\bigl(\log(m+1)-\tfrac 12\bigr)
          - 2(\log 2 - \tfrac 12)
          + \tfrac 12 (m+1)\log(m+1) + \log 2
          + R
        \]
        where
        \[
          |R| \le \tfrac 1{12} \bigl(\log(m+1)  - \log 2\bigr).
        \]

        We bound
        \begin{eqnarray*}
          && (r+1)^2\log(r+1) - (k+1)^2\log(k+1) - (r-k+1)^2\log(r-k+1) \\
          &\ge& (r+1)^2\log(r+1) - (k+1)^2\log(r+1) - (r-k+1)^2\log(r+1) \\
          &=& (2k(r-k)-1)\log(r+1)
        \end{eqnarray*}
        
        and

        \begin{eqnarray*}
          && (r+1)\log(r+1) - (k+1)\log(k+1) - (r-k+1)\log(r-k+1) \\
          &\ge& (r+1)\log(r+1) - (k+1)\log(r+1) - (r-k+1)\log(r+1) \\
          &=& -\log(r+1)
        \end{eqnarray*}

        to obtain

        \begin{eqnarray*}
          S_{r,k}(f)
          &\ge&
            (k(r-k)-\tfrac 12)\log(r+1) %
          - \tfrac 14(2k(r-k)-1) %
          + 2(\log 2 - \tfrac 12) %
          -\tfrac 12 \log(r+1) \\  %
          && - \log 2 %
          - \tfrac 1{4} \bigl(\log(r+1)  - \log 2\bigr) \\ %
          &=& k(r-k)\log(r+1) - \tfrac 12 k(r-k) - \tfrac 54 \log(r+1)
          + \tfrac 54 \log 2 - \tfrac 34.
        \end{eqnarray*}

      \item We have
        \[
          S_m(f) = (m+1)\bigl(\log(m+1)-1\bigr) - 2(\log 2 - 1)
          + \tfrac 12 \log(m+1) + \tfrac 12\log 2 + R
        \]
        where
        \[
          |R| \le \tfrac{1}{12}(\tfrac 12 - \tfrac 1{m+1}) \le \tfrac 1{24}.
        \]

        We bound

        \begin{eqnarray*}
          && (r+1)\log(r+1) - (k+1)\log(k+1) - (r-k+1)\log(r-k+1) \\
          &\le& (r+1)\log(r+1) - 2(\tfrac r2+1)\log(\tfrac r2+1) \\
          &\le& (r+1)\log(\tfrac{2r+2}{r+2}) \\
          &\le& (r+1)\log 2
        \end{eqnarray*}

        and

        \[
          \log(r+1) - \log(k+1) - \log(r-k+1) \le 0.
        \]

        We get
        \[
          S_{r,k}(f) \le 
          (r+1)\log 2 + \tfrac 32 \log 2 - 1 + \tfrac 18
          =
          (r+1)\log 2 + \tfrac 32 \log 2 - \tfrac 78.
        \]

      \item We have
        \[
          S_m(f) \le \int_0^m \frac{dt}{t+1} = \log(m+1),
        \]

        and therefore

        \[
          S_k(f)+S_{r-k}(f) \le \log(k+1) + \log(r-k+1) \le 2\log(\tfrac r2 +1).
        \]

      \item Summing the various contributions, we get
        \begin{eqnarray*}
          && S_{r,k}(f) \\
          &\ge&
          k(r-k)\log(r+1) - \tfrac 12 k(r-k) - \tfrac 54\log(r+1)
          + \tfrac 54\log 2 - \tfrac 34
          - (r+1)\log 2 \\
          && - \tfrac 32\log 2 + \tfrac 78
          -k(r-k)
          -\tfrac 16\log(\tfrac r2 +1) \\
          &\ge&
          k(r-k)\log(r+1) - \tfrac 32 k(r-k)
          -r\log 2
          -\tfrac{17}{12}\log(r+2)
          -0.626.
        \end{eqnarray*}
  \end{enumerate}
\end{proof}

\begin{proposition}\label{prop:Prkt}
  Let $t\in\R_{\ge 1}$ and~$k\in\Z_{\ge 1}$.
  The measure~$P_{r,k,t}$ of the set of module lattices that admit a $t$-destabilising
  sublattice of rank~$k$ satisfies
  \begin{eqnarray*}
    P_{r,k,t} &\le&
    \frac{\zeta_K^*(1)}{r \cdot |\Delta_K|^{\frac{k(r-k)}{2}}}
    \cdot
    \left(
      \binom{r}{k}
      \frac{\prod_{j=2}^k \zeta(j)\prod_{j=2}^{r-k} \zeta(j)\prod_{j=1}^r V_j}{\prod_{j=1}^k V_j\prod_{j=1}^{r-k} V_j}
    \right)^{r_1} \\
    && \cdot \left(
      \binom{r}{k}
      2^{k(r-k)}
      \frac{\prod_{j=2}^k \zeta(j)^2\prod_{j=2}^{r-k} \zeta(j)^2\prod_{j=1}^r V_{2j}}{\prod_{j=1}^k V_{2j}\prod_{j=1}^{r-k} V_{2j}}
    \right)^{r_2}
    \cdot t^{-r} \\
    &\le&
    \frac{\zeta_K^*(1)}{|\Delta_K|^{\frac{k(r-k)}{2}}}
    \left(
      0.8 \cdot (r+4)^2 %
      \cdot 2^{\frac{3r}{2}} \cdot
      \Bigl(
        \frac{28.2}{r+2}
      \Bigr)^{\frac{k(r-k)}{2}}
    \right)^{r_1} \\
    && \cdot \left(
      53 \cdot (r+2)^2 \cdot 4^{r} \cdot
      \Bigl(
        \frac{28.2}{r+1}
      \Bigr)^{k(r-k)}
    \right)^{r_2}
    \cdot t^{-r}.
  \end{eqnarray*}
\end{proposition}
\begin{proof}
  From Lemma~\ref{lem:crk}, write
  \[
    c(r,k) =
    \frac{\zeta_K^*(1)}{r}
    \binom{r}{k}^{\unitrk+1}
    \left(\frac{\pi^{\frac{r_1}{2}}(2\pi)^{r_2}}{|\Delta_K|^{\frac 12}}\right)^{k(r-k)}
    Z G_1^{r_1} G_2^{r_2}
  \]
  where
  \[
    Z =
    \frac{\prod_{j=2}^k\zeta_K(j)\prod_{j=2}^{r-k}\zeta_K(j)}{\prod_{j=2}^r\zeta_K(j)},
  \]
  \[
    G_1 =
    \frac{\prod_{j=1}^k\Gamma(\frac{j}{2}+1)\prod_{j=1}^{r-k}\Gamma(\frac{j}{2}+1)}{\prod_{j=1}^r\Gamma(\frac{j}{2}+1)}
  \]
  and
  \[
    G_2 =
    \frac{\prod_{j=1}^k\Gamma(j+1)\prod_{j=1}^{r-k}\Gamma(j+1)}{\prod_{j=1}^r\Gamma(j+1)}.
  \]
  By Lemma~\ref{lem:boundZ} we have~$Z \le (2.3)^{2d} = (2.3)^{2r_1}(2.3)^{4r_2}$.
  By Lemma~\ref{lem:boundG}, we have
  \[
    G_1 \le
    \exp(
        - \tfrac 12k(r-k)\log(\tfrac{r}{2}+1)
        + \tfrac 34 k(r-k)
        + r\tfrac{\log 2}{2}
        + \tfrac{47}{24} \log(r+4)
        - 1.89
      )
  \]
  and
  \[
    G_2 \le
    \exp(
          -k(r-k)\log(r+1) + \tfrac 32 k(r-k)
          +r\log 2
          +\tfrac{17}{12}\log(r+2)
          +0.626
    ).
  \]
  Using the trivial bound~$\binom{r}{k} \le 2^r$ and putting the terms together
  gives the result.
\end{proof}

We now quantify the fact that semistable lattices are balanced. This bound is
implicitly present in the proof of~\cite[Theorem~5.1]{grayson}.

\begin{lemma}\label{lem:stab-lambdai}
  Let~$L$ be a module lattice of rank~$r$ and let~$t\ge 1$.
  \begin{enumerate}
    \item If~$L$ does not admit a $t$-destabilising sublattice of rank~$1$, then
      \[
        \lambda_1(L) > t^{-\frac{1}{d}}|\Delta_K|^{-\frac{1}{2d}}\det(L)^{\frac{1}{n}}.
      \]
    \item If~$L$ does not admit a $t$-destabilising sublattice of rank~$n-1$, then
      \[
        \lambda_n(L) <
        n t^{\frac{1}{d}}|\Delta_K|^{\frac{1}{2d}}\det(L)^{\frac{1}{n}}.
      \]
  \end{enumerate}
  If $L$ is semistable then
  $\lambda_1(L) \ge |\Delta_K|^{-\frac{1}{2d}}\det(L)^{\frac{1}{n}}$
  and
  $\lambda_n(L) \le n |\Delta_K|^{\frac{1}{2d}}\det(L)^{\frac{1}{n}}$.
\end{lemma}
\begin{proof}\hfill
  \begin{enumerate}
    \item We prove the contrapositive. Suppose that the bound is not satisfied, and let $x\in L$ be such
      that
      \[
        \|x\|\le t^{-\frac{1}{d}}|\Delta_K|^{-\frac{1}{2d}}\det(L)^{\frac{1}{n}}.
      \]
      Then the rank~$1$ sublattice~$L'=\ZK x$ satisfies
      \[
        \det(\ZK x) = \|x\|^d |\Delta_K|^{1/2} \le t^{-1}\det(L)^{\frac{1}{r}},
      \]
      so that~$L'$ is $t$-destabilising.
    \item Assume $L$ does not admit a $t$-destabilising sublattice of
      rank~$n-1$, then $L^\vee$ does not admit a $t$-destabilising sublattice of
      rank~$1$. By the first part of the lemma, we have
      \[
        \lambda_1(L^\vee) >
        t^{-\frac{1}{d}}|\Delta_K|^{-\frac{1}{2d}}\det(L)^{-\frac{1}{n}}.
      \]
      Finally, Banaszczyk's theorem \cite[Theorem (2.1)]{Banaszczyk1993} gives
      \[
        \lambda_n(L) <
        n t^{\frac{1}{d}}|\Delta_K|^{\frac{1}{2d}}\det(L)^{\frac{1}{n}}.
      \]
  \end{enumerate}
  If~$L$ is semistable, then both inequalities hold for every~$t>1$, yielding
  the result by letting~$t\to 1$.
\end{proof}

Piecing together the results above, we prove the main result of this section.

\begin{proof}[Proof of Theorem \ref{thm:randbalanced}]
  We will use the notation of Proposition~\ref{prop:Prkt}.

  First suppose that~$r\ge 225$ and take~$t=1$.
  Applying Lemma~\ref{lem:louboutin-residue}, we can bound
  \[
    \frac{\zeta_K^*(1)}{|\Delta_K|^{\frac{k(r-k)}{2}}}
    \le
    \frac{1}{|\Delta_K|^{\frac{k(r-k)-1}{2}}}
    \le 1.
  \]

  We also bound
  \begin{eqnarray*}
    && 0.8 \cdot (r+4)^2
      \cdot 2^{\frac{3r}{2}} \cdot
      \Bigl(
        \frac{28.2}{r+2}
      \Bigr)^{\frac{k(r-k)}{2}} \\
    &\le& 0.8 \cdot (r+4)^2
      \cdot 2^{\frac{3r}{2}} \cdot
      \Bigl(
        \frac{28.2}{r+2}
      \Bigr)^{\frac{r-1}{2}} \text{ since }r+2>28.2\\
    &\le& O(r^{\frac 52})
      \cdot
      \Bigl(
        \frac{2^3 \cdot 28.2}{r+2}
      \Bigr)^{\frac{r}{2}} \\
    &\le& O(r^{\frac 52})
      \cdot
      \Bigl(
        \frac{225.6}{r+2}
      \Bigr)^{\frac{r}{2}} \\
    &=& 2^{-\Omega(r\log r)},
  \end{eqnarray*}

  and similarly

  \begin{eqnarray*}
    && 53 \cdot (r+2)^2 \cdot 4^{r} \cdot
      \Bigl(
        \frac{28.2}{r+1}
      \Bigr)^{k(r-k)} \\
    &\le& 53 \cdot (r+2)^2 \cdot 4^{r} \cdot
      \Bigl(
        \frac{28.2}{r+1}
      \Bigr)^{r-1} \\
    &\le& O(r^3)
    \cdot
      \Bigl(
        \frac{4\cdot 28.2}{r+1}
      \Bigr)^{r} \\
    &=& 2^{-\Omega(r\log r)},
  \end{eqnarray*}
  so that in those cases we indeed have~$P_{r,k,1} \le 2^{-\Omega(dr\log r)}$.

  Now for~$4\le r\le 224$, we apply the Odlyzko--Serre
  bound~\cite{poitou-disc}:
  \[
    |\Delta_K| \ge (A^{r_1}B^{2r_2})^{1+o(1)} \text{ as }d\to\infty,
  \]
  where~$A = 4\pi\exp(1+\gamma)$, $B = 4\pi\exp(\gamma)$ and~$\gamma$ is Euler's
  constant.
  For each such~$r$, each~$0<k<r$ and~$t=1$, we evaluate the explicit formula for the first bound in
  Proposition~\ref{prop:Prkt}, inserting~$A^{-\frac{k(r-k)-1}{2}}$ in the~$r_1$
  term and~$B^{-k(r-k)+1}$ in the~$r_2$ term, and we check that both expressions
  are strictly less than~$1$. This proves that for each such~$r$ and~$k$, we
  have~$P_{r,k,1} = 2^{-\Omega(d)}$.

  Now assume~$r=2$. The bound from Proposition~\ref{prop:Prkt} is
  \[
    P_{2,1,t} \le \frac 12 \cdot
    \frac{\zeta_K^*(1)}{|\Delta_K|^{\frac 12}}\pi^{r_1}(2\pi)^{r_2}t^{-2}.
  \]
  Using Lemma~\ref{lem:louboutin-residue} we bound
  \[
    \zeta_K^*(1)|\Delta_K|^{-\frac 12}
    \le 
    \left(\frac{e \log |\Delta_K|}{2(d-1)}\right)^{d-1}
    |\Delta_K|^{-\frac 12}
    \le
    \left(\frac{e}{2}\frac{d}{d-1} \frac{\log \delta}{\delta^{\frac
    12}}\right)^{d-1}.
  \]
  We obtain
  \[
    P_{2,1,t} \le 
    \left((1+o(1)) \frac{\pi e}{2}\frac{\log \delta}{\delta^{\frac
    12}}\right)^{d-1} \cdot O(t^{-2}).
  \]
  For the stated choice of~$t$, this is~$2^{-\Omega(d)}$.
  When~$\delta\ge 845$, we have~$t=1$.

  Finally, assume~$r=3$.
  The bound from Proposition~\ref{prop:Prkt} is
  \[
    P_{3,k,t} \le
    \frac{1}{3} \cdot 
    \frac{\zeta_K^*(1)}{|\Delta_K|}
    \left(\frac{\pi^3}{3}\right)^{r_1}
    \left(\frac{\pi^6}{18}\right)^{r_2}t^{-3}.
  \]

  Using Lemma~\ref{lem:louboutin-residue} again we bound
  \[
    \zeta_K^*(1)|\Delta_K|^{-1}
    \le
    \left(\frac{e}{2}\frac{d}{d-1} \frac{\log \delta}{\delta}\right)^{d-1}.
  \]

  We obtain
  \[
    P_{3,k,t} \le
    \left((1+o(1))\frac{\pi^3 e}{6}\frac{\log \delta}{\delta}\right)^{d-1}
    \cdot O(t^{-3}).
  \]
  For the stated choice of~$t$, this is~$2^{-\Omega(d)}$.
  When~$\delta\ge 57.5$, we have~$t=1$.

  We obtain the last statement by applying Lemma~\ref{lem:stab-lambdai} and noting
  that the values of~$t$ for~$r=2$ and~$r=3$ satisfy~$t^{\frac 1d} = O(1)$.
\end{proof}

\section{Cutting cusps: reduction to the flare} \label{sec:cutting-cusps}

The goal of this section is to prove \Cref{thm:cusp-to-flare} below, which reduces worst-case SIVP instances to SIVP in lattices which are (mildly) balanced.

\begin{theorem}[Reduction to the flare]\label{thm:cusp-to-flare}
Let $L$ be an $\ZK$-module lattice of rank $r$, and $\gamma \geq 1$.
There is a polynomial time reduction from $\gamma \cdot (1+\varepsilon)^{r-1}$-SIVP in $L$ to $\gamma$-SIVP in at most $r$ module lattices $L_1,\dots,L_t$, where each $L_i$ is of rank $r$ and $\Gamma_K^2 2^{\frac{3}{2}(rd - 1)}$-balanced, and $\varepsilon < \frac{d}{2^{(rd+1)/2}}$.
\end{theorem}

We proceed in two steps. In \Cref{subsection:red-to-lower-dim}, we prove that if the given lattice $L$ is very imbalanced (it is \emph{in the cusp}), then a polynomial time lattice-basis reduction like LLL can detect gaps between the successive minima, and exploit them to split $L$ into lattices of smaller dimension with smaller gaps. In order to preserve the dimension, we then show in \Cref{subsection:back-to-original-dim} that SIVP in these lattices of smaller dimension reduces to SIVP in lattices of the original dimension, but now with balancedness guarantees: they are now \emph{in the flare}.

\subsection{Splitting imbalanced lattices into smaller dimensions}\label{subsection:red-to-lower-dim}

To reduce to (mildly) balanced lattices, we start by showing in \Cref{lem:basis-of-dense-submodule-find-index} that large gaps between successive minima can be detected in polynomial time.
Once we know where such a gap is, we show in \Cref{lem:basis-of-dense-submodule-K-gap} how to find generators of the ``denser'' sublattice (reaching all first minima up to the gap). Then, in \Cref{lemma:skewed-split-in-two}, we show how SIVP in the original lattice reduces to SIVP in this denser sublattice, and in a lattice of complementary dimension. Essentially, this splits the original lattice around the gap, resulting in two lattices of smaller dimension and with one fewer (large) gap.

Finally, \Cref{lem:to-balanced-smalled-dim} applies this splitting recursively, resulting in a collection of lattices of smaller dimension with no remaining (large) gap.

\begin{lemma}\label{lem:basis-of-dense-submodule-find-index}
There is a polynomial time algorithm such that the following holds.
Let $L$ be an $\ZK$-module lattice of rank $r$, with successive $K$-minima $\lambda_1^K,\dots,\lambda_{r}^K$.
Given $\alpha>0$ and a basis of $L$, the algorithm either asserts that $\lambda^K_{i+1}/\lambda^K_{i}\leq  \alpha \Gamma_K 2^{rd-1}$ for all $i$, or returns an index $k$ such that $\lambda^K_{k+1}/\lambda^K_{k}>  \alpha$.
\end{lemma}

\begin{proof}
Let $(u_i)_{i=1}^{rd}$ be a family of linearly independent vectors in $L$ with $\|u_i\| = \lambda_i = \lambda_i(L)$.
One can compute in polynomial time an LLL-reduced basis $(b_i)_i$ of $L$.
By~\cite[Proposition~1.12]{lenstra82:_factor} for any $i$ we have
$$\|b_i\| \leq 2^{(rd-1)/2}\lambda_i.$$
The algorithm searches for an index $j$ such that $\|b_{j+d}\|/\|b_j\| > \alpha\Gamma_K2^{(rd-1)/2}$, and if it exists, returns $k = \lceil j/d \rceil$. If there is no such $j$, the algorithm asserts that $\lambda^K_{i+1}/\lambda^K_{i}\leq  \alpha \Gamma_K 2^{rd-1}$ for all $i$. We prove correctness in two parts:
\begin{itemize}
\item Assume a valid $j$ is found. We have
$$\frac{\lambda_{k+1}^K}{\lambda^K_k} \geq \frac{\lambda_{j+d}}{\Gamma_K\lambda_j} \geq \frac{\|b_{j+d}\|}{2^{(rd-1)/2}\Gamma_K\|b_{j}\|}> \alpha,$$
as expected.
\item Assume there exists an index $k$ such that $\lambda^K_{k+1}/\lambda^K_{k}>  \beta  2^{(rd-1)/2}$.
Let $j$ be the largest index reaching $\lambda_{j} = \lambda^K_{k}$ (in particular, $\lceil j/d \rceil = k$).
Applying Lemma~\ref{lem:Kminima}, we obtain 
$$\frac{\|b_{j+d}\|}{\|b_j\|} \geq \frac{\lambda_{j+d}}{2^{(rd-1)/2}\lambda_j}\geq \frac{\lambda_{k+1}^K}{2^{(rd-1)/2}\lambda^K_k} > \beta.$$
The contraposition, with $\beta = \alpha\Gamma_K2^{(rd-1)/2}$, states that if the algorithm finds  no valid index $j$, then $\lambda^K_{i+1}/\lambda^K_{i}\leq  \alpha \Gamma_K 2^{rd-1}$ for all $i$.
\end{itemize}
This proves that the algorithm has the claimed property.
\end{proof}

\begin{lemma}\label{lem:basis-of-dense-submodule-K-gap}
There is a polynomial time algorithm such that the following holds.
Let $L$ be an $\ZK$-module lattice of rank $r$, with successive $K$-minima $\lambda_1^K,\dots,\lambda_{r}^K$.
Given a basis of $L$ and an index $k$ such that $\lambda^K_{k+1}/\lambda^K_{k}>  \Gamma_K2^{(rd-1)/2}$, the algorithm returns a basis of the unique primitive sub-module $L' \subset L$ of rank $k$ with $\lambda_i^K(L') = \lambda_i^K$ for all $i \leq k$.
\end{lemma}

\begin{proof}
One can compute in polynomial time an LLL-reduced basis $(b_i)_i$ of $L$.
Let $j$ be the smallest index such that $\Span_K(b_1,\dots,b_j)$ has $K$-rank $k$ (in particular, $j \leq (k-1)d+1$). 
For any $i \leq j$, we have
$$\|b_i\| \leq 2^{(rd-1)/2}\lambda_i\leq 2^{(rd-1)/2}\lambda_{j} \leq \Gamma_K 2^{(rd-1)/2}\lambda_{\lceil j/d\rceil}^K \leq \Gamma_K2^{(rd-1)/2}\lambda_k^K       < \lambda_{k+1}^K.$$
Let $V = \Span_K(x \in L \mid \|x\| < \lambda_{k+1}^K)$.
The vectors $(b_1,\dots,b_{j})$ are all in $V$. Therefore, $\mathrm{span}_K (b_1,\dots,b_{j})$ is a $K$-subspace of $V$ of $K$-rank $k$. By definition of $\lambda_{k+1}^K$, the space $V$ has $K$-rank at most $k$, and
we deduce that $(b_1,\dots,b_{j})$ generates $V$. From this generating set of $V$ and the provided basis of $L$, we can deduce a basis of the sub-module $L' = L \cap V$ in polynomial time, which proves the lemma.
\end{proof}

\begin{lemma}\label{lemma:skewed-split-in-two}
Suppose $L$ is an $\ZK$-module lattice of rank $r$, with successive $K$-minima $\lambda_1^K,\dots,\lambda_{r}^K$.
Let $k$ be an index such that $\beta = \lambda^K_{k+1}/\lambda^K_{k}>  \Gamma_K 2^{(rd-1)/2}$. Then, given $k$, there is a polynomial time reduction from $\gamma \cdot (1+\varepsilon)$-SIVP in $L$ to $\gamma$-SIVP in two module lattices of rank $k$ and $r-k$, with $\varepsilon = \frac{d\Gamma_K}{2\beta} < \frac{d}{2^{(rd+1)/2}}$.
\end{lemma}

\begin{proof}
From Lemma~\ref{lem:basis-of-dense-submodule-K-gap}, one can compute in polynomial time a basis of the unique sub-module $L' \subset L$ of rank $k$ with $\lambda_i^K(L') = \lambda_i^K$ for all $i \leq k$.

Let $(u_i)_{i=1}^{rd}$ be a family of linearly independent vectors in $L$ with $\|u_i\| = \lambda_i$.
Let us start with finding a good basis of $L'$.
Applying the $\gamma$-SIVP oracle to $L'$ we can find $w_i\in L'$ such that $\|w_i\| \leq \gamma \lambda_{kd}(L')$.
By Lemma~\ref{lem:Kminima}, we have
$$ \lambda_{kd}(L') \leq \Gamma_K\lambda^K_{k}(L')\leq \Gamma_K\lambda^K_{k} \leq (\Gamma_K/\beta)\lambda^K_{k+1} \leq (\Gamma_K/\beta)\lambda_{rd}.$$
We deduce $\|w_i\| \leq  \gamma(\Gamma_K/\beta)\lambda_{rd}$.
In particular, $\|w_i\| <  \gamma\lambda_{rd}$.

Let us now complete $(w_i)_{i=1}^{kd}$ to a good basis of $L$.
Let $V = \mathrm{span}_K(L)$ and $W = \mathrm{span}_K(L')$, and consider the orthogonal projection $\pi : V \to W^\perp$. Then, $L_\pi = \pi(L)$ is a module lattice of rank $r-k$.
We have $\|\pi(u_i)\| \leq \|u_i\| = \lambda_{i}$. Applying the $\gamma$-SIVP oracle to $L_\pi$ we can find $z_i\in L$ such that $0 < \|\pi(z_i)\| \leq \gamma \lambda_{(r-k)d} (L_\pi) \leq \gamma \lambda_{rd}$. 
We can assume each $z_i$ to be reduced with respect to the basis $(w_i)_i$ of $W$, so $z_i = \pi(z_i) + \sum_{i}\mu_{i}w_i$  with $|\mu_i| < 1/2$. Recall that $\|w_i\| \leq  \gamma(\Gamma_K/\beta)\lambda_{rd}$, so 
\begin{align*}
\|z_i\| \leq \|\pi(z_i)\| + \sum_{i}|\mu_{i}|\|w_i\| 
& \leq \gamma\lambda_{rd} + (d/2)\gamma(\Gamma_K/\beta)\lambda_{rd}\\
& = \gamma\left(1 + \frac{d\Gamma_K}{2\beta}\right)\lambda_{rd}.
\end{align*}
Therefore, $(w_1,\dots,w_{kd},z_{1},\dots,z_{(r-k)d})$ is a solution of $\gamma \cdot (1+\varepsilon)$-SIVP for $L$.
\end{proof}

\begin{lemma}\label{lem:balanced-or-reduce}
Suppose $L$ is an $\ZK$-module lattice of rank $r$.
There is a polynomial time algorithm which either asserts that $L$ is $\Gamma_K^2 2^{\frac{3}{2}(rd - 1)}$-balanced, or reduces $\gamma \cdot (1+\varepsilon)$-SIVP in $L$ to $\gamma$-SIVP in two module lattices $L_1$ and $L_2$ with $\rank_K(L_1) + \rank_K(L_2) = r$ and $\rank_K(L_i) < r$, with $\varepsilon < \frac{d}{2^{(rd+1)/2}}$. 
\end{lemma}

\begin{proof}
This is a combination of \Cref{lem:basis-of-dense-submodule-find-index} (detecting gaps) and \Cref{lemma:skewed-split-in-two} (exploiting gaps).
\end{proof}

\begin{lemma}[Reduction to balanced lattices of smaller dimension]\label{lem:to-balanced-smalled-dim}
Let $L$ be an $\ZK$-module lattice of rank $r$, and $\gamma \geq 1$.
There is a polynomial time reduction from $\gamma \cdot (1+\varepsilon)^{r-1}$-SIVP in $L$ to $\gamma$-SIVP in at most $r$ module lattices $L_1,\dots,L_t$, with 
\begin{itemize}
\item $\varepsilon < \frac{d}{2^{(rd+1)/2}}$,
\item $\sum_{i=1}^t\rank_K(L_i) = r$,
\item each $L_i$ is $\Gamma_K^2 2^{\frac{3}{2}(\rank_K(L_i)d - 1)}$-balanced.
\end{itemize}
\end{lemma}

\begin{proof}
This follows from a recursive application of \Cref{lem:balanced-or-reduce}, and the fact that a rank-1 lattice is necessarily $\Gamma_K$-balanced (hence $\Gamma_K^2 2^{\frac{3}{2}(d - 1)}$-balanced). The recursion has depth at most $r-1$ since the quantity $\sum_{i=1}^t\rank_K(L_i) = r$ is constant and $t$ can only increase.
\end{proof}

\subsection{Back to the original dimension}\label{subsection:back-to-original-dim}

The previous section shows how to reduce SIVP in an imbalanced lattice into SIVP instances in balanced lattices, but these lattices have smaller dimension. We would like the computational reduction to preserve the dimension. Reducing the dimension sounds good in practice, but \emph{a priori}, there could exist $r$ such that the average case in dimension $r-1$ is harder than the average case in dimension $r$. To resolve this concern, in this section, we prove that SIVP is lattices of smaller dimension reduces to SIVP in lattices of the original dimension $r$.

\begin{lemma}[Increasing the dimension]\label{lem:small-dim-to-original-dim}
Suppose $L$ is an $\alpha$-balanced $\ZK$-module lattice of rank $k < r$.
There is a polynomial time reduction from $\gamma$-SIVP in $L$ to $\gamma$-SIVP in a $\max(\alpha, \sqrt{k}, \sqrt{d}\cdot\Gamma_K)$-balanced $\ZK$-module lattice of rank $r$.
\end{lemma}

\begin{proof}
Let $O = \ZK^{r-k}$ be the orthogonal $\ZK$-lattice of rank $r-k$.
Let $x>0$ and $M = L \oplus xO$.
Let us prove that with $x = \det(L)^{\frac{1}{kd}}$, we have that $M$ is $\max(\alpha, \sqrt{kd}, \Gamma_K)$-balanced, and $\lambda_{rd}(M) \leq \lambda_{kd}(L)$.
 We have 
\[\lambda_1^K(L) = \lambda_1(L)  \leq \sqrt{kd}\det(L)^{\frac{1}{kd}} = \sqrt{kd} \cdot x,\]
and
\[x = \det(L)^{\frac{1}{kd}} \leq \left(\prod_{i = 1}^{kd}\lambda_{i}(L)\right)^{\frac{1}{kd}} \leq \lambda_{kd}(L)  \leq \Gamma_K\lambda_{k}^K(L).\]
Since $\lambda_1(x\ZK) = x\sqrt{d}$, we deduce
that $\lambda_1^K(L)/\sqrt{k} \leq \lambda_1(x\ZK)  \leq \sqrt{d}\cdot\Gamma_K\lambda_{k}^K(L)$.
Since $M$ is an orthogonal sum of $L$ and copies of $x\ZK$, we deduce that $M$ is $\max(\alpha, \sqrt{k}, \sqrt{d}\cdot\Gamma_K)$-balanced. From $x \leq \lambda_{kd}(L)$, we deduce that $\lambda_{rd}(M) \leq \lambda_{kd}(L)$. Therefore, a solution of $\gamma$-SIVP for $M$, projected orthogonally down to $L$, is a solution of $\gamma$-SIVP for $L$.
\end{proof}

We now have all the ingredients to prove the main result of this section.
\begin{proof}[Proof of \Cref{thm:cusp-to-flare}]
This is the composition of \Cref{lem:to-balanced-smalled-dim} and \Cref{lem:small-dim-to-original-dim}, and the fact that \[\max\left(\Gamma_K^2 2^{\frac{3}{2}(\rank_K(L_i)d - 1)}, \sqrt{\rank_K(L_i)}, \sqrt{d}\cdot\Gamma_K\right) \leq \Gamma_K^2 2^{\frac{3}{2}(rd - 1)}.\]
\end{proof}

\section{Reduction from the flare to the bulk} \label{sec:middle-to-bulk}

We will use Section \ref{sec:quantitative-equidistribution-bigsec} and Section \ref{sec:self-red-bulk-algorithmic} to show that we have an algorithm, based on Hecke equidistribution, that can handle $\alpha$-balanced module lattices $L$ with $\log \alpha \ll \log d$.
Section \ref{sec:cutting-cusps} shows that we can reduce to module lattices that are $\alpha$-balanced with $\alpha \ll \Gamma_K^2 \cdot 2^{O(d)}$.
There remains a gap between these two regimes.
Thus, we are left with further reducing from lattices not too high in the cusp, with $\alpha$ exponential in $d$, to those in the bulk, where $\alpha$ is only polynomial in $d$. 
We informally call this ``intermediate'' part of the space of module lattices the \emph{flare}, see \Cref{fig:schematicimage}.

The strategy is the following.
Take an $\alpha$-balanced lattice $L$ with $\alpha$ at most $2^d$, for simplicity.
Thus, the range where the gaps $\lambda^K_{i+1}(L)/\lambda^K_i(L)$ could lie is $[1, 2^d]$.
We split this range into dyadic intervals, of which there are only $d$ many, and \emph{guess} in which of these the first gap~$\lambda^K_2(L)/\lambda^K_1(L)$ lies.
Assuming the correct guess, we apply a Hecke operator, that is, we randomly consider a certain type of sublattice of index~$p$, where $p$ lies in the respective dyadic interval.
With high probability, because $p \ll \lambda^K_2(L)/\lambda^K_1(L)$, taking such a sublattice only increases the length of the shortest vector and the result has a first gap~$\lambda^K_2(L)/\lambda^K_1(L)$ of size~$\asymp 1$.
Morally, the other successive minima are not impacted, but in practice we prove that they are only potentially multiplied by a polynomial in~$d$.

Having reduced the first gap, we continue with the second, and so forth.
However, one must use different Hecke operators for this.
For instance, if $\lambda_1^K \asymp \lambda_2^K \asymp p^{-1} \lambda_3^K$, then we wish to increase the volume of a rank $2$ dense sublattice, taking care not to reopen the first gap.
This can be achieved by considering another family of sublattices as above, with a different structure inside $L$.
Following all steps up to the last gap, with high probability, we can obtain a sublattice with gaps bounded by a polynomial in $d$, depending on $r$.

At the level of Hecke operators, closing all gaps conceptually uses an entire set of generators for the local Hecke algebra.
We also note that this procedure is expensive in terms of the rank~$r$, but provides a good algorithm in terms of~$d$.

\subsection{Closing one gap} \label{sec:closing-one-gap}
The following implements the idea that, given a gap in the successive minima, carefully choosing a sublattice in terms of the size of that gap can effectively \emph{close} or shrink it.
It is the main tool of this section.

\begin{lemma}
	Let $L$ be a module lattice of rank $r$ with $\lambda_1^K, \ldots, \lambda_r^K$ its $K$-minima. 
	Assume that there exists $n \in \Z$, $n \geq 2$, such that $\lambda_{k+1}^K \geq n \lambda_{k}^K$ for some $k < r$.
	Let $M$ be the primitive sub-module of rank $k$ containing the vectors of length at most $\lambda_k^K$.
	Assume that $L_n \subset L$ is a sub-module of rank $r$ such that $L/L_n$ is isomorphic to $(\ZK/n\ZK)^k$ and $nM$ is primitive in $L_n$.
	Then
	\begin{displaymath}
		\lambda_i^K(L_n) = n \lambda_i^K, \quad i = 1, \ldots, k, 
	\end{displaymath}
	and
	\begin{displaymath}
		\lambda_i^K \leq \lambda_i^K(L_n) \leq \left(1+\frac{\Gamma_K\sqrt{kd}}{2}\right) \lambda_i^K, \quad i = k+1, \ldots, r.
	\end{displaymath}
\end{lemma}
\begin{proof}
	We start by noting that $M$ is well-defined.
	Indeed, since $\lambda_{k+1}^K > \lambda_{k}^K$, the vectors of length up to $\lambda_{k}^K$ have a $K$-span of dimension $k$.
	We can now define $M$ to be the maximal sub-module of rank $k$ containing these vectors and we recall Definition \ref{def:primitive}.
	
	Since $nM$ is primitive in $L_n$, we can find a sub-module $M' \subset L_n$ such that $L_n = nM \oplus M'$.
	We have $L_n \subset M \oplus M' \subset L$ and, computing indices, we find that $L = M \oplus M'$.
	
	For any $i \in \{1, \ldots k\}$, we clearly have the inequality $\lambda_i^K(L_n) \leq n \lambda_i^K(L)$.
	To prove the inverse inequality, assume that there exist $K$-independent vectors $w_1, \ldots, w_i \in L_n$ with lengths strictly smaller than $n \lambda_i^K(L)$.
	The lengths of $w_1, \ldots, w_i$ are also strictly smaller than $\lambda_{k+1}^K(L)$, by assumption.
	Therefore, by definition of the successive minima, the $K$-span of these vectors is included in the $K$-span of $M$.
	We can therefore deduce that
	\begin{displaymath}
		\Span_K(w_1, \ldots w_i) \subset K \cdot M  = K \cdot n M.
	\end{displaymath}
	
	Next, because $nM$ is primitive in $L_n$, we have
	\begin{displaymath}
		K \cdot nM \cap L_n = nM.
	\end{displaymath}
	It follows that the vectors $w_1, \ldots, w_i$ lie in $nM$.
	Dividing by $n$, we obtain $i$ $K$-linearly independent vectors $w_1/n, \ldots, w_i/n$ in $M \subset L$.
	Since $n$ is a rational number, their lengths are simply $\norm{w_1}/n, \ldots, \norm{w_i}/n$.
	These are strictly smaller than $\lambda_i^K(L)$ by assumption, so we reach a contradiction.
	
	We now consider the other successive minima.
	Let $(u_i)_{i \leq r}$ be a $K$-linearly independent family of vectors in $L$ with $u_i = v_i + w_i$, where $v_i \in M$, $w_i \in M'$, and $\|u_i\| = \lambda_i^K$ for each $i \in \{1,\dots,r\}$.
	In particular, $u_i \in M$ if $i \leq k$.
	
	For each $i>k$, let $v_i' \in nM$ be the closest vector to $v_i$, so that
	$$\|v_i' - v_i\| \leq \operatorname{cov}(nM) \leq \frac{\sqrt{kd}}{2}\lambda_{kd}(nM) \leq \Gamma_K\frac{\sqrt{kd}}{2}n\lambda_k^K \leq \Gamma_K\frac{\sqrt{kd}}{2}\lambda_i^K,$$
	where $\operatorname{M}$ is the covering radius of a lattice $M$ and we use the inequality given in \cite[Thm.~7.9]{MG02}.
	Let $u_i' = v_i' + w_i \in L_n$. 
	By construction, $(u_1,\dots,u_k,u_{k+1}',\dots,u_{r}')$ are $K$-linearly independent. 
	Furthermore,
	$$\|u_i'\| = \|v_i' + w_i'\| \leq \|v_i' - v_i'\| + \|v_i + w_i'\| \leq \Gamma_K\frac{\sqrt{kd}}{2}\lambda_i^K + \lambda_i^K,$$
	which proves the result.
\end{proof}

In the previous lemma, we consider specific sublattices $L_n$ of $L$, formed by scaling a fixed sub-module $M$, containing short vectors, by $n$.
Conversely, we now consider how many sublattices with the same structure can be formed this way.
We first start with $n$ replaced by a prime ideal.

\begin{lemma} \label{lem:good-submodule-prob-local}
	The number of sublattices $L' \subset L$ such that $L/L' \cong (\ZK/\p \ZK)^k$ is given by
	\begin{displaymath}
		\frac{(1-q^r)\cdots (1-q^{r-k+1})}{(1-q) \cdots (1-q^k)},
	\end{displaymath}
	where $q = \abs{F} = N(\p)$.
	Out of these, given a fixed primitive sub-module $M \subset L$ of rank $k$, the number of sublattices $L'$ such that $\p M$ is primitive in $L'$ is
	\begin{displaymath}
		q^{k(r-k)}.
	\end{displaymath}
\end{lemma}
\begin{proof}
	Any sublattice $L'$ as in the statement satisfies $\p L \subset L' \subset L$.
	As such, they correspond bijectively to subspaces of dimension $r-k$ of the vector space $L/\p L \cong F^r$ over the field $F := \ZK / \p \ZK$.
	It is well-known that the number of such subspaces is given by the Gaussian binomial coefficient, by definition given by the formula in the first part of the lemma.
	
	For the second part of the lemma, recall that $\p M$ is primitive in $L'$ if and only if $L' \cap \Span (\p M) = \p M$, as in Definition \ref{def:primitive}.
	Since $M$ is also primitive in $L$, we have
	\begin{displaymath}
		L' \cap \Span(\p M) = L' \cap L \cap \Span(M) = L' \cap M.
	\end{displaymath}
	Thus, we are counting $L'$ as above such that $L' \cap M = \p M$.
	
	If $L' \cap M = \p M$, then $L' \cap (M + \p L) = (L' \cap M) + \p L$ (since $\p L \subset L'$), so $L' \cap (M + \p L) = \p L$.
	In other words, the images of $L'$ and $M$ inside the vector space $L/ \p L$ should have trivial intersection.
	
	Conversely, if $L' \cap (M + \p L) = \p L$, then $L' \cap M \subset \p L$.
	Since $\p L = \p M + \p M'$ for some sub-module $M'$ by primitivity, we can also deduce that $\p L \cap M = \p M$, since $\p M$ is primitive in $L$.
	Therefore, $L' \cap M \subset \p L \cap M \subset \p M$ and the reverse inclusion is obvious.
	
	Let $V = L/\p L$ and $U$ be the image of $M$ inside $V$, a subspace of dimension $k = \rank M$.
	The previous paragraphs show that the sublattices $L'$ as in the statement are in bijection with $(r-k)$-dimensional subspaces $W \subset V$ that intersect trivially with $U$.
	We can study these using the action of $\GL_r(F)$ on $(r-k)$-dimensional subspaces (the Grassmannian).
	Indeed, we can choose a basis $e_1, \ldots, e_r$, such that $(e_{r-k+1}, \ldots, e_r)$ forms a basis for $U$.
	Then any $(r-k)$-dimensional subspace of $V$ can be given as $\Span(g e_1, \ldots, g e_{r-k})$ for some $g \in \GL_r(F)$.
	
	The stabilizer of $W_0 := \Span(e_1, \ldots, e_{r-k})$ under this action is given by the subgroup
	\begin{displaymath}
		H = \left\{ g =
		\begin{pmatrix}
			A & B \\
			0 & D
		\end{pmatrix}
		\mid A \in \GL_{r-k}(F), D \in \GL_{k}(F), B \in \mathcal{M}_{r-k, k} (F)
		\right\}.
	\end{displaymath}
	Let now $W = g \cdot W_0$ be some $(r-k)$-dimensional subspace.
	Write
	\begin{displaymath}
		g = 
		\begin{pmatrix}
			S & T \\ U & V
		\end{pmatrix}
	\end{displaymath}
	as a block matrix, analogously to the description of $H$, and suppose we multiply $g$ from the right by an element
	\begin{displaymath}
		\begin{pmatrix}
			A & 0 \\ 0 & \id
		\end{pmatrix} \in H.
	\end{displaymath}
	This would replace $S$ by $S \cdot A$ and we can therefore assume that $S$ is in column echelon form (by Gauss elimination), that is, in lower triangular shape.
	
	Assume now that $W \cap U = 0$.
	This implies that, if $S = (s_{ij})_{1 \leq i,j \leq r-k}$, then $s_{r-k, r-k} \neq 0$.
	Otherwise, since $S$ is lower triangular, we would have the vector $g \cdot e_{r-k}$ in the intersection $W \cap U$.
	Multiplying by another matrix in $H$, we can assume that $s_{r-k, r-k} = 1$ (we are working over a field) and that the rest of the last row of $S$ is zero.
	The same argument now reiterates to show that $s_{r-k-1, r-k-1} \neq 0$, and so on, allowing us to assume that $S = \id_{r-k}$.
	
	In this form, we can multiply $g$ from the right by
	\begin{displaymath}
		\begin{pmatrix}
			\id & -T \\ 0 & \id
		\end{pmatrix} \in H
	\end{displaymath}
	and reduce to $T = 0$.
	This now implies that $V$ must be invertible and another multiplication by an element of $H$ allows us to assume that $V = \id_{k}$.
	 
	We have thus found representatives
	\begin{displaymath}
		g = 
		\begin{pmatrix}
			\id & 0 \\ U & \id
		\end{pmatrix}
	\end{displaymath}
	for all $(r-k)$-dimensional subspaces $W$ such that $W \cap U = 0$.
	It is easy to see that these form a system of representatives (one for each coset of $H$).
	Since $U \in \mathcal{M}_{k, r-k}(F)$ is free, we have $q^{k(r-k)}$ such representatives.
\end{proof}

\begin{lemma} \label{lem:good-submodule-prob}
	Let $p \in \Z$ be a prime and suppose we have the decomposition $p \ZK = \prod_{i=1}^g \p_i$ (where we allow ramification).
	Let $L$ be a module lattice of rank $r$ over $K$ with a given primitive sub-module $M$ of rank $k$.
	For every $i \in \{0, \ldots, g\}$, compute $L_i$ inductively and probabilistically as follows:
	\begin{itemize}
		\item define $L_0 = L$;
		\item given $L_i$, define $L_{i+1}$ as a random sub-module of $L_i$ such that $L_i/L_{i+1} \cong (\ZK/\p_i \ZK)^k$.
	\end{itemize}
	The lattice $L_g$ contains $p M$ as a primitive sub-module with probability at least $1 - d/(p-1)$.
\end{lemma}
\begin{proof}
	We use Lemma \ref{lem:good-submodule-prob-local} at each stage, with $M$ equal to $M$, $\p_1 M$, $\p_1 \p_2 M$, $\ldots$, $p M$, successively.
	Let $q_i = N(\p_i)$. 
	At step $i$, the probability of the required outcome is
	\begin{displaymath}
		\frac{q_i^{k(r-k)} (q_i-1) \cdots (q_i^k-1)}{(q_i^r - 1) \cdots (q_i^{r-k+1} - 1)} \geq \frac{(q_i-1) \cdots (q_i^k-1)}{q_i \cdots q_i^k}
	\end{displaymath}
	where we estimated $q_i^j - 1 \leq q_i^j$ in the denominator.
	It is now easy to see (e.g. inductively) that
	\begin{displaymath}
		\frac{(q_i-1) \cdots (q_i^k-1)}{q_i \cdots q_i^k} = \prod_{i=1}^k \left( 1 - \frac{1}{q^i} \right) \geq 1 - \sum_{i=1}^k \frac{1}{q^i} \geq 1 - \frac{1}{q-1}.
	\end{displaymath}
	Writing $q_i = p^{\alpha_i}$, multiplying these bounds together and applying the same reasoning as above, we obtain the bound
	\begin{displaymath}
		\prod_{i=1}^g \left(1 - \frac{1}{p^{\alpha_i}-1} \right) \geq 1 - \sum_i \frac{1}{p^{\alpha_i}-1} \geq 1 - \frac{d}{p-1},
	\end{displaymath}
	using that $\alpha_i \geq 1$ and that $g \leq d$.
\end{proof}

Putting everything together, we obtain the gap-decreasing algorithm.
\begin{proposition}  \label{cor:gapclosing}
	Let $L$ be a rank $r$ module lattice over a degree $d$ number field $K$, with $K$-minima $\lambda_1^K, \ldots, \lambda_r^K$. 
	Suppose that $\lambda_{k+1}^K \geq p \lambda_{k}^K$ for some prime $p$ and $k < r$.
	The algorithm described in Lemma \ref{lem:good-submodule-prob} then produces, with probability at least $1 - d/(p-1)$, a full-rank sub-module $L' \subset L$ of covolume $p^{dk}$, such that, if $\mu_i^K$ are its $K$-minima, then
	\begin{displaymath}
		\mu_i^K = p \lambda_i^K, \quad i = 1, \ldots, k,
	\end{displaymath}
	and
	\begin{displaymath}
		\lambda_i^K \leq \mu_i^K \leq \left(1+\frac{\Gamma_K\sqrt{kd}}{2}\right) \lambda_i^K, \quad i = k+1, \ldots, r.
	\end{displaymath}
\end{proposition}

\subsection{Reduction to balanced lattices} \label{sec:red-flare-bulk}

We now describe and analyze an algorithm for closing all gaps of a lattice.
It is adequate for reducing SIVP for lattices with gaps of size $2^d$ to SIVP for lattices with gaps of polynomial size in $d$.

\begin{algorithm}[ht]
    \caption{Finding a balanced sublattice}
    \label{alg:balanced}
    \begin{algorithmic}[1]
    	\REQUIRE A module lattice $M$ of rank $r$, and a parameter $t \in \mathbb{N}_{>1}$.
    	\ENSURE A sub-module $N \subset M$.
    		\STATE Put $N_0 = M$.
    		\FOR{$i = 1$ to $r-1$}
                \STATE Pick $g_i \in \{2^1,2^2,2^3,\ldots,2^t\}$ uniformly random.   ~~\textcolor{gray}{`Guess the gap'}
                \IF{$g_i \leq  4 d$} 
                    \STATE Put $p_i = 1$ and $N_i = N_{i-1}$.
                \ELSE
                    \STATE Pick a prime $p_i$ satisfying  $g_i/2 \leq p_i \leq g_i$.
                    \STATE Decompose $p_i = \prod_{j=1}^g \fp_j$ over $K$ (with possible ramification).
                    \STATE Put $P_i := N_{i-1}$.
                    \FOR{$j = 1$ to $g$}
                        \STATE Take a random sub-module $P_j \subset P_{j-1}$ satisfying $P_{j-1}/P_{j} \simeq (\mathbb{Z}_K/\fp_j)^i$.
                    \ENDFOR
                    \STATE Put $N_i = P_g$.
                \ENDIF
    		\ENDFOR
        \RETURN $N := N_{r-1}$.
    \end{algorithmic}
\end{algorithm}

\begin{theorem} \label{middle-to-bulk-proba} Let $M$ be a $\Z_K$-module of rank $r > 1$; and let $t \in \mathbb{N}_{>2}$ be a parameter that satisfies $2^t \geq \left(1+ \frac{\Gamma_K \sqrt{r \cdot d}}{2} \right)^{r-1} \cdot \max_j \frac{\lambda_{j+1}^K(M)}{\lambda_j^K(M)}$. Then, with probability at least $(2t)^{-(r-1)}$,
\Cref{alg:balanced} outputs a sub-module $N \subset M$ such that, for some primes $p_1, \ldots, p_{r-1}$ at most $2^t$,
\begin{itemize}
	\item $\det(N) = \det(M) \cdot \prod_{i = 1}^{r-1} p_i^{di}$.
	\item \(
	\frac{\lambda_{i+1}^K(N)}{\lambda_{i}^K(N)} \leq 4d \cdot \left( 1 + \frac{\Gamma_K \sqrt{i \cdot d}}{2} \right)  \) for all $1 \leq i \leq r -1$.

	\item \(
		\lambda_i^K(N) \leq  \prod_{j = 1}^{i-1} \left( 1 + \frac{\Gamma_K \sqrt{j \cdot d}}{2} \right) \cdot \left( \prod_{s = i}^{r-1} p_s \right)  \cdot \lambda_i^K(M).
	\) for all $1 \leq i \leq r$.
\end{itemize}
Moreover, this algorithm runs in polynomial time in the size of its input.
\end{theorem}

\begin{proof} 
	In this section, we use the notation
	\begin{displaymath}
		\gamma_i^K(L) = \frac{\lambda^K_{i+1}(L)}{\lambda^K_i(L)}
	\end{displaymath}
	for a module lattice $L$ and $i = 1, \ldots, r-1$. 
	These signify the gaps between the $K$-successive minima.
	
	For the first item, note that in the $i$-th step of the algorithm, $|P_{j-1}/P_{j}| = N(\fp_j)^i$. 
	Hence, $|N_{i-1}/N_i| = \prod_{j = 1}^g N(\fp_j)^i = p^{di}$ (with $d = [K:\Q]$). 
	Therefore, taking the product over $i$ yields $|N/M| = |N_r/N_0| = \prod_{i=1}^{r-1} p_i^{di}$, which gives the claim.

	For the second item, recall the notation $\gamma_i^K(N') = \lambda_{i+1}^K(N')/\lambda_i^K(N')$ for any module $N'$. 
	We follow the algorithm through steps $i = 1$ to $r-1$.
	We say that the `gap guessing' in step 3 (of the $i$-th loop) is successful whenever either  $g_i/2 \leq \gamma_i^K(N_{i-1}) \leq 2 \cdot g_i$.
	This happens with probability at least $1/t$. 
	After choosing a prime $p_i$, as in step 5 and step 7, note that $1\leq \gamma_i^K(N_{i-1})/p_i \leq 4d$ in this successful case.

	Assume now that we are in the non-trivial case of $g_i > 4d$ and, thus, $p_i \geq g_i / 2$.
	According to Corollary \ref{cor:gapclosing}, with probability at least $1-d/(p_i - 1) \geq 1/2$, the module $N_i$ satisfies
	$\lambda_t^K(N_i) = p_i \lambda_t^K(N_{i-1})$ for $t \leq i$ and $\lambda_t^K(N_i) \leq 
	\kappa_i \lambda_t^K(N_{i-1})$ for $t > i$, where we write $\kappa_i = 1 + \frac{\Gamma_K \sqrt{i \cdot d}}{2}$ for brevity. 
	This is also true in the trivial case of $g_i \leq 4d$, where $p_i = 1$, with probability $1$.
	
	We assume for the rest of the proof that we are indeed in such a successful `gap guessing' case, for all $i$. 
	The probability computation follows at the end of this proof.

	For all $\ell < i$, we have
	\[ \gamma^K_{\ell}(N_i) = \frac{\lambda_{\ell+1}^K(N_i)}{\lambda_\ell^K(N_i)} = \frac{\lambda_{\ell+1}^K(N_{i-1})}{\lambda_\ell^K(N_{i-1})} = \gamma^K_{\ell}(N_{i-1})\]
	whereas for $\ell = i$, we have
	\[ \gamma_{i}^K(N_i) = \frac{\lambda_{i+1}^K(N_i)}{\lambda_i^K(N_i)} \leq \frac{\kappa_i \lambda_{i+1}^K(N_{i-1})}{p_i \lambda_i^K(N_{i-1})} = \frac{\kappa_i}{p_i} \cdot \gamma_i^K(N_{i-1})  .\]
	By induction, one can then conclude that 
	\[ \gamma_i^K(N) = \gamma_\ell^K(N_{i-1}) = \frac{\kappa_i}{p_i} \cdot \gamma_i^K(N_{i-1})  \leq 4d \cdot \kappa_i = 4d \cdot \left( 1 + \frac{\Gamma_K \sqrt{i \cdot d}}{2} \right)  \]
	since we assumed that $\gamma_i^K(N_{i-1})/p_i \leq 4d$.

	For the bound on $\lambda_i^K(N)$, we use Corollary \ref{cor:gapclosing} again:
	$\lambda_j^K(N_i) = p_i \lambda_j^K(N_{i-1})$ for $j \leq i$ and $\lambda_j^K(N_i) \leq 
	\kappa_i \lambda_j^K(N_{i-1})$ for $j > i$.
	Therefore, 
	\begin{displaymath}
		\lambda_i^K(N) = \lambda_i^K(N_{r-1}) = \left( \prod_{s = i}^{r-1} p_s \right) \cdot \lambda_i^K(N_{i-1}) \leq  \left( \prod_{s = i}^{r-1} p_s \right) \cdot \left( \prod_{j = 1}^{i-1} \kappa_j \right) \lambda_i^K(N_{0}),
	\end{displaymath}
	which proves the third item.

	As promised, we finish with the probability claim. 
	For the entire algorithm to be successful, both the `gap guessing' and the `gap closing' should be successful in each of the $i$-steps. 
	These success probabilities are $1/t$ and at least $1/2$, respectively. Since these are independent events, taking the product takes the overall success probability, yielding $(2t)^{-(r-1)}$.
\end{proof}

\begin{corollary}\label{cor:flare-to-bulk}
Let $M$ be an $\ZK$-module lattice of rank $r > 1$, and $\gamma \geq 1$.
Suppose $M$ is $\alpha$-balanced. %
Let $c_K = 1+ \frac{\Gamma_K \sqrt{r \cdot d}}{2} $.
There is a polynomial time reduction which, given $M$ and $\alpha$, reduces $(c_K^{r-1}\cdot \gamma)$-SIVP in $M$ to $\gamma$-SIVP in a rank-$r$ module lattice $N$, where $N$ is $( 4d \cdot c_K)$-balanced with probability \[p = (2(r-1)\log_2(c_K) + 2\log_2(\alpha))^{-(r-1)}.\]
\end{corollary}

\begin{proof}
Let $t = (r-1)\log_2(c_K) + \log_2(\alpha)$. \Cref{alg:balanced} finds a sub-module $N \subseteq M$ satisfying the properties of \Cref{middle-to-bulk-proba} with probability
\[p = (2t)^{-(r-1)} = (2(r-1)\log_2(c_K) + 2\log_2(\alpha))^{-(r-1)}.\]
In that event, the module $N$ is $( 4d \cdot c_K)$-balanced.
Furthermore, we have
\[\lambda_r^K(N) \leq \prod_{j = 1}^{r-1} \left(  1 + \frac{\Gamma_K \sqrt{j \cdot d}}{2} \right)\cdot \lambda_r^K(M) \leq c_K^{r-1}\lambda_r^K(M),\]
so a solution of $\gamma$-SIVP for $N$ provides a solution of $(c_K^{r-1}\cdot \gamma)$-SIVP for $M$.
\end{proof}

\begin{theorem}[Reduction to the bulk] \label{theorem:stitchcuspflaretobulk}
Let $c_K = 1+ \frac{\Gamma_K \sqrt{r \cdot d}}{2}$, and $\varepsilon = \frac{d}{2^{(rd+1)/2}}$.
Let $O$ be an oracle which solves $\gamma$-SIVP for $(4d \cdot c_K)$-balanced rank-$r$ module lattices.
There is a randomized polynomial time algorithm which given access to $O$, solves $(c_K^{r-1}\cdot (1+\varepsilon)^{r-1}\cdot \gamma)$-SIVP with probability at least $1/2$. The expected number of oracle calls is $\poly_r(\log|\Delta_K|)$.
\end{theorem}

\begin{proof}
Consider a rank-$r$ module lattices or which we wish to solve $(c_K^{r-1}\cdot (1+\varepsilon)^{r-1}\cdot \gamma)$-SIVP.
By \Cref{thm:cusp-to-flare}, the problem reduces  to $t \leq r$ instances of $(c_K^{r-1}\cdot \gamma)$-SIVP in $\alpha$-balanced module lattices, with $\alpha = \Gamma_K^2 2^{\frac{3}{2}(rd - 1)}$.
Let $k \in \Z_{>0}$ be a parameter to be tuned later.
To each of these $t$ instances, apply the reduction of \Cref{cor:flare-to-bulk} independently $k$ times (and solve them using the oracle $O$), and keep the smallest response.
For each of the $t$ instances, the probability that the best-of-$k$ solutions is small enough is $1-(1-p_0)^k$ with \[p_0 = (2(r-1)\log_2(c_K) + 2\log_2(\alpha))^{-(r-1)}\]
is the success probability from \Cref{cor:flare-to-bulk}.
The probability that all $t$ instances are solved successfully is $(1-(1-p_0)^k)^t$.
We have $(1-(1-p_0)^k)^t > 1/2$ if and only if  $k > \frac{\log_2(1-2^{-1/t})}{\log_2(1-p_0)}$.
For $0< x < 1$, we have $0< x/2 < -\log(1-x)$, and for any $t \geq 1$, we have \(-\log(1-2^{-1/t}) < 1 + \log(t)\), so
\[\frac{\log_2(1-2^{-1/t})}{\log_2(1-p_0)} = \frac{\log(1-2^{-1/t})}{\log(1-p_0)} < \frac{-2\log(1-2^{-1/t})}{p_0} \leq \frac{2 + 2 \log(t)}{p_0}. \]
In particular, choosing $k > \frac{2 + 2 \log(t)}{p_0} = \poly_r(\log(\Gamma_K),d) = \poly_r(\log|\Delta_K|)$, we obtain a probability of success of at least $1/2$.
\end{proof}

\newcommand{\epsgaus}{\eps_{\mathcal{G}}}

\section{Sampling}
\label{sec:sampling}
\subsection{Road map}
In the following two sections we tackle two challenges. The first one 
regards \emph{how} to sample an element in $\GL_r(K_\R)$ with 
respect to the distribution $f_z$ as in \Cref{sec:starting-distribution}, 
assuming real arithmetic and uniform samples from $[0,1]$.
In other words, how can the distribution $f_z$ be ``built'' from known distributions.  
This is the subject of section \Cref{sec:sampling}.

On actual computers (or Turing machines), though, no real arithmetic and uniform samples are possible,
so the natural second challenge then consists of showing that \emph{discretization} 
does not impact much the final distribution of this paper's algorithm. This is the subject of \Cref{sec:discretization}.
We now elaborate more on the first of these two challenges.

We note that \Cref{sec:sampling} is more of an expository section, making clear the building blocks of the initial distribution $f_z$, whereas \Cref{sec:discretization} contains the precise procedure of sampling from a finite discretized version $\distr_z$ of $f_z$; and the proof that these two are close in some precise sense. 
In both \Cref{sec:sampling} and \Cref{sec:discretization} we use column notation for matrices and vectors.

\subsubsection*{Sampling in $\GL_r(K_\R)$ according to $f_z$.}
We will crucially rely on the fact that we can decompose 
\[ \SL_r(K_\R) =  \SU_r(K_\R) \cdot \mbox{diag}^0(K_\R) \cdot \SU_r(K_\R) \]
where $\mbox{diag}^0(K_\R)$ are the determinant $1$ diagonal matrices with coefficients in $K_\R$, 
and that the Haar-measure of a function $g$ on $\SL_r(K_\R)$ is dictated by the restriction of $g$ on the completions $K_\nu$ in $K_\R = \prod_{\nu} K_\nu$ ; which is given
by the rule \cite[Proposition 10]{MaireP}
\[ c  \int_{(k_1,k_2) \in \SU_r(K_\nu)^2} \int_{\vec{a} \in \Delta^*} \prod_{1 \leq i < j \leq r } \sinh(a_i - a_j)^{[K_\nu:\R]} g(k_1 \exp(\vec{a}) k_2) dk_1 d\vec{a} dk_2 \]
where we mean with $\exp(\vec{a})$ the $r \times r$ diagonal matrix $\diag(e^{a_1},\ldots,e^{a_r})$ and where $\Delta^* = \{ (a_1,\ldots,a_{r-1}) \in \R^{r-1} ~|~ a_1 > \ldots > a_{r-1} > -\sum_{i =1}^{r-1} a_i \}$ and $a_r = - \sum_{i = 1}^{r-1} a_i$; and where $c \in \R_{>0}$ is a normalization constant only depending on $r$ and $[K_\nu:\R]$.

By \Cref{eq:defftilde,eq:deffz} and \Cref{def:tau-rho}, 
the matrix norm part ($\rho$) and the determinant part ($\tau$) are 
independent; and both $\rho$ and $\tau$ are invariant under $\SU_r(K_\R)$. Hence, we proceed as in \Cref{alg:sampleftilde}.

\begin{remark} \label{remark:aboutXra} As explained in \Cref{sec:starting-distribution}, the initial distribution will be defined as a push-forward of a distribution on $Y_r$ (see \Cref{eq:Yr} and \Cref{eq:pushforwardpia}) under the projection $\pi_\ma$. 
 The choice of the left quotient $\GL_r(\ZK, \ida) = \Aut(\ZK^{r-1}\oplus \ida)$ in the definition of $X_{r,\ma}$ in \Cref{eq:xra} is arbitrary and done there for conciseness. 
 
 In the present section we let this quotient instead depend on the pseudo-basis $(\mathbf{B},\mI)$ of the input module lattice $M$, where $\mI = (\ma_1,\ldots,\ma_r)$ and $\mathbf{B} \in \GL_r(K_\R)$. In other words, we rather define 
 \[ X_{r,\mI} = \Aut(\ma_1 \oplus \ldots \oplus \ma_r) \lquo \GL_r(K_\R) /  (\U_r(K_\R)\cdot\R_{>0}) ,\]
 and send (the coset of) $z := \mathbf{B} \in Y_r$ to (the coset of) $z = \mathbf{B}$ in $X_{r,\mI}$, which then corresponds to the module lattice $M$.
 
 Note that the other class group components of $X_r(K)$ as in \Cref{eq:adelic-quot-conn-comps} may be chosen arbitrarily as long as the full class group is covered.
 
 In the present section, we will also see the distribution $f_z$ on $Y_r$ (see \Cref{eq:Yr}) as a distribution on $\GL_r(K_\R)$ and vice versa. This will not lead to confusion, since the support of $f_z$ consists of modules that all have the same absolute determinant, and since for any $m$ in the support of $f_z$, the entirety of $m \cdot U_r(K_\R)$ has equal density.
\end{remark}

\begin{algorithm}[ht]
    \caption{Computing a sample from $f_z$ in $\GL_r(K_\R)$}
    \label{alg:sampleftilde}
    \begin{algorithmic}[1]
    	\REQUIRE  ~\\  \vspace{-.2cm}
    	\begin{itemize}
    	 \item A pseudo-basis $(\mathbf{B},\mI)$ of a rank $r$ module lattice $M$,\vspace{-.2cm}
    	 \item $\sigma > 0$, a Gaussian parameter,\vspace{-.2cm}
    	 \item $t \in \R_{>0}$ a width parameter for the diagonal.
    	\end{itemize}
    	\ENSURE A pseudo-basis of a module lattice $R$ of rank $r$.
    	\STATE Sample $h \in H \simeq \{ h' \in \prod_{\nu} \R ~|~ \sum_{\nu}  [K_\nu:\R] \cdot h'_\nu = 0 \}$ according to a Gaussian distribution with parameter $\sigma$ (as in \Cref{eq:defftilde}), where $H$ is the hyper plane where the logarithmic units live in. \label{sample:linegauss}
    	\STATE Put $M_h = \diag(e^{h/r},\ldots,e^{h/r}) \in \GL_r(K_\R)$. Note that $\log |\det(M_h)| = h$ and thus $\tau(M_h) = \|h\|^2$.
 We denote $M_h^{(\nu)}$ for the $\nu$-th component of $M_h$ in the decomposition $\GL_r(K_\R) = \prod_{\nu} \GL_r(K_\nu)$.
 \STATE For each place $\nu$ separately, sample $\vec{a}^{(\nu)} = (a_1,\ldots,a_{r-1},a_r)$ with $(a_1,\ldots,a_{r-1}) \in \Delta^*$ from the distribution  \label{sample:lineSU}
 \begin{equation} \label{eq:simplexdistr} c' \int_{\vec{a} \in \Delta^*} \prod_{1 \leq i < j \leq r } \sinh(a_i - a_j)^{[K_\nu:\R]} 1_{[0,t]}(\rho(\exp(\vec{a})) d\vec{a}.\end{equation}
Also sample $k_1^{(\nu)},k_2^{(\nu)} \in \SU_r(K_\nu)$ uniformly (which is possible because it is a compact group) and put (for each $\nu$ separately) $g^{(\nu)} := k_1^{(\nu)} \exp(\vec{a}^{(\nu)}) M_h^{(\nu)} k_2^{(\nu)}$, where $\exp(\vec{a})$ is the $r \times r$ diagonal matrix $\diag(e^{a_1},\ldots,e^{a_r})$.
 
 \STATE Assemble the $g := (g^{(\nu)})_\nu \in \prod_\nu \GL_r(K_\nu)$ component-wise. \label{line:sampleftilde:endg} %
 \RETURN
 $(g \cdot \mathbf{B},\mI)$;
    \end{algorithmic}
\end{algorithm}

That \Cref{alg:sampleftilde} indeed yields the desired distribution $f_z$ for $z := \mathbf{B}$, is the object of \Cref{lemma:desireddistribution}. Note that, computationally, there are three distributions for which a sampling procedure is required. One, the Gaussian distribution on $h \in H$, which is already treated in an earlier work \cite{C:BDPW20} and will therefore only come up in this work in the section about discretization (\Cref{subsection:gaussiandisc}). Two,  the uniform distribution on $\SU_r(K_\nu)$, which can be computed by assembling uniform distributions on spheres in the shape of Householder transformations. This is treated in \Cref{subsection:SURK}. Three, 
the distribution on $\Delta^*$ as in \Cref{eq:simplexdistr}, which can be seen as a distribution on a polytope $\Delta^*_t$. We will sample from this distribution by a rejection sampling procedure where the proposal distribution is the uniform distribution on some polytope $\Delta^*_t$. This is treated in \Cref{subsection:simplex}.

\subsection{Sampling according to the density \texorpdfstring{$f_z$ in $\SL_r(K_\R)$}{fz in SLr(KR)}}
\begin{lemma} \label{lemma:desireddistribution} For any input pseudo-basis $(\mathbf{B},I)$, the pseudo-algorithm described in \Cref{alg:sampleftilde} indeed samples $g \from \GL_r(K_\R)$ according to the distribution $f_z$ as in \Cref{sec:starting-distribution}, with $z = \mathbf{B}$.
\end{lemma}
\begin{proof} By the definition of $f_z$ in \Cref{eq:deffz}, it enough to show that $g \in \GL_r(K_\R)$ as in line \lineref{line:sampleftilde:endg} of \Cref{alg:sampleftilde} is distributed with density $I_f^{-1} \tilde{f}$.
The definition of $\tilde{f}$ \Cref{eq:defftilde} reads
\[ \tilde{f}(x) = 1_{[0,t]}(\rho(x)) \exp( - \frac{\pi}{\sigma^2} \tau(x)),  \]
where $\rho$ and $\tau$ are defined in \Cref{def:tau-rho}. By the very definition of $\tilde{f}$, 
the determinant-part and the $\SL_r$-part are \emph{independent} (due to the product in the density function) and can hence be sampled independently.

We focus for now on sampling the $\SL_r$-part, i.e., elements $g \in \SL_r(K_\R)$ for which $\det(g) = 1 \in K_\R$ (i.e. $1$ at each local component). We decompose 
$g = (g_\nu)_\nu$ via the isomorphism $\SL_r(K_\R) \simeq \prod_{\nu} \SL_r(K_\nu)$ from which we can directly see that $\deg(g_\nu) = 1 \in K_\nu$ for all $\nu$. Hence, for $g \in \SL_r(K_\R)$,
\[ \tilde{f}(g) = 1  \Longleftrightarrow  (  \|g_\nu\|_{\op} \leq t  \mbox{ and } \|g_\nu^{-1}\|_{\op} \leq t ~ \mbox{ for all $\nu$} )\]
For each $g_\nu$, we have a unique decomposition $g_\nu = u_\nu d_\nu v_\nu$ with $u_\nu,v_\nu \in \SU_r(K_\nu)$ and $d_\nu \in D_r(\R)$ of ordered diagonal matrices (i.e., $\diag(d_1,\ldots,d_r)$ with $d_1 > \ldots > d_r$) of determinant $1$. 

The Haar measure of a function $h$ on $\SL_r(K_\nu)$ is given by \cite[Proposition 10]{MaireP}
\[ c  \int_{(k_1,k_2) \in \SU_r(K_\nu)^2} \int_{\vec{a} \in \Delta^*} \prod_{1 \leq i < j \leq r } \sinh(a_i - a_j)^{[K_\nu:\R]} h(k_1 \exp(\vec{a}) k_2) \, dk_1 \,  d\vec{a} \, dk_2 \]
for some constant $c$; here $\exp(\vec{a})$ is the $r$-dimensional diagonal matrix $\diag(e^{a_i})$ with $a_r = -\sum_{i =1 }^{r-1}$. Substituting $\tilde{f}$ for $h$, using that $\rho(g_\nu)  =\rho( u_\nu d_\nu v_\nu) =  \rho(d_\nu)$ and hence $\tilde{f}(k_1 \exp(\vec{a}) k_2  ) = 1_{[0,t]}( \max_{i = 1}^r |a_i| )$, we can deduce the following.

Sampling, for all $\nu$,  $k^{(1)}_\nu,k^{(2)}_\nu \in \SU_r(K_\nu)$ independently and uniformly, and sampling $\vec{a}_\nu \in \Delta^*_t := \{ \vec{a}_\nu \in \R^{r-1} ~|~  t > a_1 > \ldots > a_{r-1} > a_r > -t \}$ with $a_r = -\sum_{i = 1}^{r-1} a_i$ according to the distribution 
\[ c'  \prod_{1 \leq i < j \leq r } \sinh(a_i - a_j)^{[K_\nu:\R]}  d  \]
yields a $g_\nu := k^{(1)}_\nu \exp(\vec{a}_\nu) k^{(2)}_\nu \in \SL_r(K_\nu)$ such that the combination $g = (g_\nu)_\nu$ (via $\SL_r(K_\R) \simeq \prod_\nu \SL_r(K_\nu)$) 
is (Haar) distributed according to $\tilde{f}$ given a unit determinant. Here, $c'$ is defined such that $c' \int_{\vec{a} \in \Delta^*_t} \prod_{1 \leq i < j \leq r } \sinh(a_i - a_j)^{[K_\nu:\R]}  d \vec{a}$ integrates to $1$.

By sampling $h \from H$ according to a Gaussian $\mathcal{G}_{\sigma,H}$, defining 
\[ M_h = \diag(e^{h/r},\ldots,e^{h/r}) \in \GL_r(K_\R) \]  and denoting $M_h^{(\nu)}$ for the $\nu$-th component of $M_h$ in the decomposition $\GL_r(K_\R) = \prod_{\nu} \GL_r(K_\nu)$, subsequently putting $g_\nu := k^{(1)}_\nu \exp(\vec{a}_\nu) M_h^{(\nu)} k^{(2)}_\nu$
and combining $g = (t_\nu)_\nu \in \GL_r(K_\R)$ we see that $g$ is distributed according to $\tilde{f}$ (with varying determinant). 
\end{proof}

\begin{lemma} \label{lemma:randomizationbalanced} Let $t, \sigma > 0$ be parameters of \Cref{alg:sampleftilde}, and let $\eps_1 \in (0,1)$ an error parameter. Let $(\mathbf{B},\mI)$ be a pseudo-basis of an $\alpha$-balanced module lattice $M$. 
Then, with probability at least $1-\eps_1$, the output $(g \cdot \mathbf{B},\mI)$ of \Cref{alg:sampleftilde} is $(e^{2t + 2 \sigma \cdot \sqrt{2 d \log(2d/\eps_1)}} \cdot \alpha)$-balanced.
\end{lemma}

\begin{proof} We have that $g$ is of the shape $g = k_1 \cdot \delta \cdot M_h \cdot k_2$ with $k_1,k_2 \in \SU_r(K_\R)$ and $M_h$ and $\delta$ diagonal matrices (over $K$) as in \Cref{alg:sampleftilde}. Hence, by replacing $(\mathbf{B},\mI)$ by $(k_2^{-1} \mathbf{B},\mI)$ (which does not change the balancedness of $\mathbf{B}$, as $k_2$ is unitary), we may assume $k_2$ is the identity. With the same argument, as we only consider the balancedness properties of $(g \cdot \mathbf{B},\mI)$, which are the same as those of $(k_1^{-1} \cdot g \cdot \mathbf{B},\mI)$, we may assume $k_1$ is the identity as well.

Let now write $t = \delta \cdot M_h$. Our aim is to relate the successive minima of $M$ and of $tM$. We can deduce, by taking $\{ m_1,\ldots,m_j\}$ the first $j$ successive minima of $M$, that 
\[  \lambda^K_j(t M) \leq \max_{i} \| t m_i \| \leq \|t\| \cdot \|m_j\| \leq \|t\| \cdot \lambda^K_j(M).  \]
In a similar fashion, by taking $\{tm_1',\ldots, tm_j'\}$ the first $j$ successive minima of $t M$, with $m_i' \in M$, 
\[  \lambda^K_j(M) \leq \max_{i} \|t^{-1} t m_i' \| \leq \|t^{-1}\| \cdot \|t m_j'\| \leq \|t^{-1} \| \lambda^K_j(tM). \]
Hence, for all $j$, 
\[  \|t^{-1}\|^{-1} \leq  \frac{\lambda^K_j(tM)}{\lambda^K_j(M)} \leq \|t\|, \]
and so 
\[  \frac{\lambda^K_{j+1}(tM)}{\lambda^K_j(tM)} \leq  \frac{\|t\| \lambda^K_{j+1}(M)}{ \|t^{-1} \|^{-1} \lambda^K_j(M)} \leq \|t \| \|t^{-1} \| \cdot  \frac{\lambda^K_{j+1}(M)}{\lambda^K_j(M)} .\]
In other words, if $M$ is $\alpha$-balanced, $tM$ must be $(\cd(t) \cdot \alpha)$-balanced, where $\cd(t) = \|t \| \|t^{-1} \|$ is the conditioning number of $t$.

Since $t = \delta \cdot M_h$, we can use the exact same computations as in the proof of \Cref{proposition:maindiscretization}, except for the fact that $h$, in the specific continuous distribution of line \lineref{sample:linegauss} of \Cref{alg:sampleftilde}, is bounded by $\sigma \cdot \sqrt{2 d \log(2d/\eps_1)}$ with probability $\eps_1$ for any $\eps_1 \in (0,1)$, by Lemma \ref{lemma:bound-gaussian}. 
Therefore, $\cd(t) = \cd(\delta) \cdot \cd(M_h) \leq e^{2t} \cdot e^{2 \sigma \cdot \sqrt{2 d \log(2d/\eps_1)}}$, 
except with probability $\eps_1$. This finishes the proof.
\end{proof}

\subsection{Uniform sampling over \texorpdfstring{$\SU_r(K_\R)$}{SUr(KR)}} \label{subsection:SURK}
In the following lemma, we explain how we can sample uniformly 
in $\SU_r(K_\R)$ if we are allowed to use samples from $\unif([0,1])$,
the uniform distribution over $[0,1]$.

We do this by first decomposing $\SU_r(K_\R) = \prod_{\nu} \SU_r(K_\nu)$ where $K_\nu$ is the completion of $K$ at the place $\nu$, i.e., $K_\nu = \R$ if $\nu$ is real and $\C$ otherwise. Hence sampling a uniformly distributed element from  $\SU_r(K_\R)$ reduces
to sampling uniformly distributed elements from $\SU_r(\C)$ and $\SU_r(\R)$. 
As uniformly sampling in these two special orthogonal groups can be tackled similarly, we focus on the $\R$-variant: $\SU_r(\R)$.

For sampling in $\SU_r(\R)$, we note that (roughly speaking, via fibrations) $\SU_r(\R) \simeq \prod_{j = 2}^r S^{j-1}(\R)$, where $S^{r-1}$ is the unit sphere in $\R^r$. Indeed, by applying a linear transformation $T$ that sends the first column (an element of $S^{r-1}(\R)$) of a $U \in \SU_r(\R)$ to the unit vector $\mathbf{e}_1$, 
we immediately deduce that the bottom-right block of $TU$ lies in
$\SU_{r-1}(\R)$. The decomposition of $\SU_r(\R)$ then follows by induction. So,
we can conclude that uniform sampling in $\SU_{r}(\R)$ reduces to uniform samples in spheres.

To uniformly sample in $S^{r}(\R)$, we apply inverse transform sampling by 
writing the coordinates of $S^r(\R)$ in angular coordinates $(\theta_1,\ldots,\theta_r)$. By an adequate sampling of these $(\theta_1,\ldots,\theta_r)$ one then obtains a uniform distribution on $S^{r}(\R)$.

\begin{lemma} \label{lemma:samplesphere} Let $r \geq 1$. Then there is a procedure that allows to compute a uniform sample in $S^{r}(\R)$ given $r$ uniform samples $(u_1,\ldots,u_r)$ from $\unif([0,1])$.
\end{lemma}
\begin{proof}
We start by defining a map, which described the sphere in spherical coordinates \cite{Blumenson}, 
\[ [0,2\pi] \times [0,\pi]^{r-1} \mapsto  S^{r}(\R), ~(\theta_1,\ldots,\theta_r) \mapsto \vec{x} := f(\theta_1,\ldots,\theta_r) \]
by the rule 
\[ x_j = f_j(\vec{\theta}) :=  \left( \prod_{k = j}^r \sin(\theta_j) \right) \cos(\theta_{j-1}) \]
where we put $\theta_0 := 0$. 
We seek a distribution $\distr$ on $[0,2\pi] \times [0,\pi]^{r-1}$ such that $f(\vec{\theta})$ is
uniformly distributed on $S^{r}(\R)$ for $\vec{\theta} \from \distr$.
We put 
\[ \rho_j(\theta) :=  \begin{cases}
                       \frac{1}{\sqrt{\pi}} \frac{\Gamma(\frac{j+1}{2})}{\Gamma(\frac{j}{2})} \sin^{j-1}(\theta) & \mbox{ if } j > 1 \\ 
                       \frac{1}{2\pi}  & \mbox{ if } j = 1
                      \end{cases}.
 \]
And define the distribution $\rho(\vec{\theta}) := \prod_{j =1}^r \rho_j(\theta_j)$. This is indeed a distribution, since by the reduction formulae for definite integrals over powers of sines, we have\footnote{Here, $!!$ denotes the double factorial, which equals $n!! := \prod_{j = 0}^{\lfloor n/2 \rfloor} (n-2j)$.}, for $j > 1$,
\[ \int_{0}^{\pi} \sin^{j-1}(\theta) d\theta = \begin{cases}
                            \frac{2(j-2)!!}{(j-1)!!} & \mbox{ if $j$ is even} \\    \frac{(j-2)!!}{(j-1)!!} \cdot \pi & \mbox{ if $j$ is odd } 
               
                                               \end{cases}
 \]
By the fact that $\Gamma(k + 1/2) = \frac{(2k-1)!!}{2^k} \sqrt{\pi}$ and $\Gamma(k) = (k-1)!$, we see that,
\[ \frac{\Gamma(\frac{j+1}{2})}{\Gamma(\frac{j}{2})} = \begin{cases}
                                                       \frac{(2k-1)!! \sqrt{\pi} }{2^k \cdot (k-1)!} =  \frac{(2k-1)!! \sqrt{\pi} }{2 \cdot (2k-2)!!} = \frac{(j-1)!! \sqrt{\pi} }{(j-2)!! \cdot 2 }& \mbox{ for $j = 2k$ is even} \\
                                                       \frac{(k-1)! 2^{k-1} }{ \sqrt{\pi} (2k-3)!!} = \frac{(2k-2)!!}{(2k-3)!! \sqrt{\pi}} = \frac{(j-1)!!}{(j-2)!! \sqrt{\pi}} & \mbox{ for $j = 2k-1$ is odd. }  
                                                       \end{cases}
 \]
Hence, indeed, $\rho_j(\theta)$ is a distribution, and so is $\rho(\vec{\theta})$. 

Under the function $f:[0,2 \pi] \times [0,\pi]^{r-2}$ this distribution changes into a distribution $\tau$ over $S^{r}(\R)$. Our aim is to prove that this latter distribution $\tau$ is uniform. 

For $A \subseteq S^r(\R)$, we have, by the substitution formula for integrals and the inverse function theorem, 
 \begin{align}  \int_{\theta \in f^{-1}(A)} \rho(\theta) d\theta &=  \int_{a \in A} \rho(f^{-1}(a)) |D(f^{-1})(a)| da \\ & = \int_{a \in A} \rho(f^{-1}(a)) |D(f)( f^{-1}(a))|^{-1} da \label{eq:integrand} \end{align}
 hence $\tau(a) = \rho(f^{-1}(a)) |D(f)( f^{-1}(a))|^{-1} $ is the density function on $a \in S^r(\R)$. It is a fact \cite[p.~66]{Blumenson} that the Jacobian of the spherical coordinates defined by $f$ is equal to 
 \[ D(f)(\theta) := \prod_{j=1}^r \sin^{j-1}(\theta_j), \]
 and hence, for all $a \in S^r(\R)$, we have $\rho(f^{-1}(a)) = c |D(f)(f^{-1}(a))|$ for some constant $c \in \R_{>0}$. This means that $\tau(a) = \rho(f^{-1}(a)) |D(f)( f^{-1}(a))|^{-1}$ is constant, and hence is equal to the uniform distribution.

One now obtains a uniform sample $a \in S^r(\R)$ by the following procedure:
\begin{enumerate}
 \item Sample $(u_1,\ldots,u_r) \in [0,1]^r$ uniformly.
 \item Compute $F_j(x) = \int_{0}^{x} \rho_j(\theta) d\theta$ either symbolically or numerically.
 \item Compute $\theta_j = F_j^{-1}(u_j)$ for all $j$. Note that, by the inverse transform sampling principle, $\theta_j$ is now distributed with density function $\rho_j$.
 \item Compute $\vec{x} := f(\theta_1,\ldots,\theta_r) \in S^{r}(\R)$.
 \item Then $\vec{x} \in S^{r}(\R)$ is uniformly distributed.
\end{enumerate}

\end{proof}

\begin{lemma}[Uniform sampling in $\SU_r$] \label{lemma:uniformSU}
There is a procedure that transforms the $(r-1)$-tuple of uniform samples  $(u_2,\ldots,u_r) \in \prod_{j = 2}^{r} S^{j-1}(\R)$ into a uniform sample from $\SU_r(\R)$.

Likewise, there is a procedure that transforms that $r$-tuple of uniform samples $(u_1,\ldots,u_r) \in  \prod_{j = 1}^{r} S^{2j - 1}(\R)$ into a uniform sample from $\SU_r(\C)$

\end{lemma}
\begin{proof} We start with the proof of the first statement, which we prove by
  induction (where we use $\SU_1(\R) = \{1\}$). So, we assume we have a sample
  of $\SU_{r-1}(\R)$, using uniform samples $(u_2,\ldots,u_{r-1}) \in \prod_{j = 2}^{r-1} S^{j-1}(\R)$.

Since the (oriented) sphere $S^{r-1}(\R)$ is a homogeneous space for $\SU_r(\R)$, and we have the following fiber bundle \cite[I.7.6]{steenrod1999topology}
\[ SU_{r-1}(\R) \rightarrow \SU_r(\R) \rightarrow S^{r-1}(\R),  \]
we can assemble a uniform sample in $\SU_r(\R)$ by combining a uniform sample in
  $S^{r-1}(\R)$ and $\SU_{r-1}(\R)$ as follows.

We construct such $A \in \SU_r(\R)$ by the following procedure. First, sample $a \in S^{r-1}(\R)$ uniformly. This $a \in \R^{r}$ satisfies $\|a\| = 1$. Create a Householder transformation $H_{a} = I - 2 v v^\top \in U_r(\R)$ that sends $a$ to $\vec{e}_n$; that is, put $v = \frac{a - \vec{e}_n}{\|a - \vec{e}_n\|}$.

Sample $B \in SU_{r-1}(\R)$ uniformly and put 
\[ A' :=  \begin{bmatrix} 
           B & \vec{0}  \\ 
           \vec{0} &  -1
        \end{bmatrix}.
  \]
That is, the last row and the last column of $A'$ consists of zeroes, except for $A'_{11} = -1$.  
Then, output $A := H_a A'$. 

By construction, $\det(A) = \det(H_a) \det(A') = -\det(H_a) \det(B) = 1$ since Householder transformations have determinant $-1$. Hence $A \in \SU_r(\R)$.

For the second statement, about $\SU_r(\C)$, can be proven similarly, but instead with the spheres $S^{2j - 1}$, via 
the fiber bundle (for $r \geq 2$) \cite[I.7.10]{steenrod1999topology}
\[ SU_{r-1}(\C) \rightarrow SU_{r}(\C) \rightarrow S^{2r -1}(\R). \]
Note that $\SU_{1}(\C) \simeq S^1(\R)$.
The uniform sample from $\SU_r(\C)$ is then constructed by sampling $a \in S^{2r - 1}(\R)$ uniformly, and seeing it as a vector in $\C^r$ of norm $1$. Subsequently, compute the Householder transformation $H_a = I - 2 v v^*$ with $v = \frac{a - \vec{e}_1}{\|a - \vec{e}_n\|}$ (note the difference between $v^*$ and $v^\top$ between the complex and the real case). We sample $B \in SU_{r-1}(\C)$ uniformly and put
\[ A' :=  \begin{bmatrix} 
           B & \vec{0}  \\ 
           \vec{0} &  -1
        \end{bmatrix},
  \]
and define $A := H_a A'$. By similar computations, we deduce that $A$ is a uniform sample in $\SU_r(\C)$.
\end{proof}

\begin{definition} \label{def:thetanotation} For $\vec{\theta} \in \prod_{j=2}^r S^{j-1}(\R)$ we denote by $U_{\vec{\theta}} \in \SU_r(\R)$ the real unitary matrix associated with $\vec{\theta}$ defined by the procedure in \Cref{lemma:uniformSU}. 
Abusing notation, for $\vec{\theta} \in \prod_{j=1}^r S^{2j-1}(\R)$ we also denote by $U_{\vec{\theta}} \in \SU_r(\C)$ the complex unitary matrix associated with $\vec{\theta}$ defined by the procedure in \Cref{lemma:uniformSU}. 
\end{definition}

\subsection{Sampling from \texorpdfstring{$1_{[0,t]}(\rho(\exp(\vec{a})))$}{1[0,t](rho(exp(a)))} over the diagonal}
\label{subsection:simplex}
\subsubsection{The target distribution}
The goal in the following text is to derive a procedure to sample determinant one diagonal matrices 
over $K_\nu$
with operator norm (from $\rho$) bounded by some number $t \in \R_{>0}$, according to the marginal distribution inherited from 
the Haar measure on $\SL_r(K_\nu)$, as in \Cref{eq:simplexdistr}. 

This precisely coincides with sampling 
$(a_1,\ldots,a_r) \in \R$ with $a_1 > \ldots > a_{r-1} > a_r$ 
and $a_r = -\sum_{i = 1}^{r-1} a_i$,
satisfying $\max_j |a_j| < t$, 
according to the Haar measure on the diagonal in $\SL_r(\K)$ with $\K = \R$ or $\C$. 
This distribution can be shown (\cite[Proposition 10]{MaireP} where we locally instantiate $d := r$ and $e := 1$, see \cite[Section 4]{MaireP})
to have density 
\begin{equation} \label{eq:targetdistribution} g(a_1,\ldots,a_{r-1}) =   \begin{cases}
                                    c \cdot \prod_{1 \leq i < j \leq r}
                                    \sinh(a_i - a_j)^{[\K:\R]} & \mbox{ for } |a_i| < t \\ 
                                    0  & \mbox{ elsewhere}
                                    \end{cases} 
 \end{equation}
where $c \in \R_{>0}$ is a constant such that $g$ is indeed a density (with unit integral). We 
will write $\bar{g} = c^{-1} g = \prod_{1 \leq i < j \leq r} \sinh(a_i -
a_j)^{[\K:\R]}$ (restricted to $|a_i|<t$) 
for the unnormalized function.

\subsubsection{Rejection sampling}
In rejection sampling (e.g., \cite[Section II.3]{Devroye:1986}), there are two distributions: a target distribution, from which we actually would like a sample, and a proposal distribution, for which we are already able to find samples. By adequately, with a certain probability depending on the sampled value, reject samples from the proposal distribution, we arrive at a sample procedure for the target distribution.

In the case at hand, the target distribution has density function $g$ as in \Cref{eq:targetdistribution}, whereas we choose as the proposal distribution the \emph{uniform distribution} on the simplex defined by $(a_1,\ldots,a_r)$. Such a rejection sampling procedure then reads as follows.
\begin{enumerate}
 \item Compute an upper bound $M  \geq  \max_{|a_i| < t} \bar{g}(a_1,\ldots,a_{r-1})$ on $\bar{g} = c^{-1} g$.
 \item Sample $\vec{a} = (a_1,\ldots,a_{r-1}) \in \Delta^*_t$ uniformly from the set 
 \[ \Delta^*_t = \{ (a_1,\ldots,a_{r-1}) \in \R ~|~  t > a_1 > \ldots > a_{r-1} > a_{r} := -\sum_{i = 1}^{r-1} a_i > -t\}. \]
 and reject with probability $1 - \frac{\bar{g}(a_1,\ldots,a_{r-1})}{M}$. 
 \item If $\vec{a}$ is rejected, re-sample (go to line 2); if not, output $\vec{a}$.
\end{enumerate}
In line $2$ the algorithm is expected to reject $\vec{a}$ with probability $\frac{1}{\vol(\Delta^*_t)} \int_{\vec{a} \in \Delta^*_t} \left ( 1 - \frac{c^{-1} g(\vec{a})}{M} \right) d\vec{a} = 1 - \frac{1}{cM }$ and hence accepts $\vec{a}$ with probability $(cM)^{-1}$. So one can deduce 
that the expected number of uniform samples from $\Delta^*_t$ this algorithm needs, provided that $t \leq 1$, is
\[ O(cM) = O(\max_{\vec{a}} g(\vec{a})) = O\left( (16r^2)^{\frac{r(r-1)[\K:\R]}{2}} \cdot \left( \frac{4r^2}{t}\right)^{r-1}   \right) = e^{O(r^2 \log r)} \cdot t^{-(r-1)} , \] by the later \Cref{lemma:boundong} in \Cref{section:boundong}.
\subsubsection{Uniform sampling on the polytope $\Delta^*_t$} \label{subsec:uniformsamplingdelta}
Our aim is to uniformly sample in the polytope 
\[ \Delta^*_t = \{ (a_1,\ldots,a_{r-1}) \in \R ~|~  t > a_1 > \ldots > a_{r-1} > a_{r} := -\sum_{i = 1}^{r-1} a_i > -t\}. \]
We apply the change of variables $y_i = \frac{t - a_i}{2t}$ for $i \in \{1,\ldots,r-1\}$ that 
bijectively and linearly transforms $\Delta^*_t$ in the set
\begin{align} S = \{ (y_1,\ldots,y_{r-1}) \in \R ~|~ 0 < y_1 < \ldots < y_{r-1} \leq 1, \sum_{i =1}^{r-1} y_i > \frac{r-2}{2} \mbox{ and }  y_{r-1} + \sum_{i =1}^{r-1} y_i < \frac{r}{2} \}. \label{eq:Sdefinition} \end{align}
Indeed, $\sum_{i = 1}^{r-1} y_i = \sum_{i = 1}^{r-1} \frac{t - a_i}{2t} = \frac{r-1}{2} - \frac{1}{2t} \sum_{i = 1}^r a_i  >  \frac{r-1}{2} - \frac{t}{2t} = \frac{r-2}{2}$ and 
\[  y_{r-1} + \sum_{i =1}^{r-1} y_i = \frac{t - a_{r-1}}{2t} + \sum_{i = 1}^{r-1} \frac{t - a_i}{2t} = \frac{r}{2} - \frac{1}{2t} \underbrace{\left ( a_{r-1} +  \sum_{i = 1}^{r-1} a_i  \right)}_{>0} < \frac{r}{2}. \]
This set $S$ satisfies $S \subseteq \Delta^0 = \{  (y_1,\ldots,y_{r-1}) \in \R ~|~ 0 < y_1 < \ldots < y_{r-1} \leq 1 \}$, a filled simplex. A procedure for sampling in $\Delta^0$ exists \cite[\textsection I.4.3, p.~17]{Devroye:1986} by 
sampling $r-1$ uniform distributions $U_1,\ldots,U_{r-1}$ and sorting them $U_{(1)} \leq U_{(2)} \leq \ldots \leq U_{(r-1)}$.

We now sample $y$ from $\Delta^0$ in this way, and reject if $y \notin S$. We aim to compute a lower bound on the success probability of this rejection sampling procedure. Surely, if $y_1 > \frac{(r-2)}{2(r-1)}$ we have $\sum_{i = 1}^{r-1} y_i > \frac{(r-2)}{2}$. Also, if $y_{r-1} < 1/2$, we must have $y_{r-1} + \sum_{i=1}^{r-1} y_i < \frac{r}{2}$. Hence,
\begin{align*} \frac{\vol(S)}{\vol(\Delta^0)} &\geq  \underset{y \from \unif(\Delta^0)}{\mathbb{P}}\left[ y_1 > \frac{(r-2)}{2(r-1)} \mbox{ and } y_{r-1} < 1/2 \right] \\ & = \mathbb{P} \left[  \min_{i = 1,\ldots,r-1} U_i > \frac{(r-2)}{2(r-1)} \mbox{ and }  \max_{i = 1,\ldots,r-1} U_i < 1/2 \right ] 
\\ & = \mathbb{P} \left[  \max_{i = 1,\ldots,r-1} U_i < 1/2  ~\Big|~ \min_{i = 1,\ldots,r-1} U_i > \frac{(r-2)}{2(r-1)}  \right ] \cdot \mathbb{P} \left[ \min_{i = 1,\ldots,r-1} U_i > \frac{(r-2)}{2(r-1)}  \right ] 
\end{align*}
where $U_i$ are iid uniform distributions over $[0,1]$. 
We have 
\begin{displaymath}
  \mathbb{P} \left[ \min_{i = 1,\ldots,r-1} U_i > \frac{(r-2)}{2(r-1)}  \right ] = \left(1 - \frac{(r-2)}{2(r-1)}   \right)^{r-1} = \left( \frac{1}{2} + \frac{1}{2(r-1)} \right)^{r-1}
\end{displaymath}
whereas we can compute the conditional probability by defining $U'_i$ being uniform in $[\frac{(r-2)}{2(r-1)} ,1]$:
\[ \mathbb{P} \left[  \max_{i = 1,\ldots,r-1} U_i < 1/2  ~\Big|~ \min_{i = 1,\ldots,r-1} U_i > \frac{(r-2)}{2(r-1)}  \right ]  = \mathbb{P} [ \max_{i = 1,\ldots,r-1} U'_i < 1/2 ] = \left( \frac{\frac{1}{2} - \frac{(r-2)}{2(r-1)} }{1 - \frac{(r-2)}{2(r-1)} } \right)^{r-1} \]  
\[  = \left( \frac{\frac{1}{2(r-1)}}{\frac{1}{2} + \frac{1}{2(r-1)}} \right)^{r-1} \]
Hence, 
\begin{equation} \frac{\vol(S)}{\vol(\Delta^0)} \geq  (2(r-1))^{-(r-1)}. \label{eq:lowerboundSdelta0} \end{equation}
So, the expected number of uniform samples from $[0,1]$ required to compute a uniform sample in $\Delta_t^*$ via this rejection procedure, is 
\[ O((2(r-1))^{(r-1)}) = e^{O(r \log r)}.\]

\subsubsection{Bound on the maximum of $g$}
\label{section:boundong}

\begin{lemma} \label{lemma:boundong} For $t \leq 1$, we have
\[ \|\bar{g}\|_\infty \leq (4t)^{r(r-1)},~~~~ 
\| g \|_\infty \leq  
(16r^2)^{\frac{r(r-1)[\K:\R]}{2}} \cdot \left( \frac{4r^2}{t}\right)^{r-1} \]
and 
\[ \Lip(g ) \leq \frac{r^2}{t} \cdot   (16r^2)^{\frac{r(r-1)[\K:\R]}{2}} \cdot \left( \frac{4r^2}{t}\right)^{r-1} \]
\end{lemma}
\begin{proof} For the first bound we compute, using $|a_i-a_j| \leq 2t<2$. 
\[ \bar{g} = \prod_{i < j} \sinh(a_i - a_j)^{[\K:\R]} \leq  2^{r(r-1)/2} \cdot
  \prod_{i<j} (a_i - a_j)^{[\K:\R]} \leq (4t)^{r(r-1)[\K:\R]/2} \leq (4t)^{r(r-1)}\]
For the second bound we use the lower bound on the integral $I$ in \Cref{lemma:lowerboundintegral}. 
Hence, we can bound $g$ by (since $|a_i - a_j| < 2t < 2$) 
\begin{align*} g  &= I^{-1} \prod_{i < j} \sinh(a_i - a_j)^{[\K:\R]} \leq I^{-1}
\cdot   (4t)^{r(r-1)[\K:\R]/2} \\  & \leq (16r^2)^{\frac{r(r-1)[\K:\R]}{2}} \cdot \left( \frac{4r^2}{t}\right)^{r-1}. \end{align*}
For the bound on the Lipschitz constant, we bound the derivative of $g$ on $\Delta_t^*$.
\[ \frac{\partial g}{\partial a_k} = I^{-1} \frac{\partial }{\partial
  a_k}\prod_{1 \leq i < j \leq r} \sinh(a_i - a_j)^{[\K:\R]}
\] \[ = I^{-1} \prod_{\substack{1 \leq i < j \leq r \\ k \neq i, k \neq j}}
  \sinh(a_i - a_j)^{[\K:\R]} \frac{\partial }{\partial a_k}\prod_{\substack{1
  \leq i < j \leq r \\ i = k \mbox{ \scriptsize{or} } j = k}} \sinh(a_i -
  a_j)^{[\K:\R]}\]
We proceed with the right-hand side of above expression, which equals
 \[ = \frac{\partial }{\partial a_k}\prod_{i = 1}^{k-1} \sinh(a_i -
  a_k)^{[\K:\R]} \prod_{j = k+1}^{r} \sinh(a_k - a_j)^{[\K:\R]} \] 
  \[ =  \left( -\sum_{i = 1}^{k-1}[\K:\R] \frac{\cosh(a_i -
  a_k)}{\sinh(a_i-a_k)} +  \sum_{j = k+1}^{r} [\K:\R] \frac{\cosh(a_i -
  a_k)}{\sinh(a_i-a_k)} \right)  \prod_{i = 1}^{k-1} \sinh(a_i - a_k)^{[\K:\R]}
  \prod_{j = k+1}^{r} \sinh(a_k - a_j)^{[\K:\R]} \]
 Hence 
 \[ \frac{\partial g}{\partial a_k} = \left( -\sum_{i = 1}^{k-1}[\K:\R]
  \frac{\cosh(a_i - a_k)}{\sinh(a_i-a_k)} +  \sum_{j = k+1}^{r} [\K:\R] \frac{\cosh(a_i - a_k)}{\sinh(a_i-a_k)} \right)  g \]
 Since $a_i - a_j < 2t < 2$, we see that $\sinh(a_i-a_j) \leq 4t$ and $\cosh(a_i-a_j) < 2$, for all $i<j$. Hence, we can bound  
 \begin{align*} \|\frac{\partial g}{\partial a_k}\|_\infty & \leq  2 I^{-1}
 (r-1) [\K:\R]  (4t)^{[\K:\R] \frac{(r-1)r}{2} - 1} \leq 4r \cdot \frac{1}{4t}
 \cdot I^{-1} \cdot (4t)^{[\K:\R] \frac{(r-1)r}{2}} \\ & \leq \frac{r}{t} \cdot
 (16r^2)^{\frac{r(r-1)[\K:\R]}{2}} \cdot \left( \frac{4r^2}{t}\right)^{r-1} \end{align*}
 Now, $\Lip(g) \leq r \max_k \|\frac{\partial g}{\partial a_k}\|_\infty \leq
  \frac{r^2}{t} \cdot   (16r^2)^{\frac{r(r-1)[\K:\R]}{2}} \cdot \left( \frac{4r^2}{t}\right)^{r-1}$, which is what we wanted to prove.
\end{proof}

\section{Discretization} \label{sec:discretization}

\subsection{Introduction}

In \Cref{sec:sampling} we described how to sample 
from the continuous distributions that occur in the 
random walk procedure of the current work. 
On an actual computer (or Turing machine), none of these continuous distributions can be computed. 
Instead, we will compute discretized versions of these, which, in the end, will lead to a distribution $\distr$ on a finite subset $S \subseteq \GL_r(K_\R)$ instead of the distribution $\tilde{f}$. 

The discreteness of the distribution $\distr$ on $S$ and the continuity of the
distribution $\tilde{f}$ on $\GL_r(K_\R)$ cause them to be incomparable at first
glance. However, the full random walk procedure of this paper comes with a
randomization framework and at the end the rounding algorithm (see
\Cref{subsec:canonicalrep}). The output of the rounding algorithm (and thus of
the entire random walk procedure) is a \emph{distribution} in $L^1(X)$ over some discrete set of module lattices $X$. 

For $g \in \GL_r(K_\R)$ (where $g$ is sampled, for example, from $\distr$ or
from $\tilde{f}$), we can write the output of the entire random walk procedure
of this paper on input $g$ as $\psi(g) \in L^1(X)$. 

In order to show that the output distribution of the entire random walk procedure on input $g \from \tilde{f}$ differs not much from if we instead had taken the input $g \from \distr$ (on the finite set $S$), it is sufficient to show that 
\[ \underset{g \xleftarrow{\tilde{f} }\GL_r(K_\R)}{\mathbb{E}} [ \psi_g ] =   \int_g \psi_{g} \tilde{f}(g) dg \approx \sum_{g \in S} \psi_{g} \distr(g)   = \underset{g \xleftarrow{\distr }S}{\mathbb{E}} [ \psi_g ]\]
where both on the right side and the left side is a distribution over $X$, i.e.,
a function in $L^1(X)$, which is ``averaged'' over all possible $g$. Here the ``$\approx$'' sign means that we want the two distributions to be close in statistical distance.

We will show that indeed these average end distributions are close in statistical distance. We show this by changing the continuous distributions into discretized analogues one by one. So, writing $\distr_0 = \tilde{f}$, and $\distr_1$ for the distribution in which in $\tilde{f}$ the left-multiplied uniform distribution on $\SU_r(K_\nu)$ (for all $\nu$) is discretized,  $\distr_2$ for which additionally the $\vec{a} \in \Delta^*$ are discretized, $\distr_3$ for which additionally $h \in H$ is discretized, and $\distr_4 = \distr$ for which additionally the right-multiplied uniform distribution on $\SU_r(K_\nu)$ are discretized; this latter is equal to $\distr$ because then all is discretized. We will show that 
\[ \underset{g \xleftarrow{\tilde{f} }\GL_r(K_\R)}{\mathbb{E}} [ \psi_g ] \approx \underset{g \from \distr_1}{\mathbb{E}} [ \psi_g ] \approx \underset{g \from \distr_2}{\mathbb{E}} [ \psi_g ] \approx \underset{g \from \distr_3}{\mathbb{E}} [ \psi_g ]\approx \underset{g \from \distr_4}{\mathbb{E}} [ \psi_g ].\]
For each of the continuous distributions we will show how to discretize 
them appropriately and how it impacts this final distribution. The discretization of the uniform distribution on the ``left-multiplied'' $\SU_r(K_\nu)$ is treated in \Cref{sec:discSU}, the discretization of $\vec{a} \in \Delta^*$ in \Cref{sec:discdelta}, the discretization of $h \in H$ in \Cref{sec:discH} and, as it is very similar, the discretization of the ``right-multiplied'' $\SU_r(K_\nu)$ also in \Cref{sec:discSU}.

\subsection{Result}
The self-reduction of this paper on an input module lattice consists of two ingredients. The first one is a random walk procedure that both changes the input module lattice slightly geometrically and takes random prime power index sub-module lattices of it. The second ingredient is a rounding procedure, called $\Round$, that allows for efficiently computing a rational module lattice close to the input module lattice, with the virtue that the specific input pseudo-basis representation is hidden: only its module-lattice structure is known.

The random walk procedure on the space of module lattices involves random processes that can be divided into a \emph{discrete} random process and a \emph{continuous} random process. 
The discrete random process consists of choosing a random prime ideal and taking a random sub-module with quotient group isomorphic to the corresponding residue field, whereas the continuous one involve sampling from the continuous distribution $f_z$ for $z \in Y_r$.  

Recall that random process of taking submodules as above corresponds to the Hecke operator $T_\mathcal{P}$, defined in \eqref{eq:def-T-mathcal-P}.
Although $T_\mathcal{P}$ is defined on the space of lattices $X_r$, we also use $T_\mathcal{P}$ to denote the same process at the level of pseudo-bases, as in Algorithm \ref{alg:randomsublattice}, by choosing coset representatives to average over.
This should not lead to confusion, as it commutes with the push-forward through the projection $Y_r \to X_{r, \ida}$.
Recall also the rounding algorithm $\Round$, defined in \Cref{alg:canonical}, taking in parameters $\eps_0$ and a balancedness parameter $\alpha$.

Let $z = (\mathbf{B}, \mI)$ be the input corresponding to a module lattice $L$.
We define $f_z$ and $\initial_z$ by slight abuse of notation, as in Remark \ref{remark:aboutXra}, and in all that follows we interpret $f_z$ and $\initial_z$ as distributions by meaning literally $f_z \mu_{\Riem}$ and $\initial_z \mu_{\Riem}$, respectively.

The output distribution of the random walk procedure on input $z$ is given by $T_\mathcal{P} f_z$, which a priori depends on the choice of pseudo-basis.
If we additionally also apply the rounding algorithm, we get the output distribution $\Round(T_\mathcal{P} f_z)$.
Since the the output of $\Round$ is independent of pseudo-bases with high probability (see Proposition \ref{prop:rounding-algo}), we can identify this distribution with $\Round(T_\mathcal{P} \initial_z)$.

Similarly, for any other distribution $\distr_z$ on $Y_r$, we denote by $\Round(T_\mathcal{P} \distr_z)$ for the distribution that results if we took a sample from $\distr_z$ instead of $f_z$ and then subsequently applied taking random sub-module and the rounding representation algorithm. 

The goal of this section is to show that for all reasonably balanced module lattices $z$, there exists an efficiently computable \emph{finite} distribution $\distr_z$ such that $\Round(T_\mathcal{P} \distr_z)$ is statistically close to $\Round(T_\mathcal{P} \initial_z)$. This means that sampling from the continuous distribution $f_z$ (which is impossible on an actual computer) is not required per se for our reduction to work: indeed, the efficiently computable finite surrogate distribution $\distr_z$ will do, too, and causes only a tiny deviation of the end distribution.

\begin{proposition} \label{proposition:maindiscretization} 
  Let $\alpha > 1$, $0 < \varepsilon < 1$, $B \gg 1$, and let $(\mathbf{B},\mI)$ be a pseudo-basis for a module lattice $z$ be that is $\alpha$-balanced.
  Denote by $\mathcal{P}$ the set of all prime ideals of norm up to $B$.
  Then there exists a finite distribution $\distr_z$ such that 
  \[ \|\Round(T_\mathcal{P} \distr_z) - \Round(T_\mathcal{P} \initial_z) \|_1 \leq \eps + \eps_0 \]
  that is sampleable in time $\exp(8r^2 \log(r)) \cdot \poly(n, \log(1/\eps),\allowbreak \log(1/\eps_0),\allowbreak \log B,\allowbreak \size(\mathbf{B}))$, where $\eps_0>0$ is an input parameter to $\Round$, \Cref{alg:canonical}, and $\initial_z, \distr_z$ are defined through parameters $t \leq 1$ and $\sigma \leq 1$.
\end{proposition}
\begin{proof}
\textbf{Definition of $\distr_z$.} We define $\distr_z$ to be the distribution from \Cref{alg:sampleftilde}, where each continuous distribution is replaced by a finite substitute. So, the Gaussian distribution in line \lineref{sample:linegauss} is replaced by a discrete and windowed Gaussian distribution as in \Cref{def:discretegaussian}; the uniform distributions over $\SU_r(K_\nu)$ in line \lineref{sample:lineSU} are replaced by a finite counterpart defined in \Cref{def:discretizedSU} (for each place $\nu$); and the ``diagonal distribution'' in line \lineref{sample:lineSU} is replaced by a finite distribution as in \Cref{def:discretediagonaldistribution}.

\textbf{Efficiency of $\distr_z$.}
The efficiency of $\distr_z$ follows from the efficiency of all distributions involved, for which the efficiency is shown in the discussion in \Cref{sec:discHdistr}, \Cref{lemma:efficientdiagonal,lemma:efficientang}. Note that that the running time is polynomial time in $r,d, \log N$, except for the diagonal distribution, for which it is $O(d \exp(8r^2 \log(r)) \log N)$.

\textbf{Closeness of distributions.}
By the description in \Cref{alg:sampleftilde} we know that a sample from the distribution $f_z$ can be described as $z \cdot k_1 \cdot M_h \cdot a \cdot k_2$ where $k_1,k_2 \from \unif(\SU_r(K_\R))$, where $M_h = \diag(e^{h/r}, \ldots,e^{h/r})$ with $h$ sampled from a Gaussian over $H$ with parameter $\sigma$, and where $a$ is some diagonal matrix in $\GL_r(K_\R)$ sampled from a specific diagonal distribution. 

Similarly, a sample from the \emph{discretized} distribution $\distr_z$ can be described by $z \cdot \ddot{k}_1 \cdot M_{\ddot{h}} \cdot \ddot{a} \cdot \ddot{k}_2$, where $\ddot{k}_1, \ddot{k}_2$ are from the distribution described in \Cref{def:discretizedSU}, $\ddot{h}$ is sampled from a discrete Gaussian over $H$ (\Cref{def:discretegaussian}) and where  $\ddot{a}$ is from a discrete analogue of the specific diagonal distribution. 

The distributions $\Round(T_\mathcal{P}(\distr_z))$ and $\Round(T_\mathcal{P}(f_z))$ can be alternatively described by respectively 
\[ \underset{\ddot{k}_1,\ddot{h}, \ddot{a} , \ddot{k}_2}{\mathbb{E}}  [ \Round(T_\mathcal{P}(z \cdot \ddot{k}_1 \cdot M_{\ddot{h}} \cdot \ddot{a} \cdot \ddot{k}_2))] \mbox{ and } \underset{k_1,h, a ,k_2}{\mathbb{E}}  [ \Round(T_\mathcal{P}(z \cdot k_1 \cdot M_{h} \cdot a \cdot k_2))]. \]
By \Cref{alg:randomsublattice}, and since $T_\mathcal{P}$ changes the ideal part of the pseudo-basis only by multiplying one ideal by a random $\p \in \mathcal{P}$, (see also \Cref{remark:aboutXra}) we may, by the law of total probability, instead replace the operation $T_\mathcal{P}$ by a multiplication from the left by a matrix $T$. 

Writing $\bar{z} = T \cdot z$, we can measure the closeness of these distributions, we apply the triangle inequality and discretize one-by-one (starting from the right):
\begin{align}
 \|\Round(T_\mathcal{P}(\distr_z)) - \Round(T_\mathcal{P}(f_z))   \|_1
 \\ \leq  \Big \|  \underset{\ddot{k}_1,\ddot{h}, \ddot{a} , \ddot{k}_2}{\mathbb{E}}  [ \Round(\bar{z} \cdot \ddot{k}_1 \cdot M_{\ddot{h}} \cdot \ddot{a} \cdot \ddot{k}_2)] -  \underset{\ddot{k}_1,\ddot{h}, \ddot{a} , k_2}{\mathbb{E}}  [ \Round(\bar{z} \cdot \ddot{k}_1 \cdot M_{\ddot{h}} \cdot \ddot{a} \cdot k_2)]  \Big \|_1 \label{eq:triangle1} \\ 
 +  \Big \|  \underset{\ddot{k}_1,\ddot{h}, \ddot{a} , k_2}{\mathbb{E}}  [ \Round(\bar{z} \cdot \ddot{k}_1 \cdot M_{\ddot{h}} \cdot \ddot{a} \cdot k_2)]  -  \underset{\ddot{k}_1,\ddot{h}, a , k_2}{\mathbb{E}}  [ \Round(\bar{z} \cdot \ddot{k}_1 \cdot M_{\ddot{h}} \cdot a \cdot k_2)]   \Big \|_1  \label{eq:triangle2}\\ 
 +  \Big \| \underset{\ddot{k}_1,\ddot{h}, a , k_2}{\mathbb{E}}  [ \Round(\bar{z} \cdot \ddot{k}_1 \cdot M_{\ddot{h}} \cdot a \cdot k_2)] -   \underset{\ddot{k}_1,h, a , k_2}{\mathbb{E}}  [ \Round(\bar{z} \cdot \ddot{k}_1 \cdot M_{h} \cdot a \cdot k_2)]  \Big \|_1 \label{eq:triangle3} \\ 
 +  \Big \|  \underset{\ddot{k}_1,h, a , k_2}{\mathbb{E}}  [ \Round(\bar{z} \cdot \ddot{k}_1 \cdot M_{h} \cdot a \cdot k_2)]  -   \underset{k_1,h, a , k_2}{\mathbb{E}}  [ \Round(\bar{z} \cdot k_1 \cdot M_{h} \cdot a \cdot k_2)]   \Big \|_1 \label{eq:triangle4}
\end{align}
We now bound each of the components in above sum. By \Cref{lemma:mainSU}, we can bound \Cref{eq:triangle1} by \begin{equation} O( N^{-1/4} \cdot \cd(\bar{z} \cdot \ddot{k}_1 \cdot M_{\ddot{h}} \cdot \ddot{a})^{1/2} \cdot n^5 \sqrt[4]{\log(1/\eps_0)} ).  \label{eq:boundtriangle1} \end{equation}

By \Cref{lemma:maindiag}, we may deduce that \Cref{eq:triangle2} is bounded by 
\begin{equation}
N^{-1/2} O( d\exp(8r^2 \log(r)) + n^5 \cd(\bar{z} \cdot \ddot{k}_1 \cdot M_{\ddot{h}})^{1/2} \cdot \sqrt[4]{\log(1/\eps_0)}   )
\label{eq:boundtriangle2}
\end{equation}

By \Cref{lemma:maingaus} and the fact that $M_h$ and $a$ are both diagonal matrices (and thus commute), we deduce that \Cref{eq:triangle3} is bounded by 
\begin{equation}
N^{-1/2} O(n^4 \cd(\bar{z} \cdot \ddot{k}_1)^{1/2} \cdot \sqrt[4]{\log(1/\eps_0)} + n \sigma).
\label{eq:boundtriangle3}
\end{equation}

 By \Cref{lemma:mainSU}, we can bound \Cref{eq:triangle4} by 
 \begin{equation} O( N^{-1/4} \cdot \cd(\bar{z})^{1/2} \cdot n^5 \sqrt[4]{\log(1/\eps_0)} ). \label{eq:boundtriangle4} \end{equation}

Combining the bounds of \Cref{eq:boundtriangle1,eq:boundtriangle2,eq:boundtriangle3,eq:boundtriangle4}, and simplifying, we obtain
\begin{align}  \label{eq:boundwithcds}
  &\|\Round(T_\mathcal{P}(\distr_z)) - \Round(T_\mathcal{P}(f_z))   \|_1 \nonumber \\  &\leq   \cd(\bar{z} \cdot \ddot{k}_1 \cdot M_{\ddot{h}} \cdot \ddot{a})^{1/2} \cdot N^{-1/4} \cdot \sqrt[4]{\log(1/\eps_0)} \cdot n^5 \cdot (d \exp(8r^2 \log(r)) + n \sigma)
\end{align}
We will now bound the conditioning number. We have, by submultiplicativity of the conditioning number, and the fact that conditioning numbers of unitary matrices equal one, 
\begin{align} \label{eq:boundofcds}  \cd(\bar{z} \cdot \ddot{k}_1 \cdot M_{\ddot{h}} \cdot \ddot{a}) \leq \cd( \bar{z}) \cdot \cd(\ddot{k}_1) \cdot \cd(M_{\ddot{h}}) \cdot \cd(\ddot{a}) = \cd( \bar{z}) \cdot \cd(M_{\ddot{h}}) \cdot \cd(\ddot{a})
 \\ \leq \cd(\bar{z}) \cdot e^{2n^2 \sigma}  \cdot e^{2t} 
 \\ \leq 2^{8 (rd)^2} \cdot |\Delta_K|^{r+2} \cdot 2^{(2rd+3)\cdot \size(\mathbf{B}) + \size(\fp)} \cdot e^{2n^2 \sigma}  \cdot e^{2t} \\  \leq \exp( O(n^2 + n^2 \sigma + n\cdot \size(\mathbf{B}) + \size(\fp) + n \log |\Delta_K|)).
\end{align}
Indeed, since $\ddot{a}$ is diagonal, where the entries at each $\nu$-component are bounded by $[e^{-t},e^t]$, so the total conditioning number must be bounded above by $e^{2t}$.
For the bound on the (discrete and windowed) Gaussian distributed $M_{\ddot{h}}$, note that $M_{\ddot{h}} = \diag(e^{\ddot{h}/r}, \ldots,e^{\ddot{h}/r})$ and $\ddot{h}$ is bounded by $n^2 \sigma$ in absolute value, and hence $\cd(M_{\ddot{h}}) \leq e^{2n^2 \sigma}$. 
 
For the bound on the conditioning number of $\bar{z} = T \cdot z$, we use \Cref{lemma:boundconditioning} and \Cref{lemma:multiplybyp} to see that (using $t \leq 1$)
\begin{align}  \cd(\bar{z}) \leq (rd)^4 \cdot 2^{2d} \cdot |\Delta_K|^{1/d} \cdot 2^{\size(\mathbf{B}) + \size(\fp)} \cdot \cd(z) \\  \leq (rd)^4 \cdot 2^{2d} \cdot |\Delta_K|^{1/d} \cdot 2^{\size(\mathbf{B}) + \size(\fp)} \cdot 2^{4(rd)^2} \cdot |\Delta_K|^{r+1} \cdot 2^{(2rd+2) S}
\\ \leq 2^{8 (rd)^2} \cdot |\Delta_K|^{r+2} \cdot 2^{(2rd+3)\size(\mathbf{B}) + \size(\fp)}. \end{align}

Combining the bounds \Cref{eq:boundwithcds,eq:boundofcds}, using $\sigma \leq 1$, $t \leq 1$, $d \leq n = rd$, we obtain 
\begin{align}
   &\|\Round(T_\mathcal{P}(\distr_z)) - \Round(T_\mathcal{P}(f_z))   \|_1 \nonumber \\ 
   &\leq  N^{-1/4} \cdot \exp( O(n^2 \log n + n \cdot \size(\mathbf{B})+ \max_{\p \in \mathcal{P}} \size(\fp) + n \log |\Delta_K|)) \cdot  \sqrt[4]{\log(1/\eps_0)}. 
\end{align}
Hence by choosing 
\begin{equation} \label{eq:instantiateN} \log(N) = O(n^2 \log n + n \cdot \size(\mathbf{B}) + n^2\cdot \log(B)  +n \log |\Delta_K| + \log(1/\eps_0) + 4\log(1/\eps)) \end{equation} (where we use that $\max_{\p \in \mathcal{P}} \size(\p) \leq n^2 \log B$, by \Cref{lemma:rules-on-sizes}) we obtain an error  
\[ \|\Round(T_\mathcal{P}(\distr_z)) - \Round(T_\mathcal{P}(f_z))   \|_1 \leq \eps. \]
By the property (ii) in Proposition \ref{prop:rounding-algo}, we have that $\RoundPerf(T_\mathcal{P}(f_z)) = \RoundPerf(T_\mathcal{P}(\initial_z))$.
The same proposition shows that
\begin{displaymath}
  \norm{\RoundPerf(T_\mathcal{P}(f_z)) - \Round(T_\mathcal{P}(f_z))}_1 \leq \eps_0
\end{displaymath}
and we are done by the triangle inequality.
\end{proof}

\subsection{Preliminaries on sizes and conditioning numbers}
\begin{lemma} \label{lemma:boundconditioning} Let $(\mB,\mI)$ be a pseudo-basis of a module lattice $M$ and put $S = \size( \mB,\mI)$ (as in \Cref{sec:sizes}). Then $\cd(\mB) \leq 2^{4(rd)^2} \cdot |\Delta_K|^{r+1} \cdot 2^{(2rd+2)S}$.
\end{lemma}
\begin{proof} By definition, $\cd(\mB) = \| \mB \| \| \mB^{-1} \|$, where we interpret the induced norm $\| \cdot \|$ from the Euclidean norm on $K_\R^r$. It suffices to bound both $\| \mB \|$ and $\|\mB^{-1}\|$ in terms of the bound on the size $S$.

We have $\|\mB\| \leq (rd)^2 \cdot \max_{ij} \|\mB_{ij}\| \leq (rd)^2 \cdot 2^d \cdot |\Delta_K|^{1/d} \cdot 2^{S}$, since the coefficient $\mB_{ij} = \sum_{i} a_i \beta_i$, with $(\beta_1,\ldots,\beta_d)$ an LLL-reduced integral basis of $\ZK$, satisfies 
\[ \| \mB_{ij} \| \leq \max_i |a_i|  \cdot \max_j \|\beta_j\| \leq   2^d \cdot |\Delta_K|^{1/d} \cdot 2^S . \]
Using \Cref{lemma:wellconditioned}, seeing $\mB$ as a basis of a free $\ZK$-module, using that $\lambda_1( \mB \cdot \ZK^r) \geq 2^{-S}$ (since the least common multiple of the denominators occurring in $\mB$ can be at most $2^{S}$), and using the previous result on the bound on (columns of) $\mB$,
we obtain
\[  \| \mB^{-1} \| \leq  (rd)^{rd/2 + 1} \cdot 2^{S} \cdot \left( \frac{(rd)^2 \cdot 2^d \cdot |\Delta_K|^{1/d} \cdot 2^{S}}{2^{-S}} \right)^{rd} \] \[ \leq (rd)^{rd/2 + 1} \cdot 2^{(2rd + 1)S} \cdot (rd)^{2rd} \cdot 2^{rd^2} \cdot |\Delta_K|^{r} \] %
Combining the two results, we obtain
\[ \cd(\mB) \leq(rd)^{rd/2 + 1} \cdot 2^{(2rd + 1)S} \cdot (rd)^{2rd} \cdot 2^{rd^2} \cdot |\Delta_K|^{r}  \cdot (rd)^2 \cdot 2^d \cdot |\Delta_K|^{1/d} \cdot 2^{S} \]
\[ \leq  2^{4(rd)^2} \cdot |\Delta_K|^{r+1} \cdot 2^{(2rd+2)S}.   \]
Here, the last simplification in terms of $rd$ can be obtained graphically.
\end{proof}

\begin{lemma} \label{lemma:multiplybyp} Let $(\mB,\mI)$ with $\mB \in K_\R^{r \times r}$ and $\mI = (\ma_1,\ldots,\ma_r)$ be a pseudo-basis of a module lattice $M$ with $S = \size(\mB,\mI)$.  Let $M' \subseteq M$ be a sub-module lattice satisfying $M/M' \simeq \ZK/\fp$ for some prime ideal $\fp$, constructed by multiplying one of the ideals $\ma_i$ by $\fp$ and by multiplying $\mB$ from the right by $\mbox{id} + \sum_{j > i} \alpha_j \cdot e_{ij}$ with $\alpha_j \in \ma_i/(\fp \ma_i)$ (here $e_{ij}$ is the matrix that has $1$ on the $ij$-th position and zero elsewhere), see also \Cref{alg:randomsublattice}, resulting in the pseudo-basis $(\mB',\mI')$ of $M'$.

Then 
\[ \cd(\mB') \leq (rd)^4 \cdot  2^{2d} \cdot |\Delta_K|^{1/d} \cdot 2^{S + \size(\fp)} \cdot \cd(B) \] and 
\[ \size(\mB',\mI') \leq 3S + 4 \size(\fp) \cdot d \cdot \log |\Delta_K|.\] 
\end{lemma}
\begin{proof}
 Writing $\mA = \mbox{id} + \sum_{j > i} \alpha_j \cdot e_{ij}$ we have that, by submultiplicativity of the conditioning number, 
 \[  \cd(\mB') = \cd(\mB \mA) \leq \cd(\mB) \cdot \cd(\mA). \]
 Since $\mA$ has a very simple and similar inverse, namely $\mA^{-1} = \mbox{id} - \sum_{j > i} \alpha_j \cdot e_{ij}$, we can bound 
 \[  \cd(\mA) = \|\mA\| \|\mA^{-1} \| \leq (rd)^4 \max_{j} \|\alpha_j\|  \leq (rd)^4 \cdot  2^d \cdot |\Delta_K|^{1/d} \cdot 2^{\max_j \size(\alpha_j)}, \]
 by similar arguments as in \Cref{lemma:boundconditioning}. Since $\alpha_j \in \ma_i/(\fp \ma_i)$, we can deduce that (by clearing denominators of $\ma_i$ by $k$ and observing that the Hermite normal form of the ideal $k\fp \ma_i$ has coefficients at most $N(k \fp \ma_i)$) 
 we must have $\size(\alpha_j) \leq \size(\fp) + \size(\ma_i) + d$, and hence 
 \[ \cd(\mA) \leq (rd)^4 \cdot  2^{2d} \cdot |\Delta_K|^{1/d} \cdot 2^{S + \size(\fp)}, \]
 which finishes the bound on $\cd(\mB')$. For the bound on $\size(\mB',\mI')$ note that $\size(\mI') = \sum_{j = 1, j \neq i}^r \size(\ma_j) +  \size(\fp \ma_i) \leq \size(\mI) + \size(\fp) + d \leq S + \size(\fp) + d$.
 For the size of $\mB'$, note that $\mB' = \mB\mA$, with $\mA =\mbox{id} + \sum_{j > i} \alpha_j \cdot e_{ij}$, which means that for each $j > i$, the $j$-th column of $\mB$ is increased by $\alpha_j$ times the $i$-th column. Hence, the size of $B'$ can be maximally
 \[ S + \size(\fp) \cdot 2d \cdot \log |\Delta_K|  +  \size(\mB) \leq 2S + \size(\fp) \cdot 2d \cdot \log |\Delta_K| . \]
 Combining the results then yields a bound of $\size(\mB',\mI') \leq 3S + 4 \size(\fp) \cdot d \cdot \log |\Delta_K|$.
\end{proof}

\subsection{Discretization in general}

\begin{lemma} \label{lemma:separationoferrors} Let $X$ be a probability space and let $Y$ be any set. Let $h \in L^1(X)$ be a distribution and let $\ddh \in L^1(X)$ a distribution with finite support $\ddX$. Let $\{ C_{\ddx} \}$ be a collection of finite measure subsets of $X$ with $\ddx \in C_{\ddx}$ and let $T \subset X$, so that $T \cup \bigcup_{\ddx \in \ddX} C_{\ddx} = X$ is a disjoint union.
Let $\mathcal{A}_x : X \rightarrow L^1(Y)$ be a map sending $x \in X$ to a distribution on $Y$.

Then
\begin{align}
\| \underset{x \from h}{\mathbb{E}}[\mathcal{A}_x] - \underset{\ddx \from \ddh}{\mathbb{E}} [\mathcal{A}_{\ddx}]   \| = \left \|  \int_{x \in X} \mathcal{A}_x \cdot h(x) dx - \sum_{\ddx \in \ddX} \mathcal{A}_{\ddx} \cdot \ddh(\ddx) \right \|\\  \leq \Delta(h,\ddh) + \mathcal{C}(h,\ddh,\mathcal{A}) + \mathcal{T}(h),
\end{align}
with discretization error
\[  \Delta(h,\ddh) := \sum_{\ddx \in \ddX} \int_{x \in C_{\ddx}} \big| h(x) - \frac{\ddh(\ddx)}{|C_{\ddx}|} \big| dx \] 
continuity error
\[ \mathcal{C}(\ddh,\mathcal{A}) :=    \sum_{\ddx \in \ddX} \ddh(\ddx) \frac{1}{|C_{\ddx}|} \int_{x \in C_{\ddx}} \| \mathcal{A}_{x} - \mathcal{A}_{\ddx}\|_1 dx ,  \]
and tail error $\mathcal{T}(h) = \int_{x \in T} h(x) dx$. \\
 Additionally, the continuity error satisfies the bounds
 \[ \mathcal{C}(\ddh,\mathcal{A}) \leq  \max_{\ddx \in \ddX} \frac{1}{|C_{\ddx}|} \int_{x \in C_{\ddx}} \| \mathcal{A}_{x} - \mathcal{A}_{\ddx}\|_1 dx \leq  \max_{\ddx \in \ddX} \max_{x \in C_{\ddx}} \| \mathcal{A}_{x} - \mathcal{A}_{\ddx}\|_1 \]
\end{lemma}
\begin{proof} Use  
\[ \mathcal{A}_x \cdot h(x) - \mathcal{A}_{\ddx} \cdot \ddh(\ddx) = \mathcal{A}_x \cdot h(x) -  \mathcal{A}_{x} \cdot \ddh(\ddx) +  \mathcal{A}_{x} \cdot \ddh(\ddx) -  \mathcal{A}_{\ddx} \cdot \ddh(\ddx) \]
and the disjoint union $X = T \cup \bigcup_{\ddx \in \ddX} C_{\ddx}$ to obtain 
\begin{align} &\int_{x \in X} \mathcal{A}_x h(x) dx - \sum_{\ddx \in \ddX} \mathcal{A}_{\ddx} \ddh(\ddx) \nonumber \\
& = \int_{x \in T} \mathcal{A}_x h(x) dx + \sum_{\ddx \in \ddX} \int_{x \in C_{\ddx}}\mathcal{A}_x \big[ h(x) - \frac{\ddh(\ddx)}{|C_{\ddx}|^{-1}} \big] dx\\ & +   \sum_{\ddx \in \ddX} \ddh(\ddx) \frac{1}{|C_{\ddx}|} \int_{x \in C_{\ddx}} [ \mathcal{A}_{x} - \mathcal{A}_{\ddx}]dx . \label{eq:distrY}
\end{align}
Note that the expression in \Cref{eq:distrY} is a distribution in $L^1(Y)$. So, by taking the $1$-norm on $L^1(Y)$ and putting the norm within the integrals (a form of triangle inequality), one obtains the following inequality, using that $\|\mathcal{A}_x\| = 1$,
\begin{align}
 & \left \|  \int_{x \in X} \mathcal{A}_x \cdot h(x) dx - \sum_{\ddx \in \ddX} \mathcal{A}_{\ddx} \cdot \ddh(\ddx) \right \| 
 \\  & \leq \int_{x \in T} h(x) dx + \sum_{\ddx \in \ddX} \int_{x \in C_{\ddx}} \big| h(x) - \frac{\ddh(\ddx)}{|C_{\ddx}|^{-1}} \big| dx +   \sum_{\ddx \in \ddX} \ddh(\ddx) \frac{1}{|C_{\ddx}|} \int_{x \in C_{\ddx}} \| \mathcal{A}_{x} - \mathcal{A}_{\ddx}\|_1 dx  \\ 
 & =  \mathcal{T}(h) +   \Delta(h,\ddh) + \mathcal{C}(h,\ddh,\mathcal{A}).
\end{align}
Here, the last equality holds by definition. Additionally, by H\"older's inequality, and the fact that the $\ddh(\ddx)$ sum to one (it is a distribution), one obtains the bound 
\[ \mathcal{C}(\ddh,\mathcal{A}) \leq \max_{\ddx \in \ddX} \frac{1}{|C_{\ddx}|} \int_{x \in C_{\ddx}} \| \mathcal{A}_{x} - \mathcal{A}_{\ddx}\|_1 dx \leq  \max_{\ddx \in \ddX} \max_{x \in C_{\ddx}} \| \mathcal{A}_{x} - \mathcal{A}_{\ddx}\|_1.  \]
\end{proof}

\subsection{Discretization of the Gaussian distribution over \texorpdfstring{$H$}{H}}
\label{subsection:gaussiandisc}

\label{sec:discH}
\subsubsection{The continuous and the finite distribution}
\label{sec:discHdistr}
\paragraph{The continuous distribution}
We denote $H = \{ (h_\nu)_\nu \in \prod_{\nu} \R ~|~ \sum_{\nu} h_\nu = 0  \}$ for the logarithmic unit hyper plane, with standard Euclidean metric\footnote{The Euclidean length on $H$ is not consistent with that in \cite[Section 2.1]{C:BDPW20}, in which the Euclidean length is defined over the embeddings and accounts to $(\sum_{\nu} [K_\nu:\R] h^2_\nu)^{1/2}$. This does not pose a real problem, since it merely increases the hidden constant of the main result \cite[Theorem 3.3]{C:BDPW20} of that work by a small constant.}. 
The continuous distribution over $H$ is the \emph{Gaussian distribution} 
 $\Gau_{H,\sd}$ defined as in \Cref{def:gaussian}.
 
\paragraph{The finite distribution}
Choosing an ordering $\{ \nu_1, \ldots, \nu_{\ell+1} \}$ (with $\ell = \dim(H)$)
of the places, we define a basis $B_H$ of $H$ consisting of the basis elements $\mathbf{b}_j = \mathbf{e}_{\nu_{j+1}} - \mathbf{e}_{\nu_j}$ for $j = 1, \ldots, \ell$. Here, $\mathbf{e}_{\nu_j}$ is the element of $H$ that is one at the place $\nu_j$ and zero elsewhere. Given a discretization parameter $N \in \Z_{>0}$, this allows us to define the discrete Gaussian distribution, written $\Gau_{\dH,\sd}$ (see \Cref{def:discretegaussian}) with 
\[ \dH := \frac{1}{N} B_H \Z^{\ell} = \{ \sum_{j=1}^\ell z_j \mathbf{b}_{\nu_j} ~|~ z_j \in \frac{1}{N} \Z  \mbox{ for all } j  \} .\]

The finite distribution over $\dH$ that we will use in this work is a finite \emph{approximation} of this discrete Gaussian distribution \cite{SODA:Klein00,STOC:GenPeiVai08}, which we denote $\dGau_{\sd}$, that can be efficiently sampled and that deviates only slightly from $\Gau_{\dH,\sd}$. More precisely \cite[Lemma A.7]{FPMSW_eprint} states that, for any $\epsgaus >0$, by paying time polynomial in the size of the input and in $\log(1/\epsgaus)$, we can manage to have the approximation as good as $\|\dGau_{\sd} - \Gau_{\dH,\sd}\|_1 \leq 
\epsgaus$; and, additionally, any sample $v$ from $\dGau_\sd$ satisfies $\|v\| \leq \sd \cdot \sqrt{ \log(1/\eps) + 4n}$. That is, $\dGau_\sd$ is supported on vectors in $v \in \dH$ satisfying $\|v\| \leq \sd \cdot \sqrt{ \log(1/\eps) + 4n}$ (which is a finite set).

\subsubsection{The tail error and the discretization error}
\paragraph{The tail error}
We fix $N \in \Z_{>0}$ and $\epsgaus = \frac{1}{N}$ and we write 
\[ \dHf = \{ \ddh \in \dH ~|~  \|\ddh\| \leq \sqrt{2} \cdot \sd \cdot \sqrt{n} \cdot \sqrt{ \log(n/\epsgaus) +4} \}. \]
We write $F_H := \{ x \in H ~|~ x_i \in [-\frac{1}{2N},\frac{1}{2N}) \mbox{ for all } i \}$ and for each $\ddh \in \dHf$ we put $C_{\ddh} :=  \ddh + F_H$. We put $T = H \backslash ( \bigcup_{\ddh \in \dH} C_{\ddh} )$, so that $T \cup \bigcup_{\ddh \in \dH} C_{\ddh}$ is a disjoint union.

We can then reasonably bound 
\begin{equation} \label{eq:tailerrorgaussian}
 \mathcal{T}(\Gau_{\sd,H}) = \int_{h \in T} \Gau_{\sd,H}(h) dh \leq \int_{\|h\| \geq \sigma \cdot \sqrt{n(\log(1/\epsgaus) + 4)}} \Gau_{\sd,H}(h) dh \leq \epsgaus = N^{-1}. 
\end{equation}
This holds because writing $h = \sum_{i=1}^{\dim(H)} c_i h_i$ in an orthonormal basis yields that $\max_i c_i \geq \|h\|_2/\sqrt{\dim(H)} \geq \|h\|_2/\sqrt{n} \geq \sqrt{2} \sd \sqrt{\log(n/\epsgaus) + 4}$. Since the coefficients $c_i$ are all independently Gaussian distributed with the same parameter $\sd$ (but with a single variable), we have that the probability that $\max_i c_i \geq t:= \sqrt{2} \sd  \cdot \sqrt{\log(n/\epsgaus) + 4}$ is at most $n \cdot \exp( -t^2/(2 \sd^2)) \leq n \cdot \exp(-(\log(n/\epsgaus) + 4)) \leq \epsgaus$.

\paragraph{The discretization error}
To estimate the discretization error, we use that $\dGau_{\sd}$ is an $\epsgaus$-close approximation of $\Gau_{\sd,\dH}$, and hence we can conclude that 
\begin{equation} \Delta(\Gau_{\sd,H}, \dGau_{\sd}) \leq \Delta(\Gau_{\sd,H},\Gau_{\sd,\dH}) + \epsgaus . \label{eq:discretizationerrorgaussianeps} \end{equation}
So it remains to bound $ \Delta(\Gau_{\sd,H},\Gau_{\sd,\dH})$.
Before doing that, we need to apply a result on Gaussian smoothing.

We can apply \Cref{{lemma:total-gaussian-weight}} to the Gaussian sum $\Gau_{\sd,H}$ over the \emph{shifted} $\ell$-dimensional lattice $\dH + h$, where $\dH = \frac{1}{N} B_H \Z^\ell$ and $h \in H$. 
Since $\lambda_\ell(\dH) \leq \frac{2}{N}$, we can deduce that for 
\begin{equation} \sd  \geq \sqrt{\frac{\log(2n(1+1/\eps))}{\pi}} \cdot \frac{2}{N} \label{eq:sdbound} \end{equation}  (for some $\eps > 0$) holds that, for any $h \in H$, (see \Cref{def:gaussian})
\[ \Gau_{\sd,H}(\dH + h) \in [1-\eps,1+\eps] \frac{1}{ \det(\dH)} \]

\begin{lemma} \label{lemma:closegaussians} 
\label{lemma:gaussiandiscretization}
Let $N \in \Z_{>0}$ and let 
$\dH = \frac{1}{N} B_H \Z^{\ell}$, 
and let $F_H = \frac{1}{N} B_H [-1/2,1/2)^{\ell}$ be a fundamental domain of $\dH$ in $H$.
Let 
$\sd  \geq \sqrt{\frac{\log(2n(1+N))}{\pi}} \cdot \frac{4}{N}$ 
(which is twice as large as \Cref{eq:sdbound} with $\eps = 1/N$).

 Then
\[ \Delta(\Gau_{\sd,H},\Gau_{\sd,\dH}) \leq (1 + 8\ell \sigma) N^{-1/2}. \]
\end{lemma}
\begin{proof}

Writing out the definition of $\Delta(\Gau_{\sd,H},\Gau_{\sd,\dH})$ and  $\Gau_{\sd,\dH}(\ddh) = \Gau_{\sd,H}(\ddh)/\Gau_{\sd,H}(\dH)$, we have 
\begin{align}
& \Delta(\Gau_{\sd,H},\Gau_{\sd,\dH}) =  \int_{h \in F} \sum_{\ddh \in \dH} \big|\Gau_{\sd,H}(\ddh + h) - |F|^{-1} \dGau_{\sd,\dH}(\ddh)\big| dh \nonumber \\  = & |F|^{-1}\int_{h \in F} \sum_{\ddh \in \dH} \big| \frac{|F| \cdot \Gau_{\sd,H}(\dH + h) \Gau_{\sd,H}(\ddh + h)}{\Gau_{\sd,H}(\dH + h)}  -  \frac{\Gau_{\sd,H}(\ddh)}{\Gau_{\sd,H}(\dH)} \big| dh
\label{eq:gaussianweightthing1}\end{align}
By \Cref{lemma:total-gaussian-weight} and the text immediately after that lemma
  (which applies it to $\dH$), we see that, by \Cref{def:discretegaussian}, by
  the fact that $|F| = \det(\dH)$,
\[ |F| \cdot \Gau_{\sd}(\dH + h) \in [1-\frac{1}{N},1+\frac{1}{N}] \]
Hence, we obtain that \Cref{eq:gaussianweightthing1} is at most
\begin{align}
& \frac{1}{N} +  |F|^{-1}\int_{h \in F} \sum_{\ddh \in \dH} \Big| \frac{\Gau_\sd(\ddh + h)}{\Gau_\sd(\dH + h)}  -  \frac{\Gau_{\sd}(\ddh)}{\Gau_\sd(\dH)} \Big| dh \\ \leq & \frac{1}{N}  + \max_{h \in F}\sum_{\ddh \in \dH} \Big| \frac{\Gau_\sd(\ddh + h)}{\Gau_\sd(\dH + h)}  -  \frac{\Gau_{\sd}(\ddh)}{\Gau_\sd(\dH)} \Big| 
\label{eq:gaussianweightthing2} \end{align}
where we used H\"older's inequality. Now we use a result from Pellet-Mary and Stehl\'{e} \cite[Lemma 2.3]{AC:PelSte21}, by seeing $\frac{\Gau_\sd(\ddh + h)}{\Gau_\sd(\dH + h)}$ and $\frac{\Gau_\sd(\ddh)}{\Gau_\sd(\dH)}$ as distributions and the sum as the total variation distance. We will postpone the check of the premise of \cite[Lemma 2.3]{AC:PelSte21} ($\eta_{1/2}(\sd^{-1} \dH) \leq 1/2$) to the end of this proof.  We obtain that \Cref{eq:gaussianweightthing2} is bounded by 
\[ \frac{1}{N}  + 4 \sqrt{\ell} \cdot \sd \cdot \max_{h \in F} \sqrt{\|h\|} \leq \frac{1}{N} + 8 \ell \cdot \sd \cdot N^{-1/2} \leq (1 + 8\ell \sigma) N^{-1/2}, \]
where the last inequality follows from the definition of $F$.

It remains to show that $\eta_{1/2}(\sd^{-1} \dH) \leq 1/2$, i.e., $\eta_{1/2}(\dH) \leq \sd/2$. By \cite[Lemma 3.3]{MicciancioRegev2007} we have that $\eta_{1/2}(\dH) \leq \eta_{1/N}(\dH) \leq \sqrt{\frac{\log(2n(1+N))}{\pi}} \cdot \lambda_\ell(\dH) \leq \sqrt{\frac{\log(2n(1+N))}{\pi}} \cdot \frac{2}{N} \leq \sd/2$. This finishes the proof.
\end{proof}

We can conclude that the discretization error in case of the Gaussian (since $\epsgaus := N^{-1}$) is bounded as follows.
\begin{equation} \Delta(\Gau_{\sd,H}, \dGau_{\sd}) \leq N^{-1} + (1 + 8\ell \sigma) N^{-1/2} \label{eq:conclusiondiscretizationgaussian} \end{equation}
whenever
$\sd  \geq \sqrt{\frac{\log(2n(1+N))}{\pi}} \cdot \frac{4}{N}$.
\subsubsection{The continuity error}

\begin{lemma} \label{lemma:gaussiancontinuity} Let $\mathcal{A}_h$ (for $h \in H$) be the output distribution of \Cref{alg:canonical} on input $g \cdot M_h \cdot g'$ for fixed $g,g' \in \GL_r(K_\R)$. Let $N \in \Z_{>0}$ be the discretization parameter.

Then
\[ \mathcal{C}(\dGau_\sd,\mathcal{A}) \leq 92 n^{7/2} \sqrt[4]{\log(12r/\eps_0)} \cd(g)^{1/2} \cdot N^{-1/2}, \]
 where $\eps_0$ is part of the input of \Cref{alg:canonical}.
\end{lemma}
\begin{proof} Using the bound on the continuity error of \Cref{lemma:separationoferrors}, we have 
\[ \mathcal{C}(\dGau_\sd,\mathcal{A}) 
\leq \max_{\ddh \in \dHf} \max_{h \in F_H} \| \mathcal{A}_{\ddh + h} - \mathcal{A}_{\ddh}\|_1 . \]

The distribution $\mathcal{A}$ is the output distribution of \Cref{alg:canonical} on input $g \cdot M_h \cdot g'$, where $g,g' \in \GL_r(K_\R)$. Writing $\mathcal{R}$ for the output distribution of \Cref{alg:canonical}, we can show that, by \Cref{lemma:holdercontinuous}, writing $L = 92 n^3 \sqrt[4]{\log(12r/\eps_0)}$,

\begin{align*} \| \mathcal{A}_{\ddh + h} - \mathcal{A}_{\ddh} \|_1 & = \| \mathcal{R}(g \cdot M_{\ddh + h} \cdot g') - \mathcal{R}(g \cdot M_{\ddh} \cdot g') \|_1\\ &  \leq  L \| g M_{\ddh + h} g'  (g M_{\ddh} g')^{-1} - I \|^{1/2} \leq L \| g M_{\ddh + h} M_{\ddh}^{-1} g^{-1} - I \|^{1/2}  \\ & \leq L \| g M_{h}g^{-1} - I \|^{1/2} \leq L \cd(g)^{1/2} \|M_h -I\|^{1/2} \\  & \leq L\cd(g)^{1/2} \sqrt{n} \cdot N^{-1/2} \leq 92 n^{7/2} \sqrt[4]{\log(12r/\eps_0)} \cd(g)^{1/2} \cdot N^{-1/2} , \end{align*}
 where the last inequality follows from, instantiating $L = 92 n^3 \sqrt[4]{\log(12r/\eps_0)}$ and the fact that $M_h - I = \diag(e^{h/r} - 1)$ and hence $\|M_h - I \| \leq |e^{h/r} - 1|_{K_\R} \leq \sqrt{n}/N$. 
\end{proof}

\subsubsection{Concluding all errors}
\begin{lemma} \label{lemma:maingaus} Let $\mathcal{A}_h$ (for $h \in H$) be the output distribution of \Cref{alg:canonical} on input $g \cdot M_h \cdot g'$ for fixed $g,g' \in \GL_r(K_\R)$ and input $\eps_0 > 0$. 

Let $N \in \Z_{>0}$ be a discretization parameter, and let $\sd  \geq \Omega(N^{-1/2}) \geq \sqrt{\frac{\log(2n(1+N))}{\pi}} \cdot \frac{4}{N}$.
Let $\Gau_{\sd,H}$ respectively $\dGau_\sd$ be the continuous respectively finite distribution described in \Cref{sec:discHdistr}, where the finite distribution is instantiated with discretization parameter $N \in \Z_{>0}$ (and $\epsgaus = N^{-1}$).

Then, 
\begin{align*} 
& \| \underset{x \from \Gau_{\sd,H}}{\mathbb{E}}[\mathcal{A}_x] - \underset{\ddx \from \dGau_{\sd}}{\mathbb{E}} [\mathcal{A}_{\ddx}]   \|   &\leq 
N^{-1/2} \cdot O(n^{4} \cd(g)^{1/2} \log(1/\eps_0)^{1/4} + n\sd). 
\end{align*}
\end{lemma}
\begin{proof} This is just an application of \Cref{lemma:gaussiandiscretization,lemma:gaussiancontinuity} and \Cref{eq:tailerrorgaussian}, where we simplified 
  \begin{displaymath}
    N^{-1/2}(92 n^{7/2} \sqrt[4]{\log(12r/\eps_0)} \cd(g)^{1/2} +  2 + 8d \sd)
  \end{displaymath}
  into the big-O expression 
  \begin{displaymath}
    N^{-1/2} O(n^{4} \cd(g)^{1/2} \log(1/\eps_0)^{1/4} + n \sd).
  \end{displaymath}
\end{proof}

\subsection{Discretization of the distribution over \texorpdfstring{$\Delta_t^*$}{Deltat*}}
\label{sec:discdelta}

\subsubsection{The continuous and the finite distribution}
\paragraph{Continuous distribution}
\begin{definition}[Component-wise diagonal distribution] \label{def:diagonaldistributioncomp} For a fixed place $\nu$, the diagonal distribution $\distrdiag^{(\nu)}$ on the polytope $\Delta^*_t$ for $t \in \R_{>0}$ (and rank $r$) is defined by the following procedure.
\begin{enumerate}
 \item (Sample a uniform element from $\Delta^*_t$, see also \Cref{subsec:uniformsamplingdelta})
 \item \hspace{.5cm} Sample $r-1$ independent uniform variables on $[0,1]$ and sort them, yielding 
 \[ (x_1,\ldots,x_{r-1}) \in \Delta^0 \]
 \item \hspace{.5cm} If $(x_1,\ldots,x_{r-1}) \notin S$ as in \Cref{eq:Sdefinition}, goto line 1. 
 \item \hspace{.5cm} If $(x_1,\ldots,x_{r-1}) \in S$, put $a_i = t - 2tx_i$ for all $i$.
 \item (Rejection sampling with respect to the diagonal density $g$ (see \Cref{eq:targetdistribution}))
 \item \hspace{.5cm} With probability $1 -\frac{\bar{g}(a_1,\ldots,a_{r-1})}{\bar{M}}$ reject and goto line 1, where
  \[ \bar{g} = \prod_{1 \leq i < j \leq r} \sinh(a_i - a_j)^{[K_\nu:\R]}  \mbox{ and  }  \bar{M} := (4t)^{r(r-1)} \geq \|\bar{g}\|_\infty \]
 \item \hspace{.5cm} Output $(a_1,\ldots,a_{r-1})$.
\end{enumerate}
\end{definition}

\begin{definition}[Diagonal distribution] \label{def:diagonaldistribution} We denote by $\distrdiag$ the compound distribution over rank $r$ diagonal matrices over $K_\R$ where each $\nu$-component is independently distributed with $\distrdiag^{(\nu)}$.
 
\end{definition}

\paragraph{Finite distribution}

\begin{definition}[Component-wise discretized diagonal distribution] \label{def:discretediagonaldistributioncomp} For a fixed place $\nu$, the discretized diagonal distribution $\ddistrdiag^{(\nu)}$ on the polytope $\Delta^*_t$ for $t \in \R_{>0}$ (and rank $r$) and discretization parameter $N \in \Z_{>0}$ is defined by the following procedure.
\begin{enumerate}
 \item (Sample a uniform element from $\Delta^*_t$, see also \Cref{subsec:uniformsamplingdelta})
 \item \hspace{.5cm} Sample $r-1$ independent uniform variables in $[0,1) \cap \frac{1}{N} \Z$; sort them, yielding 
 \[ (x_1,\ldots,x_{r-1}) \in \Delta^0 \]
 \item \hspace{.5cm} If $(x_1,\ldots,x_{r-1}) \notin S$ as in \Cref{eq:Sdefinition}, goto line 1. 
 \item \hspace{.5cm} If $(x_1,\ldots,x_{r-1}) \in S$, put $a_i = t - 2ty_i$ for all $i$.
 \item (Rejection sampling with respect to the diagonal density $g$ (see \Cref{eq:targetdistribution}))
 \item \hspace{.5cm} Compute $\tilde{\tau} \in \frac{1}{N^2} \Z$ with $|\tilde{\tau}- \frac{\bar{g}(a_1,\ldots,a_{r-1})}{\bar{M}}| < \frac{1}{2N^2}$,  where  \[ \bar{g} = \prod_{1 \leq i < j \leq r} \sinh(a_i - a_j)^{[K_\nu:\R]}  \mbox{ and  } 
\bar{M} := (4t)^{r(r-1)} \geq \|\bar{g}\|_\infty .
\]
 \item \hspace{.5cm} With probability $1 -\tilde{\tau}$ reject and goto line 1.
 \item \hspace{.5cm} Output $(a_1,\ldots,a_{r-1})$.
\end{enumerate}

\end{definition}

\begin{definition}[Discretized diagonal distribution]
\label{def:discretediagonaldistribution} We denote by $\ddistrdiag$ the compound distribution over rank $r$ diagonal matrices over $K_\R$ where each $\nu$-component is independently distributed with $\ddistrdiag^{(\nu)}$.
\end{definition}

\paragraph{Help lemmas}
\begin{lemma} \label{lemma:helplemmaNZ} We have, for $N > 64r^2$, 
\[ |S \cap \frac{1}{N} \Z^{r-1}| \in [e^{\frac{-8(r-1)^2r}{N}},e^{\frac{8(r-1)^2r}{N}}] \cdot N^{r-1} \cdot \vol(S), \]
\[ |\Delta^0 \cap \frac{1}{N} \Z^{r-1}| \in [e^{\frac{-8(r-1)^2r}{N}},e^{\frac{8(r-1)^2r}{N}}] \cdot N^{r-1} \cdot \vol(\Delta^0). \]
Furthermore, 
\[ \Big|\big \{ x \in S \cap \frac{1}{N} \Z^{r-1} ~|~  x + (\frac{1}{2N},\frac{1}{2N}]^{r-1} \not \subseteq S \big\}\Big | \leq  \frac{64(r-1)^2 r}{N} \cdot N^{r-1} \cdot \vol(S) \]
\end{lemma}
\begin{proof} We use \Cref{lemma:latticepointsinconvexvolume} with $\Lambda = \frac{1}{N} \Z^{r-1}$, $X = S - \bt'$, $q = 1$, $\bt = 0$ and $c = \frac{4r(r-1)}{N}$ to obtain 
\begin{equation} \label{eq:countS} |S \cap \frac{1}{N} \Z^{r-1}| \in [e^{\frac{-8(r-1)^2r}{N}},e^{\frac{8(r-1)^2r}{N}}] \cdot N^{r-1} \cdot \vol(S). \end{equation}
Since $\voronoi = (-\frac{1}{2N},\frac{1}{2N}]^{r-1} \subseteq c (S - \bt')$ 
and $c = \frac{4r(r-1)}{N}$ (and similarly for $X \supset S$). Note that in order to have $1 = q > 2c$, we require $N > 8r^2$. 

For the last statement, put $c' = \frac{4r(r-1)}{N}$ note that for $x \in (1-c')[X - \bt'] + \bt'$, we have 
\[ x + \voronoi \subseteq (1-c')[X-\bt'] +\bt' + c'[X-\bt'] = X \]
Now, using \Cref{lemma:latticepointsinconvexvolume} with $\Lambda = \frac{1}{N} \Z^{r-1}$, $X = (1-c')[S - \bt']$, $q = 1$, $\bt = 0$ and $c = \frac{8r(r-1)}{N}$ we obtain 
\[ \left| \frac{1}{N} \Z^{r-1} \cap [ (1-c)(X - \bt') + \bt'] \right| \in [e^{\frac{-16(r-1)^2r}{N}},e^{\frac{16(r-1)^2r}{N}}] \cdot N^{r-1} \cdot (1-c)^{r-1} \vol(S) \]
So, by $(1-c)^{r-1} \geq e^{-c(r-1)} = e^{-16(r-1)^2 r/N}$, we deduce that 
\[ \left| \frac{1}{N} \Z^{r-1} \cap [ (1-c)(X - \bt') + \bt'] \right|  \geq  e^{-32(r-1)^2 r/N} \cdot N^{r-1} \vol(S). \]
Using \Cref{eq:countS}, we obtain 
\begin{align*} \Big|\big \{ x \in S \cap \frac{1}{N} \Z^{r-1} ~|~  x + (\frac{1}{2N},\frac{1}{2N}]^{r-1} \subsetneq S \big\}\Big | \\  \leq |S \cap \frac{1}{N} \Z^{r-1}| - \left| \frac{1}{N} \Z^{r-1} \cap [ (1-c)(X - \bt') + \bt'] \right|
\\ \leq [e^{\frac{8(r-1)^2r}{N}} - e^{\frac{-32(r-1)^2r}{N}}] \cdot N^{r-1} \vol(S) 
\\ \leq  \frac{64(r-1)^2 r}{N} \cdot N^{r-1} \vol(S) \end{align*}
whenever $\frac{32(r-1)^2r}{N} < 1/2$, since $e^{8x}-e^{-32x} < 64x$ for $x < 0.4$, which can be verified graphically.
\end{proof}

\begin{lemma} \label{lemma:volumedeltat} We have $\vol(\Delta_t^*) \leq \frac{(2t)^{r-1}}{(r-1)!}$. Similarly, $\vol(S) \leq \frac{1}{(r-1)!}$. 
\end{lemma}
\begin{proof} Write 
\[ W = \{ (x_i)_i \in \R^{r} ~|~ 1 \geq x_1 \geq x_2 \geq \ldots \geq x_{r-1} \geq -1 \mbox{ and }  x_r = -\sum_{i = 1}^{r-1} x_i \} \]. Then $\Delta_t^* \subseteq t \cdot W$. But one can prove that, by permuting the first $r-1$ indices of $S$,
that 
\[ \bigcup_{\sigma} \sigma(W) = \{ (x_i)_i \in \R^r ~|~ x_i \in [-1,1] \mbox{ for } i \in \{1,\ldots,r-1\}  \mbox{ and }  x_r = -\sum_{i=1}^{r-1} x_i \} =: U, \]  where the $\sigma$ are all permutations of the first $r-1$ indices. This is (up to sets of measure zero) a disjoint union. The volume of the latter set equals $2^{r-1}$... One can see this by applying the linear transformation $\phi$ that keeps the first $r-1$ indices intact and maps $x_r$ to $x_r \mapsto x_r - \sum_{i=1}^{r-1} x_i$ to the set $U$; we have that $\phi(U) = \{ (x_i)_i \in \R^r ~|~ x_i \in [-1,1] , x_r = 0 \}$ has volume $2^r$. Hence, $U$ itself has volume $2^r$, too, since $\det(\phi) = 1$ (by the substitution rule). Hence, $\vol(W) = \frac{2^{r-1}}{(r-1)!}$. As a consequence, $\vol(\Delta_t^*) \leq \frac{(2t)^{r-1}}{(r-1)!}$. For the bound on the volume on $S$, note that the map $y_i = \frac{t-a_i}{2t}$ for very $i \in \{1,\ldots,r-1\}$ linearly transforms $\Delta_t^*$ into the set $S$.
\end{proof}

\subsubsection{The tail error and the discretization error}
\paragraph{The tail error}
Since the space $\Delta_t^*$ is a compact space, we choose $T = 0$, which leads to a tail error of zero.
\paragraph{The discretization error}
\begin{lemma} \label{lemma:discretizationdiagonal}
 Let $N \in \Z_{>0}$ satisfy $N \geq  O(d e^{8 r^2 \log(r)} ) $ and let $t \geq 1$.
 Then
\[ \Delta(\distrdiag, \ddistrdiag) \leq N^{-1}  \cdot O(d e^{8 r^2 \log(r)} )  \]
\end{lemma}
\begin{proof} This follows from the fact that each of the components of the distributions of $\distrdiag$ and $\ddistrdiag$ are independent of each other. Applying \Cref{lemma:componentwisediscretizationdiagonal}, together with the fact that there are at most $d$ places $\nu$, we obtain the claim.
\end{proof}

\begin{lemma} \label{lemma:componentwisediscretizationdiagonal} Let $r \geq 2$, let $N \geq O( e^{8 r^2 \log(r)} ) $, let $t \leq 1$ and let $\nu$ some fixed place of $K$. Let $w = \distrdiag^{(\nu)}: \Delta_t^* \rightarrow \R$ denote the density function of the distribution as in \Cref{def:diagonaldistribution} and denote $\dw = \ddistrdiag^{(\nu)}: X \rightarrow \R$ for the probability function defined by the sampling process in \Cref{def:discretediagonaldistribution}, where $X = [t \cdot \mathbf{1} - \frac{2t}{N} \cdot \Z^{r-1}] \cap \Delta_t^*$ (where $\mathbf{1}$ is the all-one vector).
Write, for every $\da \in X$, $F_{\da} = \left( \da + [-\frac{t}{N},\frac{t}{N})^{r-1} \right) \cap \Delta_t^*$.

Then
\[ \Delta(\distrdiag^{(\nu)}, \ddistrdiag^{(\nu)}) = \sum_{\da \in X} \int_{a \in F_{\da}} \big |w(a) - \frac{\dw(\da)}{|F_{\da}|} \big| da \leq N^{-1}  \cdot O( e^{8 r^2 \log(r)} ). \]
\end{lemma}

\begin{proof} 
Note that $|X| = [t \cdot \mathbf{1} - \frac{2t}{N} \cdot \Z^{r-1}] \cap \Delta_t^*
= \psi(\frac{1}{N} \Z^{r-1} \cap S)$ with the linear bijection $\psi$ sending $y_i \mapsto t - 2ty_i$ for each component. Hence, by \Cref{lemma:helplemmaNZ},
\[ |X| = |\frac{1}{N} \Z^{r-1} \cap S| \in [e^{\frac{-8(r-1)^2r}{N}},e^{\frac{8(r-1)^2r}{N}}] \cdot N^{r-1} \cdot \vol(S) \]
which implies, together with $N \geq 8r^{10r^2} \geq 8r^3$ and \Cref{lemma:volumedeltat}, that $|X| \leq  N^{r-1} \cdot \frac{e}{(r-1)!}$.

We have $w(a) = g(a) = c \bar{g}(a)$ (for all $a \in \Delta_t^*$), as in \Cref{eq:targetdistribution}. We also would like to write $\dw(\da)$ in terms of $g(\da)$. By the procedure described in \Cref{def:discretediagonaldistribution} we can deduce that 
\begin{equation}  \dw(\da) \in  c_1 \cdot c_0 \cdot \big [ \frac{\bar{g}(\da)}{\bar{M}} - \frac{1}{N^2}, \frac{\bar{g}(\da)}{\bar{M}} + \frac{1}{N^2}   \big] ,\label{eq:interval} \end{equation}
for some constants $c_1, c_0 \in \R_{>0}$. By the fact that $\dw$ is a probability function, we also have $\sum_{\da \in X} \dw(\da) = 1$, which gives means of estimating $c_1 \cdot c_0$. As an Ansatz, we put $c_0 = \frac{\bar{M} \cdot c}{N_0^{r-1}}$, where $N_0 = N/(2t)$ and where $c \in \R_{>0}$ is the same $c$ as in the identity $g = c \bar{g}$; this choice is made in order to make $c_1$ close to one. This then yields (using $g = c \bar{g}$), and writing $\delta_0 := c_0/N^2$,
\begin{equation}  \dw(\da) \in  c_1 \big [ \frac{g(\da)}{N_0^{r-1}} - \delta_0, \frac{g(\da)}{N_0^{r-1}} + \delta_0   \big] ,\label{eq:interval2} \end{equation}
Our proof will now consist of a few technical parts.  \\
\noindent
\textbf{Claim (a):}
\begin{align}
\sum_{\da \in X} \int_{a \in F_{\da}} \big |g(a) - g(\da) \big| da
\leq \delta_1 := \frac{r \cdot \Lip(g) \cdot \vol(\Delta_t^*)}{N_0}\label{eq:lipschitzg}
\end{align}
\textbf{Proof of claim (a):}
We show that 
\[ \sum_{\da \in X} |F_{\da}| g(\da) \underbrace{\approx}_{\mbox{ \scriptsize{error} } \delta_1}  \int_{a \in \Delta_t^*} g(a) da = 1 \]
by the Lipschitz-continuity of $g$. By splitting up the space $\Delta_t^*$ into pieces $F_{\da}$, we obtain, by the fact that $|a - \da| \leq r/N_0$,
\begin{align*}  \sum_{\da \in X} |F_{\da}| g(\da) - \int_{a \in \Delta_t^*} g(a) da \leq \sum_{\da \in X}  \int_{a \in F_{\da}} | g(\da) - g(a)| da \leq \sum_{\da \in X}  \int_{a \in F_{\da}}  \Lip(g) \cdot \frac{r}{N_0} da \\ \leq \frac{r \cdot \Lip(g) \cdot \vol(\Delta_t^*)}{N_0} = \delta_1 \end{align*}
\textbf{Claim (b):}
 \begin{align}
 \sum_{\da \in X} \left| |F_{\da}| g(\da) - \frac{ g(\da)}{ N_0^{r-1}} \right| \leq \delta_2 :=  \frac{64(r-1)^2r}{N} \cdot \|g\|_\infty \cdot (2t)^{r-1} \vol(S).
 \label{eq:delta2}
\end{align}
\textbf{Proof of claim (b):}
We will show that 
\[\left|\sum_{\da \in X} \frac{g(\da)}{N_0^{r-1}} -  \sum_{\da \in X} |F_{\da}| g(\da) \right|  \leq \sum_{\da \in X} \left|\frac{g(\da)}{N_0^{r-1}} -  |F_{\da}| g(\da) \right| \leq \delta_2. \]
This inequality stems from the fact that, for most $\da \in X$ holds that $|F_{\da}| = N_0^{-(r-1)}$ (all of them, except those at the edge of $\Delta_t^*$).

Our goal is to count those $\da$ for which $|F_{\da}| < N_0^{-(r-1)}$.
By the last statement of \Cref{lemma:helplemmaNZ} (considering the linear map between $S$ and $\Delta_t^*$), there are at most $\frac{64(r-1)^2 r}{N} \cdot N^{r-1} \cdot \vol(S)$ of these.

Hence,
\begin{align*} \left| N_0^{-(r-1)} \sum_{\da \in X} g(\da) - \sum_{\da \in X} |F_{\da}| g(\da)  \right| & \leq \|g\|_\infty \cdot N_0^{-(r-1)}  \cdot \frac{64(r-1)^2 r}{N} \cdot \vol(S) \cdot N^{r-1} \\ & \leq \frac{64(r-1)^2r}{N} \cdot \|g\|_\infty \cdot (2t)^{r-1} \vol(S) = \delta_2,  \end{align*}
where we use that $N_0 = N/(2t)$. \\ \noindent
\textbf{Claim (c):}
\begin{align}
 \sum_{\da \in X} \frac{ g(\da)}{N_0^{r-1}}  \in [1-\delta_1 - \delta_2, 1 + \delta_1 + \delta_2]. \label{eq:sumboundg}
\end{align}
in particular $\sum_{\da \in X} \frac{ g(\da)}{N_0^{r-1}}  \leq 2$ if $\delta_1 + \delta_2 \leq 1$. \\ \noindent
\textbf{Proof of claim (c):} 
By using part (a) and (b) we can deduce that 
\[ N_0^{-(r-1)} \sum_{\da \in X} g(\da) \underbrace{\approx}_{\mbox{\scriptsize{error }} \delta_2}  \sum_{\da \in X} |F_{\da}| g(\da)  \underbrace{\approx}_{\mbox{\scriptsize{error }} \delta_1}   \int_{a \in \Delta_t^*} g(a) da = 1 \]
\textbf{Claim (d):}
\begin{align} |c_1 -1 | \leq \delta_3 :=  2\delta_1 + 2\delta_2 + 2\cdot |X| \cdot \delta_0. \label{eq:c1bound} \end{align} 
if $\delta_1 + \delta_2 + |X| \cdot \delta_0 \leq 1/4$.
In particular, $|c_1| \leq 2$ in that case.
This requires the assumption of $N$ being sufficiently large in the lemma.
\\ \noindent
\textbf{Proof of claim (d):} 
Combining \Cref{eq:interval2} 
with the law of total probability, we deduce
\[ 1 = \sum_{\da \in X} \dw(\da) \in c_1 \cdot \big [ \frac{\sum_{\da \in X} g(\da)}{N_0^{r-1}} -  |X| \cdot \delta_0, \frac{\sum_{\da \in X}g(\da)}{N_0^{r-1}} +  |X| \cdot \delta_0   \big] \]
Hence we can deduce that $|c_1^{-1} - N_0^{-(r-1)} \sum_{\da \in X} g(\da)| \leq   |X| \cdot \delta_0$. So, using part (c),
\[ |c_1^{-1} - 1| \leq \delta_1 + \delta_2 +  |X| \cdot \delta_0 \]
and, inverting, assuming the right-hand size being smaller than $1/4$, we obtain 
\[ |c_1 - 1| \leq 2\delta_1 + 2\delta_2 + 2 \cdot |X| \cdot \delta_0.  \]
\textbf{Conclusion:}
Combining these results, we obtain the following sequence of inequalities, assuming that $\delta_1 + \delta_2 \leq 1$ and $\delta_1 + \delta_2 + c_1 \cdot |X| \cdot \delta_0 \leq 1/4$.
\begin{align} & \sum_{\da \in X} \int_{a \in F_{\da}} \big |w(a) - \frac{\dw(\da)}{|F_{\da}|} \big| da = \sum_{\da \in X} \int_{a \in F_{\da}} \big |g(a) - \frac{\dw(\da)}{|F_{\da}|} \big| da ~~~~\mbox{(since $w(a) = g(a)$)} \nonumber \\  
\leq & \sum_{\da \in X} \int_{a \in F_{\da}} \big |g(a) - \frac{c_1 g(\da)}{|F_{\da}|N_0^{r-1}} \big| da + \sum_{\da \in X} c_1\delta_0  ~~~~\mbox{(by \Cref{eq:interval2})}
\nonumber \\ 
\leq & \sum_{\da \in X} \int_{a \in F_{\da}} \big |g(a) - \frac{g(\da)}{|F_{\da}|N_0^{r-1}} \big| da + \delta_3 \sum_{\da \in X} \frac{ g(\da)}{N_0^{r-1}}  + 2\delta_0 |X| ~~~~\mbox{(by \Cref{eq:c1bound})}
\nonumber \\ 
\leq & \sum_{\da \in X} \int_{a \in F_{\da}} \big |g(a) - \frac{g(\da)}{|F_{\da}|N_0^{r-1}} \big| da + 2\delta_3  + 2\delta_0 |X|~~~~\mbox{(by \Cref{eq:sumboundg})}
\nonumber \\ 
\leq & \sum_{\da \in X} \int_{a \in F_{\da}} \big |g(a) - g(\da) \big| da + \delta_2 +  2\delta_3  + 2 \delta_0 |X| ~~~~\mbox{(by \Cref{eq:delta2})} \nonumber \\ 
\leq & \delta_1 + \delta_2 +  2\delta_3  + 2 \delta_0 |X| ~~~~\mbox{(by \Cref{eq:lipschitzg})} \nonumber \\
\leq & 6|X| \delta_0 + 5 \delta_1 + 5 \delta_2 
\label{eq:sumintegralthing}
\end{align}
where the last inequality follows from writing out the definition $\delta_3 = 2\delta_1 + 2 \delta_2 + 2 \cdot |X| \cdot \delta_0$ and using that $|c_1|\leq 2$.

We have $\delta_0 = \frac{\bar{M}c}{N_0^{r-1} N^2} = \frac{\bar{M}c (2t)^{r-1}}{N^{r-1} N^2}$ and $|X| \leq N^{r-1} \cdot \frac{e}{(r-1)!}$. Hence, using that $\bar{M}c \leq (16r^2)^{r(r-1)} \cdot \left( \frac{4r^2}{t}\right)^{r-1}$ by \Cref{lemma:boundong}, 
\begin{equation} \label{eq:bigOdelta0} 6 |X| \delta_0 \leq \frac{6e \bar{M} c (2t)^{r-1}}{(r-1)! N^2} \leq   \frac{6e (16r^2)^{r(r-1)} \left( 8r^2\right)^{r-1}  }{(r-1)! N^2} \leq  N^{-2} \cdot O( e^{8 r^2 \log(r)} ) .\end{equation}
By the fact that $\vol(\Delta_t^*) = (2t)^{(r-1)} \vol(S)$, and $N_0 = N/(2t)$ and subsequently the bound $\Lip(g) \leq \frac{r^2}{t} \cdot   (16r^2)^{r(r-1)} \cdot \left( \frac{4r^2}{t}\right)^{r-1}$ by \Cref{lemma:boundong}, we obtain 
\begin{equation} 5 \delta_1 = \frac{5r \cdot \Lip(g) (2t)^{r} \vol(S)}{N} \leq \frac{5 r \cdot  (16r^2)^{r(r-1)} (8r^2)^{r} }{(r-1)! N} \leq N^{-1} \cdot O( e^{8 r^2 \log(r)} ).
\end{equation}
For the last error, $\delta_2$, note that, $\|g\|_\infty \leq \bar{M} c \leq (16r^2)^{r(r-1)} \cdot \left( \frac{4r^2}{t}\right)^{r-1}$ by \Cref{lemma:boundong}, we see that 
\[ 5 \delta_2 \leq \frac{64 (r-1)^2 r}{(r-1)! \cdot N} (16r^2)^{r(r-1)} \cdot \left(8r^2\right)^{r-1} \leq N^{-1}  \cdot O( e^{8 r^2 \log(r)} ) .  \]
This finishes the proof.

\end{proof}

\subsubsection{The continuity error}

\begin{lemma} \label{lemma:diagonalcontinuity} Let $\mathcal{A}_x$ (for $x \in \GL_r(K_\R)$) be the output distribution of \Cref{alg:canonical} on input\footnote{We denote by $e^x$ with a diagonal matrix $x \in \GL_r(K_\R)$ the element of $\GL_r(K_\R)$ with diagonal $\diag(e^{x_i})_i$, where $e^{x_i}$ is also component-wise over all places of $K$.} $g \cdot e^x \cdot g'$ for fixed $g,g' \in \GL_r(K_\R)$. Let $N \in \Z_{>0}$ be the discretization parameter.

Then
\[ \mathcal{C}( \ddistrdiag ,\mathcal{A}) \leq  N^{-1/2} \cdot O(n^5 \cdot \cd(g)^{1/2} \cdot \sqrt[4]{\log(1/\eps_0)})
\]
\end{lemma}
\begin{proof}
By using again that the $\nu$-components of $\ddistrdiag$ are independent and commute, we can deduce that

 \[ \mathcal{C}(\ddistrdiag,\mathcal{A}) \leq d \cdot \max_{\nu} \mathcal{C}(\ddistrdiag^{(\nu)},\mathcal{A}) .\]  
 Hence, writing out the continuity error, using the bound of \Cref{lemma:separationoferrors}, with $X = \frac{2t}{N} \Z^{r-1} \cap \Delta_t^*$ and $F_{\ddx} = (\ddx + [-\frac{t}{N},\frac{t}{N}]) \cap \Delta_t^*$,
 \[ \mathcal{C}(\ddistrdiag^{(\nu)},\mathcal{A}) \leq \max_{\ddx \in X} \max_{x \in F_{\ddx}} \| \mathcal{A}_{x} - \mathcal{A}_{\ddx}\|_1  
 \]
 Writing $\mathcal{R}$ for the output distribution of \Cref{alg:canonical}, we see that (writing $e^{x'}$ for the diagonal on the not-$\nu$-component), using \Cref{lemma:holdercontinuous},
 \begin{align*} \| \mathcal{A}_{x} - \mathcal{A}_{\ddx}\|_1 & = \|\mathcal{R}_{g e^{x} e^{x'} g'} -  \mathcal{R}_{g e^{\ddx} e^{x'}g'} \|_1 \leq  L \cdot \|ge^xe^{x'}g'(g e^{\ddx} e^{x'}g')^{-1} -I \|^{1/2}  \\ & \leq L \cdot \|ge^{x- \ddx}g^{-1} -I \|^{1/2}  \leq \cd(g)^{1/2} \cdot L  \cdot \|e^{x - \ddx} - I  \|^{1/2} \\ &\leq \frac{2 \cdot \cd(g)^{1/2} \cdot L \cdot \sqrt{r}}{N^{1/2}}. \end{align*}
 
 The last inequality follows from the fact that $\| x - \ddx\|_\infty \leq 2/N$ whenever $x \in F_{\ddx}$ (since $\ddx \in \frac{2t}{N} \Z^{r-1}$ and $t \leq 1$); and the fact that $e^a - 1 \leq 2a$ for $a < 1$. Instantiating $L = 92 n^3 \sqrt[4]{\log(1/\eps_0)}$ from \Cref{lemma:holdercontinuous} yields the claim.
\end{proof}
\subsubsection{Run time}

\begin{lemma} \label{lemma:efficientdiagonal} Let
$N \geq O(e^{8 r^2 \log(r)})$
Then the discretized diagonal distribution $\ddistrdiag$ (\Cref{def:discretediagonaldistribution}) can be sampled from using bit complexity $O(d \cdot e^{8r^2 \log(r)} \log N)$.
\end{lemma}
\begin{proof}
We show that the procedure described in \Cref{def:discretediagonaldistributioncomp} can be run with bit complexity $O(e^{8r^2 \log(r)} \log N)$. Then repeating this for each place $\nu$ (which there are at most $d$) yields the claim. 

We now focus on the algorithm in \Cref{def:discretediagonaldistributioncomp}.
The first (inner) loop is about sampling a uniform element from $\Delta_t^*$ and consists of lines 2 and 3; the acceptance probability is $ \frac{|S \cap \frac{1}{N} \Z^{r-1}|}{|\Delta^0 \cap \frac{1}{N} \Z^{r-1}|}$.

By \Cref{lemma:helplemmaNZ}, we can estimate this acceptance probability by 
\[ \frac{|S \cap \frac{1}{N} \Z^{r-1}|}{|\Delta^0 \cap \frac{1}{N} \Z^{r-1}|} \in [e^{\frac{-16(r-1)^2r}{N}},e^{\frac{16(r-1)^2r}{N}}]  \frac{\vol(S)}{\vol(\Delta^0)}  \]

Hence, using the lower bound on $\frac{\vol(S)}{e \vol(\Delta^0)}$ from \Cref{eq:lowerboundSdelta0} in \Cref{subsec:uniformsamplingdelta}, and assuming $N \geq 16r^3$, we deduce
\[  \frac{|S \cap \frac{1}{N} \Z^{r-1}|}{|\Delta^0 \cap \frac{1}{N} \Z^{r-1}|} \geq \frac{\vol(S)}{e \vol(\Delta^0)} \geq e^{-1} \cdot  (2(r-1))^{-(r-1)}. \]
Hence, using that sampling a uniform element in $[0,1) \cap \frac{1}{N} \Z$ costs time $O(\log N)$, we obtain that this 
first loop takes about $O((2(r-1))^{-(r-1)} \log(N))$ bit operations.

For the second (outer) loop, the acceptance probability is at least (using the notation $\da$ and $X$ from \Cref{lemma:componentwisediscretizationdiagonal})
\begin{align}- \frac{1}{N^2} +  \frac{1}{|S \cap \frac{1}{N} \Z^r|} \sum_{\da \in X} \frac{\bar{g}(\da)}{\bar{M}}  & \geq -\frac{1}{N^2} + \frac{1}{|S \cap \frac{1}{N} \Z^r|}\frac{1}{c\bar{M}} \sum_{\da \in X } g(\da) \nonumber \\  & \geq -\frac{1}{N^2} + \frac{N_0^{r-1}}{c\bar{M}|S \cap \frac{1}{N} \Z^r|}. 
\label{eq:lowerboundouterloop}
\end{align}
where the last lower bound comes from \Cref{eq:sumboundg} (where we need to assume $N \geq O(e^{8 r^2 \log(r)})$). We note that from $N \geq 8r^3$ and \Cref{lemma:volumedeltat}, we see that $|S \cap \frac{1}{N} \Z^{r-1}| \leq N^{r-1} \cdot \frac{e}{(r-1)!}$. And, using the bound $c \bar{M} \leq (16r^2)^{r(r-1)} \cdot \left( \frac{4r^2}{t}\right)^{r-1}$ by \Cref{lemma:boundong}, we can continue lower bounding \Cref{eq:lowerboundouterloop} by (using $N_0 = N/(2t)$)
\begin{align*} \geq - \frac{1}{N^2} + \frac{(r-1)!/e \cdot (2t)^{-(r-1)}}{c\bar{M}} \geq -\frac{1}{N^2} + \frac{(r-1)! }{e \cdot (16r^2)^{r(r-1)} \cdot \left(8r^2\right)^{r-1}}  \\
\geq 
-\frac{1}{N^2} + O(e^{-8r^2 \log(r)}) \geq O(e^{-8r^2 \log(r)}) ,
\end{align*}
by the assumption on $N$. Hence, the outer loop running time is $O( e^{8r^2 \log(r)})$, yielding a total running time of $O(e^{8r^2 \log(r)} \log N)$.
\end{proof}

\subsubsection{Concluding all errors}
\begin{lemma} \label{lemma:maindiag} Let $\mathcal{A}_x$ (for $x \in \GL_r(K_\R)$) be the output distribution of \Cref{alg:canonical} on input\footnote{We denote by $e^x$ with a diagonal matrix $x \in \GL_r(K_\R)$ the element of $\GL_r(K_\R)$ with diagonal $\diag(e^{x_i})_i$, where $e^{x_i}$ is also component-wise over all places of $K$.} $g \cdot e^x \cdot g'$ for fixed $g,g' \in \GL_r(K_\R)$. Let $N \in \Z_{>0}$ be the discretization parameter that satisfies $N \geq O(d e^{8r^2 \log(r)})$

Then
Then, 
\begin{align*} 
& \| \underset{x \from \distrdiag}{\mathbb{E}}[\mathcal{A}_x] - \underset{\ddx \from \ddistrdiag}{\mathbb{E}} [\mathcal{A}_{\ddx}]   \|   &\leq 
N^{-1/2} \cdot O( d e^{8r^2 \log(r)} + n^5 \cd(g)^{1/2} \cdot \sqrt[4]{\log(1/\eps_0)}  ). 
\end{align*}
\end{lemma}
\begin{proof} This follows from \Cref{lemma:discretizationdiagonal,lemma:diagonalcontinuity}.
\end{proof}

\subsection{Discretization of the uniform distribution in \texorpdfstring{$\SU_r(K_\R)$}{SUr(KR)}}
\label{sec:discSU}
\subsubsection{The continuous and the finite distribution}
\paragraph{The continuous distribution}
\begin{definition}[Angle distribution] \label{def:angledistribution} We denote $\Theta^{(r)} = [0,2\pi] \times [0,\pi]^{r-1}$ and we define on it a density function by the following rule: $\rho^{(r)}(\vec{\theta}) := \prod_{j =1}^r \rho_j(\theta_j)$ for $\vec{\theta} = (\theta_1,\ldots,\theta_r) \in \Theta^{(r)}$, where
\[ \rho_j(\theta) :=  \begin{cases}
                       \frac{1}{\sqrt{\pi}} \frac{\Gamma(\frac{j+1}{2})}{\Gamma(\frac{j}{2})} \sin^{j-1}(\theta) & \mbox{ if } j > 1 \\ 
                       \frac{1}{2\pi}  & \mbox{ if } j = 1
                      \end{cases}.
 \] 
\end{definition}
\begin{definition}[Uniform distribution on $\SU_r$] \label{def:uniformSU}
As in \Cref{def:thetanotation}, we define the \emph{uniform distribution} on $\SU_r(\R)$ by the distribution of $U_{\vec{\theta}} \in \SU_r(\R)$, the real unitary matrix associated with $\vec{\theta}$ defined by the procedure in \Cref{lemma:uniformSU}, where $\vec{\theta} \in \prod_{j=2}^r S^{j-1}(\R)$ is for each component $S^{j-1}$ is sampled according to the (continuous) angle distribution as in \Cref{def:angledistribution}.

Analogously, we define the \emph{uniform distribution} on $\SU_r(\C)$ by the distribution of $U_{\vec{\theta}} \in \SU_r(\C)$, the complex unitary matrix associated with $\vec{\theta}$ defined by the procedure in \Cref{lemma:uniformSU}, where $\vec{\theta} \in \prod_{j=1}^r S^{2j-1}(\R)$ is for each component $S^{2j-1}$ is sampled according to the discrete angle distribution as in \Cref{def:angledistribution}.

Both are just the uniform distribution over $\SU_r(\R)$ and $\SU_r(\C)$ respectively.
\end{definition}

\paragraph{The finite distribution}
\begin{definition}[Discretized angle distribution] \label{def:discreteangledistribution}
For $N \in \Z_{>0}$, we denote 
\[ \dTheta^{(r)} = ([0,2\pi] \cap \frac{2\pi}{N} \Z) \times ([0,\pi]^{r-1} \cap \frac{\pi}{N} \Z^{r-1}) \] 
and we define on it a density function $\drho^{(r)}$ by the following procedure:
\begin{enumerate}
 \item For each $i \in \{1,\ldots,r \}$ do:
 \item \hspace{.5cm} If $i = 1$, sample $z_1 = \frac{\theta_1}{2\pi}$ from $ [0, 1) \cap \frac{1}{N} \Z$ uniformly.
 \item \hspace{.5cm} If $i > 1$, 
 \item \hspace{1cm} Sample $z_i = \frac{\theta_i}{\pi}$ from $[0, 1) \cap \frac{1}{N}\Z$  uniformly.
 \item \hspace{1cm} Compute $q_i \in \frac{1}{2N^2}\Z$ such that $\frac{1}{2N^2} > q_i - \frac{1}{\sqrt{\pi}} \frac{\Gamma(\frac{i+1}{2})}{\Gamma(\frac{i}{2})} \geq 0$.
 \item \hspace{1cm} Sample $u \from \frac{1}{N^2} \Z \cap [0, q_i ]$ uniformly.
 \item \hspace{1cm} Compute $\tilde{\rho}_i \in \frac{1}{2N^2} \Z$ such that $-\frac{1}{2N^2} < \tilde{\rho}_i - \rho_i(\pi z_i ) \leq 0$
 \item \hspace{1cm} If $u < \tilde{\rho}_i$ proceed (accept $\theta_i$), otherwise go to line $4$.
\end{enumerate}

\end{definition}

\begin{definition}[Discretized uniform distribution on $\SU_r$] \label{def:discretizedSU}
We define the \emph{discretized uniform distribution} on $\SU_r(\R)$ by the distribution of $U_{\vec{\theta}} \in \SU_r(\R)$, the real unitary matrix associated with $\vec{\theta}$ defined by the procedure in \Cref{lemma:uniformSU}, where $\vec{\theta} \in \prod_{j=2}^r S^{j-1}(\R)$ is for each component $S^{j-1}$ is sampled according to the discrete angle distribution as in \Cref{def:discreteangledistribution}.

Analogously, we define the \emph{discretized uniform distribution} on $\SU_r(\C)$ by the distribution of $U_{\vec{\theta}} \in \SU_r(\C)$, the complex unitary matrix associated with $\vec{\theta}$ defined by the procedure in \Cref{lemma:uniformSU}, where $\vec{\theta} \in \prod_{j=1}^r S^{2j-1}(\R)$ is for each component $S^{2j-1}$ is sampled according to the discrete angle distribution as in \Cref{def:discreteangledistribution} (see also \Cref{def:thetanotation}).
\end{definition}

\begin{lemma}[Efficiency of the finite angle distribution] \label{lemma:efficientangle} For $N \geq \pi^3 r^2 + 2$, there exists an algorithm that computes a sample from the discrete angle distribution (\Cref{def:discreteangledistribution}), assuming we can sample perfect bits. This algorithm runs in time 
$\poly(\log N, r)$.
\end{lemma}
\begin{proof} Going over the lines of \Cref{def:discreteangledistribution}, we show that this is an efficient algorithm (without regarding the rejection probability). We finish the proof by showing that the algorithm has only a polynomially small (in $r$) acceptance probability.

Since sampling uniformly in $[0,1) \cap \frac{1}{N} \Z$ can be efficiently done in time $\poly(\log N)$, we deduce that lines 1 -- 4 can be handled efficiently. In line 5 we approximate $\frac{1}{\sqrt{\pi}} \frac{\Gamma(\frac{i+1}{2})}{\Gamma(\frac{i}{2})}$ within an error of $1/N^2$ which can be done in time $\poly(\log N)$. In line 6, again a uniform sample is drawn, which can be done in time $\poly(\log N)$. In line 7, $\rho_i(\pi z_i )$ is being approximated within an error range of $1/N^2$, which can be done in time $\poly(\log N)$. In line 8 two rationals $u, \tilde{\rho}_i \in \frac{1}{2N^2} \Z$ are compared, which can be done in time $\poly(\log N)$.

For the acceptance probability, we assume that $r > 1$, otherwise the proof is trivial. We compute the acceptance probability in a single loop over $i$ (starting at line 3), we can deduce that it is at least 
\begin{equation} \label{eq:lowerboundacceptance} q_i^{-1} \sum_{z_i \in [0,1) \cap \frac{1}{N} \Z} \left( \rho_i(\pi z_i) - \frac{1}{N^2} \right) \geq \frac{1}{\sqrt{r}} \left( \sum_{z_i \in [0,1) \cap \frac{1}{N} \Z}  \rho_i(\pi z_i)   - \frac{1}{N} \right) \end{equation}
since we have $q_i \leq  \frac{1}{\sqrt{\pi}} \frac{\Gamma(\frac{i+1}{2})}{\Gamma(\frac{i}{2})} + \frac{1}{2N^2} \leq \sqrt{r}$.

As $\rho_j$ is $\pi r^2$-Lipschitz, we can deduce that 
\begin{align*} &\left | \sum_{z_i \in [0,1) \cap \frac{1}{N} \Z}  \rho_i(\pi z_i) - \int_{t_i \in [0,\pi]} \rho_i ( t_i) dt_i \right | \\  \leq  & \int_{t_i \in [0,\frac{\pi}{N}]} \sum_{z_i \in [0,\pi) \cap \frac{\pi}{N} \Z} |\rho_i(z_i + t_i) - \rho_i(z_i) | dt_i \\ 
\leq  &  \int_{t_i \in [0,\frac{\pi}{N}]} \sum_{z_i \in [0,\pi) \cap \frac{\pi}{N} \Z} \pi r^2 t_i dt_i = N  \pi r^2 \cdot \frac{\pi^2}{2N^2} = \frac{\pi^3 r^2}{2N}.
\end{align*}
Hence, using that $ \int_{t_i \in [0,\pi]} \rho_i (t_i) dt_i = 1$ and \Cref{eq:lowerboundacceptance}, we obtain a lower bound on the acceptance probability of 
\[ \frac{1}{\sqrt{r}} \left(  1 - \frac{(1 + \pi^3 r^2/2)}{N} \right ) \geq \frac{1}{2 \sqrt{r}}, \]
which is inversely polynomial in $r$.

Hence, the entire algorithm runs in time $\poly(\log N, r)$.
\end{proof}

\newcommand{\Ang}{\mathrm{Ang}}
\newcommand{\Angnu}{\mathrm{Ang}^{\nu}}

\subsubsection{The tail error and the discretization error}
For this distribution, it is nicer to write $\mathcal{B}_{\theta} = \mathcal{A}_{U_{\theta}}$ where $U_\theta$ is defined by the procedure in \Cref{lemma:uniformSU}.

Recall that $\SU_r(K_\R) \simeq \prod_{\nu} \SU_r(K_\nu)$ and that an element in $\SU_r(K_\nu)$ can be encoded by a sequence of angles as in 
\[ \Angnu := \{ (\theta_1^{(1)}, (\theta_1^{(2)},\theta_2^{(2)}), \ldots, (\theta_1^{(k)}, \theta_2^{(k)} \ldots, \theta_k^{(k)}), \ldots ,( \theta_1^{(r)}, \ldots, \theta_r^{(r)})) \}, \]
where each tuple is distributed as $\rho^{(j)}$ as in \Cref{def:angledistribution}, which yields a distribution $\distr_{\Angnu}$ over $\Angnu$. The precise sizes of the tuples depend on whether $\nu$ is real or complex.

We put $\Ang = \prod_{\nu} \Angnu$ and $\distr_{\Ang}$ as the compound distribution. The discrete distribution $\ddistr_{\Ang}$ is defined as the distribution in which each of the angles in $\Angnu$ are distributed via the \emph{discrete} angle distribution.

\paragraph{The tail error} We choose the tail to be 
\[ T = \{ \vec{\theta} \in \Ang ~|~  \mbox{ there exists } i \mbox{ such that } \vec{\theta}_i \leq 2N^{-1/2}  \}. \]
By the law of total probability, the fact that $\Ang$ has at most $2dr^2$ ``angular components'' and subsequently by the $2 \pi r^2$-Lipschitz constant of the probability distributions $\rho^{(i)}(\theta^{(i)})$, we obtain 
\begin{align} \label{eq:tailang}
 \mathcal{T}( \Ang) & = \int_{\vec{\theta} \in T} \distr_{\Ang}(\vec{\theta}) d \vec{\theta} \leq 2dr^2 \max_{i \in \{1,\ldots,r\}} \int_{\theta^{(i)}\in T^{(i)}} \rho^{(i)}(\theta^{(i)}) d \theta^{(i)}  \nonumber \\ 
 & \leq 2dr^2 (2\pi r^2)  \cdot 2N^{-1/2}  \leq N^{-1/4} \cdot O(n^5) ,
\end{align}
where $T^{(j)} = \{  \theta^{(j)}~|~ \mbox{ there exists $i$ such that } \theta^{(j)}_i \leq 2N^{-1/2} \}$.

\paragraph{The discretization error}
\begin{lemma} \label{lemma:discretizationang}Writing $\distr_{\Ang}$ for the uniform distribution over $\SU_r(K_\R)$ and $\ddistr_{\Ang}$ for the discretized version of it (via the angle distributions), we have 
\[ \Delta(\distr_{\Ang}, \ddistr_{\Ang}) \leq  \frac{16 \pi^2 d r^4}{N},  \]
for $N > 8\pi^2 r^2$.
\end{lemma}
\begin{proof} Note that, by the triangle inequality, by discretizing the $\nu$-components of $\distr_{\Ang}$ (which are $\distr_{\Angnu}$) component by component, and subsequently discretizing the components of the distribution over $\Angnu$ (which are $\rho^{(i)}$ for $i \in \{1,\ldots,r\}$) component by component (which is possible by independence),  we have
\[ \Delta(\distr_{\Ang}, \ddistr_{\Ang})  \leq d \cdot \max_{\nu} \Delta(\distr_{\Angnu}, \ddistr_{\Angnu}) \leq d \sum_{i =  1}^r \Delta( \rho^{(i)}, \ddot{\rho}^{(i)} ).  \] 
The claim follows by \Cref{lemma:angleproperty}. Note that in this upper bound we included the tail space $T$, but since that can only increase the estimate, that does not matter.
\end{proof}

\newcommand{\rankvar}{r_0}

\begin{lemma} \label{lemma:angleproperty} For any $r_0 \in \Z_{>0}$ and $N > 8 \pi^2 \rankvar^2$,  the discretized angle distribution (\Cref{def:discreteangledistribution}) and the angle distribution (\Cref{def:angledistribution}) 
satisfy the following property:
\begin{equation}
\Delta(\rho^{(\rankvar)}, \ddot{\rho}^{(\rankvar)}) \sum_{\dtheta^{(\rankvar)} \in \dTheta^{(\rankvar)}}\int_{\theta^{(\rankvar)} \in  F^{(\rankvar)}} \left | \rho^{(\rankvar)}(\theta^{(\rankvar)}) -  \frac{ \drho^{(\rankvar)}(\dtheta^{(\rankvar)})}{\vol(F^{(\rankvar)})} \right| d\theta \leq \frac{16 \pi^2 \rankvar^3}{N}, 
\end{equation}
where $F^{(\rankvar)} = [0,\frac{2\pi}{N}) \times [0,\frac{\pi}{N})^{\rankvar-1}$.
 
\end{lemma}
\begin{proof} From the sample procedure in \Cref{def:discreteangledistribution} follows that the sampling probabilities of the components $\dtheta^{(\rankvar)}_i$ of $\dtheta^{(\rankvar)}$ are independent, as well as those of $\theta^{(\rankvar)}$. We just write the corresponding probability functions with $\drho_i^{(\rankvar)}$ and $\rho_i^{(\rankvar)}$. 
By the trick $ab - a'b' = (a-a')b + (b-b')a'$ (and applying this inductively to $\rho^{(\rankvar)} = \prod_i \rho_i^{(\rankvar)}$ and $\drho^{(\rankvar)} = \prod_i \drho_i^{(\rankvar)}$) we obtain 
\begin{align} \sum_{\dtheta^{(\rankvar)} \in \dTheta^{(\rankvar)}}\int_{\theta^{(\rankvar)} \in  F^{(\rankvar)}} \left | \rho^{(\rankvar)}(\theta^{(\rankvar)}) -  \frac{ \drho^{(\rankvar)}(\dtheta^{(\rankvar)})}{\vol(F^{(\rankvar)})} \right| d\theta \\  \leq \sum_{i =1}^r \sum_{\dtheta^{(\rankvar)}_i \in a_i \pi \cdot( [0,1) \cap \frac{1}{N} \Z)} \int_{\theta^{(\rankvar)}_i \in [0,a_i\pi/N)} \left | \rho_i^{(\rankvar)}(\theta_i^{(\rankvar)}+\dtheta_i^{(\rankvar)}) -  \frac{ \drho_i^{(\rankvar)}(\dtheta_i^{(\rankvar)})}{\vol([0,a_i\pi/N))} \right| \label{eq:differenceeqrho} \end{align}
 where $a_i = 2$ if $i = 1$ and $a_i = 1$ otherwise.
 
Now it is true, by the rejection sampling mechanism of \Cref{def:discreteangledistribution}, that $\drho_i^{(\rankvar)}(\dtheta_i^{(\rankvar)}) \in c_i [\rho_i^{(\rankvar)}(\dtheta_i^{(\rankvar)})- 2N^{-2},\rho_i^{(\rankvar)}(\dtheta_i^{(\rankvar)})]$ where $c_i \in \R_{>0}$ is a constant only depending on $i$, satisfying 
\[ c_i^{-1} \in [ \sum_{\dtheta_i^{(\rankvar)}  } (\rho_i^{(\rankvar)}(\dtheta^{(\rankvar)}_i) - 2N^{-2}), \sum_{\dtheta_i^{(\rankvar)} } \rho_i^{(\rankvar)}(\dtheta^{(\rankvar)}_i)] \subseteq [ - 2N^{-1}  +\sum_{\dtheta_i^{(\rankvar)}  } \rho_i^{(\rankvar)}(\dtheta^{(\rankvar)}_i), \sum_{\dtheta_i^{(\rankvar)} } \rho_i^{(\rankvar)}(\dtheta^{(\rankvar)}_i)], \]
where $\dtheta_i^{(\rankvar)}$ ranges over $a_i\pi \cdot( [0,1) \cap \frac{1}{N} \Z)$ (with $a_i = 2$ if $i = 1$ and $a_i = 1$ otherwise). 

Since $\rho_i^{(\rankvar)}$ is $\pi \cdot \rankvar^2$-Lipschitz, we can deduce that 
\[ \left| \int_{\theta_i^{(\rankvar)}} \rho_i^{(\rankvar)}(\theta) d\theta -  \frac{a_i\pi}{N} \sum_{\dtheta_i^{(\rankvar)} } \rho_i^{(\rankvar)}(\dtheta^{(\rankvar)}_i) \right| \leq 
\sum_{\dtheta_i^{(\rankvar)}} \int_{\theta_i^{(\rankvar)} \in [0,a_i\pi/N)} |\rho_i^{(\rankvar)}(\dtheta_i^{(\rankvar)} + \theta_i^{(\rankvar)}) - \rho_i^{(\rankvar)}(\dtheta_i^{(\rankvar)})| d\theta_i^{(\rankvar)}\]
\[  \leq  \frac{2 \pi^2 \rankvar^2}{N} \] 
Hence, since $\int_{\theta_i^{(\rankvar)}} \rho_i^{(\rankvar)}(\theta) d\theta  = 1$, we thus have 
\[ c_i^{-1} \in [ - 2N^{-1}  +\sum_{\dtheta_i^{(\rankvar)}  } \rho_i^{(\rankvar)}(\dtheta^{(\rankvar)}_i), \sum_{\dtheta_i^{(\rankvar)} } \rho_i^{(\rankvar)}(\dtheta^{(\rankvar)}_i)] = [-2N^{-1} + \frac{N}{a_i\pi} - 2 \pi^2 \rankvar^2 ,  \frac{N}{a_i\pi} + 2 \pi^2 \rankvar^2]. \]
Which means that $a_i \pi/(c_i N) \in [-2a_i\pi N^{-2} - 2\pi^2 \rankvar^2 N^{-1} + 1, 2\pi^2 \rankvar^2 N^{-1} + 1]$. Choosing $N > 8\pi^2 \rankvar^2$ we surely have 
$\frac{c_i N}{a_i\pi} \in [ 1- \frac{4\pi^2 \rankvar^2}{N}, 1 + \frac{4\pi^2r^2}{N}]$ and
therefore, 
\[ \frac{ \drho_i^{(\rankvar)}(\dtheta_i^{(\rankvar)})}{\vol([0,a_i\pi/N))} \in \frac{c_i N}{a_i\pi} \cdot [\rho_i^{(\rankvar)}(\dtheta_i^{(\rankvar)})- 2N^{-2},\rho_i^{(\rankvar)}(\dtheta_i^{(\rankvar)})].  \]
Thus, 
\[ \left|\frac{ \drho_i^{(\rankvar)}(\dtheta_i^{(\rankvar)})}{\vol([0,a_i\pi/N))} - \rho_i^{(\rankvar)}(\dtheta_i^{(\rankvar)}) \right|\leq  \frac{8\pi^2 \rankvar^2}{N} \]
Using again the Lipschitz constant, we therefore deduce 
\[ \left|\frac{ \drho_i^{(\rankvar)}(\dtheta_i^{(\rankvar)})}{\vol([0,a_i\pi/N))} - \rho_i^{(\rankvar)}(\theta_i^{(\rankvar)} + \dtheta_i^{(\rankvar)}) \right|\leq  \frac{16\pi^2 \rankvar^2}{N}. \]
Plugging this into \Cref{eq:differenceeqrho}, we obtain  a bound of 
\[  \sum_{\dtheta^{(\rankvar)} \in \dTheta^{(\rankvar)}}\int_{\theta^{(\rankvar)} \in  F^{(\rankvar)}} \left | \rho^{(\rankvar)}(\theta^{(\rankvar)}) -  \frac{ \drho^{(\rankvar)}(\dtheta^{(\rankvar)})}{\vol(F^{(\rankvar)})} \right| d\theta  \leq \frac{16 \pi^2 \rankvar^3}{N}. \]
\end{proof}

\newcommand{\dAng}{\ddot{\Ang}}

\subsubsection{The continuity error}
\begin{lemma} \label{lemma:continuityang} Let $\mathcal{A}_{\vec{\theta}}$ (for $\vec{\theta} \in \Ang$) be the output distribution of \Cref{alg:canonical} on input
$g \cdot U_{\vec{\theta}} \cdot g'$ for fixed $g,g' \in \GL_r(K_\R)$. Let $N \in \Z_{>0}$ be the discretization parameter.

Then the continuity error satisfies 
\[ \mathcal{C}(\ddistr_{\Ang} ,\mathcal{A}) \leq N^{-1/4} \cdot O( n^5  \cd(g)^{1/2} \cdot \sqrt[4]{\log(1/\eps_0}).  \]
\end{lemma}
\begin{proof} We have, by \Cref{lemma:separationoferrors}, writing $\dAng$ for the support of $\ddistr_{\Ang}$, 
\[ \mathcal{C}(\ddistr_{\Ang} ,\mathcal{A}) \leq \max_{\dtheta \in \dAng} \max_{\theta \in C_{\dtheta}} \| \mathcal{A}_{\vec{\theta}} - \mathcal{A}_{\vec{\dtheta}}\|_1. \]
Since we may omit the tail space $T$, we can assume, by \Cref{lemma:taulipschitz}, that $\|U_{\theta} - U_{\vartheta}\|_2 \leq 2 \pi^2 r^3 d \cdot N^{1/2} \| \theta - \vartheta \|$. Since $\|\dtheta - \theta\| \leq \frac{4 \pi}{N}$, we can deduce, by \Cref{lemma:holdercontinuous}, that, writing $L = 92n^3 \sqrt[4]{\log(1/\eps_0)}$,
\begin{align*} \| \mathcal{A}_{\vec{\theta}} - \mathcal{A}_{\vec{\dtheta}}\|_1 & \leq L \|gU_{\theta}g' (g U_{\dtheta} g')^{-1} - I\|^{1/2} \leq L \cdot \cd(g)^{1/2} \cdot \|U_{\theta} - U_{\dtheta}\|^{1/2}  \\ 
& \leq   L \cdot \cd(g)^{1/2} \cdot \left( 2 \pi^2 r^3 d \cdot N^{1/2} \| \theta - \vartheta \| \right)^{1/2} \\ 
& \leq L \cdot \cd(g)^{1/2} \cdot 2 \pi r^{3/2} d^{1/2} N^{1/4} \|\theta - \vartheta\|^{1/2} \\ & 
\leq N^{-1/4} \cdot  O( n^5  \cd(g)^{1/2} \cdot \sqrt[4]{\log(1/\eps_0)}). \end{align*}
\end{proof}

\begin{lemma} \label{lemma:taulipschitz} Let $\vec{\theta} = (\theta_1, \ldots,\theta_m) \in \Ang$ satisfy $\vec{\theta}_i \geq 2 \cdot N^{-1/2}$ for all $i \in \{1,\ldots,m\}$.  Then, for any $\vartheta \in \Ang$,
 \[\| U_{\theta} U_{\vartheta}^{-1} - I \|_2 =  \| U_{\theta} - U_{\vartheta} \|_2  \leq 2 \pi^2 r^3d \cdot N^{1/2} \|\theta - \vartheta\|, \]
 where $U$ is defined as in \Cref{lemma:uniformSU}.
\end{lemma}

\begin{proof} 
We prove this first for the space $\Ang^{\nu}$, after which it, by the triangle inequality, can be shown for the whole space $\Ang$ as well. We write $\tau = 2N^{-1/2}$.

Write $\Ang^{\nu} = \Theta^{(r)} \times \cdots \times \Theta^{(1)}$. We have the maps
\[ \Theta^{(r)} \times \cdots \times \Theta^{(1)} \xrightarrow{f} S^r \times \cdots \times S^{1} \xrightarrow{g} \SU_r(\R), \]
and we write $U_\theta = gf(\theta) \in \SU_r(\R)$ for $\theta \in \Theta := \Theta^{(r)} \times \cdots \times \Theta^{(1)}$. 

Clearly, $U_\theta = U_{\theta^{(r)}} \cdots U_{\theta^{(1)}}$, where $\theta^{(j)} \in \Theta^{(j)}$, and where $U_{\theta^{(j)}}$ is described
by the Householder transformation that sends $y_j := f_j(\theta^{(j)}) \in S^j$ to $e_j$. This Householder transformation is defined, writing $w = \frac{y_j - e_j}{\|y_j - e_j\|}$ by the rule $U_{\theta^{(j)}} := I - 2 w w^\top$.

We have, since $U_\theta,U_{\vartheta} \in \SU_r(\R)$ are unitary, that  $\|U_\theta U_\vartheta^{-1} - I\|_2 = \|U_\theta - U_{\vartheta}\|_2$; indeed,
\begin{align*} \|U_\theta - U_{\vartheta}\|_2  &=  \|(U_\theta U_\vartheta^{-1} - I)U_\vartheta\|_2 \leq \|U_\theta U_\vartheta^{-1} - I\|_2 \|U_\vartheta\|_2
\leq \|U_\theta U_{\vartheta}^{-1} - I \|_2
 \\ & = \|(U_\theta - U_{\vartheta}) U_{\vartheta}^{-1}\|_2 \leq  \|U_\theta - U_{\vartheta}\|_2 \|U_{\vartheta}^{-1}\|_2 = \|U_\theta - U_{\vartheta}\|_2 . \end{align*}
Hence, by repeated application of the trick $ab -a'b' = (a-a')b + (b-b')a'$, and the fact that the two-norm of a unitary matrix equals one, we find,
\[ \|U_\theta U_\vartheta^{-1} - I\|_2 = \|U_\theta - U_{\vartheta}\|_2 =  \| U_{\theta^{(r)}} \cdots U_{\theta^{(1)}} -  U_{\vartheta^{(r)}} \cdots U_{\vartheta^{(1)}}\|_2 \leq \sum_{j = 1}^r \| U_{\theta^{(j)}} - U_{\vartheta^{(j)}} \|_2 \]
Now, writing  $U_{\theta^{(j)}} =  I - 2 v v^\top$ and $U_{\vartheta^{(j)}} =  I - 2 w w^\top$ (with unit vectors $v = \frac{y_j - e_j}{\|y_j - e_j\|},w = \frac{y'_j - e_j}{\|y'_j - e_j\|}$, with $y_j = f_j(\theta^{(j)})$ and $y'_j = f_j(\vartheta^{(j)})$) we have, using that $\| A^\top\|_2 = \|A \|_2$, 
\begin{align*}  \| U_{\theta^{(j)}} - U_{\vartheta^{(j)}} \|_2 & = 2\| w w^\top - v v^\top\|_2 = 2 \| w(w - v)^\top + [v(w-v)^\top]^\top   \|_2 \\ 
& \leq 2 \| w(w-v)^\top\|_2 + 2 \|v(w-v)^\top\| \leq 4 \|w-v\|.  \end{align*} 
Assuming that $\|y_j - e_j\| > \tau$ or $\|y_j' - e_j\| > \tau$ and writing $d = y_j' - y_j$, and using the reverse triangle inequality $| \|a\| - \|b\| | \leq \|a - b\|$, we have that
\begin{align*} 4 \|w-v\| & = 4 \left \| \frac{(y_j - e_j)\|y_j - e_j + d\| - (y_j - e_j + d)\|y_j - e_j\|}{\|y_j - e_j\|\|y'_j - e_j\|}  \right  \|  \\
 & = 4  \frac{\left \|(y_j - e_j)(\|y_j - e_j + d\| - \|y_j - e_j\|) - d \|y_j - e_j\|  \right  \|}{\|y_j - e_j\|\|y'_j - e_j\|}  \\ 
& \leq 4  \frac{ \|y_j - e_j\| \|d\| + \|d\|  \|y_j - e_j\|}{\|y_j - e_j\|\|y'_j - e_j\|}  \leq 4 \frac{\|y_j - y'_j\|}{\|y'_j - e_j\|} \leq 4 \tau^{-1} \|y_j - y'_j\| .
\end{align*}
Since this inequality is symmetric in $y_j$ and $y'_j$, we can just assume $\|y_j - e_j\| > \tau$.
Since the function $y_j = f_j(\theta^{(j)})$ is $\pi^2 r^2$-Lipschitz, we immediately deduce, 
\[  \| U_{\theta^{(j)}} - U_{\vartheta^{(j)}} \|_2 \leq 4 \tau^{-1} \|y_j - y'_j\| \leq \frac{4 \pi^2 r^2}{\tau} \|\theta^{(j)} - \vartheta^{(j)}\|, \]
provided that $y_j = f_j(\theta^{(j)})$ satisfies $\|y_j - e_j \| > \tau$. 
Note that, for the map $f_j: \Theta^{(j)} \rightarrow S^j$, we have $x_j = f_j( \theta^{(j)}_j) = \cos(\theta^{(j)}_j)$; hence, 
for sure, if $\theta^{(j)}_j > 2\sqrt{\tau}$, we have $\|y_j - e_j\| > 1 - \cos(2\sqrt{\tau}) > \tau$ for $\tau < 1$.

So, 
\[\| U_{\theta} U_{\vartheta}^{-1} - I \|_2 =  \| U_{\theta} - U_{\vartheta} \|_2  \leq \frac{4 \pi^2 r^3}{\tau} \|\theta - \vartheta\|, \]
for $\theta = (\theta^{(r)}, \ldots,\theta^{(1)})$ with $\theta^{(j)}_j > 2\sqrt{\tau}$ for all $j \in \{1,\ldots,r\}$.

For general $\theta \in \Ang$ (instead of just $\Ang^\nu$) we arrive at the similar claim, but with an additional factor $d$, which finishes the proof.
\end{proof}

\subsubsection{Run time}
\begin{lemma} \label{lemma:efficientang} There exists an algorithm sampling from $\ddot{\Ang}$ within time $\poly(\log N , dr)$. 
\end{lemma}
\begin{proof} By \Cref{lemma:efficientangle} one can sample a single angle component (note that the components are independent) from $\ddot{\Ang}$ within polynomial time in $\log N$ and $r$. Hence, since $\Ang$ has at most $2dr^2$ angular components, we obtain at the claim of this lemma.
\end{proof}

\subsubsection{Concluding all errors}
\begin{lemma} \label{lemma:mainSU} Let $\mathcal{A}_\theta$ (for $\theta \in \Ang$) be the output distribution of \Cref{alg:canonical} on input $g \cdot U_{\theta} \cdot g'$ for fixed $g,g' \in \GL_r(K_\R)$. Let $N \in \Z_{>0}$ be the discretization parameter satisfying $N \geq 8 \pi^2 r^2 + 2$

Then, 
\begin{align*} 
& \| \underset{x \from \distr_{\Ang}}{\mathbb{E}}[\mathcal{A}_x] - \underset{\ddx \from \ddistr_{\Ang}}{\mathbb{E}} [\mathcal{A}_{\ddx}]   \|   &\leq  N^{-1/4} n^5 O( \cd(g)^{1/2} \cdot \sqrt[4]{\log(1/\eps_0)} ).
\end{align*}
\end{lemma}
\begin{proof} This follows from \Cref{eq:tailang,lemma:discretizationang,lemma:continuityang} and simplifying the expression.
\end{proof}

\section{Conclusion}\label{sec:conclusion}

We finally piece together all ingredients and prove Theorem~\ref{thm:main}, our main theorem on the worst-case to average-case reduction of SIVP.

Recall from Section \ref{sec:module-lats} the invariant measure $\mu$ on the space $X_r = X_r(K)$ of module lattices of rank $r$ over the number field $K$.
Recalling the definition of $\Round$ from Section \ref{sec:rounding}, define the average-case distribution as
\begin{displaymath}
    \distr = \Round(\mu_{\mathrm{cut}})
\end{displaymath}
where $\mu_{\mathrm{cut}}$ is the Haar-measure induced distribution on $X_r$, restricted (hence the name ``cut'') to module lattices that are $\alpha$-balanced, with $\alpha = O(B \cdot d \cdot c_K) = \exp\left( O(d \log d + \log \abs{\Delta_K})\right)$, where $B$ is as in line \lineref{line:defB} in \Cref{alg:reduction} and where $c_K = 1+ \frac{\Gamma_K \sqrt{r \cdot d}}{2}$ with $\Gamma_K \leq |\Delta_K|^{1/d}$. 

Note that, by \Cref{thm:randbalanced}, the distributions $\mu$ and $\mu_{\mathrm{cut}}$ are close in statistical distance, and that (by \Cref{prop:rounding-algo}) $\Round$ rounds its input lattice to a very close lattice. This gives a justification as to why  $\Round(\mu_{\mathrm{cut}})$ can be seen as a sound discrete average-case distribution.

\begin{algorithm}[ht]
    \caption{Reducing a fixed $(4d \cdot c_K)$-balanced instance of SIVP over module lattices to a random instance of SIVP over module lattices.}
    \label{alg:reduction}
    \begin{algorithmic}[1]
    	\REQUIRE  ~\\  \vspace{-.4cm}
    	\begin{itemize}
    	 \item A pseudo-basis $(\mathbf{B},\mI)$ of a rank $r$ module lattice $L_0$,\vspace{-.4cm}
    	 \item An oracle $\mathcal{O}$ solving $\gamma'$-SIVP for $\Round(\mu_{\mathrm{cut}})$ with probability $p = 2^{-o(n)}$. \vspace{-.2cm}
    	\end{itemize}
     	\ENSURE With probability $1 - 2^{-\Omega(n)}$, a solution to $\gamma$-SIVP for $L_0$, with $\gamma = \poly_r(|\Delta_K|^{1/d},d) \cdot \gamma'$ 
        \STATE Put $B = \exp\left( C_r (d \log d + \log \abs{\Delta_K}) \right)$ for a large enough constant $C_r > 0$ depending on $r$, $t = 1$ and $\sigma = 1/d^2$. \label{line:defB}
        \STATE Instantiate the discretization parameter $N$ as in \Cref{eq:instantiateN}: \vspace{-.2cm}
        \[ \log(N) =  O(n^2 \log n + n \cdot \size(\mathbf{B}) + n^2\cdot \log(B)  +n \log |\Delta_K| + \log(1/\eps_0) + \log(1/\eps)) . \]  \vspace{-.6cm}
        \label{line:defN}
        \REPEAT
    	\STATE Compute a module lattice $L_1 = g \cdot L_0$ using \Cref{alg:sampleftilde}  on input $L_0$ with parameters $t$ and $\sigma$ as above, and where every distribution occurring is discretized as in \Cref{sec:discretization} with discretization parameter $N$, with $\eps_0 := 2^{-\Theta(n)}$ and $\eps = 2^{-\Theta(n)}$ as in \Cref{proposition:maindiscretization}. \label{line:L1}
    	\STATE Sample uniformly random $\p$ from the set $\mathcal{P}$ of all prime ideals with norm at most $B$ (using \cite[Lemma~2.2]{C:BDPW20}) and take a random sublattice $L_2$ of $L_1$ satisfying $L_2/L_1 \cong \ZK/\p$ using \Cref{alg:randomsublattice}. \label{line:L2}
    	\STATE Sample $L_3 \from \Round(L_2)$, where $\Round$ is the algorithm given in \Cref{prop:rounding-algo}, with error parameter $\eps_0$ as above and balancedness parameter $\alpha = O(B dc_K)$. \label{line:L3}
    	\STATE Apply the oracle $\mathcal{O}$ on $L_3$. \label{line:oracle}
    	\UNTIL the output of $\mathcal{O}$ is of the form $\{ v^{(3)}_1,\ldots,v^{(3)}_n \}$ and satisfies $\|v^{(3)}_j\| \leq \gamma' \cdot O(n \cdot |\Delta_K|^{\frac{1}{2d}}) \cdot  \det(L_3)^{\frac{1}{n}}$ for all $j$. \label{line:until}
    	\STATE Use the transformation $Y$ of \Cref{prop:rounding-algo}(iii) to compute $v^{(2)}_j := Y \cdot v^{(3)}_j$ for all $j$. 
    	\label{line:unround}
    	\STATE Put $v^{(1)}_j := v^{(2)}_j \in L_2 \subseteq L_1$ for all $j$. \label{line:unsparsify}
    	\STATE Put $v^{(0)}_j := g^{-1} v^{(1)}_j$ for all $j$, to get $\{ v^{(0)}_1, \ldots,v^{(0)}_n \} \subset L_0$ with $g$ as in Line \ref{line:L1}. \label{line:undisturb}
        \RETURN $( v^{(0)}_1, \ldots,v^{(0)}_r)$.
    \end{algorithmic}
\end{algorithm}

\begin{proof}[Proof of \Cref{thm:main}]
By \Cref{theorem:stitchcuspflaretobulk}, it is sufficient to reduce $\gamma$-SIVP for $(4d \cdot c_K)$-balanced lattices to $\gamma'$-SIVP for lattices sampled from $\distr$. 
We use here that $c_K^{r-1} \cdot (1+d/2^{(rd+1)/2})^{r-1}$ from \Cref{theorem:stitchcuspflaretobulk} is $\poly_r(|\Delta_K|^{1/d},d)$, since $\Gamma_K \leq  |\Delta_K|^{1/d}$ (see \Cref{lemma:rank1prop}).

Given a $(4d \cdot c_K)$-balanced module lattice $L_0$, we randomize it using the framework from \Cref{sec:quantitative-equidistribution-bigsec}, but with discretized underlying distributions, as in \Cref{sec:discretization}. After that, we apply $\Round$ (from \Cref{prop:rounding-algo}) and feed the randomized and rounded module lattice to the oracle solving $\gamma'$-SIVP for $\Round(\mu_{\mathrm{cut}})$. This yields an output for SIVP for this rounded and randomized module lattice. By undoing the ``rounding'' and the ``randomization'', we get an SIVP solution for the original lattice $L_0$. This process is described in a precise manner in \Cref{alg:reduction}.

For this reduction in \Cref{alg:reduction} to be sound, 
we need to prove three things. One, 
we need to show that the distribution of $L_3$ in line \lineref{line:L3} is
statistically $o(p)$-close to $\Round(\mu_{\mathrm{cut}})$ in order for the
oracle in line \lineref{line:oracle} to succeed with probability $\Omega(p)$.
Two, we need to upper bound the loss in quality of the output SIVP solution caused by the randomization (and de-randomization). Three, we need to bound the expected number of queries to the oracle, and show that every step can be performed in polynomial time in the size of the input (where we assume $r = O(1)$).   \\ 
\noindent
\paragraph{(1) Statistical closeness.}
Because the final oracle solving the random instance has success probability at least $2^{-o(n)}$, it suffices to allow a statistical error of $2^{-\Omega(n)}$. In this proof, we will instantiate with parameters aiming for a statistical error of $2^{-\Theta(n)} = 2^{-\Theta(d)}$ as $r = O(1)$. 
Note that most of the ingredients of the proof can handle arbitrary errors $\varepsilon > 0$, consuming additional time $\log(1/\varepsilon)$. 

Let $z = (\mathbf{B}, \mI)$ be the input pseudo-basis for $L_0$, and we use the notation $f_z$, $\initial_z$ and $\distr_z$ as in Section \ref{sec:discretization} (see the discussion before Proposition \ref{proposition:maindiscretization}).

Note that, by construction, $L_3$ in line \lineref{line:L3} is distributed according to $\Round(T_\mathcal{P}(\distr_z))$. 
Our strategy is to bound the distance $\|\Round(T_\mathcal{P}(\distr_z)) - \Round(T_\mathcal{P}(\initial_z))\|_1$, 
as well as $\|T_\mathcal{P}\initial_z - \mu\|_1$ and $\|\mu - \mu_{\mathrm{cut}}\|_1$ by $2^{-\Omega(n)}$. 
Assuming this, by the data processing inequality (\Cref{theorem:dataprocessinginequality}) applied to $\Round$ and the triangle inequality, this reasonably yields 
\[  \| \Round(T_\mathcal{P}(\distr_{z})) - \Round(\mu_\mathrm{cut})\|_1 \leq 2^{-\Omega(n)}.  \]

First, $\|\Round(T_\mathcal{P}(\distr_z)) - \Round(T_\mathcal{P}(\initial_z))\|_1$ is bounded by $\eps_0 + \eps = 2^{-\Omega(n)}$, by Proposition~\ref{proposition:maindiscretization} and our choices of $\eps_0$ and $\eps$ in the algorithm.
Next, $\|\mu - \mu_{\mathrm{cut}}\|_1$ satisfies the same bound by Theorem~\ref{thm:randbalanced}.
Indeed, a $\mu$-random module lattice $L$ satisfies 
\begin{displaymath}
    \lambda_r^K(L) / \lambda_1^K(L) \leq \lambda_n(L) / \lambda_1(L) \ll n \abs{\Delta_K}^{1/d}
\end{displaymath} 
with probability at least $1 - 2^{\Omega(n \log r)}$, and on the other hand $\alpha = \Omega(n \cdot \abs{\Delta_K}^{1/d})$.

Next, we bound the statistical distance between $T_\mathcal{P} \initial_z \mu_{\Riem}$ and $\mu$, which is
\begin{displaymath}
    \frac{1}{2} \int_{X_r} |T_\mathcal{P} \initial_z \cdot \mu_{\Riem}(X_r) - 1| \, d\mu = \frac12 \norm{T_\mathcal{P} \initial_z - \mu_{\Riem}(X_r)^{-1} \triv_{X_r}}_1,
\end{displaymath}
where the $L^1$-norm is now with respect to $\mu_{\Riem}$.
Applying Cauchy--Schwarz, we have
\begin{align}
    \norm{T_\mathcal{P} \initial_z - \mu_{\Riem}(X_r)^{-1} \triv_{X_r}}_1 
    &\leq \sqrt{\mu_{\Riem}(X_r)} \cdot \norm{T_\mathcal{P} \initial_z - \mu_{\Riem}(X_r)^{-1} \triv_{X_r}}_2 \nonumber \\
    &\leq \exp(C(d + \log \abs{\Delta_K})) \norm{T_\mathcal{P} \initial_z - \mu_{\Riem}(X_r)^{-1} \triv_{X_r}}_2, \label{eq:cauchy}
\end{align}
for some constant $C > 0$ depending on $r$, using Lemma \ref{lem:spacevol}.

To finally apply Theorem \ref{thm:main-equidistro}, we rework its statement using the assumption that $r = O(1)$ and using that $\log x = O(x^\delta)$ for any $\delta > 0$.
Let $\kappa_d = \kappa \sigma = \kappa/d^2$ to simplify notation and note that $\kappa \geq \sqrt{d}/\sigma = d^{5/2}$ in the theorem (we make a valid choice of $\kappa$ below), so that $\kappa_d \geq \sqrt{d} \gg 1$.
For the first term, we trivially bound $\unitrk \leq d$ and $\unitrk (\log \unitrk)^3 = o(d^2)$, so we may assume that $\max(\unitrk (\log \unitrk)^3, 1/\sigma) = 1/\sigma = d^2$. 
We also bound
\begin{displaymath}
    C_1 \ll B^{-1/4} (d \log\kappa_d + \log \abs{\Delta_K}).
\end{displaymath}
For the second, denote by $\alpha(z)$ the balancedness of $L_0$.
Thus, $\alpha(z) \ll d c_K \ll d^{3/2} \abs{\Delta_K}^{1/d}$ (using \Cref{lemma:rank1prop} again), so that 
\begin{displaymath}
    C_2 \leq \exp(O(d \log d + \log \abs{\Delta_K})).
\end{displaymath}

Let
\begin{displaymath}
    \kappa_d^2 = \kappa^2 / d = \max\left(d^{5/2}, C(d + \log \abs{\Delta_K}) + d \log d + d \right).
\end{displaymath}
Since $\kappa_d$ is polynomial in $d$ and $\log \abs{\Delta_K}$, we have that $d \log\kappa_d + \log \abs{\Delta_K} = O(d \log d + \log \abs{\Delta_K})$.
All in all, applying again some trivial bounds to simplify expressions, we have
\begin{displaymath}
    \norm{T_\mathcal{P} \initial_z - \mu_{\Riem}(X_r)^{-1} \triv_{X_r}}_2^2 
    \ll  \exp(2d \log d - 2 \kappa_d^2) + B^{-1/2} \exp(C' \cdot (d \log d + \log \abs{\Delta_K}))
\end{displaymath}
for some constant $C' > 0$ depending on $r$.

We plug this last bound into \eqref{eq:cauchy} and use simplifying bounds as above, such as $\log x \leq x$, and that $\sqrt{x+y} \leq \sqrt x + \sqrt y$ for $x, y > 0$ to arrive at
\begin{displaymath}
    \norm{T_\mathcal{P} \initial_z - \mu_{\Riem}(X_r)^{-1} \triv_{X_r}}_1 \ll e^{-d} = 2^{-\Omega(n)}.
\end{displaymath}
We use here that our choice of $\kappa_d$ implies that
\begin{displaymath}
    \exp(C(d + \log \abs{\Delta_K})) \cdot \exp(d\log d - \kappa_d^2) \leq e^{-d}
\end{displaymath}
and that
\begin{displaymath}
    B^{1/4} \geq \exp\left(C'(d \log d + \log \abs{\Delta_K})/2 + C(d + \log \abs{\Delta_K}) + d\right).
\end{displaymath}
This is possible with a minimal
\begin{displaymath}
    B = \exp\left( O_r(d \log d + \log \abs{\Delta_K}) \right),
\end{displaymath}
where the implied constant depends on $r$.

\paragraph{(2) Bound on the loss in SIVP-quality.} 
The processing of the SIVP-vectors happens in lines \lineref{line:until}, \lineref{line:unround}, \lineref{line:unsparsify} and \lineref{line:undisturb}. 

We prove in part (1) of this proof that $L_3$ follows a distribution that is $2^{-\Omega(n)}$-close to $\Round(\mu_{\mathrm{cut}})$. Since module lattices $L$ sampled from $\mu$ satisfy $\lambda_n(L) \leq O(n |\Delta_K|^{\frac{1}{2d}} ) \cdot \det(L)^{\frac{1}{n}}$ by \Cref{thm:randbalanced} with probability at least $1 - 2^{-\Omega(n \log r)}$, surely module lattices $L$ sampled from $\mu_{\mathrm{cut}}$ satisfy the same inequality (even with a higher probability).

So, with probability $1 - 2^{-\Omega(n \log r)}$, we have that $\det(L_3)^{\frac{1}{n}} \leq \lambda_n(L_3) \leq O(n \cdot |\Delta_K|^{\frac{1}{2n}} )\cdot \det(L_3)^{\frac{1}{n}}$, and hence line \lineref{line:until} suffices to check whether the oracle is successful, though it might allow for an extra slack of $O(n \cdot |\Delta_K|^{\frac{1}{2n}}) = \poly_r(d, |\Delta_K|^{1/d})$, which is acceptable for our use-case. Therefore, since $p = 2^{-o(n)}$, we may assume with probability $1 - 2^{-\Omega(n)}$ that after line \lineref{line:until}, the solution $\{ v_1^{(3)},\ldots,v_n^{(3)} \}$ satisfies $\|v_j^{(3)}\| \leq \gamma' \cdot  \poly(d, |\Delta_K|^{1/d}) \cdot \lambda_n(L_3)$. Additionally, we can assume that $\lambda_n(L_3) \leq O(n |\Delta_K|^{\frac{1}{2d}} ) \cdot \det(L_3)^{\frac{1}{n}}$.

By \Cref{prop:rounding-algo}(iii), we see that $\det(L_3)^{\frac{1}{n}} \leq 2 \det(L_2)^{\frac{1}{n}}$ and that applying $Y$ only changes vector lengths by a factor $O(1)$.  
Therefore, $\{v_1^{(1)}, \ldots,v_n^{(1)}\}$ satisfy, for all $j \in \{1,\ldots,n\}$, 
\begin{align*} \|v_j^{(1)} \| & \leq  O(n |\Delta_K|^{\frac{1}{2d}} ) \cdot \gamma' \cdot \det(L_3)^{\frac{1}{n}}  \leq O(n |\Delta_K|^{\frac{1}{2d}} ) \cdot \gamma' \cdot \det(L_2)^{\frac{1}{n}} \\
& \leq O(n |\Delta_K|^{\frac{1}{2d}} ) \cdot B^{1/n} \cdot \gamma' \cdot \det(L_1)^{\frac{1}{n}}  \\ 
& \leq  \poly_r(d, |\Delta_K|^{\frac{1}{d}} ) \cdot \gamma' \cdot \det(L_1)^{\frac{1}{n}}  
\end{align*}
where the last inequality holds by our choice of $B$ in line \lineref{line:defB} of \Cref{alg:reduction}. Since $g$ has conditioning number $e^{2n^2 \sigma + 2t}$ (see the proof of \Cref{proposition:maindiscretization}), by the very same arguments as those of \Cref{prop:rounding-algo}(iii), multiplying by $g^{-1}$ changes the determinant and the lengths of vectors by at most $O(1)$. Hence, we have, for all $j \in \{1,\ldots,n\}$,  
\[ \|v_j^{(0)} \| \leq  \poly_r(d, |\Delta_K|^{\frac{1}{d}} ) \cdot \gamma' \cdot \det(L_0)^{\frac{1}{n}} \leq  \poly_r(d, |\Delta_K|^{\frac{1}{d}} ) \cdot \gamma' \cdot \lambda_n(L_0), \] 
where the last inequality holds by the fact that 
$\det(L_0)^{1/n} \leq (\prod_{j=1}^n \lambda_j(L_0))^{1/n} \leq \lambda_n(L_0)$. So, indeed, the algorithm solves $\gamma$-SIVP for $L_0$ with  $\gamma =  \poly_r(d, |\Delta_K|^{\frac{1}{d}} ) \cdot \gamma'$.

\paragraph{(3) Number of queries and efficiency.}
Line \lineref{line:L1} uses \Cref{alg:sampleftilde} with discretization, which
can be computed efficiently according to \Cref{prop:rounding-algo}. Line
\lineref{line:L2} uses the algorithm described in \cite[Lemma~2.2]{C:BDPW20} as
well as \Cref{alg:randomsublattice}, which run both within polynomial time in
the size of their input (Lemma~\ref{lem:alg:randomsublattice}). 
Line \lineref{line:L3} uses \Cref{alg:canonical}, which runs within polynomial in the size of the input. But for this algorithm to be applicable on $L_2$, we need to show that $L_2$ is $\alpha$-balanced for $\alpha = O(B \cdot d \cdot c_K)$. By similar arguments as earlier in the proof, $g$ does not change lengths of vectors much, and hence, since $L_0$ is $(4d \cdot c_K)$-balanced, we can conclude that $L_1$ is $O(d \cdot c_K)$-balanced. As $L_2$ is a sub-lattice of $L_1$ of index at most $2B$, we deduce, by \Cref{lemma:qsparsificationqbalanced}, that $L_2$ is $O(B \cdot d \cdot c_K)$-balanced, which is what was required to show. 
Since lines \lineref{line:unround}, \lineref{line:unsparsify} and \lineref{line:undisturb} are mere linear operations applied to vectors, these lines all run in polynomial time in the size of the input.

For the expected number of queries, note that the distribution of $L_3$ is $o(p)$-statistically close to $\Round(\mu_{\mathrm{cut}})$, hence we may assume that the oracle $\mathcal{O}$ gives a sound output with probability $O(p)$.
So the expected number of queries is $O(p^{-1})$. For the total reduction (which includes the cusp and the flare part) the number of queries is multiplicatively increased by $\poly_r(\log|\Delta_K|)$, which yields the total expected number of queries.
\end{proof}

\appendix

\section{Appendix} 
The lemmas in \Cref{sec:inversebasis,subsec:gaussianevenly,sec:latticepointsset} are almost literally from \cite{BPW25}, and are stated here for sake of self-containedness.

\subsection{On the matrix norm of an inverse basis}
\label{sec:inversebasis}
\begin{lemma} \label{lemma:wellconditioned} Suppose $\mathbf{B} = (\mathbf{a}_1,\ldots,\mathbf{a}_n) \in \R^{n \times n}$ is a square real matrix and a basis of a lattice $L \subset \R^n$.
Then
\[ \| \mathbf{B}^{-1} \|_2 \leq n^{n/2+1} \cdot \lambda_1(L)^{-1} \cdot \left(\prod_{j =1}^n \frac{\|\mathbf{b}_j\|}{\lambda_j(L)}\right),\]
where $\mathbf{B} = (\mathbf{b}_1,\ldots,\mathbf{b}_n)$.
\end{lemma}

\begin{proof} 
For $j \in \{1,\dots,n\}$, put $C_j = \frac{\|\mathbf{b}_j\|}{\lambda_j(\Lambda)} \geq 1$.
We have that $\mathbf{B}^{-1} = \frac{1}{\det \mathbf{B}} \mbox{adj}(\mathbf{B})$. Also, $\mbox{adj}(\mathbf{B})_{ij}$ is
defined by the determinant of the minor of $\mathbf{B}$ where the $i$-th row and $j$-th column are deleted. By the
Hadamard bound and subsequently by Minkowski's second theorem (see, e.g.,~\cite[Theorem 1.5]{MG02}),
\begin{align*} |\mbox{adj}(\mathbf{B})_{ij}| &\leq \prod_{k \neq i} \|\mathbf{b}_k\| = \prod_{k \neq i} C_k \cdot \lambda_k(L) \\
 & \leq n^{n/2}  \cdot (\prod_{k \neq i} C_k) \cdot  \det(L)/\lambda_i(L)  \leq
  n^{n/2} (\prod_{k=1}^n C_k) \cdot \det(L)/\lambda_1(L).
\end{align*}
Observing that $\det(\mathbf{B}) = \det(\Lambda)$, we obtain
\[ \|\mathbf{B}^{-1} \|_2 \leq \|\mathbf{B}^{-1}\|_F \leq \frac{1}{\det(\mathbf{B})} \cdot n \cdot \max_{ij} |\mbox{adj}(\mathbf{B})_{ij}| \leq n^{n/2+1} \cdot (\prod_{k =1}^n C_k)/\lambda_1(L). \]
\end{proof}

\subsection{On the weight of discrete Gaussians on strict sublattices}
 \label{subsec:gaussianevenly}
We show in the following two lemmas that the discrete Gaussian distribution over a lattice with an arbitrary center point has no heavy weight on any strict sublattice. This fact is used that we can compute a sample from a discrete Gaussian conditioned on the event that it is linearly independent to earlier samples.

\begin{lemma} \label{lemma:regevtechnique}
Writing $\gaussian_\sd(x) := e^{-\pi \|x\|^2/\sd^2}$, we have that for
 any lattice $\Lambda \subseteq V$ (where $V$ is a Euclidean space), any $\sd>0$ and any $t,w \in V$,
we have
\[ \gaussian_\sd(\Lambda + t + w) + \gaussian_\sd(\Lambda + t - w) \geq 2\gaussian_\sd(w)\gaussian_\sd(\Lambda + t), \]
where $\gaussian_\sd(\Lambda+t) = \sum_{\ell \in \Lambda} \gaussian_\sd(\ell + t)$.
\end{lemma}
\begin{proof} This lemma is a simple generalization of \cite[Claim 2.10]{SODA:HavReg14}, and we follow the same strategy:
\begin{align*}
\gaussian_\sd(\Lambda+t + w) + \gaussian_\sd(\Lambda+t - w) 
&=  \sum_{x \in \Lambda+t}\left(e^{-\pi\|x+w\|^2/\sd^2} + e^{-\pi\|x-w\|^2/\sd^2}\right)\\
&= 2e^{-\pi\|w\|^2/\sd^2}\sum_{x \in \Lambda+t}\left(e^{-\pi\|x\|^2/\sd^2}\cosh(2\pi\langle x,w\rangle/\sd^2)\right)\\
&\geq 2\gaussian_\sd(w)\gaussian_\sd(\Lambda + t),
\end{align*}
where the last inequality follows from $\cosh(\alpha) \geq 1$ for any real $\alpha$.
\end{proof}

\begin{lemma}
\label{lemma:Gaussian-evenly-distributed}
Let $\Lambda \subseteq V$ be a lattice and $V$ an Euclidean space, $t \in V$ and $\sd > c\cdot \sqrt{n} \cdot \lambda_n(\Lambda)$ for some $c>0$.
Then, for any strict sublattice $\Lambda' \subsetneq \Lambda$
\[ \Pr_{\substack{x \gets \Gaussian_{\Lambda+t, \sd}}}[x \in \Lambda' + t] = \frac{ \gaussian_\sd(\Lambda' + t)}{\gaussian_\sd(\Lambda + t)} \leq  \frac{1}{1+e^{-\pi c^{-2}}}.\]
\end{lemma}
\begin{proof}

Let $\Lambda' \subsetneq \Lambda$ be a sub-lattice of $\Lambda$ and
let $w \in \Lambda \setminus \Lambda'$. Then, by \Cref{lemma:regevtechnique},
\begin{align*}\gaussian_\sd(\Lambda+t) &\geq  \gaussian_\sd(\Lambda'+t) + \frac{\gaussian_\sd(\Lambda'+t+w) + \gaussian_\sd(\Lambda'+t-w)}{2} \\
&\geq (1 + \gaussian_\sd(w))\gaussian_\sd(\Lambda' + t).
\end{align*}
Writing $\Gaussian_{\Lambda+t, \sd}$ for the Gaussian distribution on $\Lambda +t$ with parameter $\sd$, we have,
\begin{align*}
\Pr_{\substack{x \gets \Gaussian_{\Lambda+t, \sd}}}[x \in \Lambda' + t] = \frac{\gaussian_\sd(\Lambda' + t)}{\gaussian_\sd(\Lambda + t)}
& \leq \frac{1}{1 + \gaussian_\sd(w)}.
\end{align*}
Note that  the set $\{ \ell \in \Lambda  ~|~ \|\ell \| \leq \sqrt{n} \lambda_n(\Lambda) \}$ contains a HKZ-basis of $\Lambda$ \cite{lagarias90:_korkin_zolot_bases_and_succes}. Hence, for any $\Lambda' \subsetneq \Lambda$ there must exist $w \in \Lambda \backslash \Lambda'$ with $\|w \| \leq \sqrt{n} \lambda_n(\Lambda)$.

So there exists $w \in \Lambda \setminus \Lambda'$ such that $\|w\| \leq \sqrt{n} \cdot \lambda_n(\Lambda) < \sd/c$, hence
$\gaussian_\sd(w) \geq \exp\left(-\pi c^{-2}\right)$, proving the lemma.
\end{proof}

\subsection{On the number of lattice points in a convex measurable volume}
\label{sec:latticepointsset}
\Cref{lemma:latticepointsinconvexvolume}, which shares some similarities with \cite[\textsection 4.2]{PKC:PlaPre21}, provides
means to estimate the number of lattice elements in a convex measurable volume. This estimate
is essential in \Cref{sec:discretization} about discretization. To prepare for
the proof of this lemma, we will need some facts on Minkowski sums of sets.

\begin{definition} \label{def:minkowskisum}Let $V$ be a Euclidean vector space. For two sets $X,Y \subseteq V$, we define
the Minkowski sum $X \minksum Y$ as follows.
\[ X \minksum Y = \{ \bx + \by ~|~ \bx \in X, \by \in Y \}.\]
For $c \in \R_{>0}$ we denote by $cX$ the set
\[ cX = \{ c \cdot \bx ~|~ \bx \in X \}. \]
\end{definition}

\begin{lemma} \label{lemma:minkowskisum} Let $V$ be a Euclidean vector space and let $r,s > 0$ and let $X \subseteq V$ be a convex volume.
Then 
\[ (rX) \minksum (sX) = (r + s) X. \]
\end{lemma}
\begin{proof} We start with inclusion to the right. Suppose $\by \in (rX) \minksum (sX)$, i.e., $\by = r\bx + s\bx'$ where $\bx,\bx' \in X$.
Then $\tfrac{\by}{r+s} = \tfrac{r\bx + s\bx'}{r+s} \in X$, since it is a weighted average of two
points in $X$ and $X$ is convex. So $\by \in (r+s) X$. Inclusion to the left holds because $\by \in (r+s)X$ means that $\by = (r+s)\bx = r\bx + s\bx \in (rX) \minksum (sX)$.
\end{proof}
\begin{lemma} \label{lemma:minkowskisymmetric} Let $V$ be a Euclidean vector space, let $r > 0$, let $X,Y \subseteq V$ be sets 
and let $S \subseteq V$ be a symmetric set, i.e., $\bx \in S \Leftrightarrow -\bx \in S$. Then
\[  (X \minksum S) \cap Y \subseteq [X \cap (Y \minksum S)] \minksum S . \]
\end{lemma}
\begin{proof} Suppose $\bx + \bs = \by  \in (X \minksum S) \cap Y$. Then $\bx =\by - \bs \in X \cap (Y \minksum S)$, so $\by = \bx + \bs \in [X \cap (Y \minksum S)]\minksum S$.
\end{proof}

\begin{lemma}
\label{lemma:latticepointsinconvexvolume}%
Let $V$ be a $n$-dimensional Euclidean vector space,
let $\Lambda \subseteq V$ be a full-rank lattice,
let $X \subseteq V$ be a convex measurable volume for which $\voronoi \subseteq cX$ for some $c \in \R_{>0}$, where $\voronoi$ is the (origin-centered)
Voronoi cell of $\Lambda$. Then, for all $\bt,\bt' \in V$ and all $q >2 c$,
\[ |(\Lambda + \bt) \cap q( X + \bt')| \in
[e^{-2nc/q},e^{2nc/q}]
\cdot \frac{q^n \cdot \vol(X)}{\det(\Lambda)},\]
where $q(X+ \bt') = \{ q \cdot (x + \bt') ~|~ x \in X \}$ is the scaling of the (translated) set $X + \bt'$ by $q \in \R_{>0}$.
\end{lemma}

\begin{proof} As $|(\Lambda + \bt) \cap (qX + q\bt')| = |(\Lambda + \bt - q\bt') \cap qX|$, we just assume, without loss of generality, that $\bt' = 0$.
Note that $\voronoi \subseteq cX$, and that $X$ is convex. So, by \Cref{lemma:minkowskisum}, we have
$(qX) \minksum \voronoi \subseteq (qX) \minksum (cX) = (q+c) X$. Similarly, $(q-c) X \minksum \voronoi \subseteq qX$. Therefore
\begin{equation} [(\Lambda + \bt) \cap qX] \minksum \voronoi \subseteq (q+c)X. \label{eq:minklower} \end{equation}
Note that $\voronoi$ is symmetric and $(\Lambda + \bt) \minksum \voronoi = V$, the whole vector space. So, by \Cref{lemma:minkowskisymmetric}  and $(q-c)X \minksum \voronoi \subseteq qX$,
\begin{align}  (q-c)X & = [(\Lambda + \bt) \minksum \voronoi] \cap (q-c)X  \\ & \subseteq  [(\Lambda + \bt) \cap ((q-c)X \minksum \voronoi)] \minksum \voronoi \subseteq  [(\Lambda + \bt) \cap qX] \minksum \voronoi  \label{eq:minkupper}    \end{align}
By  \Cref{eq:minklower,eq:minkupper} and the fact that $\voronoi$ is a fundamental domain of $\Lambda$ with
volume $\det(\Lambda)$, we obtain
\[ (q-c)^n \vol(X) \leq |(\Lambda + \bt) \cap qX | \cdot \det(\Lambda) \leq (q+c)^n \vol(X) .\]
Dividing by $\det(\Lambda)$ and using the estimate $e^{-2nc/q} \leq (1-c/q)^n \leq (1+c/q)^n \leq e^{2nc/q}$ (note that $q > 2c$) we
arrive at the final claim.
\end{proof}

\subsection{Gaussian tails}
\begin{lemma}
\label{lemma:bound-gaussian}
Let $V$ be a real vector space of dimension $n$ and $s > 0$. For any $\eps \in (0,1]$, it holds that $\Pr_{\vec x \leftarrow \Gaussian_{V,s}}(\|\vec x\| \geq s \cdot \sqrt{2n \cdot \log(2n/\eps)}) \leq \eps$.
\end{lemma}

\begin{proof}
Let $\vec B$ be an orthonormal basis of $V$ and write $\vec x = (x_1, \cdots, x_n)$ the coordinates of $\vec x$ in this basis. Then the random variables $x_i$ are linearly independent Gaussian distributions over $\R$ with standard deviation $s$. Moreover, for any $t > 0$, if $\|\vec x \| \geq t$, there should exist some $i$ such that $|x_i| \geq t/\sqrt{n}$. Hence, we obtain
\[\Pr_{\vec x \leftarrow \Gaussian_{V,s}}\big(\|\vec x\| \geq t\big)   \leq n \cdot \Pr_{x\leftarrow \Gaussian_{\R,s}}\big(|x| \geq t/\sqrt{n}\big) \leq 2n \cdot \exp\Big(-\frac{t^2}{2n \cdot s^2}\Big),\]
where the first inequality comes from the union bound and the last one comes from Chernoff's bound. Taking $t = s \cdot \sqrt{2n \cdot \log(2n/\eps)}$ leads to the desired result.
\end{proof}

\subsection{Sizes of elements} \label{sec:appendix-sizes}
\begin{lemma}[Rules on sizes of elements] 
\hfill
\begin{enumerate}
 \item For $m_j \in \Z$, $\size( \prod_{j = 1}^k m_j) \leq \sum_{j=1}^k \size(m_j)$ and $\size( \sum_{j=1}^k m_j) \leq \log_2(k) +  \max_j(\size(m_j))$.
 \item For $q_i \in \Q$, $\size( \prod_{i =1}^k q_i) \leq \sum_{i = 1}^k \size(q_i)$ and $\size( \sum_{i=1}^k q_i) 
\leq 3 \sum_{j = 1}^k \size(q_j)$.
 \item For $\gamma_j \in K$, we have $\size(\sum_j \gamma_j) \leq 3\sum_j \size(\gamma_j)$. Additionally, 
 \[ \size(\prod_{j=1}^k \gamma_j) \leq  k \cdot 3d^2 \cdot  ( \lceil \log |\Delta_K| \rceil + \sum_{j=1}^k\size(\gamma_j)). \]
\item For fractional $\ZK$ ideals $\ma, \ma_i$ of $K$, we have $\size(\ma) \leq d^2 \size(N(\ma))$ and $\size(\prod_{i=1}^k \ma_i) \leq d^2 \sum_{i =1}^k \size(\ma_i)$.
\end{enumerate} 
\end{lemma}
\begin{proof} 
\begin{enumerate}
 \item We have $\size(mn) = 1 + \lceil \log_2(|m|) + \log_2(|n|) \rceil \leq \size(m) + \size(n)$. This generalizes to larger products.
 Assume without loss of generality that $m_1$ is the largest (in absolute value) among the $m_j$. Then we have 
 \[ \size(\sum_{j=1}^k m_j ) = 1 + \lceil \log_2 |\sum_{j=1}^k m_j| \rceil \leq 1 +  \lceil \log_2(k \cdot |m_1|) \rceil \leq \size(k) + \size(m_1).\]
 
 \item We have $\size( \frac{a}{b} \cdot \frac{c}{d}) = \size(ac) + \size(bd) \leq \size(a) + \size(c) + \size(b) + \size(d) \leq \size(\frac{a}{b}) + \size(\frac{c}{d})$, (by part (i)) which generalizes to larger products.
 
 Write $q_j = \frac{a_j}{b_j}$, and write 
 \[ q =  \sum_{j = 1}^k q_i = \sum_{j = 1}^k \frac{a_j}{b_j} = \frac{\sum_{j = 1}^k \left( a_j \prod_{t \neq j} b_t \right) }{\prod_{j=1}^k b_j} \]
 Then, by definition and by part (i), 
 \[ \size(q) \leq \size(\prod_{j=1}^k b_j) + \size(\sum_{j = 1}^k ( a_j \prod_{t \neq j} b_t ) ) \leq \sum_{j =1}^k \size(b_j) + \max_j \size(a_j \prod_{t \neq j} b_t) + \size(k) . \]
 \[ \leq 2 \sum_{j =1 }^k \size(q_j) + \size(k) \leq 3 \sum_{j =1 }^k \size(q_j) .  \]
 \item Note that the size of $\gamma = \sum_{j =1}^k \gamma_j$ is dictated by its rational coefficients in the $\ZK$-basis $(\beta_1,\ldots,\beta_d)$, which are just the rational coefficients of $\gamma_j$ added together. Hence, $\size(\gamma) \leq 3 \sum_{j=1}^k \size(\gamma_j)$ by part (ii).

 Writing $\gamma = \sum_{i=1}^d g_i \beta_i$ and $\delta = \sum_{i=1}^d d_i \beta_i$, we obtain 
 \[ \gamma \cdot \delta =  (\sum_i g_i \beta_i)(\sum_i d_i \beta_i) = \sum_{ij} g_i d_j \beta_i \beta_j = \sum_{k = 1}^d \left( \sum_{ij} g_i d_j [\beta_i \beta_j]_k  \right) \beta_k, \]
 where $[\beta_i \beta_j]_k \in \Z$ denotes the coefficient in $\Q$ of $\beta_i \beta_j$ in terms of the basis element $\beta_k$, i.e., $\beta_i \beta_j =\sum_{k=1}^d [\beta_i \beta_j]_k  \beta_k$.
 By the fact that $\beta_i \beta_j$ can be written in the $(\beta_1,\ldots,\beta_d)$-basis with integer coefficients bounded by $\sqrt{d} |\Delta_K|^{d+2}$, i.e., $|[\beta_i\beta_j]_k | \leq \sqrt{d} |\Delta_K|^{d+2} \leq |\Delta_K|^{3d}$ (this follows from the assumptions in the beginning of \Cref{sec:sizes}) for all $i,j,k$,  we see that, by part (ii), 
 \[  \size(\gamma \delta) = \sum_{k = 1}^d \size\left( \sum_{ij} g_i d_j [\beta_i \beta_j]_k  \right) \leq  3d \cdot  \sum_{i,j} (\size( g_i d_j) + 3d \lceil \log |\Delta_K| \rceil   ) \]
 \[ \leq 3d^3 \lceil \log |\Delta_K| \rceil  + 3d \sum_{i,j} (\size( g_i) + \size( d_j)) \leq 3d^3 ( \lceil \log |\Delta_K| \rceil + \size(\gamma) + \size(\delta))\]
 For larger products, by dividing the products in two in a binary fashion, we obtain, by induction,
 \[  \size(\prod_{j=1}^k \gamma_j) \leq  k \cdot 3d^2 \cdot  ( \lceil \log |\Delta_K| \rceil + \sum_{j=1}^k\size(\gamma_j)). \]
\item For any integral ideal $\mathfrak{a} \subseteq \ZK$, we have that $\size(\mathfrak{a}) \leq d^2 \size( N(\mathfrak{a}))$, since $\mathfrak{a}$ is represented by its HNF generating matrix (with entries in $\Z$), of which the product of the diagonal entries must equal $N(\mathfrak{a})$. Hence, by the HNF properties, each of the coefficients must be bounded in absolute value by $N(\mathfrak{a})$. For fractional ideals, the scaling of the generating matrix of $\mathfrak{a}$ can be chosen to be the denominator of $N(\mathfrak{a})$. Hence, also for fractional ideals $\mathfrak{a}$ holds that $\size(\mathfrak{a}) \leq d^2 \size( N(\mathfrak{a}))$. By the fact that the product of the diagonals equals the norm, we also have $\size(N(\mathfrak{a})) \leq \size(\mathfrak{a})$.

It follows then that $\size(\prod_{i=1}^k \mathfrak{a}_i ) \leq d^2 \sum_{i=1}^k \size(\mathfrak{a}_i)$.

\end{enumerate}

\end{proof}

\printbibliography[heading=bibintoc,title={References}]

\cleardoublepage
\section*{List of symbols}
\addcontentsline{toc}{section}{List of Symbols}
\markboth{List of Symbols}{List of Symbols}
 \renewcommand{\arraystretch}{1.2}

\begin{longtable}{{|l|p{12cm}|}}
\hline 
\textbf{Symbol} & \textbf{Description} \\
\hline 
\endfirsthead

\hline
\textbf{Symbol} & \textbf{Description} \\
\hline
\endhead

$\mathfrak{a},\mathfrak{b},\mathfrak{c}, \ldots$ & Ideals of the ring of integers $\ZK$ of a number field $K$ \\
\hline
$\adel_K$ & The ad\`eles of the number field $K$  \\
\hline

$B(t)$ & The ball of radius $t$ in $\SL_r(K_\R)$, with respect to the the distance notation $\rho$ (\cpageref{def:Bt}) \\
\hline
$B$ & Bound in the definition of the set of all prime ideals $\mathcal{P}(B)$ whose norm is bounded by $B$ (\cpageref{def:boundB}) \\
\hline
$\mB$ & Basis part in $(\mB,\mI)$, with $\mB \in K^{r\times r}$, a pseudo-bases of a module-lattice $M$, sometimes also just a basis in $\Q^{r\times r}$ \\
\hline
$d$ & Degree $d = [K:\Q]$ of the number field $K$ \\
\hline
$\Gaussian_{\cdot}$ & Either discrete Gaussian or continuous Gaussian distribution, depending on the subscript (\cpageref{def:gaussian}) \\
\hline
$H$ & The hyperplane where the logarithmic embedding of the units $\ZK^\times$ lives in (\cpageref{def:hyperplane}) \\
\hline
$\mI$ & Ideal part in $(\mB,\mI)$, with $\mI = (\ma_1,\ldots,\ma_r)$, where $(\mB,\mI)$ is a pseudo-basis of a rank $r$ module-lattice $M$ \\
\hline
$K$ & Number field of degree $d = [K:\Q]$ and discriminant $\Delta_K$ \\
\hline
$L$ & Generally, a lattice \\
\hline
$L^p(\cdot)$ & The space of $L^p$-integrable functions over the specified space \\
\hline
$M$ & A module (lattice) of rank $r$ over the field $K$ \\
\hline
$n$ & The dimension $n= d \cdot r$ of the module lattices occurring in this work over $\R$ \\
\hline
$N(\cdot)$ & The absolute norm of elements of ideals of the field $K$ \\
\hline
$O(\cdot),o(\cdot)$ & Landau's big-O and small-o notation \\
\hline
$\ZK$ & The ring of integers of $K$ \\
\hline
$\ZK^\times$ & The unit group of $K$ \\
\hline
$\fp$ & A prime ideal of $K$ \\
\hline
$\mathcal{P}, \mathcal{P}(B)$ & A set of prime ideals of $K$, generally $\mathcal{P} = \mathcal{P}(B)$, the set of all prime ideals with norm bounded by $B$ \\
\hline
$r$ & The rank as a module over $K$, of modules (module lattices) occurring in this work \\
\hline
$r_1$ & The number of real embeddings of $K$ \\
\hline
$r_2$ & The number of complex embeddings of $K$ \\
\hline
$\unitrk$ & The rank of the unit group $\ZK^\times$ \\
\hline
$\Round,\RoundPerf$ & The algorithm rounding a module lattice to a close rational module lattice (\cpageref{sec:rounding}) \\
\hline
$\size(\cdot)$ & The number of bits required to represent the algebraic object at hand (\cpageref{sec:sizes}) \\
\hline
$t$ & A parameter in the continuous randomization (or initial distribution) of the input module lattice of the random walk method of this paper (see also $B(t)$) (\cpageref{def:Bt}) \\
\hline
$T_\fp$ & The Hecke operator corresponding to uniform averaging over submodules $N\subset M$ such that  at $M/N\cong \ZK/\p$ (\cpageref{sec:Hecke-ops}) \\
\hline
$T_\mathcal{P},T_{\mathcal{P}(B)}$ & The Hecke operator corresponding to averaging over all $T_\fp$ with $\fp \in \mathcal{P}$ or $\mathcal{P}(B)$ (\cpageref{eq:def-T-mathcal-P}) \\
\hline
$X_r, X_r(K)$ & The space of similarity classes of modules lattices over $K$ of rank $r$ (\cpageref{eq:adelic-quot-conn-comps}) \\
\hline
$X_{r,\ma}$ & The component of the space of similarity classes of modules lattices over $K$ of rank $r$, dictated by the ideal class of $\mathfrak{a}$ (\cpageref{eq:xra}) \\
\hline
$Y_r$ & The space of invertible $r\times r$ matrices over $K$ up to rotation and scaling (\cpageref{eq:Yr}) \\
\hline

$\alpha$ & Balancedness parameter for a module lattice (\cpageref{def:alpha-bal}) \\
\hline
$\Gamma_K$ & The maximum of the quotient between the outermost successive minima $\lambda_n(I)/\lambda_1(I)$ over all ideal lattices $I$ of the number field $K$ (\cpageref{def:imbalance-gamma}) \\
\hline
$\Delta_K$ & The discriminant of the number field $K$ \\
\hline
$\varepsilon$ & A small error parameter in $[0,1]$, often indicating the failure probability of an algorithm \\
\hline
$\varepsilon_0$ & The closeness of the $\Round$-algorithm to the perfect distribution $\RoundPerf$ (\cpageref{prop:rounding-algo}) \\
\hline
$\lambda_j(\Lambda)$ & The $j$-th successive minimum of the lattice $\Lambda$ with respect to the $2$-norm (\cpageref{subsec:prelimminima}) \\
\hline
$\lambda_j^{(\infty)}(\Lambda)$ & The $j$-th successive minimum of the lattice $\Lambda$ with respect to the $\infty$-norm (\cpageref{subsec:prelimminima}) \\
\hline
$\lambda_j^{K}(M)$ & The $j$-th successive $K$-minimum of the module lattice $M$ with respect to the canonical norm (\cpageref{def:Kminima}) \\
\hline
$\mu$ & The Haar measure on $X_r$ (\cpageref{def:riemriem}) \\
\hline
$\mu_{\mathrm{cut}}$ & The `cut' Haar measure on $X_r$ (\cpageref{sec:conclusion}) \\
\hline
$\mu_{\mathrm{Riem}}$ & The Riemannian measure on $X_r$ (\cpageref{sec:riem-geom} and \cpageref{def:riemriem}) \\
\hline
$\rho_\sigma$ & The Gaussian function $x\mapsto e^{-\pi \|x\|^2/\sigma^2}$ (\cpageref{def:gaussian}) \\
\hline
$\sigma$ & The deviation for the Gaussian function or the (discrete) Gaussian distribution \\
\hline
$\initial_z$ & The initial distribution on $X_r$ by `folding' the distribution $f_z$ (\cpageref{def:initial-distro}) \\
\hline

\end{longtable}

\end{document}